\documentclass[12pt]{article}

\usepackage{amsmath,amssymb,amsfonts,amsthm}
\usepackage{graphics}
\usepackage{epsfig}
\usepackage{amsmath}
\usepackage{amsfonts}
\usepackage{float}
\usepackage{comment}
\allowdisplaybreaks
\font\elevensf=cmss10 scaled\magstephalf

\newtheorem{theorem} {{\elevensf THEOREM}}[section]

\newtheorem{lemma} {{\elevensf LEMMA}}[section]
\newtheorem{corollary} {{\elevensf COROLLARY}}[section]
\newtheorem{example} {{\elevensf EXAMPLE}}[section]

\newtheorem{remark} {{\elevensf REMARK}}[section]

\newtheorem{definition} {{\elevensf DEFINITION}}[section]

\newcommand{\bpr}{\begin{proof}}
\newcommand{\epr}{\end{proof}}
\newcommand{\lb}{\left(}
\newcommand{\rb}{\right)}
\newcommand{\A}{\chi_{{\omega}_{A^{\epsilon}}}(x)}
\newcommand{\AC}{\chi_{{\omega}_{A^{\epsilon}}}(x_1)}
\newcommand{\B}{\chi_{{\omega}_{B^{\epsilon}}}(x)}
\newcommand{\AB}{\chi_{{\omega}_{A^{\epsilon}} \cap {\omega}_{B^{\epsilon}}} (x)}

\renewcommand\qed{$\blacksquare$}

\def\CC{{\rm \kern.24em \vrule width.02em height1.4ex depth-.05ex \kern-.26emC}}

\def\TagOnRight

\def\AA{{it I} \hskip-3pt{\tt A}}

\def\QQ{\rlap {\raise 0.4ex \hbox{$\scriptscriptstyle |$}} {\hskip -0.1em Q}}

\catcode`\@=11

\def\theequation{\@arabic{\c@section}.\@arabic{\c@equation}}

\begin{document}

\baselineskip 14pt
\parindent.4in
\catcode`\@=11

\begin{center}

{\Huge \bf Convergence Relative to a Microstructure: Properties, Optimal
Bounds and Application} \\[5mm]

{\bf Tuhin GHOSH and Muthusamy VANNINATHAN }\\[4mm]

\textit{Centre for Applicable Mathematics, Tata Institute of Fundamental Research, India.}\\[2mm]

Email : \textit{tuhin@math.tifrbng.res.in\ ,\ vanni@math.tifrbng.res.in }

\end{center} 

\begin{abstract}
\noindent
In this work, we study a new notion involving convergence of microstructures
represented by matrices $B^\epsilon$ related to the classical $H$-convergence of
$A^\epsilon$. It incorporates the interaction between the two microstructures. This
work is about its effects on various aspects : existence,
examples, optimal bounds on emerging macro quantities, application etc. Five among
them are highlighted below : $(1)$ The usual arguments based on translated inequality, 
$H$-measures, Compensated Compactness etc for obtaining optimal bounds are
not enough. Additional compactness properties are needed. $(2)$ Assuming  two-phase microstructures, the bounds define naturally four optimal regions in the phase space of macro quantities. The classically known single region in the self-interacting case , namely $B^\epsilon= A^\epsilon$ can be recovered from them, a result that indicates we are dealing with a true extension of the $\mathcal{G}$-closure problem.
$(3)$ Optimality of the bounds is not immediate because of (a priori)
non-commutativity of macro-matrices, an issue not present in the
self-interacting case. Somewhat surprisingly though, commutativity follows
a posteriori. $(4)$ From the application to ``Optimal Oscillation-Dissipation Problems'', it emerges that oscillations and dissipation can co-exist
optimally and the microstructures behind them need not be the same though
they are closely linked. Furthermore, optimizers are found among $N$-rank
laminates with interfaces. This is a new feature. $(5)$ Explicit computations in the case of canonical microstructures are performed, in which we make use of $H$-measure in a novel way. 
\end{abstract}

\section{Introduction}
\setcounter{equation}{0}
In this work, we are concerned with a new notion of convergence related to the classical $H$-convergence of the homogenization theory.
We begin by recalling the notion of $H$-convergence \cite{T}. Let $\mathcal{M}(\alpha, \beta;\Omega)$ with
$0<\alpha<\beta$ denote the set of all real $N\times N$ symmetric matrices $A(x)$ of functions defined almost
everywhere on a bounded open subset $\Omega$ of $\mathbb{R}^N$ such that if $A(x)=[a_{kl}(x)]_{1\leq k,l\leq N}$ then 
\begin{equation*} a_{kl}(x)=a_{lk}(x)\ \forall l, k=1,..,N \ \mbox{and }\   (A(x)\xi,\xi)\geq \alpha|\xi|^2,\ |A(x)\xi|\leq \beta|\xi|,\  \forall\xi \in \mathbb{R}^N,\ \mbox{ a.e. }x\in\Omega.\end{equation*}
Let $A^\epsilon$ and $A^{*}$ belong to $\mathcal{M}(a_1,a_2, \Omega)$ with $0 < a_1 < a_2$. We say $A^\epsilon \xrightarrow{H} A^{*}$ or $H$-converges to 
a homogenized matrix $A^{*}$, if $A^\epsilon\nabla u^\epsilon \rightharpoonup A^{*}\nabla u$ in $L^2(\Omega)$ weak, for all test sequences $u^\epsilon$ satisfying 
\begin{equation}\begin{aligned}\label{ad13}
u^{\epsilon} &\rightharpoonup u \quad\mbox{weakly in }H^1(\Omega)\\
-div(A^\epsilon\mathbb\nabla u^\epsilon(x))& \mbox{ is strongly convergent in } H^{-1}(\Omega).
\end{aligned}\end{equation}
Convergence of the canonical energy densities follows as a consequence:
\begin{equation}\label{ad12} A^\epsilon\nabla u^\epsilon\cdot\nabla u^\epsilon  \rightharpoonup  A^{*}\nabla u\cdot\nabla u \mbox{ in } \mathcal{D}^\prime(\Omega). \end{equation}
Further, total energies over $\Omega$ converge if $u^\epsilon$ and $u$, for example, lie in $H^1_0(\Omega)$ :
\begin{equation}\label{ad15} \int_\Omega A^\epsilon\nabla u^\epsilon\cdot\nabla u^\epsilon dx  \rightarrow \int_\Omega A^{*}\nabla u\cdot\nabla u\ dx. \end{equation}

In this work, we are concerned with other oscillating quadratic energy densities; more precisely, let
us take another sequence of matrices $B^\epsilon$ and consider the corresponding energy density :
\begin{equation}\label{bs14}
B^\epsilon \nabla u^\epsilon\cdot \nabla u^\epsilon
\end{equation}
and study its behaviour as $\epsilon$ tends to zero. Just like \eqref{ad12}, 
we may expect the appearance of new macro quantities in its weak limit. This motivates
the following notion:
\begin{definition}\label{sid}
Let $A^\epsilon$ $H$-converge to $A^{*}$. Let $B^\epsilon$ 
be given in $\mathcal{M}(b_1,b_2;\Omega)$ where $0 < b_1 < b_2$. 
We say $B^\epsilon$  converges to the matrix of functions $B^{\#}(x)$ relative to $A^\epsilon$ 
(denoted $B^\epsilon\xrightarrow{A^\epsilon}B^{\#}$) if for all test sequences $u^\epsilon$ satisfying \eqref{ad13} we have 
\begin{equation}\label{ad17} B^\epsilon\nabla u^\epsilon\cdot\nabla u^\epsilon  \rightharpoonup  B^{\#}\nabla u\cdot\nabla u \mbox{ in } \mathcal{D}^\prime(\Omega). \end{equation}
\end{definition}
\noindent
Further, analogous to \eqref{ad15}, we have as a consequence 
\begin{equation*} \int_\Omega B^\epsilon\nabla u^\epsilon\cdot\nabla u^\epsilon dx  \rightarrow \int_\Omega B^{\#}\nabla u\cdot\nabla u dx,
\end{equation*}
if $ u^\epsilon$ and $u$ belong to $H^1_0(\Omega)$ (see Remark \ref{ub19}). 
We will also show $B^{\#}\in\mathcal{M}(b_1,\widetilde{b_2};\Omega)$ with $\widetilde{b_2}=b_2\frac{a_2}{a_1}\geq b_2$ (see \eqref{Sd3}).\\
\\
The significance of the limit $B^{\#}$ in the context of Calculus of Variations will be clear in the application in Section \ref{qw10}.
In some sense, $B^{\#}$ can be interpreted as macro coefficients associated with the optimal dissipation of energy of oscillations associated with $A^{*}$.

The main characteristic feature behind $B^{\#}$ is the interaction between microstructures of $A^\epsilon$ and $B^\epsilon$
through the triple product of oscillating quantities \eqref{bs14}. The case $A^\epsilon=B^\epsilon$is classical and is
referred to as self- interacting case. In this case, comparing \eqref{ad12} with \eqref{ad17}, we have  $B^{\#}=A^{*}$ and thus the new notion of relative convergence extends $H$-convergence. 
We may say that $H$-convergence is based on both flux-energy convergence whereas the new notion is energy-based. 
(In general the weak limit of $B^\epsilon\nabla u^\epsilon$ is not equal to $B^{\#}\nabla u$). It is 
worthwhile to note that $B^{\#}$ is nontrivial and interesting even if $B^\epsilon$ is 
independent of $\epsilon$. On the other hand, if $A^\epsilon$ is independent of $\epsilon$
then $B^{\#}$ is somewhat trivial because it coincides with the weak limit $\overline{B}$
of $B^\epsilon$. Since $B^{*}$ ($H$-limit of $B^\epsilon$) is defined using solely the
canonical oscillating test functions of $B^\epsilon$ (not involving those of $A^\epsilon$) $B^{\#}$ is in general different
from $B^{*}$. In fact, we give an expression of $B^{\#}$ using the oscillating test 
functions associated to $A^\epsilon$ in the next section (see Corollary \ref{Sd14}) 
which will show that $B^{\#}$ and $B^{*}$ need not be the same. 

There is another type of interaction between $A^\epsilon$ and $B^\epsilon$.
Indeed, it is possible to associate with $B^\epsilon$ macro quantities $B^{\#}_{FL}$
describing effective behaviour of fluxes relative to the given microstructure $A^\epsilon$; 
More precisely, $B^{\#}_{FL}$ possesses the following property : 
\begin{equation*}B^\epsilon\nabla u^\epsilon \rightharpoonup B^{\#}_{FL}\nabla u\ \mbox{ in } L^2(\Omega) \mbox{ weak, } \end{equation*}
for all test sequences $u^\epsilon$ satisfying \eqref{ad13}.
Existence of such macro coefficients is obtained using the
correctors associated with $A^\epsilon$ (cf. Remark \ref{ad16}). 
It is immediately seen that $B^{\#}_{FL}$ coincides with $A^{*}$ if 
$B^\epsilon = A^\epsilon$ and on the other hand if $B^\epsilon=B$ is independent of $\epsilon$, 
then $B^{\#}_{FL} = B$. Unfortunately, the above matrix $B^{\#}_{FL}$ is not capable of describing
the macro behaviour of energy density $B^\epsilon \nabla u^\epsilon \cdot \nabla  u^\epsilon$; 
that is, we cannot say that the weak limit of $B^\epsilon \nabla u^\epsilon\cdot \nabla u^\epsilon$ 
is  given by $B^{\#}_{FL} \nabla u\cdot \nabla u$. Since the applications that we envisage in Calculus of Variations involve energy
densities in their objective functionals, $B^{\#}_{FL}$  does not seem to play any role in the examples
of Section \ref{qw5}. In contrast, $B^{\#}$ will play a role, as we shall see below. For more elaborate 
study of $B^{\#}_{FL}$, see \cite{TG-MV}. In general, $B^{\#}$ is different from $B^{\#}_{FL}$, even for
simple microstructures as shown below (see Remark \ref{ot10}). This can also be seen in more general terms 
: $B^{\#}_{FL}$ is born out of interaction between $B^\epsilon$ and the canonical oscillations associated with
$A^\epsilon$ (namely, the correctors) in a linear fashion whereas $B^{\#}$ is an outcome of quadratic interaction
with the correctors of $A^\epsilon$. This entire article is concerned with the effects of this later interaction.

There are advantages in separating the macro behaviour of flux density and energy density which are ``mixed'' in $H$-convergence. For one thing, we can take $B^\epsilon$ more general and this has consequences in Calculus of Variations as shown in Section \ref{qw5}.  

Our first result is a compactness result proving the existence of $B^{\#}$. (See Theorem \ref{bs8}) The
main difference with $H$-limit is that the oscillations in $\nabla u^\epsilon$ are
restricted by conditions \eqref{ad13} in which $A^\epsilon$ appears and $B^\epsilon$ does not. 
The proof introduces adjoint state associated with the constraint \eqref{ad13}. Repeated application
of div-curl lemma then enables us to prove the nice behaviour of the energy density \eqref{ad17}. This
result implicit in \cite{KP,KV} is given a quick proof by different arguments. The point to note is that
even if the entire sequence $A^\epsilon$ $H$-converges to $A^{*}$ and $B^\epsilon = B$ is independent of $\epsilon$, 
there may be several limit points $B^{\#}$ (see \eqref{FL19}). All such limit points lie on the fibre over $A^{*}$, an object defined in \eqref{kfab}.

Basic objectives in the theory of Homogenization include not only to show the
existence of the homogenized limit $A^{*}$ but also establish optimal bounds on
it independent of microgeometric details along with application to optimal design problem (ODP). This has 
been completely accomplished in the case of two-phase composites \cite{A,GB,LC,ML1,T}.
We intend to develop a similar program for the new object $B^{\#}$. Our main assumption
in this task will be that $A^\epsilon$ and $B^\epsilon$ have two phases with prescribed volume proportions :
\begin{equation*} A^{\epsilon}(x) = a^{\epsilon}(x)I\ \mbox{ where, }\ a^{\epsilon}(x)= a_1\A +a_2 (1-\A)\ \mbox{ a.e. in }\ x\in\Omega.\end{equation*}
$\A$ defines the microstructures in the medium, with $\A\rightharpoonup \theta_A(x)$ in $L^{\infty}$ weak* topology, $\theta_A(x)$ satisfying 
$0\leq \theta_A(x)\leq 1$ a.e. $x\in\Omega$ defines the local proportion for the component $a_1 I$ in $\Omega$ (or, $(1-\theta_A(x))$ for the component $a_2 I$). 
\begin{equation*} B^{\epsilon}(x)= b^{\epsilon}(x)I\ \mbox{ where }\ b^{\epsilon}(x)= b_1\B + b_2 (1-\B)\ \mbox{ a.e. in }\ x\in\Omega.\end{equation*}
and we assume that $\B\rightharpoonup \theta_B(x)$ in $L^{\infty}$ weak* topology, $\theta_B(x)$ satisfying $0\leq \theta_B(x)\leq 1$ a.e. $x\in\Omega$
defines the local proportion for the component $b_1 I$ in $\Omega$ (or, $(1-\theta_B(x))$ for the component $b_2 I$).

Before describing our results, let us cite few situations where interactions of similar kind arise and the above notion and our results may be of interest. 

From the works of Busse, Howard, Doering and Constantin cited in \cite{BUS,DCR}, we see that there is a general interest of averaging
as many observables as possible in complex/oscillating systems and deriving bounds on the emerging macro parameters independent of the
details of the complexities. Some of these observables have the triple product structure as in \eqref{bs14}. In the present context, the
model complex system is defined by $A^\epsilon$ via the conditions \eqref{ad13}. Thanks to the theory of $H$-convergence, we could study the 
behaviour of the canonical energy density \eqref{ad12} which corresponds to self- interacting case. The $H$-limit $A^{*}$ describes its behaviour. 
Estimates on $A^{*}$ constitute one of the celebrated results in the theory (see the references \cite{A,GB,LC,ML1,T}). Within the linear framework, 
other basic quantities of interest are quadratic densitites of the form \eqref{bs14} for the same system \eqref{ad13}. Since $(A^{*},B^{\#})$ are the macro quantities in their weak limits, it is natural to seek optimal bounds on $(A^{*},B^{\#})$ independent of microstructural details. Such bounds are stated in 
Section \ref{ad18} and this article is devoted to their proof and applications. 

Secondly, let us cite the example of an Optimal Control Problem (OCP) which is a linear-quadratic regulator problem in which both the state
equation and the cost functional involves finely oscillating coefficients $A^\epsilon$ and $B^\epsilon$ respectively \cite{KP,KV}. In such a
case, there is an interaction between the cost functional and the state equation through their microstructures. The limit homogenized problem 
is again an OCP of similar type with $(A^{*},B^{\#})$ replacing $(A^\epsilon,B^\epsilon)$ respectively. This shows that $B^{\#}$ is an inevitable
macro quantity appearing in such fundamental optimization processes. 

Next, let us mention the question of distinguishing two microstructures admitting the same homogenized matrix $A^{*}$.
(Such questions could be of interest in Inverse Problems). Since $B^{\#}$ incorporates some features of the microstructures
underlying $A^\epsilon$ through the interaction of $A^\epsilon$ and $B^\epsilon$, the idea is to exploit it in the process of
distinction. Thus, let us consider two sequences $A^\epsilon$ and $A^\epsilon_1$ admitting the same $H$-limit $A^{*}$. Further,
we assume both $A^\epsilon$ and $A^\epsilon_1$ have same two-phases $(a_1,a_2)$ with the same local volume proportions $(\theta_A(x)
,(1-\theta_A(x))$. Their microgeometries need not however be the same. The question of interest is the following : Can one distinguish
the microgeometries of $A^\epsilon$ and $A^\epsilon_1$ by means of the macro quantity $B^{\#}$ for some suitable choice of $B^\epsilon$ ?

Let $B^\epsilon = B$ be independent of $\epsilon$, that is, $B^\epsilon$ admits no microstructure.
Suppose  $B\xrightarrow{A^\epsilon} B^{\#}$ and $B\xrightarrow{A^\epsilon_1} B^{\#}_1$ relative to 
$A^\epsilon$ and $A^\epsilon_1$ respectively. Then it can be shown that $B^{\#}
= B^{\#}_1$. Indeed, in this case, $B^{\#}$ can be expressed as a (complicated) function of ${A^{*},B,
\theta_A,a_1,a_2}$ (see Remark \ref{ad19}). Thus distinction is not possible in this case. The next natural
step is to take $B^\epsilon$ with two-phase microstructure as a diagnostic tool to distinguish between the 
two microgeometries of $A^\epsilon$ and $A^\epsilon_1$. In this case, we can indeed choose a microstructure 
$B^\epsilon$ such that $B^{\#}$ and $B^{\#}_1$ are different (see Remark \ref{Sd12}). 

As the fourth situation, let us mention one novel application to Calculus of Variations 
which is carried out in this paper (see Section \ref{qw5}). Our bounds in Section \ref{ad18} 
will be used in full force.  Application of $H$-Convergence to Calculus of Variations is classical 
and relaxed solutions are described in terms of $H$-limits $A^{*}$ \cite{A,CHA,MT1,GB,LC,ML1}.
The objective functional in this class of problems are defined on single characteristic function
$\chi_{\omega_A}$ and depends only on the state but not on its gradient.  The usefulness of the 
$H$-convergence in such problems is well-known; it provides a method to obtain relaxed version of
the problem \cite{A,CHA,MT1}. Indeed, homogenization limits and optimal bounds on them describe all relaxed microstructures
along with the constraints on them. Minimizers are found among them and further one can write down the optimality
conditions at minimizers. For various methods of relaxation, one may refer to \cite{KS1,KS2,KS3}. 

Our objective functional which is minimized involves the gradient of the state with variable coefficients. 
More precisely, it is
\begin{equation*}\int_{\Omega}B(x)\nabla u\cdot \nabla u\ dx\ \ \ (cf.\eqref{ub11}). \end{equation*}
The case $B(x) = I$ is treated in the literature. Its relaxation will involve the functional  
$\int_\Omega I^{\#}\nabla u\cdot \nabla u$ and optimal bounds on it.  See \cite{GRB,LV}; One of our objectives in this
work is to derive stronger results which yield optimal point-wise bounds on such energy densities
(cf.\eqref{bs1},\eqref{Sd11}) and we get thereby the relaxed objective functional (see Section \ref{qw5}).
Earlier work was done without formally introducing $I^{\#}$; introduction of $I^{\#}$ simplifies somewhat the relaxation process. 

We now consider the case $B(x)\neq I$; More precisely we assume $B(x)$ has $2$-phases given by 
$B(x)= b_1\chi_{\omega_B}(x) +b_2(1-\chi_{\omega_B}(x))$. Now the above functional depends on the
pair of characteristic functions $(\chi_{\omega_A},\chi_{\omega_B})$ and we minimize over them. Such
problems are called Optimal Oscillation-Dissipation Problems because oscillations are present in the gradient fields of the state and we try to minimize/dissipate their energy although in a non-uniform way via $B(x)$. Interacting microstructures naturally appear in the minimization process.
Following interpretation may be given to this type of problems : when $B(x) = I$, that is when $\chi_{\omega_B}
\equiv 1$, one seeks a generalized microstructure on which
the energy of the oscillation of gradient of the state is suppressed/dissipated to the extent possible within the 
constraints. Since $B(x)=I$, this dissipation occurs uniformly in the space. Anticipating non-uniformity in the
energy dissipation in heterogeneous media, it is more natural to admit a variable matrix $B(x)$ : roughly, larger
value of $B(x)$ means smaller energy of the oscillations in the gradient. Furthermore, by minimizing with respect
to $(\chi_{\omega_A},\chi_{\omega_B})$, we allow the system to choose its own ``conductivity matrix'' $A^{*}$ and
an associated inverse ``dissipative matrix''$B^{\#}$.  Basic difficulty in such Optimal Oscillation-Dissipation
Problems is as follows : To minimize the objective functional, ideally speaking, we would like to have minimal
value of $B$ in regions of large gradient of the state. The difficulty is that such regions can be large and 
are not known a priori. Moreover, the minimal value material in $B$ can have small volume. Such issues make
the problem somewhat nontrivial. As usual, one does not expect ``classical solution'' to this problem and
we need to relax it. The new notion of relative limit $B^{\#}$ and the bounds on it stated in Section \ref{ad18} 
are needed to formulate the relaxed version. 

Minimizers for this problem would be pairs $(A^{*},B^{\#})$ in which $A^{*}$ captures oscillatory effects and $B^{\#}$ describes 
dissipation effects. Both these effects co-exist optimally in the problem. 
Quite remarkably, minimizers are found among $N$-rank laminates across which
the core-matrix values get switched. This seems to be a new phenomenon when
compared with the fact that minimizers for the classical ODP usually do
not have such interfaces. It may be mentioned that there are other problems of the type min-max for which our results apply. See Remark \ref{bs6}.

As the last situation, we mention the following extension of the classical \textit{G-closure} problem for conductivities : finding all possible conductivities at a given location in the physical domain when two conductors are mixed
in a given volume proportion $(\theta_A,1-\theta_A)$. The set is classically denoted as $\mathcal{G}_{\theta_A}$.  Its extension incorporating  interaction between $A^\epsilon$ and $B^\epsilon$ is as follows : Given $\{a_1,a_2,\theta_A\}$, $\{b_1,b_2,\theta_B\}$, find optimal bounds on $(A^{*},B^{\#})$ independent of microgeometric details of $A^\epsilon$ and $B^\epsilon$. They can depend on $A^{*}$, but otherwise independent of microgeometric details of $A^\epsilon$. When $B^\epsilon = A^\epsilon$, we know that $B^{\#} = A^{*}$ and so the later problem obviously extends the first one. In this work, we solve the extended problem, namely, we estimate $(A^{*},B^{\#})$ independent of
microstructures in terms of $\{a_1,a_2, \theta_A\}$, $\{b_1,b_2,\theta_B\}$.
As an easy example, we can choose $B^\epsilon$ to be any function of $A^\epsilon$ without altering its microstructure. In this context, we cite \cite{AM} in which a different problem involving $A^\epsilon,B^\epsilon$ with the same underlying microstructure is considered without introducing $B^{\#}$. Obviously, more challenging case is to deal with the change of microstructures; namely that $A^\epsilon$ and $B^\epsilon$ have different microstructures. While this case is treated completely here, the more difficult problem of deriving optimal bounds on $B^{\#}$ in terms of $(a_1,a_2,\theta_A),(b_1,b_2,\theta_B)$, 
independent of microstructures and without involving $A^{*}$ is left open. 

Next, we highlight some of main points of our work. 

First point is concerned with the proof of optimal bounds on $(A^{*},B^{\#})$ assuming
that $B^\epsilon$ and $A^\epsilon$ have two-phases. This is an interesting mathematical
challenge because it involves the study of interaction of two microstructures $A^\epsilon$ and $B^\epsilon$. We need to choose a suitable method capable of handling interacting microstructures. Of various methods available to obtain optimal bounds on $A^{*}$, we found that the method of using the combination of translated inequality and
Compensated Compactness Theory along with $H$-measure \cite{T,T1} can be extended to include interaction between microstructures and establish the required bounds on $(A^{*},B^{\#})$.

Two key elements of the method are important : one is the choice of a
suitable macro field \eqref{abc1} and the associated oscillating gradient field
satisfying the basic compactness condition \eqref{ad13}. This will be used to test the appropriate translated inequality. The self-interacting case suggests the use of the same  oscillating field for both $A^{*}$ and $(A^{*},B^{\#})$.

The second element is the introduction of a suitable oscillating system
with differential constraints to bound tightly the so-called $H$-measure
term. Due to interaction of microstructures, this step is complicated and
we need additional compactness \eqref{eiz} to carry it out; mere \eqref{ad13} is not enough. The fundamental problem is that generally compactness suppresses
oscillation. Therefore the question arises whether the oscillating field
satisfying all the previous requirements has also the additional
compactness. This is a delicate point. While the oscillation around \eqref{abc1}
exists along with the compactness \eqref{ad13} for all $A^{*}$, the additional
compactness holds only for some special $A^{*}$. In fact our results show that
for $A^{*}$ in the interior of the phase region, such additional compactness
does not hold. Fortunately, for $A^{*}$ on the boundary of the phase region, it
does hold and it is a consequence of optimality nature of such $A^{*}$. Thus
both additional compactness and the usual oscillations co-exist for $A^{*}$
belonging to boundary. Using this property of the system,we are able to
deduce optimal lower bound on $(A^{*},B^{\#})$ if  $A^{*}$ lies on boundary of the region. Exploiting the phase space structure of $A^{*}$, same estimates on
$(A^{*},B^{\#})$ are later extended  if $A^{*}$ lies in the interior by other means.

The second point is about the bounds themselves. Estimates on $(A^{*},B^{\#})$ keeping $\{a_1,a_2,\theta_A\}$, $\{b_1,b_2,\theta_B\}$ fixed are quantified
in the form of inequalities in Section \ref{ad18}. Naturally, the analysis divides 
the physical domain $\Omega$ into four sub-domains denoted as $\Omega_{(Li,Uj)}$ $i,j=1,2$ (see \eqref{FL1}) with interfaces separating them depending on local volume proportions $\{\theta_A,\theta_B\}$. Figures showing these regions are included in Section \ref{ad18}. These estimates seem complicated. We arrived at them via a two-pronged strategy : on one hand, computing $B^{\#}$
for canonical microstructures of $A^\epsilon$ and on the other, trying to produce a rigorous proof. In our procedure of proving optimal estimates, we go though several intermediate easy cases before taking up the general case pointing out
the difficulties. The estimates obtained in the four cases are labeled as $L1$, $L2$, $U1$, $U2$. They define four optimal regions denoted as $(Li,Uj)$ in the phase space of macro parameters $(A^{*},B^{\#})$, whose union is denoted as 
$\mathcal{G}_{(\theta_A,\theta_B)}$. This constitutes an important structural change in the product space $(A^{*},B^{\#})$ when compared with the single known region $\mathcal{G}_{\theta_A}$ in the $A^{*}$-space. It is instructive to imagine $\mathcal{G}_{(\theta_A,\theta_B)}$ consisting of pairs $(A^{*},B^{\#})$ in which $B^{\#}$ is in a fibre sitting over each $A^{*}\in\mathcal{G}_{\theta_A}$. While each fibre is convex (see Remark \ref{qw1}), it is not clear what kind of other convexity properties each region in  $\mathcal{G}_{(\theta_A,\theta_B)}$ enjoys.  The link between the regions $\mathcal{G}_{(\theta_A,\theta_B)}$ and  $\mathcal{G}_{\theta_A}$ is interesting. In the self-interacting case, namely if $A^\epsilon = B^\epsilon$, the bounds L1 and L2 coincide with the well-known lower and upper bounds on $A^{*}$ respectively, thus we recover $\mathcal{G}_{\theta_A}$ from them. On the other hand, the bounds U1, U2 on $A^{*}$ do not seem to have any special significance in the classical phase space. 
(For details, see Remark \ref{eiu} and Figure 2). Finally, while relaxed solutions for Optimal Design Problems (ODP) are found in $\mathcal{G}_{\theta_A}$, we need the regions in $(A^{*},B^{\#})$ -space to capture relaxed solutions for Optimal Oscillation-Dissipation Problems (OODP). It is worthwhile to point out that relaxed solutions are found among $N$-rank laminates with an interface across which the core-matrix values are interchanged. This appears to be a new feature. 

As the third point, let us recall  that explicit computation of $A^{*}$ on various canonical microstructures such as simple and $N$-rank laminates, Hashin-Shtrikman coated balls are found in the literature \cite{T}. These are useful in proving the optimality of bounds on $A^{*}$. New computations for $B^{\#}$ associated to the above $A^{*}$ are made in our work in Section \ref{ub12} and in Section \ref{ts}. They reveal the structure of $B^{\#}$ which is used in establishing that various bounds in Section \ref{ad18} are indeed saturated. They also help us find relaxed solutions of OODP. As an example, let us mention laminates.  Assuming that $A^{\epsilon}$ is governed with $p-$sequential laminate microstructures, and that $\theta_A\leq \theta_B$, we construct the relative limits $B^{\#}_p$ and $B^{\#}_{p,p}$ associated to the $p$-laminate $A^{*}_p$ with the same lamination directions $\{e_i\}_{1\leq i\leq p}$ and with the same proportions $\{m_i\}_{1\leq i\leq p}$ as in $A^{*}_p$ but  with  $\omega_{A^\epsilon}\subseteq \omega_{B^\epsilon}$ (see \eqref{bs16}). To carry this out, we first establish a general relation (eg. \eqref{ad3}) between $A^{*}$ and $B^{\#}$ using $H$-measure techniques and we exploit it subsequently to compute the relative limits $B^{\#}_p$ and $B^{\#}_{p,p}$ corresponding to $p$-laminate $A^{*}_p$.

The final point stated in Section \ref{ad18} and proved in Section \ref{qw4} is about optimality of the regions $(Li,Uj)$ in the sub-domain $int(\Omega_{(Li,Uj)})$, $i,j=1,2$. This means that given $(A^{*},B^{\#})$ lying in a region $(Li,Uj)$, they can be realized as $H$-limit and a relative limit for suitable sequences $(A^\epsilon$, $B^\epsilon)$ with two phases and with their local volume proportions satisfying appropriate inequalities. This task has been accomplished in the case of $A^{*}$ in the literature (see \cite{A} for instance). Because of the presence of two matrices $(A^{*},B^{\#})$, non-commutativity is an extra difficulty which is not present in the self-interacting case. This is tackled with what we call Optimality-Commutativity Lemma (Lemma \ref{zz14}) which exploits the specific structure of our bounds. Apart from this, we need one more new element and that is fibre-wise convexity of the region (Li,Uj) in the phase space (see Remark \ref{qw1}). Modulo these new elements, the proof of optimality follows an established procedure. It has three parts. In the first, we work with macroscopically homogeneous cases. In the second part, we treat the general case by using piecewise constant approximation of function. 
We exploit here the fibre-wise convexity of the regions (Li,Uj). Third part proves optimality of all four regions taken together in $\Omega$. Somewhat strangely, another consequence of Optimality-Commutativity Lemma is that the $H$-limit $A^{*}$ and the relative limit $B^{\#}$ lying on the fibre over $A^{*}$ commute with each other in $\Omega$, i.e. $A^{*}(x)B^{\#}(x)=B^{\#}(x)A^{*}(x)$. This commutativity property is not a priori clear to start with ; our arguments need the optimal bounds to show it. An easy consequence of commutativity is that the optimal bounds stated in Section \ref{ad18}, can be formulated in terms of eigenvalues of $A^{*}$ and $B^{\#}$ (see \eqref{eg6},\eqref{eg7},\eqref{eg10},\eqref{eg5}).  

It is well-known \cite[Page No. 690]{ML} that the \textit{G-closure problem} for elasticity is open due to the fact that laminated microstructures are insufficient to generate all possible homogenized materials and that this difficulty does not exist for $H$-limits $A^{*}$ of the conductivity problem with two phases. In this context, let us state that the \textit{G-closure} type difficulty does not exist for relative limits $B^{\#}$ either. 
   
Let us now describe briefly the plan of this paper. After proving the existence of $B^{\#}$ in the next section, we present also its straight-forward properties. We take up the one-dimensional
case in some detail in Section \ref{hsf}. Optimal estimates involving the pair $(A^{*},B^{\#})$ which
constitute the main results of this work are stated in Section \ref{ad18}. There are four cases, which
naturally arise in the analysis. Proof of these bounds is given in Section \ref{ts}. In the same section, we treat separately the easy case when $B^\epsilon = B$ independent of $\epsilon$. Here, we distinguish two types of arguments: one to bound the energy density 
associated with $B^{\#}$, namely $B^{\#}\nabla u\cdot\nabla u$ and another to estimate
the matrix $B^{\#}$ itself via its quadratic form $B^{\#}\eta\cdot\eta$. Computation
of new macro coefficients $B^{\#}$ on classical microstructures associated with
$A^\epsilon$ can be found in Section \ref{ub12} and Section \ref{ts}. There is a separate discussion in Section \ref{qw4} about the optimality of the regions $(Li,Uj)$ $i,j=1,2$ defined by the bounds L1,L2,U1,U2. 
Final Section \ref{qw5} is devoted to two applications to problems of Calculus of Variations.\\
\\
For other comments on the contents of this paper, the reader is refer to various sections below. 

\tableofcontents

\section{Existence of $B^{\#}$ and its properties}\label{Sd13}
\setcounter{equation}{0}
\subsection{Existence of $B^{\#}$}\label{bs8}
\begin{theorem}\label{ot1}
 Given $B^\epsilon\in\mathcal{M}(b_1,b_2;\Omega)$, there exist a subsequence and $B^{\#}\in\mathcal{M}(b_1,\widetilde{b_2};\Omega)$ with $\widetilde{b_2}=b_2\frac{a_2}{a_1}$ such that 
 \begin{equation}\label{dc1} B^\epsilon\nabla u^\epsilon\cdot\nabla u^\epsilon \rightharpoonup B^{\#}\nabla u\cdot\nabla u \ \mbox{ in }\mathcal{D}^\prime(\Omega).\end{equation}
\end{theorem}
\begin{proof}
\textbf{Step(1): Definition of $B^{\#}$:}
Let $\{e_k\}_k,k=1,2,..,N$ be the standard basis vectors in $\mathbb{R}^N$.  
We define the oscillatory test functions $\chi_{k}^{\epsilon}$, $\zeta_{k}^{\epsilon}$ and 
$\psi_{k}^{\epsilon}$ in $H^1(\Omega)$ to define $A^{*}$, $B^{*}$, $B^{\#}$ 
as follows. Let $A^\epsilon$ $H$- converges to $A^{*}$, then upto a subsequence still denoted by $A^\epsilon$, there exists 
a sequence $\{\chi_k^\epsilon\}\in H^1(\Omega)$ such that
\begin{equation*}\begin{aligned}
\chi_{k}^{\epsilon} &\rightharpoonup 0 \mbox{ weakly in }H^1(\Omega),\\ 
 A^{\epsilon}(\nabla\chi_{k}^{\epsilon} + e_k ) &\rightharpoonup A^{*}{e_k} \mbox{ weakly in }L^2(\Omega) \mbox{ with }\\
 -div (A^{\epsilon}(\nabla\chi_{k}^{\epsilon} + e_k )) &= -div(A^{*}{e_k})\ \mbox{ in }\Omega, \ \ k=1,2,..,N.
\end{aligned}\end{equation*}
We consider the matrix $X^{\epsilon}$ defined by its columns $(\nabla\chi_{k}^{\epsilon})$
is called the corrector matrix for $A^{\epsilon}$, with the following property :
\begin{equation}\label{ot9}\nabla u^{\epsilon} - (X^{\epsilon}+I)\nabla u \rightarrow 0 \quad\mbox{ in }L_{loc}^1(\Omega).\end{equation}
The existence of such sequence $\{\chi_{k}^{\epsilon}\}$ is well known in homogenization theory, for more details one may look at \cite{A,T}.\\
\\
Similarly, let $B^\epsilon$ $H$-converges to $B^{*}$, then upto a subsequence still denoted by $B^\epsilon$, we define the corrector matrix $Y^{\epsilon}$ defined by its columns 
$(\nabla\zeta_{k}^{\epsilon})$ satisfying
\begin{equation*}\begin{aligned}
\zeta_{k}^{\epsilon} &\rightharpoonup 0 \mbox{ weakly in }H^1(\Omega),\\
 B^{\epsilon}(\nabla\zeta_{k}^{\epsilon} + e_k ) &\rightharpoonup B^{*}e_k \mbox{ weakly in }L^2(\Omega) \mbox{ with }\\
-div(B^{\epsilon}(\nabla\zeta_{k}^{\epsilon} + e_k )) &= -div(B^{*}e_k)\ \mbox{ in }\Omega, \ \ k=1,2,..,N.
\end{aligned}\end{equation*}
And finally we define the test functions $\psi_{k}^{\epsilon}$ bounded uniformly with respect to $\epsilon$ in $H^1(\Omega)$, satisfying
\begin{equation*} div(A^{\epsilon} \nabla \psi_{k}^{\epsilon}  - B^{\epsilon}(\nabla\chi_{k}^{\epsilon} + e_k)) =0\ \mbox{ in }\Omega, \ \ k=1,2,..,N.\end{equation*}
Such test functions $\{\psi_{k}^{\epsilon}\}$ have been introduced in \cite{KP,KV} subject to an optimal control problem.
Then upto a subsequence we consider the limit as 
\begin{equation*}\begin{aligned}
\psi_{k}^{\epsilon} &\rightharpoonup \psi_{k} \mbox{ weakly in }H^1(\Omega),\\
A^{\epsilon}\nabla \psi_{k}^{\epsilon}  - B^{\epsilon}(\nabla\chi_{k}^{\epsilon} + e_k) &\rightharpoonup\ \varsigma_k \mbox{\ (say) weakly in }L^2(\Omega)\mbox{ with }\\
div(\varsigma_k) &=0\ \mbox{ in }\Omega, \ \ k=1,2,..,N.
\end{aligned}\end{equation*}
Next, we define the limiting matrix $B^{\#}$ : For each $k=1,2,..,N$ 
\begin{equation}\label{Sd2}
 B^{\#}e_k :=  A^{*}\nabla \psi_{k} - \varsigma_k = A^{*}\nabla \psi_{k}- lim \{ A^{\epsilon}\nabla \psi_{k}^{\epsilon}  - B^{\epsilon}(\nabla\chi_{k}^{\epsilon} + e_k)\}; 
\end{equation}
and as a perturbation of $H$-limit $B^{*}$ we write
\begin{equation*}
B^{\#}e_k = B^{*}e_k +  A^{*}\nabla \psi_k - lim \{ A^{\epsilon} \nabla \psi_k^\epsilon - B^{\epsilon}(\nabla\chi_k^{\epsilon} - \nabla\zeta_k^{\epsilon})\}. 
\end{equation*}
\ \ \ \ (The above limits are to be understood as  $L^2(\Omega)$ weak limit).\\
\\
\textbf{Step(2):}
We introduce $p^\epsilon\in H^1(\Omega)$ such that 
\begin{equation}\begin{aligned}\label{FG1}
\nabla p^\epsilon \rightharpoonup \nabla p \mbox{ weakly in } L^2(\Omega),&\\
 div(A^{\epsilon}\nabla p^{\epsilon} - B^{\epsilon}\nabla u^{\epsilon}) =\ 0 \mbox{ in }\Omega.&
\end{aligned}\end{equation}
We introduce the new flux $z^{\epsilon} = A^{\epsilon}\nabla p^{\epsilon} - B^{\epsilon}\nabla u^{\epsilon}$, 
and say $ z^{\epsilon}\rightharpoonup  z$ weakly in  $L^2(\Omega)$.\\
\\
\textbf{Step(3):}
We apply the well-known div-curl lemma \cite{T} several times to simply have the following convergences :
\begin{equation*} (A^{\epsilon}\nabla \psi_{k}^{\epsilon}- B^{\epsilon}(\nabla\chi_{k}^{\epsilon} + {e_k}))\cdot (\nabla\chi_{k}^{\epsilon} + {e_k})\rightharpoonup  (A^{*}\nabla \psi_{k}- B^{\#}e_k)\cdot e_k \ \mbox{ in }\mathcal{D}^{\prime}(\Omega)\end{equation*}
and
\begin{equation*} A^{\epsilon}(\nabla\chi_{k}^{\epsilon} + {e_k})\cdot \nabla\psi_{k}^{\epsilon}\rightharpoonup  A^{*}e_k\cdot\nabla\psi_{k} \ \mbox{ in }\mathcal{D}^{\prime}(\Omega).  \end{equation*}
Since $A^\epsilon = (A^\epsilon)^t$ and $A^{*}=(A^{*})^t$, thus combining the above two convergences we obtain
\begin{equation}\label{FL12} B^{\epsilon}(\nabla\chi_{k}^{\epsilon} + {e_k})\cdot (\nabla\chi_{k}^{\epsilon} + {e_k})\rightharpoonup B^{\#}e_k\cdot e_k  \ \mbox{ in }\mathcal{D}^{\prime}(\Omega), \ k=1,2..,N.  \end{equation}
On the other hand, thanks to div-curl lemma we also have 
\begin{equation*} (A^\epsilon\nabla p^\epsilon- B^\epsilon\nabla u^\epsilon)\cdot (\nabla\chi_{k}^{\epsilon} + {e_k})\rightharpoonup  z\cdot e_k \ \mbox{ in }\mathcal{D}^{\prime}(\Omega) \end{equation*}
and
\begin{equation*} A^\epsilon (\nabla\chi_{k}^{\epsilon} + {e_k})\cdot\nabla p^\epsilon \rightharpoonup  A^{*}e_k\cdot\nabla p \ \mbox{ in }\mathcal{D}^{\prime}(\Omega).  \end{equation*}
Thus one gets,
\begin{equation}\label{eir} B^\epsilon\nabla u^\epsilon\cdot (\nabla\chi_{k}^{\epsilon} + {e_k})\rightharpoonup A^{*}\nabla p\cdot e_k- z\cdot e_k \ \mbox{ in }\mathcal{D}^{\prime}(\Omega). \end{equation}
Similarly, having 
\begin{equation*} (A^{\epsilon}\nabla \psi_{k}^{\epsilon}- B^{\epsilon}(\nabla\chi_{k}^{\epsilon} + {e_k}))\cdot \nabla u^{\epsilon}\rightharpoonup  (A^{*}\nabla \psi_{k}- B^{\#}e_k)\cdot\nabla u \ \mbox{ in }\mathcal{D}^{\prime}(\Omega)  \end{equation*}
and
\begin{equation*} A^\epsilon\nabla u^\epsilon\cdot\nabla \psi_{k}^{\epsilon} \rightharpoonup  A^{*}\nabla u\cdot\nabla \psi_k \ \mbox{ in }\mathcal{D}^{\prime}(\Omega), \end{equation*}
one obtains,
\begin{equation}\label{FG13} B^\epsilon (\nabla\chi_{k}^{\epsilon} + {e_k})\cdot \nabla u^\epsilon\rightharpoonup B^{\#}e_k\cdot\nabla u \ \mbox{ in }\mathcal{D}^{\prime}(\Omega). \end{equation}
Now by simply combining \eqref{eir} and \eqref{FG13}, we determine the expression of $z$ : 
\begin{equation*} z\cdot e_k = (A^{*}\nabla p -B^{\#}\nabla u)\cdot e_k, \ \ k=1,2,..N.\end{equation*}
Thus, 
$ z = A^{*}\nabla p -B^{\#}\nabla u $. Since $div\ z^\epsilon =0$ and $z^\epsilon \rightharpoonup z\ \mbox{ in }L^2(\Omega) \mbox{ weak }$, so $div\ z=0$.\\

\noindent
Therefore, we conclude that
\begin{equation*} B^\epsilon\nabla u^\epsilon\cdot\nabla u^\epsilon =\ A^\epsilon\nabla u^\epsilon\cdot\nabla p^\epsilon - z^\epsilon\cdot\nabla u^\epsilon \rightharpoonup A^{*}\nabla u\cdot\nabla p - z\cdot\nabla u = B^{\#}\nabla u\cdot\nabla u\ \mbox{ in } \mathcal{D}^\prime(\Omega).\end{equation*}
Hence, \eqref{dc1} follows. We prove later that $B^{\#}\in \mathcal{M}(b_1,\widetilde{b_2};\Omega)$. (See \eqref{Sd3}).
\hfill\end{proof}
\noindent
The above result justifies the Definition \ref{sid} stated in Introduction.  
\begin{remark}\label{ub19}
 If $u^\epsilon$ satisfies Dirichlet boundary condition i.e. $u^\epsilon\in H^1_0(\Omega)$, then it is natural to impose same
 Dirichlet boundary condition on $p^\epsilon$ also. Multiplying the equation $-div(A^\epsilon\nabla u^\epsilon)= f^\epsilon$ in $\Omega$ for $u^\epsilon$ (where
 $f^\epsilon$ strongly converges to $f$ in $H^{-1}(\Omega)$), by $p^\epsilon$ and the equation $div(A^\epsilon\nabla p^\epsilon-B^\epsilon\nabla u^\epsilon)=0$ in $\Omega$ for $p^\epsilon$
 by $u^\epsilon$, and passing to the limit after subtraction,
 we obtain the desired result, namely 
 \begin{equation*}
  \int_\Omega B^\epsilon \nabla u^\epsilon\cdot\nabla u^\epsilon\ dx \rightarrow \int_\Omega B^{\#}\nabla u\cdot\nabla u\ dx. 
 \end{equation*}
\hfill\qed\end{remark}
\begin{remark}\label{abc3}
The quadratic quantities \eqref{bs14}
are bounded in $L^1(\Omega)$ and so their weak
limit points are a priori  Radon measures. However, due to the special
nature of the sequence $\nabla u^\epsilon$, (\eqref{ad13} is satisfied), they are in fact $L^1(\Omega)$ functions according to  Definition \ref{sid} and  Theorem \ref{ot1}.
\hfill\qed\end{remark}
\begin{remark}\label{ad16}
Let us consider the flux sequence $\{B^\epsilon\nabla u^\epsilon\}_{\epsilon>0}$, which is bounded uniformly in $L^2(\Omega).$ 
Then using \eqref{ot9}, upto a subsequence we find 
\begin{equation*} B^\epsilon\nabla u^\epsilon \rightharpoonup B^{\#}_{FL}\nabla u \ \mbox{ in }L^2(\Omega) \end{equation*}
where, 
\begin{equation*} B^\epsilon(\nabla\chi^\epsilon_k + e_k)\rightharpoonup B^{\#}_{FL}e_k \mbox{ in }\mathcal{D}^\prime(\Omega), \ \ k=1,2,..,N.\end{equation*}
However, in general the limiting macro quantities $B^{\#}_{FL}$ and $B^{\#}$ appear in flux and energy convergence
respectively need not be same, i.e. $B^{\#}_{FL} \neq B^{\#}$ (cf. Remark \ref{ot10}).
\hfill\qed\end{remark}
\begin{remark}\label{zz17}
Above arguments show that the following characteristic property of $(A^{*},B^{\#})$ :
$v^\epsilon\in H^1(\Omega)$ such that $\nabla v^\epsilon \rightharpoonup \nabla v$ in $L^2(\Omega)$ and $div(A^\epsilon\nabla v^\epsilon)$ in $H^{-1}(\Omega)$ convergent then 
\begin{equation*}
A^\epsilon\nabla v^\epsilon \rightharpoonup A^{*}\nabla v \mbox{ in } L^2(\Omega)\ \ 
\mbox{and }\ \  B^\epsilon\nabla v^\epsilon\cdot\nabla v^\epsilon \rightharpoonup B^{\#}\nabla v\cdot\nabla v\mbox{ in }\mathcal{D}^\prime(\Omega). 
\end{equation*}
Moreover if $u^\epsilon$ is another test sequence like $v^\epsilon$ then 
\begin{equation*} B^\epsilon\nabla u^\epsilon\cdot\nabla v^\epsilon \rightharpoonup B^{\#}\nabla u\cdot\nabla v\mbox{ in }\mathcal{D}^\prime(\Omega). \end{equation*}
This follows simply due to   
\begin{equation*}B^\epsilon\nabla u^\epsilon\cdot\nabla v^\epsilon =\ A^\epsilon\nabla v^\epsilon\cdot\nabla p^\epsilon - z^\epsilon\cdot\nabla v^\epsilon \rightharpoonup A^{*}\nabla v\cdot\nabla p - z\cdot\nabla v = B^{\#}\nabla u\cdot\nabla v\ \mbox{ in } \mathcal{D}^\prime(\Omega). \end{equation*}
\hfill\qed\end{remark}

\subsection{Properties of $B^{\#}$}
Let us first define $B^{\#}$ element wise i.e. to define $(B^{\#})_{lk}=B^{\#}e_l\cdot e_k.$ 
We consider the sequences  $(A^{\epsilon}\nabla\psi_{k}^{\epsilon} - B^{\epsilon}(\nabla\chi_{k}^{\epsilon} + e_k))$ and $(\nabla\chi_{l}^{\epsilon} +{e_l})$ 
where $e_k,e_l\in\mathbb{R}^N$ are the canonical basis vectors, then by applying the div-curl lemma as before,
we have the following elements wise convergence
\begin{equation}\label{dc2}
B^{\epsilon}(\nabla\chi_k^{\epsilon}+e_k)\cdot(\nabla\chi_l^{\epsilon}+ e_l) \rightharpoonup  B^{\#}e_k\cdot e_l\quad\mbox{in }\mathcal{D}^{\prime}(\Omega)
\end{equation}
\begin{remark}
Notice that if $A^{\epsilon}$ is independent of $\epsilon$ i.e. $A^{\epsilon} = A$ then we have $X^{\epsilon}=0$, so $B^{\#} = \overline{B}$,
where, $\overline{B}$ is the $L^{\infty}(\Omega)$ weak* limit of $B^{\epsilon}$.
\hfill\qed
\end{remark}
\noindent
As an application of the above distributional convergence \eqref{dc2} one has the following :
\begin{enumerate}
\item[(i)]
Since $\{B^{\epsilon}\}_{\epsilon>0}$ is symmetric, $B^{\#}$ is also a symmetric matrix.
\item[(ii)] Let $B^{\epsilon} \in \mathcal{M}(b_1,b_2;\ \Omega)$. 
Then the ellipticity constant of $B^{\#}$ remains same 
as for $B^\epsilon$ i.e. $B^{\#}\geq b_1I$. However, the upper bounds for $B^\epsilon$ and $B^{\#}$ need not be the same. We have $B^{\#}\leq\widetilde{b_2}I$ with $\widetilde{b_2}=b_2\frac{a_2}{a_1}$. In this context, let us recall that the homogenized matrix $B^{*}$ admits bounds : $b_1I\leq B^{*}\leq b_2I$. 
\end{enumerate}
Let $\lambda=(\lambda_1,\lambda_2,..,\lambda_N)\in\mathbb{R}^N$ be an arbitrary vector, we define the corresponding oscillatory test function 
$\chi^\epsilon_\lambda= \sum_{k=1}^N \lambda_k\chi^\epsilon_k\in H^1(\Omega)$ and $\zeta^\epsilon_\lambda= \sum_{k=1}^N \lambda_k\zeta^\epsilon_k\in H^1(\Omega)$ to have 
\begin{equation}\label{ot11}B^{\#}\lambda\cdot\lambda =\ limit\ B^{\epsilon}(\nabla\chi_{\lambda}^{\epsilon}+ \lambda)\cdot(\nabla\chi_{\lambda}^{\epsilon}+\lambda);\end{equation}
and as a perturbation of the $H$-limit $B^{*}$,  
\begin{equation}\label{os18}
B^{\#}\lambda\cdot {\lambda}=\ B^{*}\lambda\cdot {\lambda} + limit\ B^{\epsilon}(\nabla\chi_{\lambda}^{\epsilon} - \nabla\zeta_{\lambda}^{\epsilon})\cdot(\nabla\chi_{\lambda}^{\epsilon} - \nabla\zeta_{\lambda}^{\epsilon}). \\                 
\end{equation}
\ \ \ (The above `limits' are to be understood in the sense of distributions.) \\
\begin{corollary}\label{Sd14}
\begin{equation}\label{Sd4}
B^{\#} \geq B^{*}
\end{equation}
where the equality holds if and only if $\nabla(\chi_{\lambda}^{\epsilon} - \zeta_{\lambda}^{\epsilon})\rightarrow 0 $ in $L^2(\Omega)$ for each $\lambda\in\mathbb{R}^N.$
\hfill\qed\end{corollary}
\noindent 
In the following lemma as an application of the distributional converge \eqref{ot11}, we provide the general bounds on $B^{\#}$.  
\begin{lemma}[\textbf{General Bounds}]\label{hsk}
Let $A^{\epsilon}\in \mathcal{M}(a_1,a_2;\Omega)$ with $0 <a_1\leq a_2 <\infty $ and $B^{\epsilon}\in\mathcal{M}(b_1,b_2;\ \Omega)$ with $0 <b_1\leq b_2 <\infty $,
$H$-converges to $A^{*}\in\mathcal{M}(a_1,a_2;\ \Omega)$ and $B^{*}\in\mathcal{M}(b_1,b_2;\ \Omega)$ respectively.
Then we have the following bounds 
\begin{equation}\label{Sd3} b_1I\leq \underline{B} \leq B^{*} \leq B^{\#} \leq \frac{b_2}{a_1}A^{*} \leq \frac{b_2}{a_1}\overline{A}\leq b_2\frac{a_2}{a_1}I\end{equation}
where, $(\underline{B})^{-1}$ is the $L^{\infty}$ weak* limit of the matrix sequence $(B^{\epsilon})^{-1}$ and $\overline{A}$ is the $L^{\infty}$ weak* limit of the matrix sequence $A^{\epsilon}$.
\end{lemma}
\begin{proof}
In general we have the following trivial bounds with respect to $L^{\infty}(\Omega)$ weak* limit of the sequence and its inverse
\begin{equation*}\underline{A} \leq A^{*} \leq \overline{A} \quad\mbox{ and }\quad \underline{B} \leq B^{*}\leq \overline{B}\end{equation*}
where, $A^{\epsilon}$, $(A^{\epsilon})^{-1}$, $B^{\epsilon}$, $(B^{\epsilon})^{-1}$ converges to $\overline{A}$, $(\underline{A})^{-1}$, $\overline{B}$, $(\underline{B})^{-1}$ in $L^{\infty}(\Omega)$ weak* limit respectively.\\
So by using \eqref{Sd4} we have the lower bound for $B^{\#}$
\begin{equation*} b_1I\leq \underline{B} \leq B^{*} \leq B^{\#}.\end{equation*} 
Next we seek the upper bound for $B^{\#}$. In general, $B^{\#}\leq \overline{B}$ is not true (cf. Remark \ref{Sd8}).
In fact, in the one-dimension problem (Example: \ref{hsf}) for two phase material ($a^{\epsilon}(x)=a_1\chi_{a_\epsilon}(x)+a_2(1-\chi_{a_\epsilon})$, and $b^{\epsilon}(x)=b_1\chi_{b_\epsilon}(x)+ b_2(1-\chi_{b_\epsilon})$ ) by taking $(\frac{a_2}{a_1} -1)$ small enough or $(\frac{b_2}{b_1} -1)$ is large enough 
it could be possible to get even $b^{\#}\geq b_2 \ ( > \overline{b} )$. 
It's a new phenomena. However we derive upper bounds of it in terms of $A^{*}$ (or $\overline{A}$).
Let us consider the following inequality
\begin{equation*}B^{\epsilon}(\nabla\chi_{\lambda}^{\epsilon} +\lambda)\cdot(\nabla\chi_{\lambda}^{\epsilon} + \lambda) \leq\ b_2(\nabla\chi_{\lambda}^{\epsilon} + \lambda)\cdot(\nabla\chi_{\lambda}^{\epsilon} + \lambda) \leq\ \frac{b_2}{a_1} A^{\epsilon}(\nabla\chi_{\lambda}^{\epsilon} +{\lambda})\cdot(\nabla\chi_{\lambda}^{\epsilon} +\lambda).\end{equation*}
So by using the distributional convergence we have
\begin{equation*}B^{\#}{\lambda}\cdot {\lambda} \leq\ \frac{b_2}{a_1} A^{*}{\lambda}\cdot {\lambda} \leq\  \frac{b_2}{a_1}\overline{A}\ \lambda\cdot {\lambda}\leq \frac{b_2}{a_1}a_2 \lambda\cdot {\lambda}\end{equation*}
which gives the desired upper bound on $B^{\#}$.
\hfill\end{proof}
\begin{remark}
A bound which is not sharp was found earliar in \cite{KR}.
Notice that if $B^{\epsilon} = \frac{b_2}{a_1}A^{\epsilon}$, then clearly $B^{\#} =\frac{b_2}{a_1}A^{*}$ and the fourth inequality becomes equality in \eqref{Sd3}
and Corollary \ref{hsk} provides the condition for the third inequality to become equality in \eqref{Sd3}.
\hfill\qed\end{remark}
\noindent
\begin{remark}[Localization]\label{ub13}
If $A^\epsilon\in\mathcal(a_1,a_2;\Omega)$ $H$-converges to $A^{*}$ and $\omega$ is an open subset of $\Omega$, then the sequence $A^\epsilon|_{\omega}$ restrictions of $A^\epsilon$ to $\omega$ $H$ converges to $A^{*}|_{\omega}$ 
(see \cite[Lemma 10.5]{T}). Similarly we remark that
if $B^\epsilon\in\mathcal(b_1,b_2;\Omega)$ converges to $B^{\#}$
relative to $A^\epsilon$ in $\Omega$, then $B^\epsilon|_\omega$ converges to 
$B^{\#}|_\omega$ relative to $A^\epsilon|_\omega$. 
\hfill\qed\end{remark}
\begin{remark}\label{sii}
If $A^\epsilon\in\mathcal{M}(a_1,a_2;\Omega)$ 
strongly  converges to $A^{*}$ in $L^p(\Omega)$ for $1\leq p<\infty$ then 
$A^\epsilon$ $H$-converges to $A^{*}$ (see, \cite[Lemma 1.2.22]{A}). 
Similarly, we remark that, 
if $A^\epsilon\in\mathcal{M}(a_1,a_2;\Omega)$ and $B^\epsilon\in\mathcal{M}(b_1,b_2;\Omega)$ strongly converge to $A^{*}$ and $B^{\#}$ respectively in $L^p(\Omega)$ for $1\leq p<\infty$ then $B^\epsilon$ converges to $B^{\#}$ relative to $A^\epsilon$. 
\hfill\qed\end{remark}
\begin{lemma}\label{pol5}
 If $B^\epsilon_i \xrightarrow{A^\epsilon} B^{\#}_i$ for $i=1,..,k$ then $\sum_{i=1}^k c_i B^\epsilon_i\xrightarrow{A^\epsilon} \sum_{i=1}^k c_i B^{\#}_i$, where $c_i>0$ are constants and $\sum_{i=1}^k c_i =1$.   
\end{lemma}
\begin{proof}
The proof simply follows the definition of the relative convergence  \eqref{dc1}. 
\hfill\end{proof}
We end this section by establishing that the macro quantities $(A^{*},B^{\#})$ can be estimated in terms of underlying microstructure. Such results will be needed in our study of 
$(A^{*},B^{\#})$ and in establishing optimal bounds on $(A^{*},B^{\#})$.
(See Section \ref{ad18}).
The estimate on $A^{*}$ are classical \cite{A} and we extend it for the
relative limit $B^{\#}$. 
\begin{lemma}\label{zz15}
Let $A^{\epsilon,i}\in\mathcal{M}(a_1,a_2;\Omega)$, $B^{\epsilon,i}\in\mathcal{M}(b_1,b_2;\Omega)$ and 
\begin{center}
$A^{\epsilon,i} = \{a_1\chi_{\omega_{A^{\epsilon,i}}} + a_2(1-\chi_{\omega_{A^{\epsilon,i}}})\}I \xrightarrow{H} A^{*,i}$
and
$B^{\epsilon,i} = \{b_1\chi_{\omega_{B^{\epsilon,i}}} + b_2(1-\chi_{\omega_{B^{\epsilon,i}}})\}I \xrightarrow{A^{\epsilon,i}} B^{\#,i}$
\end{center}
for $i=1,2$. We assume that $A^{*,i}$ and $B^{\#,i}$ are the constant matrices. 
Then there exist positive constants $C>0$ and $\delta_A>0,\delta_B>0$ (independent of microstructures) such that, 
\begin{equation}\label{FL16}
||(A^{*,1}-A^{*,2})||\ \leq\  C\ \underset{\epsilon\rightarrow 0}{limsup}\lb \int_\Omega |(\chi_{\omega_{A^{\epsilon,1}}}(x) - \chi_{\omega_{A^{\epsilon,2}}}(x))|\  dx\rb^{\delta_A} ;
\end{equation}
\begin{align}\label{FL14}
||(B^{\#,1}-B^{\#,2})||\ \leq\ & C\ \underset{\epsilon\rightarrow 0}{limsup}\ \{\lb \int_\Omega |(\chi_{\omega_{A^{\epsilon,1}}}(x) - \chi_{\omega_{A^{\epsilon,2}}}(x))|\  dx\rb^{\delta_A}\notag\\
&\qquad\qquad\qquad\qquad+ \lb \int_\Omega |(\chi_{\omega_{B^{\epsilon,1}}}(x) - \chi_{\omega_{B^{\epsilon,2}}}(x))|\ dx \rb^{\delta_B}\}.
\end{align}
\end{lemma}
\bpr
As $A^{*,i}$, $i=1,2$ are constant homogenized matrices we can take the following oscillatory test functions (See \cite[Page no. 6]{KP}) $\{w^{\epsilon,i}_k\}_{k=1}^N\in (H^1(\Omega))^N$ satisfying : 
\begin{equation}\label{zz16}
-div A^{\epsilon,i}\nabla w^{\epsilon,i}_k = 0 \mbox{ in } \Omega, \quad w^{\epsilon,i}_k = x_k \ \mbox{ on }\partial\Omega, \ i=1,2
\end{equation}
so that,
\begin{equation*}
\nabla w^{\epsilon,i}_k \rightharpoonup e_k \mbox{ in } L^2(\Omega)\ \mbox{ and }\ A^{\epsilon,i}\nabla w^{\epsilon,i}_k \rightharpoonup A^{*,i}e_k \mbox{ in }L^2(\Omega),\ i=1,2. 
\end{equation*}
Using $w^{\epsilon,1}_k- w^{\epsilon,2}_k =0$ on $\partial\Omega$ and integrating by parts in \eqref{zz16} we have 
\begin{equation*}
 \int_\Omega A^{\epsilon,1}\nabla(w^{\epsilon,1}_k -w^{\epsilon,2}_k)\cdot\nabla(w^{\epsilon,1}_k -w^{\epsilon,2}_k)\ dx 
 = \int_\Omega (A^{\epsilon,2}-A^{\epsilon,1})\nabla w^{\epsilon,2}_k\cdot\nabla (w^{\epsilon,1}_k -w^{\epsilon,2}_k)\ dx.
\end{equation*}
Then by using coercivity of $A^{\epsilon,1}$ we simply get   
\begin{equation}\label{zz19}
 ||\nabla(w^{\epsilon,1}_k -w^{\epsilon,2}_k)||_{L^2(\Omega)} \leq C ||(\chi_{\omega_{A^{\epsilon,1}}}-\chi_{\omega_{A^{\epsilon,2}}})\nabla w^{\epsilon,2}_k||_{L^2(\Omega)}.
\end{equation}
By using the Meyers theorem (cf. \cite[Theorem 1.3.41]{A}), there exists an exponent $p>2$ such that $||\nabla w^{\epsilon,i}_k||_{L^p(\Omega)}\leq C$, $i=1,2$ (independent of $\epsilon$),
and further using Young's inequality in \eqref{zz19} we deduce
\begin{equation*}
 ||\nabla(w^{\epsilon,1}_k -w^{\epsilon,2}_k)||_{L^2(\Omega)} \leq C ||(\chi_{\omega_{A^{\epsilon,1}}}-\chi_{\omega_{A^{\epsilon,2}}})||_{L^{\frac{2p}{p-2}}(\Omega)}. 
\end{equation*}
\begin{align*}
|(A^{*,1}-A^{*,2})e_k| &\leq\ \underset{\epsilon\rightarrow 0}{limsup}\ |\{\int_\Omega A^{\epsilon,1}(\nabla w^{\epsilon,1}_k -\nabla w^{\epsilon,2}_k)\cdot e_k\ dx + \int_\Omega (A^{\epsilon,1}-A^{\epsilon,2})\nabla w^{\epsilon,2}_k\cdot e_k\ dx\}| \\
&\leq C\ \underset{\epsilon\rightarrow 0}{limsup} \lb||\nabla(w^{\epsilon,1}_k -w^{\epsilon,2}_k)||_{L^2(\Omega)} + ||(\chi_{\omega_{A^{\epsilon,1}}}-\chi_{\omega_{A^{\epsilon,2}}})||_{L^{\frac{2p}{p-2}}(\Omega)}\rb\\
&\leq C\ \underset{\epsilon\rightarrow 0}{limsup} ||(\chi_{\omega_{A^{\epsilon,1}}}-\chi_{\omega_{A^{\epsilon,2}}})||_{L^{\frac{2p}{p-2}}(\Omega)}.
\end{align*}
On the other hand, by the definition of relative convergence, we have 
\begin{equation*}
B^{\epsilon,i}\nabla w^{\epsilon,i}_k\cdot\nabla w^{\epsilon,i}_k \rightharpoonup B^{\#,i}e_k\cdot e_k \mbox{ in }\mathcal{D}^\prime(\Omega), \ i=1,2.
\end{equation*}
Moreover, we claim :
\begin{equation}\label{zz18}
\int_\Omega B^{\epsilon,i}\nabla w^{\epsilon,i}_k\cdot\nabla w^{\epsilon,i}_k\ dx \rightarrow \int_\Omega B^{\#,i}e_k\cdot e_k\ dx, \ i=1,2.
\end{equation}
Introducing the adjoint state $p^{\epsilon,i}_k\in H^1_0(\Omega)$ via the relation :
\begin{equation*}
 div (A^{\epsilon,i}\nabla p^{\epsilon,i}_k -B^{\epsilon,i}\nabla w^{\epsilon,i}_k) =0 \mbox{ in }\Omega ; \ \ i=1,2.
\end{equation*}
Multiply the above equation of $p^{\epsilon,i}_k$ by $w^{\epsilon,i}_k$ and the equation of $w^{\epsilon,i}_k$ by $p^{\epsilon,i}_k$, we obtain 
\begin{equation*}
 \int_\Omega B^\epsilon\nabla w^{\epsilon,i}_k\cdot\nabla w^{\epsilon,i}_k\ dx = -\int_{\partial\Omega}(z^{\epsilon,i}_k\cdot \nu)\ x_k \ d\sigma \mbox{ with }z^{\epsilon,i}_k = (A^{\epsilon,i}\nabla p^{\epsilon,i}_k -B^{\epsilon,i}\nabla w^{\epsilon,i}_k).
\end{equation*}
Taking limit $\epsilon \rightarrow 0$ and using that $z^{\epsilon,i}_k \rightharpoonup z^{i}_k = A^{*,i}\nabla p^{i}_k -B^{\#,i}e_k$ in $L^2(\Omega)$ weak, we get 
\begin{equation*}
 \int_\Omega B^\epsilon\nabla w^{\epsilon,i}_k\cdot\nabla w^{\epsilon,i}_k\ dx \rightarrow  -\int_{\partial\Omega}(z^i_k\cdot \nu)\ x_k\ d\sigma
\end{equation*}
By repeating the above steps, we can compute the above limit and establish our claim \eqref{zz18}. \\
\\
Then it follows that,
\begin{align*}
|(B^{\#,1}-B^{\#,2})e_k\cdot e_k| &\leq\ \underset{\epsilon\rightarrow 0}{limsup}\ |\{\int_\Omega B^{\epsilon,1}(\nabla w^{\epsilon,1}_k\cdot\nabla w^{\epsilon,1}_k -\nabla w^{\epsilon,2}_k\cdot\nabla w^{\epsilon,2}_k)\ dx \\
&\qquad\qquad\qquad\qquad + \int_\Omega (B^{\epsilon,1}-B^{\epsilon,2})\nabla w^{\epsilon,2}_k\cdot \nabla w^{\epsilon,2}_k\ dx\}| \\
&\leq C\ \underset{\epsilon\rightarrow 0}{limsup} \lb||\nabla(w^{\epsilon,1}_k -w^{\epsilon,2}_k)||_{L^2(\Omega)} + ||(\chi_{\omega_{B^{\epsilon,1}}}-\chi_{\omega_{B^{\epsilon,2}}})||_{L^{\frac{p}{p-2}}(\Omega)}\rb\\
&\leq C\ \underset{\epsilon\rightarrow 0}{limsup} \lb||(\chi_{\omega_{A^{\epsilon,1}}}-\chi_{\omega_{A^{\epsilon,2}}})||_{L^{\frac{2p}{p-2}}(\Omega)} + ||(\chi_{\omega_{B^{\epsilon,1}}}-\chi_{\omega_{B^{\epsilon,2}}})||_{L^{\frac{p}{p-2}}(\Omega)}\rb.
\end{align*}
Thus \eqref{FL16}, \eqref{FL14} follows with $\delta_A = \frac{p-2}{2p}$ and $\delta_B =  \frac{p-2}{p}$.   
\hfill\epr
\noindent
Finally, we add one covariance property for the relative convergence.  
\begin{lemma}[Covariance property]\label{pol6}
Let $A^\epsilon \in \mathcal{M}(a_1,a_1;\Omega)$ $H$-converges to $A^{*}$, and $B^\epsilon\in\mathcal{M}(b_1,b_2;\Omega)$ converges to $B^{\#}$ relative to $A^\epsilon$, and $\phi$ is a diffeomorphism $\Omega$ onto $\phi(\Omega)$. Let us define :
\begin{equation*}
 \widetilde{A}^\epsilon(\phi(x)) \stackrel{def}{=} \frac{1}{\mbox{det }(\nabla \phi(x))}\nabla \phi(x) A^\epsilon(x)\nabla\phi^{T}(x),\
 \widetilde{B}^\epsilon(\phi(x)) \stackrel{def}{=} \frac{1}{\mbox{det }(\nabla \phi(x))}\nabla \phi(x) B^\epsilon(x)\nabla\phi^{T}(x).
\end{equation*} 
Then we have  
\begin{align*}
\widetilde{A}^\epsilon(\phi(x)) &\xrightarrow{H} \frac{1}{\mbox{det }(\nabla \phi(x))}\nabla \phi(x) A^{*}(x)\nabla\phi^{T}(x) = \widetilde{A}^{*}(\phi(x)) \mbox{ in }\phi(\Omega)\\
\mbox{and }\ \widetilde{B}^\epsilon(\phi(x)) &\xrightarrow{\widetilde{A}^\epsilon(\phi(x))} \frac{1}{\mbox{det }(\nabla \phi(x))}\nabla \phi(x) B^{\#}(x)\nabla\phi^{T}(x) = \widetilde{B}^{\#}(\phi(x)) \mbox{ in }\phi(\Omega).
\end{align*}
\end{lemma}
\begin{proof}
The proof is analogous to the proof of the covariance property of $H$-convergence given in \cite[Lemma 21.1]{T}.
\hfill\end{proof}

\section{One-dimensional Case}\label{hsf}
\setcounter{equation}{0}

Let us present the one-dimensional case where we compute the macro limits $A^{*}$ and $B^{\#}$ explicitly. We point out that the relation \eqref{FL19} was earliar obtained in \cite{KP}.

Let $0 < a_1 \leq a^{\epsilon} \leq  a_2 \mbox{ and } 0 < b_1 \leq b^{\epsilon} \leq b_2$ in some bounded open interval $I\subset \mathbb{R}$. 
We have $u^{\epsilon},p^{\epsilon}\in H^1_0(I)$ solving the following system of equations with $f\in H^{-1}(I)$ : 
\begin{equation*}\begin{aligned}
&-\frac{d}{dx}(a^{\epsilon}\frac{du^{\epsilon}}{dx}) = f  \mbox{ in }I,\\
& \frac{d}{dx}(a^{\epsilon}\frac{dp^{\epsilon}}{dx} - b^{\epsilon}\frac{du^{\epsilon}}{dx}) = 0 \mbox{ in }I.
\end{aligned}\end{equation*}
Then from the one-dimensional homogenization, we know that 
\begin{equation*}u^{\epsilon}\rightharpoonup u \mbox{ weakly in }H^1_0(I)\mbox{ and the flux }\sigma^{\epsilon} = a^{\epsilon}\frac{du^{\epsilon}}{dx}\rightarrow \sigma=a^{*}\frac{du}{dx} \mbox{ strongly in }L^2(I),\end{equation*}
where $(a^{*})^{-1}= L^{\infty}(I)\mbox{ weak* limit of }(a^{\epsilon})^{-1}= \underline{a}^{-1}$ (say). \\  
Similarly, from the adjoint equation, we have 
\begin{align*}p^{\epsilon} \rightharpoonup p\mbox{ in }H^1_0(I)\mbox{ and the flux }z^{\epsilon} = (a^{\epsilon}\frac{dp^{\epsilon}}{dx} - b^{\epsilon}\frac{du^{\epsilon}}{dx}) = c^{\epsilon}\mbox{ (constant) converges strongly}&\quad\\
\mbox{ to $z=c$ (constant) in }L^2(I)&.\end{align*}
We combine these two strongly convergent sequences $\sigma^{\epsilon}$, $z^{\epsilon}$ in the following way : 
\begin{equation}\label{qw3}
\frac{dp^{\epsilon}}{dx} - \frac{b^{\epsilon}}{(a^{\epsilon})^2}\sigma^{\epsilon} = \frac{c^{\epsilon}}{a^{\epsilon}}.\end{equation}
Passing to the limit as $\epsilon \rightarrow 0$ in the above equation we have, 
\begin{equation*}
\frac{dp}{dx} - (lim^{*}\frac{b^{\epsilon}}{(a^{\epsilon})^2})\sigma = \frac{c}{\underline{a}}, \quad\mbox{ where }(lim^{*}\frac{b^{\epsilon}}{(a^{\epsilon})^2}) = \mbox{$L^{\infty}(I)$ weak* limit of }\frac{b^{\epsilon}}{(a^{\epsilon})^2}.
\end{equation*}
Then by using it in \eqref{qw3}, we have 
\begin{equation*}\begin{aligned}
&z = \underline{a}\frac{dp}{dx} - \{(\underline{a})^2(lim^{*}\frac{b^{\epsilon}}{(a^{\epsilon})^2})\}\frac{du}{dx} = c\\  
\mbox{or, }\ \ &\frac{d}{dx}(\underline{a}\frac{dp}{dx} - \{(\underline{a})^2(lim^{*}\frac{b^{\epsilon}}{(a^{\epsilon})^2})\}\frac{du}{dx}) = 0  
\end{aligned}\end{equation*}
Therefore, 
\begin{equation}\label{FL19}
b^{\#} = (\underline{a})^2\hspace{1pt}(lim^{*}\frac{b^{\epsilon}}{(a^{\epsilon})^2}).
\end{equation}
Even if $a^\epsilon \xrightarrow{H} a^{*} =\underline{a}$ and $b^\epsilon \rightharpoonup \overline{b}$ in $L^{\infty}(I)$ weak* for the entire sequence, $\frac{b^\epsilon}{(a^\epsilon)^2}$ may oscillate as microstructures vary and so its limit need not be unique. In other words, the behaviour is different from  
$a^{*}=\underline{a}$. 
\begin{remark}\label{ot10}
One also notices that, the flux
\begin{equation*} b^\epsilon\frac{du^{\epsilon}}{dx} =\frac{b^\epsilon}{a^\epsilon}\sigma^\epsilon \rightharpoonup (lim^{*}\frac{b^{\epsilon}}{a^{\epsilon}})\hspace{1.5pt}\sigma = (lim^{*}\frac{b^{\epsilon}}{a^{\epsilon}})\hspace{2pt}\underline{a}\frac{du}{dx}, \mbox{ weakly in }L^2(I),\end{equation*}
and the energy
\begin{equation*} b^\epsilon\frac{du^{\epsilon}}{dx}\frac{du^{\epsilon}}{dx} \rightharpoonup (lim^{*}\frac{b^{\epsilon}}{(a^{\epsilon})^2})\hspace{2pt}(\underline{a})^2\frac{du}{dx}\frac{du}{dx} \mbox{ weakly in }\mathcal{D}^\prime(I).\end{equation*}
The limiting macro quantities that appear in flux and energy convergence respectively need not be same. For instance, let $a^\epsilon,b^\epsilon$ be two-phase medium having the same microstructure i.e.  
\begin{equation*} a^\epsilon = a_1 \chi^\epsilon + a_2(1-\chi^\epsilon),\ (0<a_1<a_2<\infty)\ \mbox{ and } b^\epsilon = b_1\chi^\epsilon + b_2(1-\chi^\epsilon), \ (0<b_1\leq b_2<\infty).\end{equation*}
Then one can show 
\begin{equation*} b^{\#}_{FL} = (lim^{*}\frac{b^{\epsilon}}{a^{\epsilon}})\hspace{2pt}\underline{a} \neq  (lim^{*}\frac{b^{\epsilon}}{(a^{\epsilon})^2})\hspace{2pt}(\underline{a})^2 = b^{\#}.
\end{equation*}
\hfill\qed\end{remark}
\begin{remark}\label{Sd12}
Let us consider two-phase media, $a^\epsilon_1$ and $a^\epsilon_2$ given by two different microstructures, and possessing the same homogenized limit $\underline{a}$.
Then there exists  $b^\epsilon$ such that $a^\epsilon_j \xrightarrow{H} \underline{a}$, $b^\epsilon\xrightarrow{a^\epsilon_j}
b^{\#}_j $ for $j=1,2$ with $b^{\#}_1 \neq b^{\#}_2$.
However, if $b^\epsilon$ is independent of $\epsilon$, then $b^{\#}$ is constant. 

\hfill\qed\end{remark}
\subsection{One-dimensional Bounds}\label{Sd6}
Here we find bounds on $b^{\#}$ where both $a^\epsilon(x)$ and $b^\epsilon(x)$ are given by two phase medium.
Let $I$ be an interval in $\mathbb{R}$, we consider
\begin{equation*}\begin{aligned}
a^{\epsilon}(x) &= a_1\A + a_2(1-\A)\mbox{ with }(a_1 < a_2 ),\ \ x\in I\\
\mbox{and }\quad b^{\epsilon}(x) &= b_1\B + b_2(1-\B)\mbox{ with }(b_1\leq b_2),\ \ x\in I.
\end{aligned}\end{equation*}
Let us assume that,
\begin{equation*}
\A\rightharpoonup \theta_A(x) \mbox{ and } \B \rightharpoonup \theta_B(x) \quad\mbox{ in }L^{\infty}(I)\mbox{ weak*. }
\end{equation*}
Then following \eqref{FL19} we find
\begin{align}
b^{\#}(x) &=(\underline{a})^{2}\underset{\epsilon\rightarrow 0}{lim}\ \frac{b^{\epsilon}}{(a^{\epsilon})^2}(x)\notag \\
&=(\underline{a})^{2} \underset{\epsilon\rightarrow 0}{lim}\ \{ \frac{b_2}{a_2^2} + \frac{(b_1-b_2)}{a_2^2}\B + (\frac{b_2}{a_1^2}-\frac{b_2}{a_2^2})\A - {(b_2-b_1)}(\frac{1}{a_1^2} -\frac{1}{a_2^2})\AB \}\notag\\
&=(\underline{a})^{2}\{\frac{b_2}{a_2^2} + \frac{(b_1-b_2)}{a_2^2}\theta_B + {b_2}(\frac{1}{a_1^2} -\frac{1}{a_2^2})\theta_A - {(b_2-b_1)}(\frac{1}{a_1^2} -\frac{1}{a_2^2})\theta_{AB}\}\label{FL18}
\end{align}
where, $\theta_{AB}$ = $L^{\infty}(I)$ weak* limit of $\A.\B (= \AB$) and it satisfies the following bounds :
\begin{equation}\begin{aligned}\label{FG2} 
\mbox{ Upper Bound : }\ \ & \theta_{AB}\ \leq\ min\ \{\theta_A,\theta_B\}.\\
\mbox{ Lower Bound : }\ \ & \theta_{AB}\ \geq\ 0, \ \mbox{ always, }\\
\mbox{and }\ \ \ \           &  \theta_{AB}\ \geq\ \theta_A + \theta_B -1, \ \mbox{ whenever }\theta_A + \theta_B \geq 1.
\end{aligned}\end{equation}  
Using this, we find maximum and the minimum value of $b^{\#}$ (cf.\eqref{FL18}) for given $(a_1,a_2,\theta_A)$ and $(b_1,b_2,\theta_B)$ fixed.
Notice that, the final term in the expression \eqref{FL18} has the negative sign. So in order to get the lower (or upper) bound of the left hand side
we need to use upper (or lower) bound of $\theta_{AB}$. Thus we obtain the following lower bound
\begin{equation}\begin{aligned}\label{FL20}
\mbox{when $\theta_A \leq \theta_B$,  }\ b^{\#} \geq\ &\ (\underline{a})^{2}\{\frac{b_2}{a_2^2} + \frac{(b_1-b_2)}{a_2^2}\theta_B + {b_1}(\frac{1}{a_1^2} -\frac{1}{a_2^2})\theta_A \}= l^{\#}_1(x) \mbox{ (say)} \\
\mbox{when $\theta_B \leq \theta_A $, }\ b^{\#} \geq\ &\ (\underline{a})^{2}\{\frac{b_2}{a_2^2} + \frac{(b_1-b_2)}{a_1^2}\theta_B + {b_2}(\frac{1}{a_1^2} -\frac{1}{a_2^2})\theta_A\}= l^{\#}_2(x) \mbox{ (say).}
\end{aligned}\end{equation}
Similarly, we find the upper bounds
\begin{equation}\begin{aligned}\label{FL17}
\mbox{when $\theta_A + \theta_B \leq 1,$ }\ b^{\#} \leq\ &\ (\underline{a})^{2}\{\frac{b_2}{a_2^2} + \frac{(b_1-b_2)}{a_2^2}\theta_B + {b_2}(\frac{1}{a_1^2} -\frac{1}{a_2^2})\theta_A\} = u^{\#}_1(x) \mbox{ (say) }\\
\mbox{when $\theta_A + \theta_B \geq 1,$ }\ b^{\#} \leq\ &\ (\underline{a})^{2}\{\frac{(b_2-b_1)}{a_1^2} + \frac{b_1}{a_2^2} + \frac{(b_1-b_2)}{a_1^2}\theta_B + {b_1}(\frac{1}{a_1^2} -\frac{1}{a_2^2})\theta_A\}= u^{\#}_2(x) 
\end{aligned}\end{equation}           
Moreover, a simple computation shows that
\begin{equation*}  max\ \{\ l^{\#}_1, l^{\#}_2\ \}\ \leq \  min\ \{\ u^{\#}_1,u^{\#}_2\ \}.
\end{equation*}
\textbf{One-dimensional Bound :} 
\begin{equation}\label{FL13}
 min\ \{\ l^{\#}_1, l^{\#}_2\ \} \leq b^{\#} \leq   max\ \{\ u^{\#}_1, u^{\#}_2\ \}.
\end{equation}
\textbf{Optimality of the Bounds : }
We define
\begin{align*}
u_{\#}=\begin{cases} u^{\#}_1 \quad\mbox{if }\theta_A +\theta_B \leq 1 \\
        u^{\#}_2  \quad\mbox{if }\theta_A +\theta_B \geq 1 
       \end{cases}
,\quad l_{\#}=\begin{cases}
   l^{\#}_1 \quad\mbox{if }\theta_A \leq \theta_B \\      
   l^{\#}_2  \quad\mbox{if }\theta_B \leq \theta_A
\end{cases}.
\end{align*}
Then the following theorem provides the optimality of the region defined by the bounds \eqref{FL13}.
\begin{theorem}\label{Sd10}
By varying microstructures for $\{a^\epsilon, b^\epsilon \}$, $b^{\#}(x)$ attains all values in the interval $[l_{\#}(x),u_{\#}(x)]$. 
\end{theorem}
\begin{proof}
The equality of the above bound (lower and upper) can be achieved easily. For upper bound equality $u_{\#}$, one considers 
the ${\omega}_{A^{\epsilon}}$ and ${\omega}_{B^{\epsilon}}$ in such a way that they intersect (${\omega}_{A^{\epsilon}} \cap {\omega}_{B^{\epsilon}}$)  in a least way,
i.e. whenever we have $\theta_A +\theta_B \leq 1$, we take ${\omega}_{A^{\epsilon}} \cap {\omega}_{B^{\epsilon}} =\emptyset$ and for $\theta_A +\theta_B \geq 1$, 
we consider ${\omega}_{A^{\epsilon}}$, ${\omega}_{B^{\epsilon}}$ in such a way that $L^{\infty}$ weak* limit of $\AB$ becomes  $(\theta_A +\theta_B-1)$.
Similarly for the lower bound equality $l_{\#}$, one takes ${\omega}_{B^{\epsilon}} \subseteq {\omega}_{A^{\epsilon}}$ for $\theta_A\geq \theta_B$. When $\theta_B\geq \theta_A$ we take ${\omega}_{A^{\epsilon}} \subseteq {\omega}_{B^{\epsilon}}$.

Another important thing to notice that for given any value $v$ between $v\in (l_{\#},u_{\#})$, there are corresponding microstructures $\A,\B$ such that simply by fixing the value of the weak* limit of $\AB$ to be $\theta_{AB}$, we can achieve the intermediate value $v$. The following lemma from \cite{T} allows us to do that. For this purpose, it is enough to express $v$ in terms of $\theta_{AB}$ using the relation \eqref{FL18}.
\hfill\end{proof}
\begin{lemma}[Lypunov]
If $\theta \in L^{\infty}(\Omega)$ satisfies $0 \leq \theta \leq 1$ a.e. in $\Omega$ then there exists
a sequence of characteristic functions $\chi_{\epsilon}$ satisfying $\chi_{\epsilon}(x)\rightharpoonup \theta(x)$ in $L^{\infty}(\Omega)$ weak*.
\hfill\qed
\end{lemma}
We conclude this section by making some general comments on $N$ dimensional case. In general the homogenized matrices $A^{*}$, $B^{\#}$ are difficult to calculate. 
Description of the set of all possible macro quantities is analogous to the known famous \textit{G-closure problem}
\cite{A,AM} in homogenization theory. It largely depends the microstructure or microgeometry which provides the structure 
or geometry in microscale in which way the components are mixed. Our goal is to obtain bounds on the set of all possible limit matrices from below and above which are independent of microstructure. We also want the bounds to be optimal in the sense of Theorem \ref{Sd10}.

\section{Statement of Main Results}\label{ad18}
\setcounter{equation}{0}
Here in this section we announce the results concerning the bounds on $(A^{*},B^{\#})$. We assume $A^{\epsilon}$ and $B^{\epsilon}$
are governed with the two phase medium. 
More precisely let $A^{\epsilon}\in \mathcal{M}(a_1,a_2;\Omega)$ with $0 < a_1\leq a_2 <\infty$ and 
$B^{\epsilon}\in\mathcal{M}(b_1,b_2;\Omega)$ with $0< b_1\leq b_2 < \infty$ are governed with two phase medium :
\begin{equation}\label{ta}\begin{aligned}
A^{\epsilon}(x) = a^{\epsilon}(x)I\ \mbox{ where, }\ a^{\epsilon}(x) &= a_1\A +a_2 (1-\A)\ \mbox{ a.e. in }\ x\in\Omega \\
\A &\rightharpoonup \theta_A(x) \mbox{ in }L^{\infty}(\Omega)\mbox{  weak* topology }.
\end{aligned}\end{equation} 
\begin{equation}\label{tb}\begin{aligned}
B^{\epsilon}(x)= b^{\epsilon}(x)I\ \mbox{ where }\ b^{\epsilon}(x) &= b_1\B + b_2 (1-\B)\ \mbox{ a.e. in }\ x\in\Omega\\
\B &\rightharpoonup \theta_B(x)\mbox{ in } L^{\infty}(\Omega) \mbox{ weak* topology }.
\end{aligned}
\end{equation}
In this section, we state four bounds labeled as L1, L2, U1, U2 involving $(A^{*},B^{\#})$, $\{a_1,a_2,\theta_A\}$, $\{b_1,b_2,\theta_B\}$. They are shown to be valid in sub-domains $int(\Omega_{L1})$, $int(\Omega_{L2})$, $int(\Omega_{U1})$, $int(\Omega_{U2})$ which are as follows :
\begin{align*}
&\Omega_{L1}:=  \{x\in\Omega : \theta_A(x) \leq \theta_B(x)\},\ \ 
\ \ \ \ \  \Omega_{L2}:=  \{x\in\Omega : \theta_B(x) < \theta_A(x)\},\\
&\Omega_{U1}:=  \{x\in\Omega : \theta_A(x) + \theta_B(x)\leq 1\},\ \ 
\Omega_{U2}:=  \{x\in\Omega : \theta_B(x) + \theta_A(x)> 1\}.
\end{align*}
The bounds define four regions in the phase space denoted by $(Li,Uj)$, $i,j=1,2$. Physical sub-domain on which the region $(Li,Uj)$ is optimal is $int(\Omega_{(Li,Uj)})$, where 
$ \Omega_{(Li,Uj)} = \Omega_{Li} \cap \Omega_{Uj} \mbox{ for }i,j=1,2$.

Compare this situation with the classical case involving $A^{*}$ in which there is only one region in the phase space and it is optimal in the entire physical domain $\Omega$. 
Their proofs presented in Section \ref{ts} use a combination of translated inequality, $H$-measure techniques and
Compensated Compactness. 
We start by recalling the optimal bounds  on the homogenized matrices $A^{*}$ for two-phase medium. 
\paragraph{Optimal Bounds on $A^{*}(x)$ :} 
Let $\mathcal{G}_{\theta_A(x)}$ (known as the \textit{G-closure set}) be the set of all possible effective tensors $A^{*}(x)$ obtained by the homogenization
of two phases $a_1,a_2$ associated with the volume fraction $\theta_A(x)$ and $(1-\theta_A(x))$ 
respectively, which has the following pointwise characterization in terms of the trace inequalities (see \cite[Proposition 10]{MT1}) :
The set $\mathcal{G}_{\theta_A(x)}$ is the set of all symmetric matrices with eigenvalues $\lambda_1(x),..,\lambda_N(x)$ satisfying pointwise 
\begin{equation}\begin{aligned}\label{FL11}
\underline{a}(x) \leq \lambda_i(x)&\leq \overline{a}(x),\quad \forall 1\leq i\leq N\\
\mbox{(Lower Trace Bound) : }\quad \ \sum_{i=1}^N \frac{1}{\lambda_i(x)-a_1} &\leq \frac{1}{\underline{a}(x)-a_1} + \frac{N-1}{\overline{a}(x)-a_1}\\
\mbox{(Upper Trace Bound) : }\quad \ \sum_{i=1}^N \frac{1}{a_2 -\lambda_i(x)} &\leq \frac{1}{a_2-\underline{a}(x)} + \frac{N-1}{a_2-\overline{a}(x)},
\end{aligned}\end{equation}
where $\underline{a}(x), \overline{a}(x)$ are the harmonic and arithmetic means of $a_1,a_2$ respectively, defined as 
$\underline{a}(x) = (\frac{\theta_A(x)}{a_1} +\frac{1-\theta_A(x)}{a_2})^{-1}$ and $\overline{a}(x)=a_1\theta_A(x)+a_2(1-\theta_A(x))$.  \\
\\
$\mathcal{G}_{\theta_A(x)}$ is a convex region for fixed $x$ (see \cite[Remark 2.2.15]{A}). Let the boundaries of $\mathcal{G}_{\theta_A(x)}$ be $\partial\mathcal{G}^{L}_{\theta_A(x)}$ and $\partial\mathcal{G}^{U}_{\theta_A(x)}$ i.e. 
($\partial\mathcal{G}_{\theta_A(x)} = \partial\mathcal{G}^{L}_{\theta_A(x)} \cup \partial\mathcal{G}^{U}_{\theta_A(x)}$).
The set $\partial\mathcal{G}^{L}_{\theta_A(x)}$ (respectively, $\partial\mathcal{G}^{U}_{\theta_A(x)}$) represents the lower bound equality (respectively, the upper bound equality)
of \eqref{FL11}. \hfill\qed
\\
\\
We also introduce $\mathcal{K}_{\theta_A}\subset \mathbb{R}^N$ consisting of $(\lambda_1,..,\lambda_N)$ satisfying the above three inequalities \eqref{FL11}. 
Having the above characterization of the set $\mathcal{G}_{\theta_A(x)}$,  we will present our results concerning the trace bounds involving $(A^{*},B^{\#})$ in arbitrary dimension. It is divided into four cases and their optimality will be taken up in Section \ref{qw4} (see \eqref{qw2}).
\subsection{Optimal Trace Bounds : $A^{\epsilon}$ is governed by a two phase medium and $B^{\epsilon}$ is independent of $\epsilon$ }\label{FL7}
We consider $A^{\epsilon}\in\mathcal{M}(a_1,a_2;\Omega)$ is given by \eqref{ta} and $B^{\epsilon}(x)=\ b(x)I\in\mathcal{M}(b_1,b_2;\Omega)$ for some $b\in L^{\infty}(\Omega)$. That is, $B^\epsilon$ does not have microstructure. Then we have following bounds in terms of the trace of the matrices $A^{*}(x)$ and $B^{\#}(x)$ hold almost everywhere in $x\in\Omega$ :
\paragraph{Lower Trace Bound L :}
\begin{equation}\label{tw}
tr\ \{b(x)(a_2I-A^{*}(x))(a_2B^{\#}(x)-b(x)A^{*}(x))^{-1}(a_2I-A^{*}(x))\} \leq\ N\theta_A(x)(a_2-a_1).
\end{equation}  
\noindent
\textbf{Upper Trace Bound U :}
\begin{equation}\label{tq}
tr\ \{ b(x)(A^{*}(x) - a_1I)(b(x)A^{*}(x)-a_1 B^{\#}(x))^{-1}(A^{*}(x) - a_1I)\} \leq\ N(1-\theta_A(x))(a_2-a_1).
\end{equation}
In the sequel, when we talk about lower/upper trace bound, we always mean the above bounds supplemented with the bound \eqref{Sd3}.
The above bounds are saturated/optimal in the sense of
Theorem \ref{qw6} below. We will not write down explicitly the proof because more
complicated case of $B^\epsilon$ having  microstructures will be discussed in Section \ref{qw4}.\\
\\
We further discover (see Remark \ref{eg3}) that for $A^\epsilon(x)$ satisfying \eqref{ta}, the corresponding $H$-limit $A^\epsilon(x) \xrightarrow{H} A^{*}(x)$ in $\Omega$ and the relative limit $b(x)I\xrightarrow{A^\epsilon(x)} B^{\#}(x)$ in $\Omega$, commute with each other, i.e.
\begin{equation}
A^{*}(x)B^{\#}(x) = B^{\#}(x)A^{*}(x),\quad x\in\Omega\ \mbox{ a.e.}
\end{equation}
That means that they have a common set of eigenvector basis, say $\{u_1(x),\ldots,u_N(x)\}$, $x\in\Omega$. If $\lambda_i(x)$ and the $\mu_i(x)$ are the eigenvalues of $A^{*}(x)$ and $B^{\#}(x)$ corresponding to the common eigenvector $u_i(x)$ i.e.
\begin{equation} 
A^{*}(x)u_i(x) = \lambda_i(x) u_i(x) \mbox{ and } B^{\#}(x)u_i(x)=\mu_i(x) u_i(x) ,\ \ i=1,\ldots,N, \ x\in\Omega \mbox{ a.e.},
\end{equation}
then the above bounds \eqref{tw} and \eqref{tq} are also written as :
\begin{align}
&\textbf{Lower Trace Bound L :}\ \sum_{i=1}^N \frac{b(x)(a_2-\lambda_i(x))^2}{(a_2\lambda_i(x)-b(x)\mu_i(x))} \leq N\theta_A(x)(a_2-a_1).\label{eg6}\\
&\textbf{Upper Trace Bound U :}\ \sum_{i=1}^N \frac{b(x)(\lambda_i(x) - a_1)^2}{(b(x)\lambda_i(x)-a_1\mu_i(x))} \leq N(1-\theta_A(x))(a_2-a_1). \label{eg7}
\end{align}

\subsection{Optimal Trace Bounds : $A^{\epsilon}$ and $B^{\epsilon}$ are both governed by two phase medium}\label{Sd9}
Here we assume $A^{\epsilon}$ and $B^{\epsilon}$ are given by \eqref{ta} and \eqref{tb} respectively. Then we have the following optimal trace bounds that hold on various sub-domains of Omega specified in each case.
\paragraph{Lower Trace Bound  L1 :}
Let $A^{*}(x)\in\mathcal{G}_{\theta_A(x)}$. Then by the structure of the phase space of $A^{*}$, we know that there exists a unique $\theta(x)$ with $\theta(x)\leq \theta_A(x)$ almost everywhere $x\in\Omega$ such that $A^{*}(x)\in\partial\mathcal{G}_{\theta(x)}^L$ (see Figure 4) :
\begin{equation}\label{eib}
tr\ (A^{*}(x)-a_1I)^{-1} = tr\ (\overline{A}_{\theta}(x)-a_1I)^{-1}  + \frac{\theta(x)}{(1-\theta(x))a_1}
\end{equation}
where $\overline{A}_{\theta}(x) =(a_1\theta(x) +a_2(1-\theta(x))I$.\\
\\
Then the corresponding $(A^{*},B^{\#})$ with $A^{*}\in\mathcal{G}_{\theta_A}$ satisfies  
\begin{align}\label{tt}
tr\ (B^{\#}(x)-b_1I)(\overline{A}_{\theta}(x)-a_1I)^2(A^{*}(x)-a_1I)^{-2}\geq\ & N(b_2-b_1)(1-\theta_B(x))\notag\\
&\quad+\frac{b_1(a_2-a_1)^2}{a_1^2}\theta(x)(1-\theta(x)).
\end{align}
where $\theta(x)$ is given by \eqref{eib} in terms of $A^{*}(x).$\\
\\
By eliminating $\theta$ from \eqref{tt}, the above lower bound is equivalent to the following :
\begin{align}\label{ub6}
tr\ (B^{\#}(x)-b_1I)(A^{*}(x)-a_1I)^{-2}\geq\ & \frac{N(b_2-b_1)(1-\theta_B(x))\lb a_1tr\ (A^{*}(x)-a_1I)^{-1} +1\rb^2}{(a_2 +a_1(N-1))^2}\notag\\
                                        &\quad +\frac{b_1}{a_1}\frac{\lb(a_2-a_1)\hspace{2pt} tr\ (A^{*}(x)-a_1I)^{-1} - N\rb}{(a_2 +a_1(N-1))}.
\end{align}
The above lower bound is valid and optimal in the sense of Theorem \ref{qw6} in the sub-domain $int(\Omega_{L1})$. 
\paragraph{Lower Trace Bound  L2 :}
Let $A^{*}(x)\in\mathcal{G}_{\theta_A(x)}$. Then we know that
there exists a unique $\theta(x)$ with $\theta(x)\geq \theta_A(x)$ almost everywhere $x\in\Omega$ such that
$A^{*}\in\partial\mathcal{G}_{\theta(x)}^U$ (see Fig 2) :
\begin{equation}\label{eic}
tr\ ({A^{*}(x)}^{-1}-a_2^{-1}I)^{-1} =  tr\ ({\underline{A}_{\theta}}^{-1}(x) - a_2^{-1}I)^{-1}  + (N-1)\frac{(1-\theta(x))a_2}{\theta(x)}
\end{equation}
where $\underline{A}_{\theta}^{-1}(x) = (\frac{\theta(x)}{a_1}+\frac{(1-\theta(x))}{a_2})I.$\\
\\
Then the corresponding $(A^{*},B^{\#})$ with $A^{*}\in\mathcal{G}_{\theta_A}$ satisfies :\\
\\
\textbf{Case (a):\ when $\frac{b_2}{a_2^2}\leq \frac{b_1}{a_1^2}$}
\begin{align}\label{to}
tr\ \{({A^{*}}^{-1}(x)B^{\#}(x)&{A^{*}}^{-1}(x)- \frac{b_2}{a_2^2} I)(\underline{A}_{\theta}^{-1}(x) - a_2^{-1}I)^2({A^{*}}^{-1}(x)- a_2^{-1}I)^{-2}\} \notag\\
&\qquad\geq\ N(l(x)-\frac{b_2}{a_2^2})+\frac{b_2(a_2-a_1)^2}{(a_1a_2)^2}\theta(x)(1-\theta(x))(N-1)  . 
\end{align}
\textbf{Case (b):\ when $\frac{b_2}{a_2^2}\geq \frac{b_1}{a_1^2}$}
\begin{align}\label{tn}
&tr\ \{({A^{*}}^{-1}(x)B^{\#}(x){A^{*}}^{-1}(x)- \frac{b_1}{a_1^2} I)(\underline{A}_{\theta}^{-1}(x) - a_2^{-1}I)^2({A^{*}}^{-1}(x)- a_2^{-1}I)^{-2}\}\notag\\
&\geq  N(l(x)-\frac{b_1}{a_1^2}) +\{ \frac{b_1(a_2-a_1)^2}{a_1^4}\theta(x)+2(\frac{b_2}{a_2^2}-\frac{b_1}{a_1^2})\frac{(a_2-a_1)}{a_1}\}(1-\theta(x))(N-1) 
\end{align}
where,
\begin{equation*} l(x)= \ \frac{b_1}{a_1^2}\theta_B(x) + \frac{b_2}{a_1^2}(\theta(x) - \theta_B(x)) + \frac{b_2}{a_2^2}(1-\theta(x))\end{equation*}
and $\theta(x)$ is given by \eqref{eic} in terms of $A^{*}$.\\
\\
The above lower bound is valid and  optimal in the sense of Theorem \ref{qw6} in the sub-domain $int(\Omega_{L2})$.
\paragraph{Upper Trace Bound U1 :}
Let $A^{*}(x)\in\mathcal{G}_{\theta_A(x)}$. 
Then we know that there exists a unique $\theta(x)$ with $\theta(x)\leq \theta_A(x)$ almost everywhere $x\in\Omega$ such that
$A^{*}\in\partial\mathcal{G}_{\theta(x)}^L$, satisfying \eqref{eib}. Then the corresponding $(A^{*},B^{\#})$ with $A^{*}\in\mathcal{G}_{\theta_A}$ satisfies
\begin{align}\label{hsm}
tr\ \{(\frac{b_2}{a_1}A^{*}(x)-B^{\#}(x))(\overline{A}_{\theta}(x)-a_1I)^2(A^{*}(x)- a_1 I)^{-2}\} \geq\ &N(b_2-b_1)\theta_B(x) \notag\\
                                                                                            & + N\frac{b_2}{a_1}(a_2-a_1)(1-\theta(x)).
\end{align}
The above upper bound is valid and optimal in the sense of Theorem \ref{qw6} in the sub-domain $int(\Omega_{U1})$. 
\paragraph{Upper Trace Bound U2 :}
Let $A^{*}(x)\in\mathcal{G}_{\theta_A(x)}$. Then we know that there exists a unique $\theta(x)$ with $\theta(x)\geq \theta_A(x)$ almost everywhere $x\in\Omega$ such that
$A^{*}\in\partial\mathcal{G}_{\theta(x)}^U$,  satisfying \eqref{eic}. Then the corresponding $(A^{*},B^{\#})$ with $A^{*}\in\mathcal{G}_{\theta_A}$ satisfies
\begin{align}\label{tm}
&tr\ \{(\frac{b_2a_2}{a_1^2}{A^{*}}^{-1}(x)-{A^{*}}^{-1}(x)B^{\#}(x){A^{*}}^{-1}(x))(\underline{A}_{\theta}(x)^{-1} - a_2^{-1}I)^2({A^{*}(x)}^{-1} - a_2^{-1}I)^{-2}\}\notag\\
&\geq\  N(\frac{b_2}{a_1^2}-\theta^{*}(x))+ N\frac{b_2(a_2-a_1)}{a_1^3}(1-\theta(x))- 2\frac{(b_2-b_1)(a_2-a_1)}{a_1^3}(1-\theta(x))(N-1)
\end{align}
where, 
\begin{equation*}\theta^{*}(x)= \frac{b_2}{a_1^2} + \frac{(b_1-b_2)}{a_1^2}\theta_B(x) + (\frac{b_1}{a_2^2}-\frac{b_1}{a_1^2})(1-\theta(x))\end{equation*}
The above upper bound is valid and optimal in the sense of Theorem \ref{qw6} in the sub-domain $int(\Omega_{U2})$.\\
\\
Note that, all the above lower/upper trace bounds are supplemented with the bound \eqref{Sd3}.
\begin{remark}
As we have done in the case of  \eqref{ub6},  we can express other bounds also  by eliminating $\theta(x)$ in terms of $(a_1,a_2,A^{*},b_1,b_2,\theta_A,\theta_B)$.
\hfill\qed\end{remark}
\begin{remark}
We further discover (see Remark \ref{eg3}) that for any two sequences $A^\epsilon(x)$ and $B^\epsilon(x)$ satisfying \eqref{ta} and \eqref{tb} respectively, the corresponding $H$-limit $A^\epsilon(x) \xrightarrow{H} A^{*}(x)$ in $\Omega$ and the relative limit $B^\epsilon(x)\xrightarrow{A^\epsilon(x)} B^{\#}(x)$ in $\Omega$, commute with each other locally in $\Omega$, i.e. 
\begin{equation}
A^{*}(x)B^{\#}(x) = B^{\#}(x)A^{*}(x),\quad x \mbox{ a.e. in }\Omega.
\end{equation}
That means that they have a common set of eigenvector basis, say $\{u_1(x),\ldots,u_N(x)\}$, $x\in\Omega$. If $\lambda_i(x)$ and the $\mu_i(x)$ are the eigenvalues of $A^{*}(x)$ and $B^{\#}(x)$ respectively corresponding to the common eigenvector $u_i(x)$, i.e.
\begin{equation} 
A^{*}(x)u_i(x) = \lambda_i(x) u_i(x) \mbox{ and } B^{\#}(x)u_i(x)=\mu_i(x) u_i(x) ,\ \ i=1,\ldots,N, \  x \mbox{ a.e. in }\Omega, 
\end{equation}
then the above $L1,L2,U1,U2$ bounds can be expressed in terms of eigenvalues $\{\lambda_i\}_{1\leq i\leq N}$ and $\{\mu_i\}_{1\leq i \leq N}$. For example, the L1 bound \eqref{ub6} and the U1 bound \eqref{hsm} can be written as :
\begin{align}
&\textbf{Lower Trace Bound L1 :}\notag\\
&\sum_{i=1}^N \frac{(\mu_i(x)-b_1)}{(\lambda_i(x)-a_1)^{2}}\ \geq\  \frac{N(b_2-b_1)(1-\theta_B(x))\lb a_1\sum_{i=1}^N (\lambda_i(x)-a_1)^{-1} +1\rb^2}{(a_2 +a_1(N-1))^2}\notag\\
&\qquad\qquad\qquad\qquad\qquad\qquad\qquad\qquad +\frac{b_1}{a_1}\frac{\lb(a_2-a_1)\sum_{i=1}^N (\lambda_i(x)-a_1)^{-1} - N\rb}{(a_2 +a_1(N-1))}. \label{eg10}\\
&\textbf{Upper Trace Bound U1 :}\notag\\
&\sum_{i=1}^N \frac{(\frac{b_2}{a_1}\lambda_i(x)-\mu_i(x))}{(\lambda_i(x)- a_1)^{2}}\ \geq\  \frac{ N(b_2-b_1)\theta_B(x)\lb a_1\sum_{i=1}^N (\lambda_i(x)-a_1)^{-1} +1\rb^2}{(a_2 +a_1(N-1))^2} \notag\\
&\qquad\qquad\qquad\qquad\qquad\qquad\qquad\qquad\qquad\quad + N\frac{b_2}{a_1}\frac{\lb a_1\sum_{i=1}^N (\lambda_i(x)-a_1)^{-1} +1\rb}{(a_2 +a_1(N-1))}.\label{eg5}                                                                                      
\end{align}
\end{remark}
\begin{figure}[H]
 \begin{center}
  \includegraphics[width = 13cm]{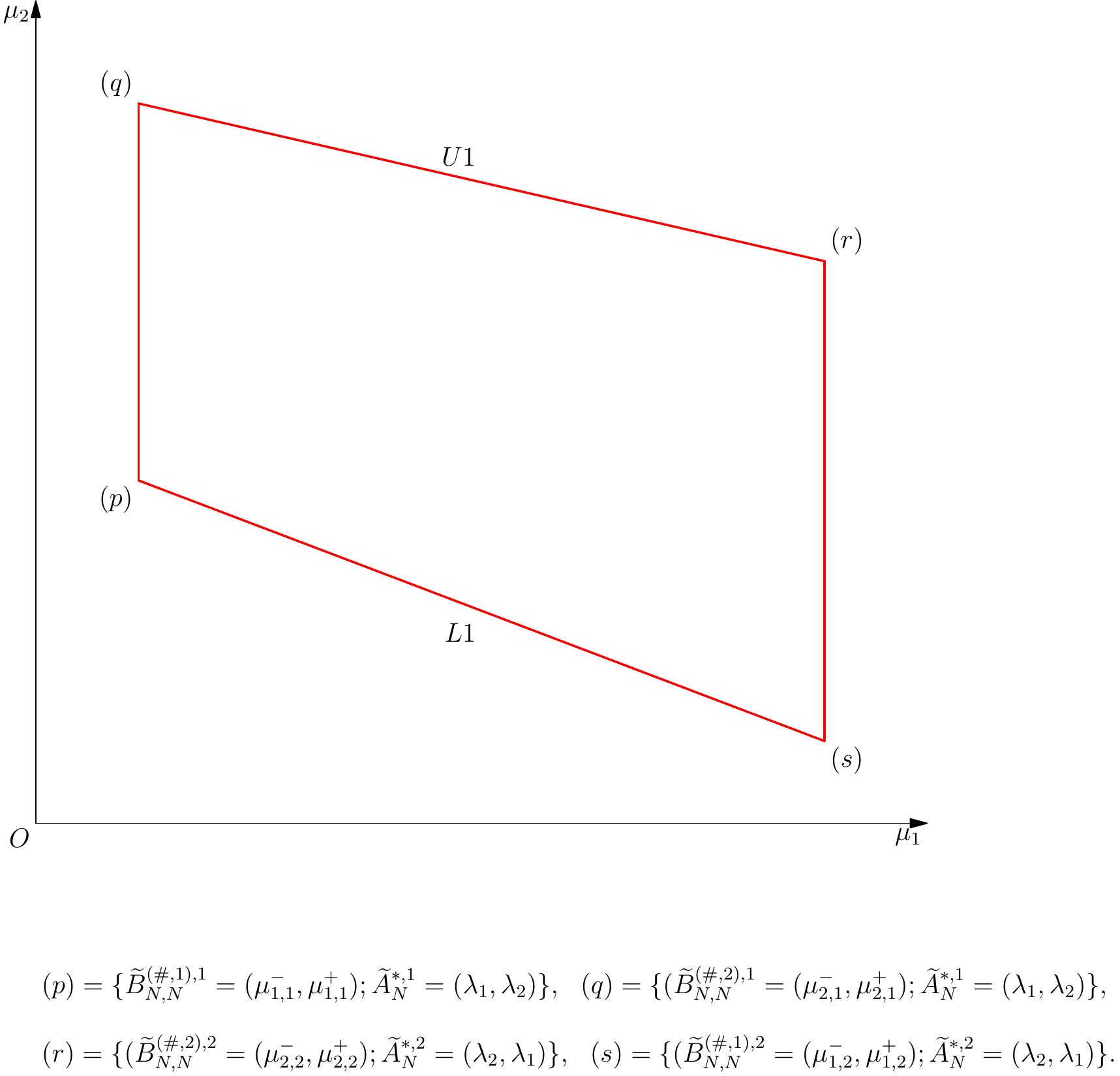}
 \end{center}
\caption{\textit{$N=2$ : $(L1,U1)$ bounds for $B^{\#}$ with eigenvalues $(\mu_1,\mu_2)$ when $(a_1,a_2,\theta_A)$, $(b_1,b_2,\theta_B)$ and $A^{*}$ with eigenvalues $(\lambda_1,\lambda_2)$  $(\lambda_1\leq\lambda_2)$ are given.}}  
\hfill\qed \end{figure}
\begin{remark}\label{eiu}
Let us consider the self-interacting case i.e. $B^\epsilon = A^\epsilon$, then we have $B^{\#}=A^{*}$ and the corresponding inequalities in lower bound L1 
(cf.\eqref{ub6}) and lower bound $L2(a)$ coincide with the classical optimal bounds
for $A^{*}$ given in \eqref{FL11}. 
On the other hand,  $U1,U2$ define two regions in the \textit{G-closure set} of $A^{*}$ 
whenever $\theta_A\leq \frac{1}{2}$ and $\theta_A> \frac{1}{2}$ respectively.
These regions do not seem to have any special significance in the classical phase diagram for $A^{*}$. For example  
the trace bound U1 says that the following inequality folds for all $A^{*}\in\underset{\theta_A\leq\frac{1}{2}}{\cup}\mathcal{G}_{\theta_A}$:
\begin{equation*}
 tr\ (A^{*}-a_1I)^{-1} + a_1tr\ (A^{*}-a_1I)^{-2}\geq \frac{(a_1tr\ (A^{*}-a_1I)^{-1} +1)}{(a_2 +a_1(N-1))}+ a_1\frac{(a_1tr\ (A^{*}-a_1I)^{-1} +1)^2}{(a_2 +a_1(N-1))^2}.
\end{equation*}
\hfill\qed\end{remark}
\begin{remark}
If $\theta_A \rightarrow 0 $ i.e. $A^\epsilon$ becomes homogeneous/ independent of $\epsilon$ then the bounds L1 and U1 imply $tr\ B^{\#} = tr\ \overline{B}$. Actually,
the matrix inequalities \eqref{bs10},\eqref{OP4} from which L1, U1 are deduced by taking 
trace, give $B^{\#}=\overline{B}$. We can see this result directly by our arguments in Section 2. 
Similar conclusions can be reached by taking $\theta_A\rightarrow 1$ in L2, U2. 
\hfill\qed\end{remark}
\begin{figure}[H]
 \begin{center}
  \includegraphics[width = 14cm]{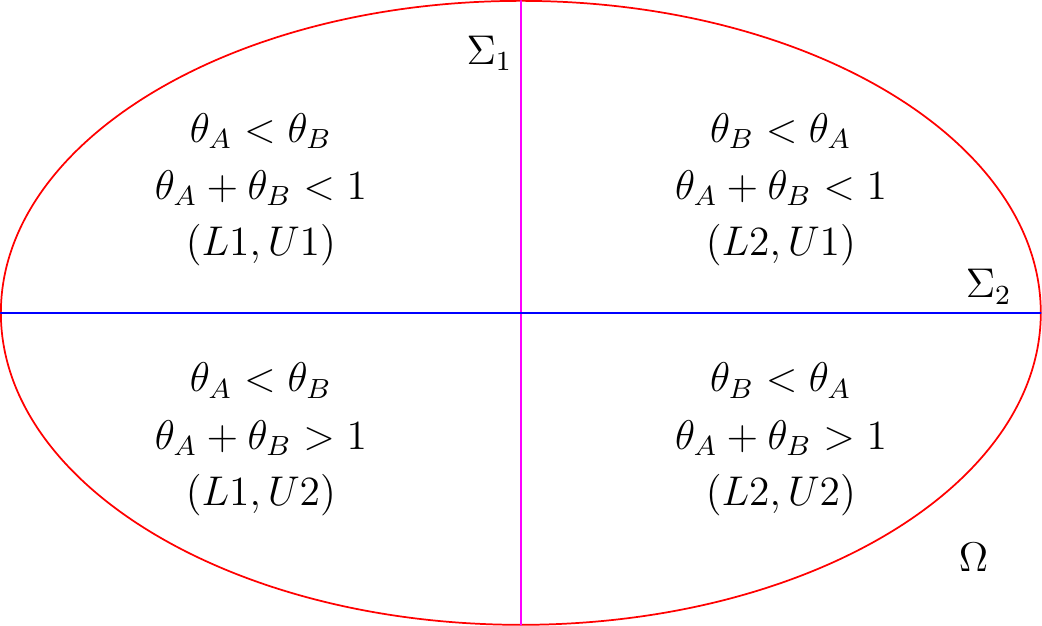}
 \caption{\textit{$N=2$ : $\Omega$ is partitioned into four disjoint sub domains by $\Sigma_1 :=\{x\in\Omega\ :\ \theta_A(x)=\theta_B(x)\}$ and $\Sigma_2 :=\{x\in\Omega\ :\ \theta_A(x)+\theta_B(x)=1\}$ which are assumed to be hypersurfaces transverse to each other as shown.}}
 \end{center}
We introduce the following four regions in the phase space of $(A^{*},B^{\#})$ :
 \begin{align*}
 (L1,U1) &:= \{(A^{*},B^{\#}) : A^{*}\in\mathcal{G}_{\theta_A} \mbox{ and } (A^{*},B^{\#})\mbox{ satisfies both L1, U1  with }\theta_A\leq\theta_B,\\ &\qquad\qquad\qquad\qquad\qquad\qquad\qquad\qquad\qquad\qquad\qquad\qquad\qquad\quad\theta_A+\theta_B\leq 1.\}\\
 (L1,U2) &:= \{(A^{*},B^{\#}) : A^{*}\in\mathcal{G}_{\theta_A} \mbox{ and } (A^{*},B^{\#})\mbox{ satisfies both L1, U2  with }\theta_A\leq\theta_B,\\ &\qquad\qquad\qquad\qquad\qquad\qquad\qquad\qquad\qquad\qquad\qquad\qquad\qquad\quad\theta_A+\theta_B> 1.\}\\
 (L2,U1) &:= \{(A^{*},B^{\#}) : A^{*}\in\mathcal{G}_{\theta_A} \mbox{ and } (A^{*},B^{\#})\mbox{ satisfies both L2, U1  with }\theta_B <\theta_A,\\ &\qquad\qquad\qquad\qquad\qquad\qquad\qquad\qquad\qquad\qquad\qquad\qquad\qquad\quad\theta_A+\theta_B\leq 1.\}\\
 (L2,U2) &:= \{(A^{*},B^{\#}) : A^{*}\in\mathcal{G}_{\theta_A} \mbox{ and } (A^{*},B^{\#})\mbox{ satisfies both L2, U2  with }\theta_B< \theta_A,\\ &\qquad\qquad\qquad\qquad\qquad\qquad\qquad\qquad\qquad\qquad\qquad\qquad\qquad\quad\theta_A+\theta_B> 1.\}
 \end{align*}
\end{figure}
\begin{remark}[Fibre-wise Convexity]\label{qw1}
We define the union of the four regions in the phase space corresponding to four quadrants in the physical domain shown in Figure 2 :  
\begin{align}\label{kab}
 \mathcal{K}_{(\theta_A,\theta_B)} =\{ (A^{*},B^{\#})\mbox{ constant matrices such that }A^{*}\in\mathcal{G}_{\theta_A} \mbox{ and } (A^{*},B^{\#}) \mbox{ is in the  }\quad &\notag\\
 \mbox{regions }(Li,Uj), \mbox{ for some }i,j=1,2
 \mbox{ with constant proportions }(\theta_A,\theta_B)\}&
\end{align}
and for $A^{*}\in\mathcal{G}_{\theta_A}$ we set 
\begin{align}\label{kfab}
 \mathcal{K}^{f}_{(\theta_A,\theta_B)}(A^{*}) =\{ B^{\#};\ (A^{*},B^{\#})\in \mathcal{K}_{(\theta_A,\theta_B)}\}.
 \end{align}
Note $\mathcal{K}^{f}_{(\theta_A,\theta_B)}(A^{*})$ is nothing but the fibre over $A^{*}$. It can be easily verified that the fibre is a closed convex set for fixed $A^{*}$ using the linearity of the bounds with respect to  $B^{\#}$.
It's not clear what sort of other convexity properties $\mathcal{K}_{(\theta_A,\theta_B)}$ possesses. Fortunately, we will be needing in the sequel only the convexity of  $\mathcal{K}^{f}_{(\theta_A,\theta_B)}(A^{*})$. 
\end{remark}
\begin{figure}[H]
 \begin{center}
  \includegraphics[width = 14cm]{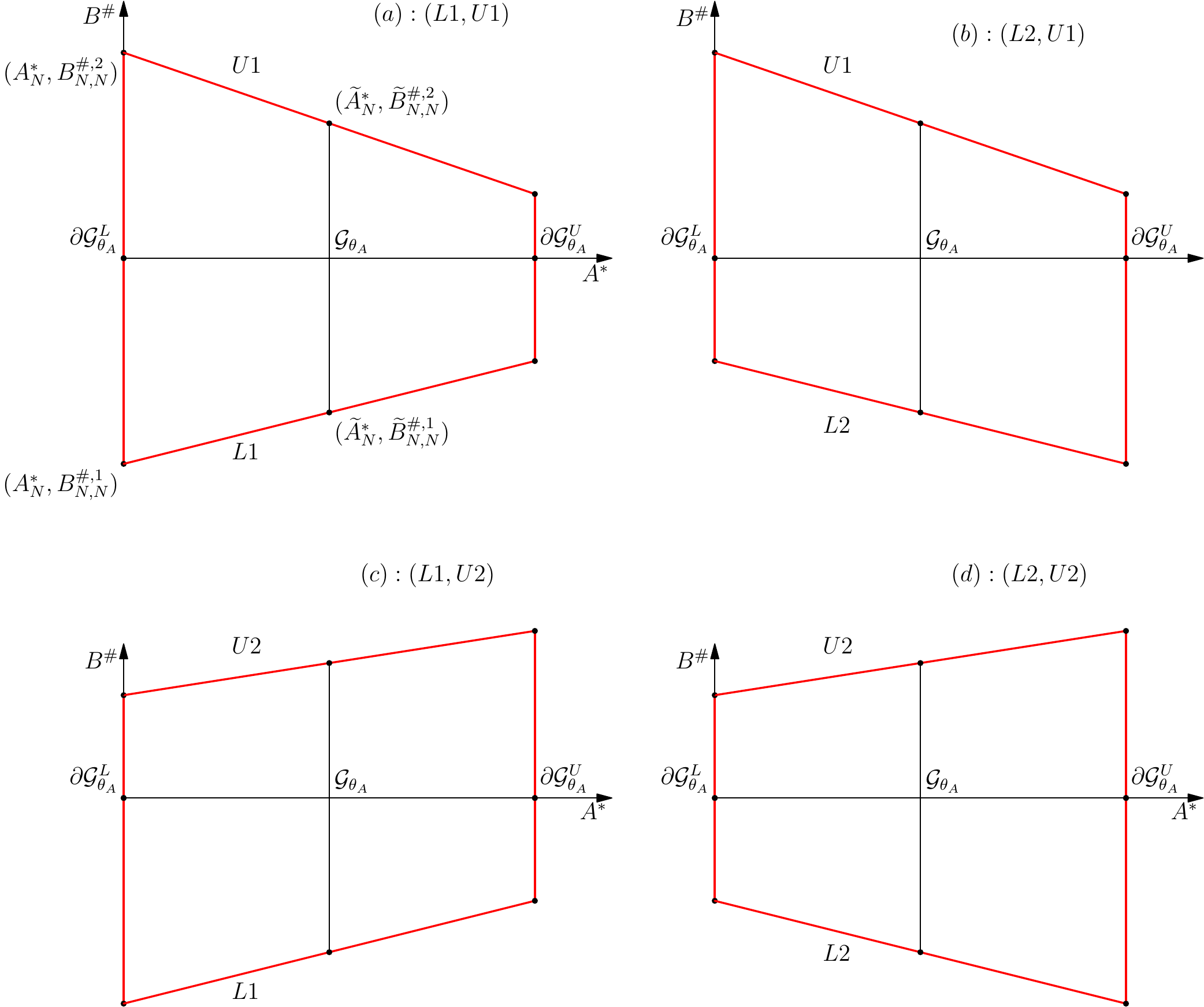}
\caption{\textit{Horizontal (resp.vertical) axis represents the phase space of 
$A^{*}$ (resp. $B^{\#}$). Some fibres standing on $A^{*}$ are shown.}
}
 \end{center}
\hfill\qed\end{figure}
\noindent
Let $\mathcal{G}_{(\theta_A(x),\theta_B(x))}$ be the set of all possible pairs of effective tensor fields $(A^{*}(x),B^{\#}(x))$ with $A^{*}(x)\in\mathcal{G}_{\theta_A(x)}$ obtained by the homogenization of two phases $\{a_1,a_2\}$ with volume fractions $\{\theta_A(x),(1-\theta_A(x))\}$ associated with $A^{*}(x)$, and $B^{\#}(x)$ the corresponding relative limit obtained with two phases $\{b_1,b_2\}$ with  volume fractions $\{\theta_B(x),(1-\theta_B(x))\}$ respectively. Then it is established in Section \ref{sil} that if  $(A^{*}(x),B^{\#}(x))\in \mathcal{G}_{(\theta_A(x),\theta_B(x))}$ then it lies in one of the regions $(Li,Uj)$ ($i,j=1,2$) pointwise, depending upon the values of $\theta_A(x)$ and $\theta_B(x)$, i.e. $\mathcal{G}_{(\theta_A(x),\theta_B(x))}\subseteq \mathcal{K}_{(\theta_A(x),\theta_B(x))}$, $x\in\Omega$ almost everywhere. The reverse inclusion is also true, which is known as optimality of the bounds. 
\begin{theorem}[Optimality]\label{qw6}
\noindent
\begin{enumerate}
\item Optimality of L1\ :\  Let $A^{*}(x)$, $B^{\#}(x)$ be symmetric positive definite matrices and $L^\infty$ functions $\theta_A(x)$, $\theta_B(x)$ be defined in domain $\Omega_1$, an open subset of $\Omega$. We assume that $A^{*}(x)\in\mathcal{G}_{\theta_A(x)}$ and the pair $(A^{*}(x),B^{\#}(x))$ satisfies equality of the bound L1 with \eqref{Sd3} for $x$ a.e. in $\Omega_1$. In addition, it is assumed that $0\leq \theta_A(x)\leq
\theta_B(x) \leq 1$ for $x$ a.e in $\Omega_1$. Then there exist sequences $A^\epsilon(x)$, $B^\epsilon(x)$, defined in $\Omega_1$ and measurable subsets  $\omega_{A^\epsilon}$, $\omega_{B^\epsilon}$  of $\Omega_1$ satisfying  \eqref{ta}, \eqref{tb} with $\omega_{A^\epsilon} \subset \omega_{B^\epsilon}$ and that $A^\epsilon\xrightarrow{H}A^{*}$ and $B^\epsilon\xrightarrow{A^\epsilon}B^{\#}$ in the domain $\Omega_1$.
\item  Statements analogous to the above one hold with regard to the optimality of other bounds L2, U1, U2.
\item  Optimality of the region $(L1,U1)$\ :\ 
 Let $A^{*}(x)$, $B^{\#}(x)$ be symmetric positive definite matrices and $L^\infty$ functions $\theta_A(x)$, $\theta_B(x)$ be defined in some domain $\Omega_{1,1}$, an open subset of $\Omega$. We assume that $A^{*}(x)\in\mathcal{G}_{\theta_A(x)}$ and the pair $(A^{*}(x),B^{\#}(x))$ satisfies the bounds L1 and U1 with \eqref{Sd3} for $x$ a.e. in $\Omega_{1,1}$. In addition, it is assumed that $0\leq \theta_A(x)\leq \theta_B(x) \leq 1$ and $0\leq \theta_A(x)+\theta_B(x)\leq 1$ for $x$ a.e in $\Omega_{1,1}$. Then there exist sequences $A^\epsilon(x)$, $B^\epsilon(x)= \widetilde{\beta_1}(x)B^\epsilon_1(x) +\widetilde{\beta_2}(x)B^\epsilon_2(x)$ with $\widetilde{\beta_1},\widetilde{\beta_2}\geq 0$ and $\widetilde{\beta_1}+\widetilde{\beta_2}=1$, defined in $\Omega_{1,1}$ and measurable subsets  $\omega_{A^\epsilon}$, $\omega_{B^\epsilon_1}$, $\omega_{B^\epsilon_2}$  of $\Omega_1$ satisfying  \eqref{ta}, \eqref{tb} with $\omega_{A^\epsilon} \subset \omega_{B^\epsilon_1}$ and $\omega_{A^\epsilon} \subset \omega^c_{B^\epsilon_2}$, and  that $A^\epsilon\xrightarrow{H}A^{*}$ and $B^\epsilon\xrightarrow{A^\epsilon}B^{\#}$ in the domain $\Omega_{1,1}$.
\item Statements analogous to the above one hold with regard to the optimality of the other regions $(L1,U2)$, $(L2,U1)$, $(L2,U2)$.
\item  Optimality of all regions taken together\ :\  
Let $A^{*}(x)$, $B^{\#}(x)$ be symmetric positive definite matrices and $L^\infty$ functions $\theta_A(x)$, $\theta_B(x)$ be defined in domain $\Omega_1$ an open subset of $\Omega$ with $0\leq \theta_A(x),\theta_B(x)\leq 1$. We assume that $A^{*}(x)\in\mathcal{G}_{\theta_A(x)}$ and the pair $(A^{*}(x),B^{\#}(x))\in \mathcal{K}_{(\theta_A(x),\theta_B(x))}$ $x$ a.e. with \eqref{Sd3} for $x$ a.e. in $\Omega_1$. Then there exist sequences $A^\epsilon(x)$, $B^\epsilon(x)$, defined in $\Omega_1$ and measurable subsets  $\omega_{A^\epsilon}$, $\omega_{B^\epsilon}$  of $\Omega_1$ satisfying  \eqref{ta}, \eqref{tb} and that $A^\epsilon\xrightarrow{H}A^{*}$ and $B^\epsilon\xrightarrow{A^\epsilon}B^{\#}$ in the domain $\Omega_1$.
\end{enumerate}
\end{theorem}
\noindent
The proof  will be presented in Section \ref{qw4}.\hfill\qed
\begin{remark}\label{ED4}
In the context of point (5) of Theorem \ref{qw6}, we can impose the volume   proportion $\delta_A$ and $\delta_B$ of the $a_1$ material and $b_1$-material in $\Omega$, i.e.
\begin{equation}\label{ED2}
\frac{|\omega_{A^\epsilon}|}{|\Omega|}=\delta_A\mbox{ and }\frac{|\omega_{B^\epsilon}|}{|\Omega|} =\delta_B,\ \forall \epsilon.
\end{equation}
Consequently, the $L^\infty(\Omega)$ weak* limit of $\chi_{\omega_{A^\epsilon}}$ and $\chi_{\omega_{B^\epsilon}}$ are $\theta_A$ and $\theta_B$ respectively and they satisfy
\begin{equation}\label{ED3}
\frac{1}{|\Omega|}\int_\Omega \theta_A = \delta_A\mbox{ and  }\frac{1}{|\Omega|}\int_\Omega\theta_B =\delta_B. 
\end{equation}
Conversely, given $\{\theta_A,\theta_B,\delta_A,\delta_B\}$ satisfying \eqref{ED3}, we can choose $A^\epsilon, B^\epsilon$ as in point (5) of Theorem \ref{qw6} and in addition they satisfy the condition \eqref{ED2}.
\begin{proof}
The proof lies in the fact that, for a given $\delta_A>0$, and a sequence of measurable sets $\{\omega_{A^\epsilon}\}_\epsilon$, $\omega_{A^\epsilon}\subset \Omega$, such that $\frac{|\omega_{A^\epsilon}|}{|\Omega|}\rightarrow \delta_A$ as $\epsilon \rightarrow 0$, then there exist a sequence of measurable sets $\widetilde{\omega}_{A^\epsilon}\subset \Omega$ such that 
$\frac{|\widetilde{\omega}_{A^\epsilon}|}{|\Omega|}=\delta_A $, $\forall \epsilon$  and $|\widetilde{\omega}_{A^\epsilon}\smallsetminus \omega_{A^\epsilon} | \rightarrow 0 $ as $\epsilon \rightarrow 0$. \\
We construct $\widetilde{\omega}_{A^\epsilon}$ in the folllowing way :  We define for each $\epsilon$, 
$\mbox{if }\delta_A - \frac{|\omega_{A^\epsilon}|}{|\Omega|} >0, \mbox{ then }$
\begin{align*}  
\widetilde{\omega}_{A^\epsilon} = \omega_{A^\epsilon} \cup T_1^\epsilon
\mbox{ where, }  T_1^\epsilon \mbox{ is mesurable and }T^\epsilon_1  \subset (\Omega\smallsetminus \omega_{A^\epsilon}), \mbox{ with }|T_1^\epsilon| = \delta_A-\frac{|\omega_{A^\epsilon}|}{|\Omega|}>0,
\end{align*}
$\mbox{and if }\frac{|\omega_{A^\epsilon}|}{|\Omega|}-\delta_A>0, \mbox{ then }$
\begin{align*}
\widetilde{\omega}_{A^\epsilon} =  \omega_{A^\epsilon}\smallsetminus T_2^\epsilon 
\mbox{ where, }  T_2^\epsilon \mbox{ is mesurable and }T^\epsilon_2 \subset  \omega_{A^\epsilon} \mbox{ with }|T_2^\epsilon| = \frac{|\omega_{A^\epsilon}|}{|\Omega|}-\delta_A>0.
\end{align*} 
Thus we have $\frac{|\widetilde{\omega}_{A^\epsilon}|}{|\Omega|}=\delta_A $, $\forall \epsilon$  and $|\widetilde{\omega}_{A^\epsilon}\smallsetminus \omega_{A^\epsilon} | \rightarrow 0 $ as $\epsilon \rightarrow 0$.
Similar construction can be done for the sequence $\frac{|\omega_{B^\epsilon}|}{|\Omega|}\rightarrow\delta_B$, to have $\frac{|\widetilde{\omega}_{B^\epsilon}|}{|\Omega|}=\delta_B $, $\forall \epsilon$  and $|\widetilde{\omega}_{B^\epsilon}\smallsetminus \omega_{B^\epsilon} | \rightarrow 0 $ as $\epsilon \rightarrow 0$.\\
\\
Now for given $A^{*}\in \mathcal{G}_{\theta_A}$ and $B^{\#}\in \mathcal{K}^f_{(\theta_A,\theta_B)}(A^{*})$ with $\theta_A,\theta_B$ are satisfying \eqref{ED3}, then following (5) of Theorem \ref{qw6} there exist $A^\epsilon$ and $B^\epsilon$ satisfying \eqref{ta} and \eqref{tb} respectively such that $A^\epsilon\xrightarrow{H}A^{*}$ and $B^\epsilon\xrightarrow{A^\epsilon}B^{\#}$ in $\Omega$. Now we consider the sequences
$$ \widetilde{A}^\epsilon = \{a_1\chi_{\widetilde{\omega}_{A^\epsilon}}+a_2(1-\chi_{\widetilde{\omega}_{A^\epsilon}})\}I \mbox{ and }\widetilde{B}^\epsilon = \{b_1\chi_{\widetilde{\omega}_{B^\epsilon}}+b_2(1-\chi_{\widetilde{\omega}_{B^\epsilon}})\}I$$
where the respective microstructures satisfies \eqref{ED2}. And as we have from our construction  
$$ ||\widetilde{A}^\epsilon - A^\epsilon||_{L^1(\Omega)} \rightarrow 0  \mbox{ and } ||\widetilde{B}^\epsilon - B^\epsilon||_{L^1(\Omega)} \rightarrow 0 \mbox{ as }\epsilon\rightarrow 0. $$
Then using the Remark \ref{sii} we have,
$\widetilde{A}^\epsilon\xrightarrow{H}A^{*}$ and  $\widetilde{B}^\epsilon\xrightarrow{\widetilde{A}^\epsilon}B^{\#}$ in $\Omega$.\\
Alternatively, one can also exploit the fact that the construction in Section \ref{qw4} produces infact $B^\epsilon$ converging $B^{\#}$ in $L^1(\Omega)$. This is stronger than the relative convergence of $B^\epsilon\xrightarrow{A^\epsilon}B^{\#}$ in $\Omega$.
\hfill\end{proof}
\end{remark}
\noindent
The above Remark \ref{ED4} together with the Theorem \ref{qw6} of optimality, plays a crucial role in the applications of problems of Calculus of Variations  discussed in Section \ref{qw5}. 
\section{Optimal Microstructures}\label{ub12}
\setcounter{equation}{0}
Before establishing the above bounds L1, L2, U1, U2 we present the analysis of $A^{*}$ and $B^{\#}$
for two important class of microstructures known as laminates and Hashin-Shtrikman constructions. Recall that, these microstructures are optimal microstructures for providing optimality of the 
bounds in the classical case i.e. for $A^{*}$ bound. We will see that they are useful for the study of $B^{\#}$ as well.
\subsection{Laminated Microstructures : Simple Laminates}\label{bs12}
We begin with the laminates. The laminate microstructures are defined when the geometry of the problem varies only in a single direction that is the 
sequence of matrices $A^{\epsilon}$ depends on a single space variable, say $x_1$, $A^{\epsilon}(x) = A^{\epsilon}(x_1)$ 
and the homogenized composite is called laminate. If the component phases are stacked in slices orthogonal to the $e_{1}$ direction, 
in that case it is a generalization of the one-dimensional settings. In particular the $H$-convergence can be reduced to the usual weak convergence 
of some combinations of entries of the matrix $A^{\epsilon}$. In effect, this yields another type of explicit formula for the homogenized matrix as 
in the one-dimensional case. Let us recall this result : let $A^{\epsilon} \in \mathcal{M}(a_1,a_2,\Omega)$ satisfy the assumption $A^{\epsilon}(x) = A^{\epsilon}(x_1)$. 
Then $A^{\epsilon}$ $H$-converges to a homogenized matrix $A^{*}$ iff the following convergences hold in  $L^{\infty}(\Omega)$ weak* : (See \cite{A})
\begin{equation*}
\begin{aligned}
&\frac{1}{A^{\epsilon}_{11}} \rightharpoonup \frac{1}{A^{*}_{11}},\quad \frac{A^{\epsilon}_{1j}}{A^{\epsilon}_{11}} \rightharpoonup \frac{A^{*}_{1j}}{A^{*}_{11}} \mbox{ for }2\leq j \leq N,\quad \frac{A^{\epsilon}_{i1}}{A^{\epsilon}_{11}} \rightharpoonup \frac{A^{*}_{i1}}{A^{*}_{11}} \mbox{ for }2\leq i \leq N,\\
&(A^{\epsilon}_{ij}- \frac{A^{\epsilon}_{1j} A^{\epsilon}_{i1}}{A^{\epsilon}_{11}}) \rightharpoonup (A^{*}_{ij}- \frac{A^{*}_{1j} A^{*}_{i1}}{A^{*}_{11}}) \mbox{ for }i\neq j\mbox{, }2\leq i \leq N\mbox{, }1\leq j\leq N. 
\end{aligned}
\end{equation*}
where $(A^{\epsilon}_{ij})_{1\leq i,j\leq N}$ and $(A^{*}_{ij})_{1\leq i,j\leq N}$ denote the entries of $A^{\epsilon}$ and $A^{*}$, respectively.\\
The oscillating test functions matrix $X^{\epsilon}$ in \eqref{dc2} can also be explicitly written in the case of laminated structures. Indeed, it is easy to check that 
\begin{equation*}
\begin{aligned}
&X^{\epsilon}_{11} = \frac{A^{*}_{11}}{A^{\epsilon}_{11}},\quad X^{\epsilon}_{1j} = \frac{A^{*}_{1j} - A^{\epsilon}_{1j}}{A^{\epsilon}_{11}}\mbox{ for }2\leq j \leq N\\
& X^{\epsilon}_{ii} = 1 \mbox{ for }2\leq i \leq N,\quad X^{\epsilon}_{ij} = 0 \mbox{ for }i\neq j\mbox{, }2\leq i \leq N\mbox{, }1\leq j\leq N. 
\end{aligned}
\end{equation*}
Then using the convergence result \eqref{dc2}, we obtain $B^{\#}$ matrix explicitly : 
\begin{equation*}
\begin{aligned}
&B^{\#}_{11} =\ (A^{*}_{11})^2 \underset{\epsilon \rightarrow 0}{lim}\ \frac{B^{\epsilon}_{11}}{(A^{\epsilon}_{11})^2},\quad  B^{\#}_{1j} =\ \underset{\epsilon \rightarrow 0}{lim}\ {B^{\epsilon}_{1j}}(\frac{A^{*}_{1j} - A^{\epsilon}_{1j}}{A^{\epsilon}_{11}} )^2 \mbox{ for }2\leq j \leq N,\\
&B^{\#}_{ii} =\ \overline{B} \mbox{ for }2\leq i \leq N,\quad B^{\#}_{ij} =\ 0 \mbox{ for }i\neq j\mbox{, }2\leq i \leq N\mbox{, }1\leq j\leq N.
\end{aligned}
\end{equation*}
(The above limits have been taken in $L^{\infty}$ weak* sense.)\\

\noindent
Now from the above result we can deduce a number of special cases of particular interest.
\paragraph{Laminated microstructures for isotropic case, i.e. $A^\epsilon=a^\epsilon I$ and $B^\epsilon=b^\epsilon I$ are scalar matrices :}
Let us assume that the matrix $A^{\epsilon}(x)$ is isotropic and depends only on $x_1$, i.e. $ A^{\epsilon}(x) = a^{\epsilon}(x_1)I$
and $B^{\epsilon}(x)=b^{\epsilon}(x)I$ is also isotropic but unlike $A^{\epsilon}$ it may depend on other variables also $x=(x_1,..,x_N)$. Then, we see
\begin{equation}\begin{aligned}\label{FG17}
&A^{*} = diag\{\hspace{1pt}\underline{a}(x_1),\overline{a}(x_1),...,\overline{a}(x_1)\}\\
&B^{\#} = diag\{\hspace{1pt}(\underline{a}(x_1))^2\ \underset{\epsilon\rightarrow 0}{lim}\frac{b^{\epsilon}(x)}{(a^{\epsilon}(x_1))^2},\ \overline{b}(x),..,\overline{b}(x)\}
\end{aligned}\end{equation}
where, $(\underline{a})^{-1}$ is the $L^{\infty}(\Omega)$ weak* limit of $ (a^{\epsilon})^{-1}$ and $\overline{a}$ and $\overline{b}$ is the $L^{\infty}(\Omega)$ weak* limit of $a^{\epsilon}$ and
$b^{\epsilon}$ respectively.\\

\noindent
If $b^{\epsilon}(x)$ is independent of $\epsilon$ say $b^{\epsilon}(x)=b(x)$, then corresponding $B^{\#}$ will be
\begin{equation}B^{\#} = diag\{\hspace{1.5pt}b(x)(\underline{a}(x_1))^2\ \underset{\epsilon\rightarrow 0}{lim}(a^{\epsilon})^{-2}(x_1),\ b(x),..,b(x)\}.\end{equation}
Following the formula \eqref{FG17}, we would like to consider few more cases as follows.
\paragraph{(a): Laminated microstructures for $a^{\epsilon}$ and $b^{\epsilon}$ both in the same direction $e_1$:}
We consider $ a^{\epsilon}(x) = a^{\epsilon}(x_1)$ and $b^{\epsilon}(x) = b^{\epsilon}(x_1)$ then from \eqref{FG17} 
\begin{equation*}B^{\#}_{11}=\  (\underline{a}(x_1))^2\ \underset{\epsilon\rightarrow 0}{lim}\frac{b^{\epsilon}}{(a^{\epsilon})^2}(x_1)\end{equation*}
$-$ exactly the formula what we have derived in $1$-dim case, and $B^{\#}_{ii} = \overline{b}(x_1)$ for $2\leq i\leq N.$\\
\textbf{(b): Laminated microstructures for $a^{\epsilon}$ and $b^{\epsilon}$ in the mutually transverse direction :}
By that we mean if $ a^{\epsilon}(x) = a^{\epsilon}(x_1)$ then we are considering $b^{\epsilon}(x) = b^{\epsilon}(x_2,..,x_N)$, i.e.independent of $x_1$ variable. 
Then 
\begin{equation*}
B^{\#}_{11} =  (\underline{a}(x_1))^2\ \overline{b}(x_2,..,x_N)\ \underset{\epsilon\rightarrow 0}{lim}\frac{1}{(a^{\epsilon})^2}(x_1)
\ \ \mbox{and } \ B^{\#}_{ii} = \overline{b}(x_2,.,x_N)\ \mbox{ for }\ 2\leq i\leq N.
\end{equation*}
Now using the fact 
\begin{equation*}\underset{\epsilon\rightarrow 0}{lim}\ \frac{1}{(a^{\epsilon})^2}(x_1) \geq \frac{1}{(\underline{a}(x_1))^2}.\end{equation*} 
We get 
\begin{equation*} B^{\#} \geq \overline{b}I.\end{equation*} 
\begin{remark}\label{Sd8}
The above property is different from that of the usual homogenized limit $B^{*}$ which is always bound above by $\overline{b}I$.
It indicates that the upper bound of $B^{\#}$ is bigger than the upper bound of $B^{*}$.
\hfill\qed\end{remark}
\noindent\textbf{(c): $a^\epsilon$ and $b^\epsilon$ are governed by two phase medium :}
Let us consider,  
\begin{equation}\begin{aligned}\label{bs17}
a^{\epsilon}(x)&= a^\epsilon(x_1) = a_1\AC + a_2(1-\AC)\mbox{ with }(a_1 < a_2 )\\
\mbox{and }\ b^{\epsilon}(x) &= b_1\B + b_2(1-\B)\mbox{ with }(b_1 < b_2)
\end{aligned}\end{equation}
where
\begin{equation*}
\AC\rightharpoonup \theta_A(x_1) \mbox{ and } \B \rightharpoonup \theta_B(x) \quad\mbox{ in }L^{\infty}(\Omega)\mbox{ weak* limit. }
\end{equation*}
Notice that just like $\theta_A$, $\theta_B$ we need a new information on the microstructures, 
i.e. $\theta_{AB}$ = $L^{\infty}$ weak* limit of $\AC\B$ and it satisfies the bounds \eqref{FG2}.
Then using \eqref{FG17} we compute the simple laminates $B^{\#}=diag\{B^{\#}_{kk}\}_{1\leq k\leq N}$. 
Following the one-dimensional case computations (cf. Section \ref{Sd6}), we have 
\begin{equation}\label{Sd7}\begin{aligned}
B^{\#}_{11} &= (\underline{a}(x_1))^2 \underset{\epsilon\rightarrow 0}{lim} \frac{b^{\epsilon}(x)}{(a^{\epsilon}(x_1))^2}\\
&=(\underline{a})^{2}\{\frac{b_2}{a_2^2} + \frac{(b_1-b_2)}{a_2^2}\theta_B + (\frac{b_2}{a_1^2} -\frac{b_2}{a_2^2})\theta_A - {(b_2-b_1)}(\frac{1}{a_1^2} -\frac{1}{a_2^2})\theta_{AB}\}\\
B^{\#}_{kk} &= b_1\theta_B + b_2(1-\theta_B)\ \ \mbox{for  }k=2,..,N. 
\end{aligned}\end{equation}
Next by using the bounds \eqref{FG2} over \eqref{Sd7}, we get the following lower bounds as :
\begin{equation}\begin{aligned}\label{lb10}
\mbox{when $\theta_A \leq \theta_B$,  }\ B^{\#}_{11}\ \geq\ &\ (\underline{a})^{2}\{\frac{b_2}{a_2^2} + \frac{(b_1-b_2)}{a_2^2}\theta_B + {b_1}(\frac{1}{a_1^2} -\frac{1}{a_2^2})\theta_A \}= L^{\#}_1 \mbox{ (say)} \\
\mbox{when $\theta_B \leq \theta_A $, }\ B^{\#}_{11}\ \geq\ &\ (\underline{a})^{2}\{\frac{b_2}{a_2^2} + \frac{(b_1-b_2)}{a_1^2}\theta_B + {b_2}(\frac{1}{a_1^2} -\frac{1}{a_2^2})\theta_A\}= L^{\#}_2 \mbox{ (say)}.
\end{aligned}\end{equation}
Similarly, the upper bounds as :
\begin{equation}\begin{aligned}\label{tp}
\mbox{when $\theta_A + \theta_B \leq 1,$ }\ B^{\#}_{11}\ \leq\ &\ (\underline{a})^{2}\{\frac{b_2}{a_2^2} + \frac{(b_1-b_2)}{a_2^2}\theta_B + {b_2}(\frac{1}{a_1^2} -\frac{1}{a_2^2})\theta_A\} = U^{\#}_1 \mbox{ (say) }\\
\mbox{when $\theta_A + \theta_B \geq 1,$ }\ B^{\#}_{11}\ \leq\ &\ (\underline{a})^{2}\{\frac{(b_2-b_1)}{a_1^2} + \frac{b_1}{a_2^2} + \frac{(b_1-b_2)}{a_1^2}\theta_B + {b_1}(\frac{1}{a_1^2} -\frac{1}{a_2^2})\theta_A\}= U^{\#}_2 \mbox{ (say). }
\end{aligned}\end{equation}           
Moreover, a simple computations shows that
\begin{equation*}  max\ \{ L^{\#}_1, L^{\#}_2\}\ \leq \  min\ \{U^{\#}_1, U^{\#}_2\}.\end{equation*}
Thus in the class of simple laminations we have obtained the bounds for $B^{\#}$ as, 
\begin{align}\label{FL10}
& min\ \{L^{\#}_1, L^{\#}_2\} \leq B^{\#}_{11}\ \leq   max\ \{U^{\#}_1, U^{\#}_2\}\\
\mbox{ and }& B^{\#}_{kk} = b_1\theta_B + b_2(1-\theta_B) \mbox{ for }k=2,..,N.\notag
\end{align}
Following the Theorem \ref{Sd10} established in one-dimensional case, the above inequality \eqref{FL10} provides an optimal bound for simple laminates.  
\begin{remark}
If $B^{\epsilon}=b(x)I$ i.e. independent of $\epsilon$, then we obtain :
\begin{equation}\label{bs11} B^{\#}=\ diag\{b(x)(\underline{a}(x_1))^2(\frac{\theta_A(x_1)}{a_1^2}+\frac{1-\theta_A(x_1)}{a_2^2}),\ b(x),..,b(x)\}.\end{equation}
Similarly if 
\begin{equation*}B^{\epsilon}=b^\epsilon(x_2,..,x_N)I=\ \{b_1\chi_{\omega_{B^\epsilon}}(x_2,..,x_N)+ b_2(1-\chi_{\omega_{B^\epsilon}}(x_2,..,x_N))\}I\ \ \mbox{ (independent of $x_1$), }\end{equation*}
then as 
\begin{equation*}\chi_{\omega_{B^\epsilon}}(x_2,..,x_N)\chi_{\omega_{A^\epsilon}}(x_1) \rightharpoonup \theta_B(x_2,..,x_N)\theta_A(x_1) \ \ L^{\infty}(\Omega)\mbox{ weak* }; \end{equation*}
We obtain : 
\begin{equation}\begin{aligned}\label{lb6}
B^{\#}_{11}&=\ (\underline{a})^{2}\{\frac{b_2}{a_2^2} + \frac{(b_1-b_2)}{a_2^2}\theta_B + {b_2}(\frac{1}{a_1^2} -\frac{1}{a_2^2})\theta_A - {(b_2-b_1)}(\frac{1}{a_1^2} -\frac{1}{a_2^2})\theta_B\theta_A\}\\
\mbox{and }\ B^{\#}_{kk}&=\ b_1\theta_B + b_2(1-\theta_B),\ \ k=2,..,N. 
\end{aligned}
\end{equation}
Note that, $min \{L^{\#}_1,L^{\#}_2\}<B^{\#}_{11} < max \{U^{\#}_1,U^{\#}_2\}$. Thus, the inequality \eqref{FL10} are strict in this case. 
\hfill\qed
\end{remark}
\subsection{Sequential Laminates}\label{Sd19}
Let us mention a subclass of laminated homogenized tensor which is known as sequential laminates. (See \cite[Section 2.2]{A}).
\begin{example}[Rank-$p$ Sequential Laminates]
Let $\{e_i\}_{1\leq i \leq p}$ be a collection of unit vectors in $\mathbb{R}^{N}$ and $\{\theta_i\}_{1\leq i \leq p}$ 
the proportions at each stage of the lamination process, i.e. first we laminate $a_1$ and $a_2$ in $e_1$ direction 
with the proportion $\theta_1$ and $(1-\theta_1)$ respectively to get $A^{*}_1$ then 
we laminate $A^{*}_1$ and $a_2$ in $e_2$ direction with the proportion $\theta_2$ and $(1-\theta_2)$ 
respectively to get $A^{*}_2$ and this is repeated $p$ times.\\
\\
(a): We then have the following formulas from \cite{A} for a rank-p sequential laminate with matrix $a_2 I$ and core $a_1 I$.
\begin{equation*}(\prod_{j=1}^{p} \theta_j)( A^{*}_p - a_2 I)^{-1} = (a_1 - a_2)^{-1}I + \sum_{i=1}^{p}\lb (1-\theta_i)(\prod_{j=1}^{i-1} \theta_j)\rb\frac{(e_i\otimes e_i)}{a_2}.\end{equation*}
(b): For rank-p sequential laminate with matrix $a_1 I$ and core $a_2 I$, we have
\begin{equation*}(\prod_{j=1}^{p} (1-\theta_j))( A^{*}_{p} - a_1 I)^{-1} = (a_2 - a_1)^{-1}I + \sum_{i=1}^{p}\lb \theta_i(\prod_{j=1}^{i-1}(1-\theta_j)\rb\frac{(e_i\otimes e_i)}{a_1}.\end{equation*}
\hfill\qed\end{example}
\noindent
The following lemma from \cite{A} is important for proving the saturation / optimality of the bound of $A^{*}$. 
\begin{lemma}\cite{A}
Let  $\{e_i\}_{1\leq i \leq p}$ be a collection of unit vectors. Let $\theta_A \in (0,1)$. 
Now for any collection of non-negative real numbers  $\{m_i\}_{1\leq i \leq p}$ satisfying $\sum_{i=1}^{p} m_i =1 $, 
there exists a rank-$p$ sequential laminate $A^{*}_p$ with matrix $a_2 I$ and core $a_1 I$ in proportion $(1-\theta_A)$ 
and $\theta_A$ respectively and with lamination directions $\{e_i\}_{1\leq i \leq p}$ such that
\begin{equation}\label{OP3}
\theta_A( A^{*}_{p} - a_2 I)^{-1} = (a_1 - a_2)^{-1}I + (1-\theta_A)\sum_{i=1}^{p} m_i\frac{e_i\otimes e_i}{a_2 (e_i\cdot e_i)}.
\end{equation}
\noindent
An analogous result holds when the roles of $a_2$ and $a_1$ (in proportions $(1-\theta_A)$ and $\theta_A$ respectively) 
in the lemma above are switched. The formula above is replaced by
\begin{equation}\label{OP2}
(1-\theta_A)(A^{*}_p - a_1 I)^{-1} = (a_2 - a_1)^{-1}I + \theta_A\sum_{i=1}^{p} m_i\frac{e_i\otimes e_i}{a_1( e_i\cdot e_i)}.
\end{equation}
\hfill\qed\end{lemma}

In order to define the sequential laminates for $B^{\#}$, we assume that $A^{\epsilon}$ is governed with 
$p-$sequential laminate microstructures, (say $A^\epsilon_p$) and by considering any $B^{\epsilon}(x)$ independent of $\epsilon$
one defines the limit matrix say $B^{\#}_p$. We call it  a quasi-sequential laminate $B^{\#}_p$ which corresponds to the
sequential laminate $A^{*}_p$. Now we consider $B^\epsilon$ with two-phases and with a microgeometry with layer corresponding to each layer of $p-$sequential laminated microstructures of $A^\epsilon_p$ 
there is an associated layered microstructure for $B^\epsilon$ also (say $B^\epsilon_p$).  
Then we define the corresponding limit matrix as $B^\epsilon_p\xrightarrow{A^\epsilon_p} B^{\#}_{p,p}$, by saying $(p,p)-$ sequential laminates.\\
One possible way to get these $B^{\#}_p$  would be finding the corrector matrix $X^{\epsilon}$ corresponding to $A^\epsilon_p\xrightarrow{H} A^{*}_p$  
(i.e.$(X^{\epsilon})^t A^{\epsilon}_p X^{\epsilon} \rightharpoonup A^{*}_p $ in $\mathcal{D}^{\prime}(\Omega)$) 
and then apply the convergence result \eqref{dc2} to get the the limit matrix $B^{\#}_p$.
It is known that for $p\geq2$, getting corrector matrix $X^{\epsilon}$ is not easy. Briane \cite{B}
gave an iteration procedure to obtain $X^{\epsilon}$. Getting the
expression for $B^{\#}_p, B^{\#}_{p,p}$ is even more complicated.
Fortunately, one gets an explicit expression for $B^{\#}_{p,p}$ through $H$-measure techniques which can be  
carried out in four different cases which are optimal for our bounds L1, L2, U1, U2. We postponed the construction of $B^{\#}_p, B^{\#}_{p,p}$ to Section \ref{ts} because we need $H$-measure, a tool which we have not yet introduced. Here we merely give expression for these two matrices. 
\begin{example}
(a) Let us consider $B^{\epsilon}$ is independent of $\epsilon$ say $B^{\epsilon}= b(x)I$ with $b\in L^{\infty}(\Omega)$ and bounded below by a positive constant. 
In this case we define the $p$-sequential laminates $B^{\#}_p$, with matrix $a_1 I$ and core $a_2 I$ for $A^{*}_p$ (cf.\eqref{OP2})  as follows :
\begin{equation*}b(\overline{A}-A^{*}_p)(B^{\#}_p-bI)^{-1}(\overline{A}-A^{*}_p) =\  \theta_A(1-\theta_A)(a_2 -a_1)^2 (\sum_{i=1}^{p} m_i\frac{e_i\otimes e_i}{e_i.e_i}) ;\ \mbox{ with }\sum_{i=1}^p m_i =1.\end{equation*}
(b) By considering $B^{\epsilon}=b^{\epsilon}I$ defined in \eqref{bs17}, we define the
$(p,p)-$sequential laminates $B^{\#}_{p,p}$, whenever $\omega_{A^{\epsilon}},\omega_{B^{\epsilon}}$ be the $p$-sequential laminate
microstructures with $\omega_{A^{\epsilon}}\subseteq \omega_{B^{\epsilon}}$ in the same directions $\{e_i\}_{1\leq i\leq p}$ and with
matrix $a_1 I$ and core $a_2 I$ for $A^{*}_p$ (cf. \eqref{OP2}) as follows :
\begin{align}\label{bs16}
&\{\frac{(\overline{B}-b_1I)}{(\overline{A}-a_1I)}(A^{*}_p-a_1I)+\frac{b_1}{a_1}(\overline{A}-A^{*}_p)\}(B^{\#}_{p,p}-b_1I)^{-1}\{\frac{(\overline{B}-b_1I)}{(\overline{A}-a_1I)}(A^{*}_p-a_1I)+\frac{b_1}{a_1}(\overline{A}-A^{*}_p)\}\notag\\
&=\ (\overline{B}-b_1 I)+\frac{b_1(a_2-a_1)^2}{a_1^2}\theta_A(1-\theta_A)(\sum_{i=1}^{p} m_i\frac{e_i\otimes e_i}{e_i.e_i});\ \mbox{ with }\sum_{i=1}^p m_i =1.
\end{align}
\hfill\qed\end{example}
\subsection{Hashin-Shtrikman Constructions}\label{hsl}
In this section, sequences are indexed by $n$ (not by $\epsilon$) and $n\rightarrow\infty$. 
Before we move into finding bounds on $B^{\#}$, let us mention another important class of microstructures which is known as Hashin-Shritkman microstructures.
In the beginning of the theory of homogenization it played a very crucial role to provide the bounds on the two-phase isotropic medium without even development
of $H$-convergence and so on.
We follow \cite[chapter 25]{T} to define the Hashin-Shtrikman microstructures. 
\begin{definition}\label{lb14}
Let $\omega \subset \mathbb{R}^N$ be a bounded open with Lipschitz boundary.
Let $A_{\omega}(x)=[a^{\omega}_{kl}(x)]_{1\leq k,l\leq N} \in \mathcal{M}(a_1,a_2;\ \omega)$  
be such that after extending $A_{\omega}$ by $A_{\omega}(y) = M$ for $y \in \mathbb{R}^N\smallsetminus \omega$ where $M \in L_{+}(\mathbb{R}^N ; \mathbb{R}^N)$
(i.e. $M = [m_{kl}]_{1\leq k,l \leq N}$ is a constant positive definite $N\times N$ matrix), 
if for each $\lambda \in \mathbb{R}^N$ there exists a $w_{\lambda}\in  H^{1}_{loc}(\mathbb{R}^N)$ satisfies
\begin{equation}\label{hsw}
- div (A_{\omega}(y)\nabla w_{\lambda}(y)) = 0 \quad\mbox{in }\mathbb{R}^N,\quad  w_{\lambda}(y) = \lambda\cdot y \quad\mbox{in }\mathbb{R}^N \smallsetminus \omega;
\end{equation}
Then $A_{\omega}$ is said to be \textit{equivalent} to $M$.
\end{definition}
Then one uses a sequence of Vitali coverings of $\Omega$ by reduced copies of $\omega,$
\begin{equation}\label{hso}  meas(\Omega \smallsetminus \underset{p\in K}{\cup}(\epsilon_{p,n}\omega + y^{p,n}) = 0, \mbox{ with } \kappa_n = \underset{p\in K}{sup}\hspace{2pt} \epsilon_{p,n}\rightarrow 0\end{equation}
for a finite or countable $K$. These define the microstructures in $A^{n}$. One defines for almost everywhere $x\in \Omega$ 
\begin{equation}\label{tl} A^{n}_{\omega}(x) = A_{\omega}(\frac{x - y^{p,n}}{\epsilon_{p,n}}) \mbox{ in } \epsilon_{p,n}\omega + y^{p,n},\quad p\in K\end{equation}
which makes sense since for each $n$ the sets $\epsilon_{p,n}\omega + y^{p,n},\ p\in K$ are disjoint.
The above construction \eqref{tl} represents the so-called Hashin-Shtrikman microstructures.\\
\\
The $H$-limit of the entire sequence $A^{n}_{\omega}(x)$ exist and in particular we have (see \cite{T})
\begin{equation*}A^{n}_{\omega} \xrightarrow{H} M .\end{equation*}
It can be seen as follows: One defines $v^{n} \in H^1(\Omega)$ by
\begin{equation*}
v^{n}(x) = \epsilon_{p,n}w_{\lambda}(\frac{x-y^{p,n}}{\epsilon_{p,n}})+ \lambda\cdot y^{p,n},\quad\mbox{ in  }\epsilon_{p,n}\omega + y^{p,n}
\end{equation*}
which satisfies 
\begin{equation}\begin{aligned}\label{hsc}
& v^{n}(x) \rightharpoonup \lambda\cdot x \mbox{ weakly in }H^1(\Omega;\mathbb{R}^N)\\
\mbox{and}\ \ & A^{n}_{\omega}\nabla v^{n}\rightharpoonup M\lambda \mbox{ weakly in } L^2(\Omega;\mathbb{R}^N).
\end{aligned}\end{equation}
We have the following integral representation of the homogenized matrix $M$ 
\begin{equation*} Me_k\cdot e_l = \frac{1}{|\omega|}\int_{\omega} A_{\omega}(y)\nabla w_{e_k}(y)\cdot e_l\ dy  = \frac{1}{|\omega|}\int_{\omega} A_{\omega}(y)\nabla w_{e_k}\cdot\nabla w_{e_l}\ dy\end{equation*}
where $w_{e_k}, w_{e_l}$ are the solution of \eqref{hsw} for $\lambda= e_k$ and $\lambda=e_l$ respectively.
\begin{example}[Spherical Inclusions in two-phase medium]\label{tk}
If $\omega=B(0,1)$ a ball of radius one, and
\begin{equation}\begin{aligned}\label{sia}
A_{\omega}(y)= a_B(r)I &=\ a_1 I \quad\mbox{if } |y| \leq R\\
     &=\ a_2 I \quad\mbox{if }  R < |y| \leq 1
\end{aligned}
\end{equation}
with the volume proportion $\theta_A = R^N$ for $a_1I.$ In the literature, $\{a_1,a_2\}$
are called core-coating values respectively.\\
\\
Then $A_{\omega}$ is equivalent to $m I$ or $A^{n}_{\omega} \xrightarrow{H} m I$, where $A^{n}_{\omega}$ is defined as \eqref{tl} and $m$ satisfies \cite{HS} 
\begin{equation}\label{sim}
\frac{m - a_2}{m + (N-1)a_2} = \theta_A \frac{a_1 - a_2}{a_1 + (N-1)a_2}.
\end{equation} 
\hfill\qed\end{example}
\begin{example}[Elliptical Inclusions in two-phase medium]\label{tj}
For $m_1,..,m_N\in \mathbb{R}$, and $\rho + m_j >0 $ for $j=1,..,N$, the 
family of confocal ellipsoids $S_\rho$ of equation
\begin{equation*} \sum_{j=1}^N \frac{y^2_j}{\rho_2 + m_j} = 1,\end{equation*}
defines implicitly a real function $\rho$, outside a possibly degenerate ellipsoid in
a subspace of dimension $<$ $N$.\\
Now if we consider $\omega=E_{\rho_2+m_1,..,\rho_2+m_N}= \{ y\ | \ \sum_{j=1}^N \frac{y^2_j}{\rho_2 + m_j} \leq 1\},$ 
with $\rho_2 + \underset{j}{min}\ m_j >0 $  and  
\begin{align*}
 A_\omega(y)= a_E(\rho)I &=\ a_1 I \quad\mbox{if } \rho \leq \rho_1\\
     &=\ a_2 I \quad\mbox{if }  \rho_1 < \rho \leq \rho_2
\end{align*}
then $A_\omega$ is equivalent to a constant diagonal matrix $\Gamma= [\gamma_{jj}]_{1\leq j\leq N}$ satisfying
\begin{equation*}\sum_{j=1}^N \frac{1}{a_2 - \gamma_{jj}} =\ \frac{(1-\theta_A)a_1 + (N+\theta_A-1)\beta}{\theta_A a_2(a_2- a_1)}\ \mbox{ where }\theta_A = \underset{j}{\varPi}\ \sqrt{\frac{\rho_1 + m_j}{\rho_2 + m_j}}.\end{equation*}
\hfill\qed
\end{example}

Let us assume that $B^n_{\omega}$ is consistent with $A^n_{\omega}$ in the sense that both are defined w.r.t. the same Vitali covering of $\Omega$ (cf. \eqref{hso}).
We consider a $B_{\omega}=[b^{\omega}_{kl}]\in \mathcal{M}(b_1,b_2;\ \omega)$ and define a sequence $B^n_{\omega} \in \mathcal{M}(b_1,b_2;\ \Omega)$ by 
\begin{equation*} B^{n}_{\omega}(x) = B_{\omega}(\frac{x - y^{p,n}}{\epsilon_{p,n}}) \quad\mbox{ in } \epsilon_{p,n}\omega + y^{p,n}\end{equation*}
Then we want to find the corresponding $B^{\#}$ for this Hashin-Shtrikman microstructures.\\
As long as we have the corrector results \eqref{hsc} then by using the convergence result \eqref{dc2}
we define the $B^{\#}$ as follows : For each $\lambda\in \mathbb{R}^N$
\begin{equation*}
B^n_{\omega}\nabla v^{n}\cdot \nabla v^{n} \rightharpoonup B^{\#}\lambda\cdot\lambda \quad\mbox{in }\mathcal{D}^{\prime}(\Omega).
\end{equation*}
Moreover, by taking any $\varphi\in \mathcal{D}(\Omega)$
\begin{align*}
\frac{1}{|\Omega|}\int_{\Omega} B^{n}_{\omega}\nabla v^{n}\cdot \nabla v^{n}&\varphi(x)dx\\
=\ &\frac{1}{|\Omega|}\sum_{p\in K} \int_{\epsilon_{p,n}\omega +y^{p,n}} B_{\omega}(\frac{x - y^{p,n}}{\epsilon_{p,n}})\nabla w_{k}(\frac{x-y^{p,n}}{\epsilon_{p,n}})\cdot\nabla w_{l}(\frac{x-y^{p,n}}{\epsilon_{p,n}})\ \varphi(x)dx\\
=\ &\frac{1}{|\Omega|}\sum_{p\in K} \epsilon_{p,n}^N. \int_{\omega}B_{\omega}(y)\nabla w_{e_k}(y)\cdot\nabla w_{e_l}(y)\ \varphi(\epsilon_{p,n}y+ y^{p,n}) dy\\ 
\rightarrow\ & \lb\frac{1}{|\omega|}\int_{\omega}B_{\omega}(y)\nabla w_{e_k}(y)\cdot\nabla w_{e_l}(y) dy\rb.\lb\frac{1}{|\Omega|}\int_{\Omega}\varphi(x)  dx\rb.
\end{align*}
Thus we have the following integral representation of $B^{\#}$ :
\begin{equation}\label{hse}
 B^{\#}e_k\cdot e_l = \frac{1}{|\omega|}\int_{\omega} B_{\omega}(y)\nabla w_{e_k}(y)\cdot \nabla w_{e_l}(y)\ dy . 
\end{equation}
\begin{remark}
 Note that $B^{\#}$ is a constant matrix just as $A^{*}$ is. However, it is not clear whether $B^{\#}$ can be obtained using the idea of equivalence 
 in the sense of Definition \ref{lb14}.
\hfill\qed\end{remark}
Now we consider the previous Example \ref{tk} of spherical inclusion where we know $w_{e_l}$ explicitly and will find out $B^{\#}$ explicitly in various cases.
We seek the solution of \eqref{hsw} for the spherical inclusion i.e. when $A_\omega(y)=a_{B(0,1)}(y)I$ is given by \eqref{sia}, in the form of 
\begin{equation}\label{sis} w_{e_l}(y) =\ y_lf(r),\ y\in B(0,1);\end{equation}
where, $f(r)$ is of the form of 
\begin{equation}\begin{aligned}\label{sif}
f(r) &=\ \widetilde{b_1} \ \mbox{ if }r< R,\\
 &=\ \widetilde{b_2} + \frac{\widetilde{c}}{r^N} \ \mbox{ if }R < r < 1\\
 &=\ 1 \ \mbox{ if }1 < r.
\end{aligned}\end{equation}
In order to keep the solution $w_{e_l}(y)$ and flux $a(r)(f(r)+rf^{\prime}(r))$ to be continuous across
the inner boundary $(r=R)$ and the outer boundary $(r=1)$
we have these following conditions to satisfy
\begin{equation}\begin{aligned}\label{ad5}
\widetilde{b_1} =\ \widetilde{b_2} + \frac{\widetilde{c}}{r_1^N}, &\ \mbox{ and }\ a_1\widetilde{b_1} =\ a_2(\widetilde{b_2} + \frac{(1-N)\widetilde{c}}{r_1^N} )\\
\widetilde{b_2} + \widetilde{c} =\ 1, &\ \mbox{ and }\ a_2(\widetilde{b_2} + (1-N)\widetilde{c}) =\ m. \\
\end{aligned}\end{equation}
Then solving $(\widetilde{b_1},\widetilde{b_2},\widetilde{c})$ in terms of $(a_1, a_2, \theta_A)$ from the first three equation of \eqref{ad5},  we have
\begin{equation}\label{sic}
\widetilde{b_1} = \ \frac{Na_2}{(1-\theta_A)a_1 + (N+\theta_A -1)a_2},\quad \widetilde{b_2} = \ \frac{(1-\widetilde{b_1}\theta_A)}{(1-\theta_A)} \quad\mbox{and}\quad \widetilde{c} = \frac{(\widetilde{b_1} -1)\theta_A}{(1-\theta_A)}
\end{equation}
and finally putting it into the fourth equation of \eqref{ad5}, $`m$' can be written as in \eqref{sim}.\\
\\
Based on this, next we derive the expression of $B^{\#}$ for various cases.
\paragraph{(1): $B_{B(0,1)}(x)I = b I$ for some constant $b >0$ : }
In this case $B^{n}_{B(0,1)}$ becomes independent of $n$ and equal to $bI$.
 $A_\omega= a_{B(0,1)}I$ defined as in \eqref{sia}, i.e. considering $a_1$ as a core and $a_2$ as a coating. 
Then from \eqref{hse} it follows that $B^{\#} = b^{\#}I$ with
\begin{equation*} b^{\#} =  b\int_{B(0,1)} \nabla w_{e_l}(x)\cdot \nabla w_{e_l}(x) dx. \end{equation*}
Now as we see, 
\begin{align*}
m &=\ \int_{B(0,1)} a_{B(0,1)}(x)\nabla w_{e_l}(x)\cdot\nabla w_{e_l}(x) dx \\
  &=\ a_2 \int_{B(0,1)} \nabla w_{e_l}(x)\cdot\nabla w_{e_l}(x) dx + (a_1 -a_2) \int_{B(0,R)} \nabla w_{e_l}(x)\cdot\nabla w_{e_l}(x) dx
  \end{align*}
or, by using \eqref{sif} one gets
\begin{equation}\label{siw} \int_{B(0,1)} \nabla w_{e_l}(x)\cdot\nabla w_{e_l}(x) dx  =\ \frac{m - (a_1 -a_2){\widetilde{b_1}}^2 \theta_A}{a_2}. \end{equation}
Plugging this, in the expression of $b^{\#}$ and using \eqref{sim} and \eqref{sic}, finally one gets
\begin{equation}\label{hsb} b^{\#} =\ b\ [\ 1 + \frac{N\theta_A(1-\theta_A)(a_2-a_1)^2}{((1-\theta_A)a_1 + (N+\theta_A-1)a_2)^2}\ ]. \end{equation}                           
Similarly, changing the role of $a_1$ and $a_2$ (make $a_2$ as core and coated with $a_1$) while keeping the volume fraction $\theta_A$ fixed for $a_1$, and  
say $m_{*}$ be the new homogenized coefficient for $A^{n}$ : 
\begin{equation}\label{FL9} \frac{m_{*} - a_1}{m_{*} + (N-1)a_1} = (1-\theta_A) \frac{a_2 - a_1}{a_2 + (N-1)a_1}.\end{equation}
In this case with core $a_2 I$ and matrix $a_1 I$, the expression for $B^{\#}= b^{\#}I$ becomes  
\begin{equation}\label{hsd} b^{\#} =\ b\ [\ 1 + \frac{N\theta_A(1-\theta_A)(a_2-a_1)^2}{(\theta_Aa_2 + (N-\theta_A)a_1)^2}\ ]. \end{equation}          
It is well known that both of these formula \eqref{sim} and \eqref{FL9} provide the saturation / optimality of the upper bound and lower bound respectively
for the two-phase $(a_1, a_2, \theta_A)$ homogenization. In a same spirit we will show the formula \eqref{hsb} and \eqref{hsd} provide the 
saturation/optimality of the lower bound and upper bound for $B^{\#}$ respectively when $B^{\epsilon}=bI$ independent of $\epsilon$.
\paragraph{(2): $B_{B(0,1)}$ is governed with two-phase medium :} 
Let us consider, for any measurable set $\omega_B\subset B(0,1)$ 
\begin{equation}\begin{aligned}\label{sib}
B_{B(0,1)}(y)= b_B(y)I &=\ b_1 I \quad\mbox{if } y\in\omega_B \\
     &=\ b_2 I \quad\mbox{if }y\in B(0,1)\smallsetminus \omega_B  
\end{aligned}
\end{equation}
with the volume proportion $\theta_B = |\omega_B|$ for $b_1I.$\\
\\
Next we consider few straight forward cases only with $\omega_A =B(0,R)$ with $0<R<1$ (spherical inclusion, cf.\eqref{sia}),
which are indeed very useful for showing the saturation / optimality of the bounds $(A^{*},B^{\#})$ announced in Section \ref{Sd9}.
\paragraph{(2a): Core $a_1 I$ with coating $a_2 I$ for $A_{B(0,1)}$ and core $b_1 I$ with coating $b_2 I$ for $B_{B(0,1)}$, with $\omega_B\subseteq \omega_A$ :}
Using the integral representation \eqref{hse} for $B^{\#} = b^{\#}I$ we get
\begin{align*}
b^{\#} &=\  \int_{B(0,1)} \{b_1\chi_{\omega_B}(x) + b_2(1-\chi_{\omega_B}(x))\}\ \nabla w_{e_l}(x)\cdot \nabla w_{e_l}(x) dx \\
       &=\   b_2\int_{B(0,1)} \nabla w_{e_l}(x)\cdot \nabla w_{e_l}(x) dx  + (b_1-b_2)\int_{\omega_B}\nabla w_{e_l}(x)\cdot\nabla w_{e_l}(x) dx\\
       &=\ \frac{b_2}{a_2}\{m- (a_1-a_2){\widetilde{b_1}}^2\theta_A\} + (b_1 - b_2) {\widetilde{b_1}}^2\theta_B,\ \ \mbox{ as $\omega_B\subseteq\omega_A$; with using }\eqref{sif} \mbox{ and } \eqref{siw}.
\end{align*}
Finally by using the expression of $m$ (cf. \eqref{sim}) and $\widetilde{b_1}$ (cf. \eqref{sic}) we get,
\begin{equation}\label{FG6}
b^{\#} =\  b_2\ [\ 1 + \frac{N\theta_A(1-\theta_A)(a_2-a_1)^2}{((1-\theta_A)a_1 + (N+\theta_A-1)a_2)^2}\ ] - \frac{(b_2-b_1)(Na_2)^2\theta_B}{((1-\theta_A)a_1 + (N+\theta_A-1)a_2)^2}.                          
\end{equation}
First we notice that for $N=1$ the above formula becomes identical with the lower bound $l^{\#}_2$ for $\theta_B\leq\theta_A$ case in \eqref{FL20}. 
It goes same for higher dimension also.\\
Similar formula can be derived by changing the core and coating for $A_{B(0,1)}$ or $B_{B(0,1)}$ or both, while keeping
the volume fractions $\theta_A,\theta_B$ same as before.\\
\\
\textbf{(2b): Core $a_2 I$ with coating $a_1 I$ for $A_{B(0,1)}$ and core $b_2 I$ with coating $b_1 I$ for $B_{B(0,1)}$, with $\omega_A \subseteq \omega_B$ :}
Compared to the previous case $(2a)$, here we are changing the role of $a_1$ and $a_2$  while keeping the volume fraction $\theta_A$ fixed for $a_1$,
also the role of $b_1$ ,$b_2 $ by keeping the volume fraction $\theta_B$ fixed for $b_1$,  with $\omega_A \subseteq\omega_B$ (or, $\omega_B^c \subseteq \omega_A^c$).  
Then doing the above mentioned changes in the above formulation \eqref{FG6} we simply get,
\begin{equation}\label{ED1}
b^{\#} =\  b_1\ [\ 1 + \frac{N\theta_A(1-\theta_A)(a_2-a_1)^2}{(\theta_A a_2 + (N-\theta_A)a_1)^2}\ ] - \frac{(b_1-b_2)(Na_1)^2(1-\theta_B)}{(\theta_A a_2 + (N-\theta_A)a_1)^2}.                          
\end{equation}
It can be seen that for $N=1$ the above formula becomes identical with the lower bound $l^1_{\#}$ for $\theta_A\leq\theta_B$ case in \eqref{FL20}. 
It goes same for higher dimension too.
\begin{remark}
As an important comparison between two cases $(2a)$ and $(2b)$, we see both are giving saturation / optimality of the lower bound of $B^{\#}$
whenever $\theta_B \leq \theta_A$ and $\theta_A \leq \theta_B$ respectively 
by switching the roles of $a_1$, $a_2$ and $b_1$, $b_2$ both. 
\hfill\qed\end{remark}
\noindent\textbf{(2c): Core $a_2 I$ with coating $a_1 I$ for $A_{B(0,1)}$ and core $b_1 I$ with coating $b_2 I$ for $B_{B(0,1)}$, with $\omega_A \subseteq \omega_B^c$ :}
Compared to the previous case $(2b)$, here we are only changing the role of $b_1$ and $b_2$ 
while keeping the volume fraction $\theta_B$ fixed for $b_1$, with $\omega_A \subseteq \omega_B^c$.
Then making this change in the above formulation \eqref{ED1} we simply get,
\begin{equation}\label{FG7}
b^{\#} =\  b_2\ [\ 1 + \frac{N\theta_A(1-\theta_A)(a_2-a_1)^2}{(\theta_A a_2 + (N-\theta_A)a_1)^2}\ ] - \frac{(b_2-b_1)(Na_1)^2\theta_B}{(\theta_A a_2 + (N-\theta_A)a_1)^2}.                          
\end{equation}
It provides the saturation / optimality of the upper bound of $B^{\#}$  for $\theta_B + \theta_A \leq 1$ case. 
For $N=1$ the above formula becomes identical with the upper bound $u^{\#}_1$ in \eqref{FL17}. \\
\\
\textbf{(2d): Core $a_1 I$ with coating $a_2 I$ for $A_{B(0,1)}$ and core $b_2 I$ with coating $b_1 I$ for $B_{B(0,1)}$, with $\omega_A^c \subseteq \omega_B$ :}
Here we show one more variation. Compared to the previous case $(2c)$, here we are changing the role of $b_1$, $b_2$ and 
$a_1$, $a_2$ both under the condition $\omega_A^c \subseteq \omega_B$.
Then making these changes in the above formulation \eqref{FG7} we get,
\begin{equation}\label{bs19}
b^{\#} =\  b_1\ [\ 1 + \frac{N\theta_A(1-\theta_A)(a_2-a_1)^2}{((1-\theta_A) a_1 + (N +\theta_A-1)a_2)^2}\ ] - \frac{(b_1-b_2)(Na_2)^2(1-\theta_B)}{((1-\theta_A)a_1 + (N+\theta_A-1)a_2)^2}.                          
\end{equation}
It provides the saturation / optimality of the upper bound of $B^{\#}$  for $\theta_B + \theta_A \geq 1$ case. 
For $N=1$ the above formula becomes identical with the upper bound $u^{\#}_2$ in \eqref{FL17}. 
\begin{remark}
 Note that, in the above two cases $(2c)$ and $(2d)$, the core and coating combination has switched as a pair 
 for both $A_{B(0,1)}$ and $B_{B(0,1)}$, in order to give the saturation / optimality of the upper bound of $B^{\#}$. 
\hfill\qed\end{remark}
\begin{remark}
One can also consider the above core and coating combinations with different inclusion conditions but for those construction the corresponding
$B^{\#}$ will not stand as an optimal one. It remains inside in between optimal lower bound and optimal upper bound. 
\hfill\qed\end{remark}
\begin{remark}\label{ad14}
We also notice that, for a given fixed laminate microstructures or Hashin-Shtrikman constructions of $A^{\epsilon}$
providing the saturation / optimality for $A^{*}$ bound there are plenty of microstructures for $B^{\epsilon}$ are available
with satisfying the required inclusion conditions such that the resulting $(A^{*},B^{\#})$ is unique and provides the 
saturation / optimality of its bound (see Section \ref{qw4}).
\hfill\qed\end{remark}
\noindent
Now we move into the final part to establish the optimal bounds on $(A^{*},B^{\#})$ announced in Section \ref{Sd9}.

\section{Proof of Optimal Bounds :}\label{ts}
\setcounter{equation}{0}
Here we are going to establish the optimal trace bounds announced in Section \ref{ad18}.
We shall begin with recalling few key components of our main tool namely the $H$-measure
techniques and its application to compensated compactness theory. 
\subsection{$H$-measure}\label{hsn}
The notion of $H$-measure has been introduced by Gerard and Tartar. It is a defect measure which quantifies
the lack of compactness of weakly converging sequences in $L^2(\mathbb{R}^N)$ in the phase space. 
More specifically, it indicates where in the physical space, and at which frequency in the Fourier Space are the obstructions to strong convergence \cite{AM,PG2,T1,T2}.
Let $v_\epsilon$ be a sequence of functions defined in $\mathbb{R}^N$ with values in $\mathbb{R}^P$. The components of the vector valued function 
$v_{\epsilon}$ are denoted by $({v_{\epsilon}}_i)_{1\leq i\leq P}$. We assume that $v_\epsilon $ converges weakly to $0$ in $L^2(\mathbb{R}^N)^P$.
Then there exist a subsequence (still denoted by $\epsilon $) and a family of complex-valued Random measures $(\mu_{ij}(x,\xi))_{1\leq i,j\leq P}$ 
on $\mathbb{R}^N \times \mathbb{S}^{N-1} $ such that, for test functions $ \varphi_1(x),\varphi_2(x)$ in $C_0(\mathbb{R}^N)$ (the space of continuous functions 
vanishes at infinity), and $\psi(\xi)$ in $C(\mathbb{S}^{N-1})$ (the space of continuous functions over a $N$ dimensional sphere in $\mathbb{R}^N$), it satisfies 
\begin{equation}\label{dc6}
\int_{\mathbb{R}^N} \int_{\mathbb{S}^{N-1}} \varphi_1(x)\overline{\varphi_2(x)}\psi(\frac{\xi}{|\xi|})(\mu_{ij}(dx,d\xi)) 
=\ \underset{\epsilon\rightarrow 0}{\textrm{lim}}\ \int_{\mathbb{R}^N} \widehat{F}(\varphi_1(x){v_{\epsilon}}_i)(\xi)\overline{\widehat{F}(\varphi_2(x){v_{\epsilon}}_j)}\psi(\frac{\xi}{|\xi|})\ d\xi
\end{equation}  
where $\widehat{F}$ is the usual Fourier transform operator defined in $L^2(\mathbb{R}^N)$ by 
\begin{equation*}(\widehat{F}\varphi)(\xi) = \int_{\mathbb{R}^N} \varphi(x) e^{-2i\pi(x\cdot \xi)}\ dx.\end{equation*}
\noindent
The matrix of measures $\mu = (\mu_{ij}$) is called the $H$-measure of the subsequence $v_{\epsilon}$. 
It takes its values in the set of hermitian and non-negative matrices 
\begin{equation*}\mu_{ij} = \overline{\mu}_{ji},\quad \sum_{i,j =1}^P \lambda_i \overline{\lambda_j}\hspace{2pt}\mu_{ij}\geq 0 \ \ \ \forall \lambda \in \mathbb{C}^P.\end{equation*}
In \eqref{dc6} the role of the test functions $\varphi_1$ and $\varphi_2$ are to localize in space and $\psi$ is to localize in the directions of oscillations.
When we take $\psi =1$, we recover the usual defect measure in the physical space, i.e.      
\begin{center}
$\int_{\mathbb{S}^{N-1}}d\mu_{ij}(x,d\xi)$ is just the weak* limit measure of the sequence 
$v_{\epsilon}^i\overline{v_{\epsilon}^j}$, 
\end{center}
which is bounded in $L^1(\mathbb{R}^N)$.
Thus, the $H$-measure gives 
a more precise representation of the lack of compactness, by taking into account oscillation directions.
An important property of the $H$-measure is its so-called localization principle which is indeed a generalization
of the compensated compactness theory of Murat and Tartar.
\begin{theorem}[Localization principle]\label{dc9}
Let $v_{\epsilon}$ be a sequence which converges weakly to $0$ in $L^2(\mathbb{R}^N)^P$ and defines a $H$-measure $\mu $.
If $v_{\epsilon}$ is such that for $ 1\leq m\leq m_0,$ 
\begin{equation}\label{dc7}
\sum_{j=1}^{P} \sum_{k=1}^{N} \frac{\partial}{\partial x_k} ( A_{jk}^{m}(x){v_{\epsilon}}_j) \longrightarrow 0  \quad\mbox{in } H^{-1}_{\textrm{loc}} (\Omega)\mbox{ strongly }
\end{equation}
where the coefficients $ A_{jk}^m $ are continuous in an open set $\Omega $ of $\mathbb{R}^N$ then the $H$-measures satisfies 
\begin{equation}\label{dc8}
\sum_{j=1}^{P} \sum_{k=1}^{N} A_{jk}^{m}(x) \xi_k \mu_{ji}=\ 0 \quad\mbox{in } \Omega\times \mathbb{S}^{N-1} \ \ \forall i,m \quad\mbox{satisfying } 1\leq i\leq P ,\ \  1\leq m \leq m_0. 
\end{equation}
\hfill\qed\end{theorem}
\noindent
As a consequence of the localization principle we state an another important result which associates with the quadratic form.
Let us define a standard pseudo-differential operator $q$ which is defined through its symbol
$(q_{ij}(x,\xi))_{1\leq i,j\leq P} \in C^{\infty}(\mathbb{R}^N \times \mathbb{R}^N)$ by 
\begin{equation*} (qv)_i(x)= \sum_{j=1}^P \textit{F}^{-1}\{q_{ij}(x,.)\textit{F}v_j(.)\}(x) \end{equation*}
for any smooth and compactly supported function $v$.
Here in particular we are concerned with the so called poly-homogeneous pseudo-differential 
operator of order $0$, i.e. whose principal symbol is homogeneous of degree $0$ in $\xi$ and 
with compact support in $x$. We borrow the following remark from \cite{AM}. 
\begin{remark}
Being homogeneous of degree $0$, the symbol 
$(q_{ij}(x,\xi))_{1\leq i,j\leq P}$ is not smooth at $\xi=0$. 
This is not a problem since any
regularization by a smooth cut-off at the origin gives rise to the same pseudo-
differential operator up to the addition of a smoothing operator. 
\hfill\qed\end{remark}
\noindent
Then we have the following theorem.
\begin{theorem}[Compensated Compactness]\label{dc10}
(a) Let $v_{\epsilon}$ be a sequence which converges weakly to $0$ in  
$L^2(\mathbb{R}^N)^P$. Then there exist a subsequence and a $H$-measure $\mu$ 
such that, for any poly-homogeneous pseudo-differential operator $q$ of 
order $0$ with principal symbol $(q_{ij}(x,\xi))_{1\leq i,j\leq P} $, 
we have \begin{equation*}
\underset{\epsilon \rightarrow 0}{\textrm{lim}} \int_{\mathbb{R}^N}q(v_{\epsilon})\cdot\overline{v_{\epsilon}} dx = \int_{\mathbb{R}^N}\int_{\mathbb{S}^{N-1}}\sum_{i,j=1}^P q_{ij}(x,\xi)\mu_{ij}(dx,d\xi). 
\end{equation*}
(b) Suppose the sequence $v_{\epsilon}$ also satisfies the differential constraints given by \eqref{dc7}. 
So, the corresponding $H$-measure $\mu$ is satisfying \eqref{dc8}.\\
Introduce the wave cone
\begin{equation*}
\Lambda = \{\ (x,\xi, V)\in \mathbb{R}^N \times \mathbb{S}^{N-1} \times \mathbb{C}^P\mbox{ s.t. }\sum_{j=1}^{P} \sum_{k=1}^{N} A_{jk}^{m}(x) \xi_k V_{j}=0, 1\leq m\leq m_0\ \}.
\end{equation*}
Now 
\begin{equation*} \mbox{if }\ \sum_{i,j=1}^P q_{ij}(x,\xi)V_i\overline{V_j} \geq 0 \quad\mbox{for any } (x,\xi, V)\in \Lambda, \mbox{ then } \sum_{i,j=1}^P q_{ij}\mu _{ij}\geq 0. \end{equation*}
\hfill\qed\end{theorem}
\noindent
As a matter of consequence of the previous theorem we state the result of compensated compactness with variable coefficients which will be the main technical ingredient of this section.
Let $v_{\epsilon}$ be a sequence converging weakly to $v$ in $L^2(\mathbb{R}^N)^P$. Assume that, for $1\leq m \leq m_0,$ $v_{\epsilon}$ satisfies 
\begin{equation*}
\sum_{j=1}^P\sum_{k=1}^N\frac{\partial}{\partial x_k}(A_{jk}^{m}(x){v_{\epsilon}}_j)\longrightarrow \sum_{j=1}^P\sum_{k=1}^N\frac{\partial}{\partial x_k}(A_{jk}^{m}(x)v_j) \mbox{ in }H^{-1}_{loc}(\Omega)\mbox{ strongly. }
\end{equation*}
where the coefficients $A_{jk}^m$ are continuous in an open set $\Omega$ of $\mathbb{R}^N$. Let $\Lambda$ be the Characteristic set defined by
\begin{equation*}
\Lambda =\{\ (x,\xi,V)\in \mathbb{R}^N \times \mathbb{S}^{N-1} \times \mathbb{C}^P\mbox{ such that }\sum_{j=1}^P\sum_{k=1}^N A_{jk}^{m}\xi_{k}V_j=0; 1\leq m \leq m_0 \ \}.
\end{equation*}  
Let $q$ be a poly-homogeneous pseudo-differential operator of order $0$ with hermitian principal symbol $ (q_{ij}(x,\xi))_{1\leq i,j \leq p}$ such that
\begin{equation*}
\sum_{i,j=1}^{P} (q_{ij}(x,\xi))\overline{V_i}V_j \geq 0\mbox{ for any }(x,\xi,V)\in \Lambda.
\end{equation*}
Then for any non-negative smooth $\varphi$ with compact support in $\Omega$ 
\begin{equation*}
\underset{\epsilon \rightarrow 0}{\textrm{lim\hspace{1.5pt}inf}}\ \int_{\mathbb{R}^N} \varphi(x)q(v_{\epsilon})\cdot\overline{v_{\epsilon}}\ dx \geq \int_{\mathbb{R}^N} \varphi(x)q(v)\cdot\overline{v}\ dx.
\end{equation*}
\hfill\qed
\subsection{Bounds : $B^{\epsilon}$ is independent of $\epsilon$ }\label{os9} 
Let us begin with the case when the sequence $B^{\epsilon}$ is independent of $\epsilon$. 
We assume $B^{\epsilon}(x) = b(x)I\in \mathcal{M}(b_1,b_2;\Omega)$. To start with we recall from \eqref{dc1}, 
\begin{equation}\label{FG3} b(x)\nabla u^{\epsilon}\cdot\nabla u^{\epsilon} \rightharpoonup B^{\#}(x)\nabla u\cdot\nabla u \quad\mbox{ in }\mathcal{D}^{\prime}(\Omega).\end{equation}  
Let us remark that $B^{\#} = b(x)I^{\#}(x)$. We have the constraints on $\nabla u^\epsilon $ : 
\begin{equation}\label{abc2}
 \nabla u^\epsilon \rightharpoonup\nabla u \mbox{ weakly in } L^2(\Omega),\ \ -div(A^\epsilon\nabla u^\epsilon) \mbox{ is }H^{-1}(\Omega)\mbox{ convergent,}
\end{equation}
and by homogenization theory, we have the following limit :
\begin{equation}\label{ot15}A^{\epsilon}\nabla u^{\epsilon} \rightharpoonup A^{*}\nabla u \mbox{ in }L^2(\Omega)\mbox{ weak.}\end{equation}
We derive here the optimal lower bound \eqref{tw} and optimal  upper bound \eqref{tq} on $(A^{*},B^{\#})$ respectively.
\paragraph{Lower Bound :}
Let us introduce the constant vector $\eta\in\mathbb{R}^N$ and start with the so-called translated inequality for $A^\epsilon\in\mathcal{M}(a_1,a_2;\Omega)$ :
\begin{equation*} (a_2I-A^{\epsilon})(\nabla u^{\epsilon} + \eta)\cdot(\nabla u^{\epsilon} + \eta)\ \geq\ 0 \quad\mbox{a.e. in }\Omega.\end{equation*}
Then multiplying by $b(x)$ on both sides and expanding it, we get 
\begin{equation*} a_2\hspace{2pt} b\nabla u^{\epsilon}\cdot\nabla u^{\epsilon}-bA^{\epsilon}\nabla u^{\epsilon}\cdot\nabla u^{\epsilon} + 2b(a_2\nabla u^\epsilon\cdot\eta -A^{\epsilon}\nabla u^{\epsilon}\cdot\eta) + b(a_2I-A^{\epsilon})\eta\cdot\eta\ \geq\ 0. \end{equation*}
By passing to the limit as $\epsilon\rightarrow 0$ in the above inequality, we simply get 
\begin{equation}\label{ot13} a_2B^{\#}\nabla u\cdot\nabla u - bA^{*}\nabla u\cdot\nabla u + 2b(a_2I-A^{*})\nabla u\cdot\eta + b(a_2I-\overline{A})\eta\cdot\eta\ \geq\ 0. 
\end{equation}
To get lower bound on $B^{\#}$, we  minimize w.r.t. $\nabla u$. It is recalled
that any fixed vector in $\mathbb{R}^N$ can be realized as $\nabla u$ with $u^\epsilon$ satisfying the usual constraints \eqref{ad13}. Hence we have 
\begin{equation}\label{ot12}
 b\ (a_2I-A^{*})(a_2B^{\#}-bA^{*})^{-1}(a_2I-A^{*})\eta\cdot\eta \leq\ \theta_A(a_2-a_1)\eta\cdot\eta.
\end{equation}
Note that, $(a_2B^{\#}-bA^{*})$ is positive definite matrix (cf. \eqref{Sd3}). Now, by taking trace on both sides, \eqref{tw} follows.
\hfill\qed
\paragraph{Upper Bound :}
Here we consider another translated inequality as
\begin{equation}\label{bs7}(A^{\epsilon} - a_1I)(\nabla u^{\epsilon} + \eta)\cdot(\nabla u^{\epsilon} + \eta)\ \geq\ 0 \quad\mbox{a.e. in }\Omega,\ \mbox{ with constant vector }\eta \in \mathbb{R}^N. \end{equation}
Then by multiplying $b(x)$ on both sides and expanding it, we get  
\begin{equation*}bA^{\epsilon}\nabla u^{\epsilon}\cdot\nabla u^{\epsilon} - a_1 b\nabla u^{\epsilon}\cdot\nabla u^{\epsilon} + 2b(A^{\epsilon}-a_1I)\nabla u^{\epsilon}\cdot\eta + b(A^{\epsilon}-a_1 I)\eta\cdot\eta\ \geq\ 0. \end{equation*}
Now by passing to the limit in the above inequality, we simply get 
\begin{equation}\label{ot17} bA^{*}\nabla u\cdot\nabla u - a_1 B^{\#}\nabla u\cdot\nabla u + 2b(A^{*}-a_1I)\nabla u\cdot\eta + b(\overline{A}-a_1 I)\eta\cdot\eta\ \geq\ 0. \end{equation}
Finally minimizing with respect to $\nabla u$, we get 
\begin{equation}\label{bs2}
 b\ (A^{*} - a_1I)(bA^{*}-a_1 B^{\#})^{-1}(A^{*} - a_1I)\eta\cdot\eta \leq\ (1-\theta_A)(a_2-a_1)\eta\cdot\eta.
\end{equation}
Note that, $(bA^{*}-a_1B^{\#})$ is positive definite matrix (cf. \eqref{Sd3}). Now, by taking trace on both sides, \eqref{tq} follows. 
\hfill\qed
\begin{remark}
Let us note the reverse : for deriving lower bound (upper
bound) on $(A^{*},B^{\#})$, we used the translated inequality usually meant for
deriving upper bound (lower bound) for $A^{*}$.
\hfill\qed\end{remark}
\begin{remark}[Pointwise bounds on energy density $B^{\#}\nabla u\cdot\nabla u$]\label{hsg}
Having found the bounds \eqref{ot12},\eqref{bs2} on the matrix $B^{\#}$ itself, we now find the bounds on energy density $B^{\#}\nabla u\cdot\nabla u$, which is useful in solving optimal design problems, as shown in Section \ref{qw5}. We consider the inequalities obtained in \eqref{ot13} and \eqref{ot17} and minimize them with respect to $\eta\in\mathbb{R}^N$ to
obtain the lower bound and upper bound respectively as follows :
\begin{align}
(i):\mbox{(Lower bound)}\ \ \ B^{\#}\nabla u\cdot\nabla u \geq \frac{b}{a_2}\{A^{*} + (a_2I-\overline{A})^{-1}(a_2I-A^{*})^2\}\nabla u\cdot \nabla u. \label{bs1}\\
(ii):\mbox{(Upper bound)}\ \ \ B^{\#}\nabla u\cdot\nabla u \leq \frac{b}{a_1}\{A^{*} + (\overline{A}-a_1I)^{-1}(A^{*}-a_1I)^2\}\nabla u\cdot \nabla u. \label{Sd11}
\end{align}
The corresponding $A^{*}$ and $B^{\#}$ (constructed in Section \ref{ub12}) achieve equality in the above bounds. This shows the saturation/optimality property of these special structures.
\hfill\qed\end{remark}
\noindent
\textbf{Saturation/Optimality of the Bounds :}
A simple calculation shows that, the simple laminates that we have defined in \eqref{bs11} provides the equality of both bounds \eqref{ot12},\eqref{bs2} (and hence equality in \eqref{tw},\eqref{tq})  for all lamination directions $e_1,e_2,..,e_N$. They are the points of intersections the lower trace bound \eqref{tw} and the upper trace bound \eqref{tq} as in the case of $\mathcal{G}_{\theta_A}$ (see Figure 4).\\
Similarly, the equality of the lower bound \eqref{ot12} of $(A^{*},B^{\#})$
is achieved by the composite based on Hashin-Shtrikman construction with core $a_1I$ and coating $a_2I$
given in \eqref{hsb} for $B^{\#}$ and \eqref{sim} for $A^{*}$. 
Similarly, the equality in the upper bound \eqref{bs2} of $(A^{*},B^{\#})$ is achieved by the 
Hashin-Shtrikman construction with core $a_2I$ and coating $a_1I$ described in \eqref{hsd} for $B^{\#}$ and \eqref{FL9} for $A^{*}$.
Notice that, the Hashin-Shtrikman microstructure which gave the upper bound (or, lower bound )
equality for $A^{*}$ gives the lower bound (or, upper bound) equality for $(A^{*},B^{\#})$. 
\hfill\qed

\noindent
Having mentioned optimal properties of simple laminates and Hashin-Shtrikman constructions, we introduce and discuss the saturation / optimality of rank $N$ laminates. 
\paragraph{Construction of Sequential Laminates :}
Here we construct the $p$-rank or sequential laminates formula for $B^{\#}$ (denoted by $B^{\#}_p$) using $H$-measure techniques. We do it in three steps. The crucial step is to obtain a relation between $B^{\#}$, $A^{*}$ and the $H$-measure term (see \eqref{ad3} below). Subsequently we exploit it to define the laminate $B^{\#}_p$. We now provide details.  
\paragraph{Step 1 :}
We start with recalling the optimal lower bound on $A^{*}$.
If a sequence of characteristic function $\A \rightharpoonup \theta_A $ in $L^{\infty}(\Omega)$ weak*, with
\begin{equation*}A^{\epsilon}=\ a^{\epsilon}I=\ (a_1\A+a_2(1-\A)I \mbox{ $H$-convergence to } A^{*}\mbox{ in }\Omega,\quad (a_1 < a_2).\end{equation*}
Then we know that the optimal lower bound on $A^{*}$ is as follows \cite{T} :
\begin{equation}\label{lb3}
tr\ (A^{*}-a_1 I)^{-1} \leq \frac{N}{(1-\theta_A)(a_2 -a_1)} + \frac{\theta_A}{(1-\theta_A)a_1}.
\end{equation}
The $N$-sequential laminate \eqref{OP2} (with $p=N$) with matrix $a_1 I$ 
and core $a_2 I$ achieves the lower bound \eqref{lb3} with equality. Conversely,
any point on the lower bound equation \eqref{lb3} with equality can be achieved through $N$-sequential laminates \eqref{OP2}, see  \cite{A}.\\
\\
Let us recall the translated inequality \eqref{bs7}, which implies,
\begin{equation}\label{lb1}
A^{\epsilon}\nabla u^{\epsilon}\cdot\nabla u^{\epsilon} + (A^{\epsilon}-a_1 I)\eta \cdot \eta - 2a_1 \nabla u^{\epsilon}\cdot \eta\ \geq\ a_1\nabla u^{\epsilon}\cdot\nabla u^{\epsilon}- 2A^{\epsilon}\nabla u^{\epsilon}\cdot \eta.
\end{equation}
One derives the optimal lower bound \eqref{lb3} by passing to the limit in this inequality. \\
\\
Using \eqref{abc2} and \eqref{ot15}, passing to the limit on the left hand side of \eqref{lb1} is rather easy.  We focus our attention on the right hand side of \eqref{lb1}. By multiplying it by $b(x)$, we get  
\begin{equation}\label{ad1}
 b(x)a_1\nabla u^{\epsilon}\cdot\nabla u^{\epsilon} - 2b(x)A^{\epsilon}\nabla u^{\epsilon}\cdot \eta  
\end{equation}
and we pass to the limit in the above equation \eqref{ad1} in two different ways.\\
\\
Firstly, by using \eqref{FG3} and \eqref{ot15} we get
\begin{equation}\label{ad2}
 a_1 B^{\#}\nabla u\cdot\nabla u - 2bA^{*}\nabla u\cdot \eta .
\end{equation}
Secondly, we use the notion of $H$-measure in order to pass to the limit in \eqref{ad1} or in the right hand side of \eqref{lb1}.\\
\\
To this end, introducing a coupled variable $V_{\epsilon}=(\nabla u^{\epsilon}, A^{\epsilon}\eta)$, we write the right 
hand side of \eqref{lb1} in a quadratic form $q(V_{\epsilon})\cdot V_{\epsilon} $. 
Here $q$ is the following linear map :
\begin{equation*} q:\mathbb{R}^N \times \mathbb{R}^N \longmapsto \mathbb{R}^N \times \mathbb{R}^N\ \mbox{ defined by }\ q(\nabla u^\epsilon,A^\epsilon\eta)=\ (a_1\nabla u^\epsilon-A^\epsilon\eta, -\nabla u^\epsilon ). \end{equation*}
Introducing $V_0 =(\nabla u,\overline{A}\eta)$, we have 
\begin{equation*} q(V_{\epsilon})\cdot V_{\epsilon}=2q(V_{\epsilon})\cdot V_0 -q(V_0)\cdot V_0 + q(V_{\epsilon}- V_0)\cdot(V_{\epsilon}- V_0).\end{equation*}
Let $\varPi_V$ denote the $H$-measure of $(V_{\epsilon}- V_0)$. Its $x$-projection is a Radon measure known as defect measure. However, due to Remark \ref{abc3}, it is in fact a $L^1(\Omega)$ function. This will be useful below. We pass to the limit in \eqref{ad1} by virtue of Theorem \ref{dc10} to get
\begin{equation}\label{lb2}
b(x) a_1\nabla u\cdot\nabla u- 2b(x)\overline{A}\nabla u\cdot \eta + b(x) X, 
\end{equation}
where $X$ is the $H$-measure correction term defined by 
\begin{equation*} X = \underset{\epsilon \rightarrow 0}{\textrm{lim}}\ q(V_{\epsilon}- V_0)\cdot(V_{\epsilon}- V_0) = \int_{\mathbb{S}^{N-1}} tr\ (q\varPi_V (x,d\xi)) \end{equation*}
or equivalently,
\begin{equation}\label{eid} X = \langle\langle \varPi_V, Q(U,A\eta)\rangle\rangle  \quad\mbox{with } Q(U,A\eta)=q(U,A\eta)\cdot(U,A\eta). \end{equation}
Here we are using the angular brackets to denote the directional average, a notation introduced in \cite{T}.\\
\\
Equating the limit \eqref{lb2} with \eqref{ad2}, we get
\begin{equation*}
a_1 B^{\#}\nabla u\cdot\nabla u - 2b\hspace{1.5pt}A^{*}\nabla u\cdot \eta = b\lb a_1\nabla u\cdot\nabla u-2\overline{A}\nabla u\cdot \eta + X \rb.
\end{equation*}
Now minimizing with respect to $\nabla u$, we obtain a useful relation linking $B^{\#}$ with $A^{*}$ with $H$-measure term : 
\begin{equation}\label{ad3}
b(\overline{A}-A^{*})(B^{\#}-bI)^{-1}(\overline{A}-A^{*})\eta\cdot\eta= a_1X.
\end{equation}
On the other hand by passing to the limit in \eqref{lb1}, it yields
\begin{equation}\label{ot14}
A^{*}\nabla u\cdot\nabla u + (\overline{A}-a_1 I)\eta \cdot \eta - 2a_1\nabla u\cdot \eta\ \geq\ a_1\nabla u\cdot\nabla u -2\overline{A}\nabla u\cdot \eta + X .
\end{equation}
\textbf{Step 2 :}
Next, the idea is now to use Theorem \ref{dc10} in order to get a lower bound on $X$. 
However, the quadratic form $q$, which is defined by 
\begin{equation*} q(V)\cdot V= a_1 |U|^2 - 2U\cdot A\eta \mbox{ with } V=(U,A\eta) \in \mathbb{R}^N \times \mathbb{R}^N \end{equation*}
is not coercive with respect to the variable $A\eta$. 
So there is no hope to prove that $q$ is non-negative on wave cone, as required by Theorem \ref{dc10}. 
Here, we apply this result to a slightly different quadratic form. The following arguments were inspired from \cite{AM}.\\
As the gradient $\nabla u^{\epsilon}$ satisfies $curl(\nabla u^{\epsilon}) =0 $, 
in the spirit of Theorem \ref{dc9}, we introduce the oscillation variety 
\begin{equation}\label{os16}
 \vartheta=\ \{\ (\xi,U,A\eta) \in \mathbb{S}^{N-1} \times \mathbb{R}^N\times \mathbb{R}^N;\ \xi_{i}U_j -\xi_{j}U_i = 0\ \forall i,j \}
\end{equation}
and its projection (wave cone)
\begin{align*}
\Lambda &=\ \{(U,A\eta) \in \mathbb{R}^N\times \mathbb{R}^N;\ \exists\ \xi\in \mathbb{S}^{N-1}  \mbox{ such that }(\xi,U,A\eta) \in \vartheta \}\\
        &=\ \{(U,A\eta) \in \mathbb{R}^N\times \mathbb{R}^N;\ U \parallel \xi,\mbox{ for some }\xi\in \mathbb{S}^{N-1}\}.
\end{align*}
We define $\Lambda_{\xi}\subset \Lambda$, $\xi\in \mathbb{R}^N\smallsetminus \{0\}$ as 
\begin{equation}\label{sir} \Lambda_{\xi}=\ \{(U,A\eta) \in \mathbb{R}^N\times \mathbb{R}^N;\ U= t\hspace{1.5pt}\xi,\ t\in \mathbb{R} \}=\ \mathbb{R}\xi;\quad\mbox{So, }  \underset{\xi\neq 0}{\cup} \Lambda_{\xi}=\Lambda.\end{equation}
We introduce a new linear form $q^{\prime}_{\xi}$, whose associated quadratic form defined as 
\begin{equation}\label{sit} Q^\prime_\xi(U,A\eta):= q^{\prime}_{\xi}(U,A\eta)\cdot(U,A\eta)=\ q(U,A\eta)\cdot(U,A\eta)-\underset{U \in \Lambda_{\xi}}{min}\  q(U,A\eta)\cdot(U,A\eta). \end{equation}
The above quadratic form is non-negative on the wave cone $\Lambda_\xi$. Thus applying Theorem \ref{dc10} we get 
\begin{equation}\label{eiy} trace\ (q^{\prime}_{\xi}\varPi_V) \geq 0 \end{equation}
which implies that 
\begin{equation}\label{eie} X =\ \langle\langle\varPi_V, Q(U,A\eta)\rangle\rangle \geq  \langle\langle\varPi_V,\underset{U \in \Lambda_{\xi}}{min}\ Q(U,A\eta)\rangle\rangle. \end{equation}
Introducing $\varPi_A $, the $H$-measure of $ (A^{\epsilon} - \overline{A})\eta$, since $ \underset{U \in \Lambda}{min}\ Q(U,A\eta) $ depends only on $A\eta$, to obtain :
\begin{equation*} X \geq\ \langle\langle\varPi_V,\underset{U \in \Lambda_{\xi}}{min}\ Q(U,A\eta)\rangle\rangle = \langle\langle\varPi_{A},\underset{U \in \Lambda_{\xi}}{min}\ Q(U,A\eta)\rangle\rangle .\end{equation*}
It remains to compute the right hand side of the above inequality. We simply have
\begin{equation}\label{sij} \underset{U \in \Lambda_\xi}{min}\ Q(U,A\eta)=\ -\frac{(A\eta\cdot \xi)^2}{a_1 |\xi|^2} \end{equation}
with the minimizers $U_{min}$ in $\Lambda_\xi$ :
\begin{equation}\label{sio} U_{min} :=\ \frac{(A\eta\cdot\xi)}{a_1|\xi|^2}\hspace{1.5pt}\xi.\end{equation}
Since $(A^{\epsilon}-\overline{A})\eta =(a_1 -a_2)(\chi_{\epsilon} -\theta_A)\eta $, the $H$-measure $\varPi_A$ reduces to 
\begin{equation*}
(\varPi_A)_{ij} = (a_2 -a_1)^2(\nu_A)\eta_i\eta_j  \ \ \forall i,j =1,..,N
\end{equation*}
where $\nu_A$ is a $H$-measure of the sequence $(\A-\theta_A)$ with
\begin{equation}\label{hsz}  \nu_A(x,d\xi)\geq 0 \ \mbox{ and }\ \int_{\mathbb{S}^{N-1}} \nu_A (x,d\xi) = \theta_A(x)(1-\theta_A(x)).\end{equation}
We finally obtain 
\begin{equation}\label{os10} \langle\langle\varPi_A, \frac{(A\eta\cdot\xi)^2}{|\xi|^2} \rangle\rangle =\ (a_2 -a_1)^2\int_{\mathbb{S}^{N-1}} \frac{(\eta \cdot \xi)^2}{|\xi|^2}\nu_A (x,d\xi). \end{equation}
Introducing a matrix $M_A$ defined by 
\begin{equation}\label{os19} M_A = \frac{1}{(1-\theta_A)\theta_A}\int_{\mathbb{S}^{N-1}} \xi \otimes \xi\ \nu_A (x,d\xi) \end{equation}
which is non-negative and has unit trace.\\
\\
Therefore 
\begin{equation}\label{FL5} X\ \geq\ - \frac{\theta_A(1-\theta_A)(a_2 -a_1)^2}{a_1} M_A\eta\cdot \eta\ =: X_{min}   \end{equation}
Clearly, $X$ achieves its minimum $X_{min}$ with $U=U_{min}$ defined in \eqref{sio}.\\
\\
Thus from \eqref{ot14} we obtain 
\begin{equation}\label{lb4}
(A^{*}-a_1 I)\nabla u\cdot\nabla u + (\overline{A}-a_1 I)\eta \cdot \eta + 2(\overline{A}-a_1 I)\nabla u\cdot \eta\ \geq\ \frac{-\theta_A(1-\theta_A)(a_2 -a_1)^2}{a_1} M_A\eta\cdot \eta .
\end{equation}
Above inequality is pointwise which allows us to localize the problem of estimation. To do so, let us take the point $x=x_0$, where $A^{*}(x_0)$ is defined as a constant homogenized matrix with the proportion $\theta_A(x_0)$. 
Construction of classical oscillatory test function shows that $\nabla u$ is an arbitrary vector in $\mathbb{R}^N$ :
\begin{equation}\label{zz11}
(i)\ \ \nabla u^\epsilon \rightharpoonup \zeta\in\mathbb{R}^N \mbox{ arbitrary and }(ii)\ \    div(A^\epsilon\nabla u^\epsilon) \mbox{ converges }H^{-1}(\Omega) \mbox{ strong.}
\end{equation}
We minimize \eqref{lb4} with respect to $\nabla u=\zeta$ with its minimizer 
\begin{equation}\label{abc1}
\zeta = - (\overline{A}(x_0)-a_1I)(A^{*}(x_0)- a_1 I)^{-1}\eta ;
\end{equation}
so \eqref{lb4} yields the following estimate at $x_0$ : 
\begin{equation}\label{lb5}
(A^{*}-a_1 I)^{-1}\eta\cdot \eta\ \leq\ \frac{|\eta|^2}{(1-\theta_A)(a_2 -a_1)} + \frac{\theta_A}{(1-\theta_A)a_1}M_A\eta\cdot\eta. 
\end{equation}
Since $x_0$ is arbitrary, the matrix bound \eqref{lb5} is pointwise bound for $x$ almost everywhere.  
Since $\eta $ is an arbitrary vector, by taking trace we get the well known optimal lower bound \eqref{lb3} for $A^{*}$.
We will use later the same minimizer $\zeta = - (\overline{A}-a_1I)(A^{*}- a_1 I)^{-1}\eta$ later for our purposes.
\paragraph{Step 3 :}
Note that the expression \eqref{OP2} for the
$p$-sequential laminates $A^{*}_p$ with matrix $a_1I$ and core $a_2I$ can be written as
\begin{equation}\label{UD10} 
(1-\theta_A)( A^{*}_{p} - a_1 I)^{-1} = (a_2 - a_1)^{-1}I + \frac{\theta_A}{a_1} \int_{\mathbb{S}^{N-1}} \frac{e\otimes e}{e\cdot e}d\nu_A(e), 
\end{equation}
for the probability measure $\nu_A$ on the unit sphere $\mathbb{S}^{N-1}$ given by $\nu_A = \sum_{j=1}^p m_j\delta_{e_j}$
with, $\sum_{j=1}^p m_j=1$. See \cite{T}. Now going back to \eqref{os19},\eqref{FL5}, one sees the above integrals over $\mathbb{S}^{N-1}$ appears naturally through the $H$-measure techniques and the equality takes place with $X=X_{min}$. \\
\\
It naturally enables us to define the corresponding $p$-sequential laminate for $B^{\#}$ as follows : We recall \eqref{ad3} with $A^{*}=A^{*}_p$ (cf.\eqref{UD10}) and $X=X_{min}$ (cf.\eqref{FL5}) with  $M_A=\sum_{i=1}^p m_i\frac{e_i\otimes e_i}{e_i\cdot e_i}$ where $\{e_i\}$ are the canonical basis vectors and $\sum_{i=1}^p m_i =1$. 
Then the resulting $B^{\#}$ defines the $p$-sequential laminate say $B^{\#}_p$ with matrix $a_1 I$ and core $a_2 I$ satisfying, 
\begin{equation}\label{FG4}
b(B^{\#}_p-bI)^{-1}(\overline{A}-A^{*}_p)^2 =\  \theta_A(1-\theta_A)(a_2 -a_1)^2\ \lb\sum_{i=1}^{p} m_i\frac{e_i\otimes e_i}{e_i\cdot e_i}\rb,\ \  \mbox{ with }\ \sum_{i=1}^p m_i =1.
\end{equation}
Note that $B^{\#}_p$ is diagonal, since $A^{*}_p$ is diagonal. \\
\\
Similarly, by changing the role of $a_1$,$a_2$, one defines the $p$-sequential laminate $B^{\#}_p$ with matrix $a_2 I$ and core $a_1 I$ as 
\begin{equation}\label{bs13}
b(B^{\#}_p-bI)^{-1}(\overline{A}-A^{*}_p)^2 =\  \theta_A(1-\theta_A)(a_2 -a_1)^2\ \lb\sum_{i=1}^{p} m_i\frac{e_i\otimes e_i}{e_i\cdot e_i}\rb,\ \  \mbox{ with }\ \sum_{i=1}^p m_i =1.
\end{equation}
where $A^{*}_p$ is the $p$-sequential laminates \eqref{OP3} with matrix $a_2 I$ and core $a_1 I$.\hfill\qed
\begin{remark}[Uniqueness of $B^{\#}$]\label{ad19}
As we see from the relation \eqref{ad3}, if $A^{*}$ and $X$ are fixed then
$B^{\#}$ is also fixed. Now suppose for two different microstructures if we have same 
homogenized limit $A^{*}\in \partial\mathcal{G}_{\theta_A}$ then associated matrix $M_A$ is also the same (because equality holds in \eqref{lb5}). Hence both the microstructures possess the same $X=X_{min}$ (cf.\eqref{FL5}).
Therefore, both microstructures lead to the same macro relative limit $B^{\#}$ with $A^{*}\in\partial\mathcal{G}_{\theta_A}$. 
If $A^{*}\in int\ \mathcal{G}_{\theta_A}$, it can be realized as $A^{*}\in\partial\mathcal{G}_{\widetilde{\theta_A}}$
for some $\widetilde{\theta_A}$ uniquely determined by $\theta_A,A^{*}$ (see Figure 4). Then using the previous
arguments in Step $2$ once again, replacing $\theta_A$ by $\widetilde{\theta_A}$ everywhere we conclude that
$B^{\#}$ is uniquely determined. 
\hfill\qed\end{remark}
\begin{remark}[Optimality]\label{ot18}
\noindent
$(1)\ :\ $ The $N$-sequential laminates $(A^{*}_N,B^{\#}_N)$ given in \eqref{OP3},\eqref{bs13} and \eqref{OP2},\eqref{FG4} achieve the equality in the lower bound \eqref{tw} and the upper
bound \eqref{tq} respectively. Conversely, it is known that any point on $\partial\mathcal{G}_{\theta_A}^L$ or $\partial\mathcal{G}_{\theta_A}^U$ can be achieved by $N$-sequential laminates $A^{*}_N$. 
Further, the relative limit $B^{\#}_N$ constructed above along with $A^{*}_N$ achieves the equality in the upper bound \eqref{tq} or the lower bound \eqref{tw} respectively. \\
\noindent
$(2)\ :\ $ Given $A^{*}$, $B^{\#}$ satisfying both inequalities \eqref{tq},\eqref{tw}, there exists $A^\epsilon$ satisfying \eqref{ta} and $b(x)$ with $b_1\leq b(x)\leq b_2$ such that $A^\epsilon \xrightarrow{H} A^{*}$ and $b\xrightarrow{A^\epsilon} B^{\#}$. We will not prove such assertions here. We will deal with more complicated cases of two-phase $\{A^\epsilon, B^\epsilon\}$ with $B^\epsilon$ depending on $\epsilon$. See Theorem \ref{qw6} and its proof in Section \ref{qw4}.
\hfill\qed\end{remark}

\subsection{Bounds : $B^{\epsilon}(x)$ is governed by two-phase medium.}\label{sil} 
Now we move to the case when $B^{\epsilon}(x)$ is governed with the two-phase medium. 
We consider the sequences $\A \rightharpoonup \theta_A $ and $\B \rightharpoonup \theta_B $ in $L^{\infty}(\Omega)$ weak*, with
\begin{equation*}A^{\epsilon}=\ a^{\epsilon}I=\ (a_1\A+a_2(1-\A)I \mbox{ $H$-converges to } A^{*}\mbox{ in }\Omega,
\end{equation*}
and
\begin{equation*}
B^{\epsilon} =\ b^{\epsilon}I=\ (b_1\B + b_2(1-\B))I \ \mbox{converges to }B^{\#} \mbox{ relative to }A^\epsilon \mbox{ in }\Omega.
\end{equation*}
We assume $(0<a_1<a_2<\infty)$ and $(0<b_1<b_2<\infty)$.\\
\\
Let us begin with by considering two particular cases namely, when $\omega_{B^\epsilon}=\omega_{A^\epsilon}$ and $\omega_{B^\epsilon}=\omega^c_{A^\epsilon}$ respectively.
Let us first derive the lower bound when $\omega_{B^\epsilon}=\omega_{A^\epsilon}$.
\paragraph{1. Lower Bound : when $\omega_{B^\epsilon}=\omega_{A^\epsilon}$ :}
We introduce the constant vector $\eta$ in $\mathbb{R}^N$ and consider the simple translated inequality for $B^\epsilon$ with oscillation field $\nabla u^\epsilon$ associated to $A^\epsilon$ :
\begin{equation}\label{FL8}
(B^{\epsilon}-b_1 I)(\nabla u^{\epsilon} +\eta)\cdot(\nabla u^{\epsilon} +\eta)\geq 0 \quad\mbox{a.e. in }\Omega. 
\end{equation}
which is rewritten as
\begin{equation}\label{Sd18}
B^{\epsilon}\nabla u^{\epsilon}\cdot\nabla u^{\epsilon} + (B^{\epsilon}- b_1 I)\eta \cdot \eta + 2(B^\epsilon-b_1 I) \nabla u^{\epsilon}\cdot \eta\ \geq\ b_1\nabla u^{\epsilon}\cdot\nabla u^{\epsilon}.
\end{equation}
We impose the constraints on $\nabla u^\epsilon $ : 
\begin{equation}\label{six}
 \nabla u^\epsilon \rightharpoonup\nabla u \mbox{ weakly in } L^2(\Omega),\ \ -div(A^\epsilon\nabla u^\epsilon) \mbox{ is }H^{-1}(\Omega)\mbox{ convergent. }
\end{equation}
Thanks to the distributional convergence \eqref{dc1}, we know
\begin{equation*}B^{\epsilon}\nabla u^{\epsilon}\cdot\nabla u^{\epsilon}\ \rightharpoonup\ B^{\#}(x)\nabla u\cdot\nabla u\ \mbox{ in }\ \mathcal{D}^{\prime}(\Omega).\end{equation*}
Passing to the limit in the left hand side of \eqref{Sd18} is rather easy. It is enough to use the relation between the two fluxes and we have : 
\begin{equation}\label{sik} (B^\epsilon-b_1I)\nabla u^\epsilon = \frac{(b_2-b_1)}{(a_2-a_1)}(A^\epsilon-a_1I)\nabla u^\epsilon\ \rightharpoonup\ \frac{(b_2-b_1)}{(a_2-a_1)}(A^{*}-a_1I)\nabla u\ \mbox{ in } L^2(\Omega).\end{equation}
On the other hand, in order to pass to the limit in the right hand side of \eqref{Sd18} we use 
the $H$-measure techniques. Recalling that the limit of right hand side of \eqref{lb1} obtained
in the previous Section \ref{os9} as 
\begin{equation}\label{siy} a_1\nabla u^{\epsilon}\cdot\nabla u^{\epsilon}- 2A^{\epsilon}\nabla u^{\epsilon}\cdot \eta \rightharpoonup a_1\nabla u\cdot\nabla u- 2\overline{A}\nabla u\cdot \eta + X\ \mbox{ in }\ \mathcal{D}^{\prime}(\Omega)\end{equation}
where, $X$ is the $H$-measure corrector term. Hence, 
\begin{equation*} b_1\nabla u^{\epsilon}\cdot\nabla u^{\epsilon} \rightharpoonup b_1\nabla u\cdot\nabla u + 2\frac{b_1}{a_1}(A^{*}-\overline{A})\nabla u\cdot \eta + \frac{b_1}{a_1}X \ \mbox{ in }\ \mathcal{D}^{\prime}(\Omega).\end{equation*}
Therefore passing to the limit in \eqref{Sd18} we simply get, 
\begin{equation*}B^{\#}\nabla u\cdot\nabla u + (\overline{B}-b_1 I)\eta \cdot \eta + 2\frac{(b_2-b_1)}{(a_2-a_1)}(A^{*}-a_1I)\nabla u\cdot \eta\ \geq\ b_1\nabla u\cdot\nabla u + 2\frac{b_1}{a_1}(A^{*}-\overline{A})\nabla u\cdot \eta + \frac{b_1}{a_1}X.\end{equation*}
Next, by using the lower bound on $X$ (cf.\eqref{FL5}) we obtain 
\begin{align*}
(B^{\#}-b_1I)\nabla u\cdot\nabla u + (\overline{B}-b_1 I)\eta \cdot \eta + 2\{\frac{(b_2-b_1)}{(a_2-a_1)}(A^{*}-a_1I)+\frac{b_1}{a_1}(\overline{A}-A^{*})\}\nabla u\cdot \eta\ \qquad\qquad &\\
+\frac{b_1}{a_1^2}(a_2-a_1)^2\theta_A(1-\theta_A)M_A\eta\cdot\eta\ \geq\ 0. &
\end{align*}
Finally, minimizing over $\nabla u$ we obtain the lower bound as
\begin{align}
&\{\frac{(b_2-b_1)}{(a_2-a_1)}(A^{*}-a_1I)+\frac{b_1}{a_1}(\overline{A}-A^{*})\}(B^{\#}-b_1I)^{-1}\{\frac{(b_2-b_1)}{(a_2-a_1)}(A^{*}-a_1I)+\frac{b_1}{a_1}(\overline{A}-A^{*})\}\eta\cdot\eta\notag\\
&\quad\qquad\qquad\qquad\qquad\qquad\quad\leq\ (\overline{B}-b_1 I)\eta \cdot \eta  +\frac{b_1}{a_1^2}(a_2-a_1)^2\theta_A(1-\theta_A)M_A\eta\cdot\eta. \label{siq}
\end{align}
Since $\eta $ is an arbitrary vector in $\mathbb{R}^N$, by taking trace on both sides of \eqref{siq} and using the fact that the matrix $M_A$ has unit trace, 
one obtains the following lower trace bound for $(A^{*},B^{\#})$ whenever the two corresponding microstructures $\omega_{A^\epsilon}=\omega_{B^\epsilon}$.
\begin{align}
tr\ \{[\frac{(b_2-b_1)}{(a_2-a_1)}(A^{*}-a_1I)+\frac{b_1}{a_1}(\overline{A}-A^{*})]^2(B^{\#}-b_1I)^{-1}\}\leq\ &N(b_2-b_1)(1-\theta_A)\notag \\
&+\frac{b_1}{a_1^2}(a_2-a_1)^2\theta_A(1-\theta_A). \label{bs20}
\end{align}
\textbf{2. Lower Bound : when $\omega_{B^\epsilon}=\omega^c_{A^\epsilon}$ : }
In this case we have the following flux convergence :
\begin{equation*}(B^\epsilon-b_1I)\nabla u^\epsilon = \frac{(b_2-b_1)}{(a_2-a_1)}(a_2I-A^\epsilon)\nabla u^\epsilon\ \rightharpoonup\ \frac{(b_2-b_1)}{(a_2-a_1)}(a_2I-A^{*})\nabla u\ \mbox{ in } L^2(\Omega).
\end{equation*}
Starting with the above information,one follows the procedure
of the previous case to obtain the following inequality : 
\begin{align}
&\{\frac{(b_2-b_1)}{(a_2-a_1)}(a_2I-A^{*})+\frac{b_1}{a_1}(\overline{A}-A^{*})\}(B^{\#}-b_1I)^{-1}\{\frac{(b_2-b_1)}{(a_2-a_1)}(a_2I-A^{*})+\frac{b_1}{a_1}(\overline{A}-A^{*})\}\eta\cdot\eta\notag\\
&\quad\qquad\qquad\qquad\qquad\qquad\quad\leq\ (\overline{B}-b_1 I)\eta \cdot \eta  +\frac{b_1}{a_1^2}(a_2-a_1)^2\theta_A(1-\theta_A)M_A\eta\cdot\eta. \label{eia}
\end{align}
Then by taking trace on both sides of \eqref{eia}, we simply obtain 
\begin{align}\label{FL15}
tr\ \{[\frac{(b_2-b_1)}{(a_2-a_1)}(a_2I-A^{*})+\frac{b_1}{a_1}(\overline{A}-A^{*})]^2(B^{\#}-b_1I)^{-1}\}\leq\ &N(b_2-b_1)(1-\theta_B)\notag\\
&+\frac{b_1}{a_1^2}(a_2-a_1)^2\theta_A(1-\theta_A). 
\end{align}
\textbf{Saturation/Optimality of the above lower bounds : }
Simple calculation shows that the equality of the above lower bound \eqref{bs20} of $(A^{*},B^{\#})$ 
is achieved by the simple laminates, Hashin-Shtrikman construction given in the Section \ref{bs12} and Section \ref{hsl}
respectively.

However, in the case $\omega_{B^\epsilon}=\omega_{A^\epsilon}^c$, the equality of the above lower bound \eqref{FL15} 
is achieved by the simple laminates, Hashin-Shtrikman construction when $\theta_A \leq \frac{1}{2}$
(i.e. when $\theta_B =(1-\theta_A) \geq \theta_A$). 
Later we will see in order to find the optimal bound when $\theta_A >\frac{1}{2}$, one needs to start with the dual version of \eqref{FL8} given below in \eqref{tc}.

With regard to $N$-sequential laminates, the construction of $A^{*}_N$ is
classical. The construction of the corresponding relative limit $B^{\#}_{N,N}$ can be based on a relation linking $A^{*}_N$  and  $B^{\#}_{N,N}$  which is analogous to \eqref{FG4}. This relation is given in \eqref{hsa}  for the case  $\omega_{A^\epsilon} = \omega_{B^\epsilon}$. Similar thing can be done for $\omega_{B^\epsilon}=\omega_{A^\epsilon}^c$. These structures
$A^{*}_N,B^{\#}_{N,N}$ provide the saturation / optimality of the above bounds.
A useful observation is that $A^{*}$ belongs to boundary of $\mathcal{G}_{\theta_A}$ in all three cases above.
\hfill\qed
\paragraph{General Case :}
Although our ultimate goal is to find bounds over the arbitrary microstructures $\omega_{A^\epsilon},\omega_{B^\epsilon}$,
it is a difficult task to perform. The source of difficulty lies in the fact that  we do not have convergence result for the flux $B^\epsilon\nabla u^\epsilon$ as in  \eqref{sik} for arbitrary microstructures  $\omega_{A^\epsilon},\omega_{B^\epsilon}$. We see the products involving $B^{\epsilon}$ and $\nabla u^{\epsilon}$, the oscillations in $\nabla u^{\epsilon}$ is controlled by $A^{\epsilon}$ through the state equation governed with its microstructures.
In order to obtain the optimal bounds, one has to take into account of this fact. In the previous cases, such convergences (cf.\eqref{sik})
are used crucially in order to obtain optimal bounds. However, from these particular cases 
we make an useful observation that the optimal structures of $A^{*}$ (i.e. $A^{*}\in\partial\mathcal{G}_{\theta_A}=\partial\mathcal{G}^L_{\theta_A}\cup \partial\mathcal{G}^U_{\theta_A}$) 
provides the saturation / optimality of the $(A^{*},B^{\#})$ bounds. Motivated by this, in order to treat the general case we would like to take the following strategy : 
First we will be finding the bounds on $B^{\#}$ with arbitrary $\omega_{B^\epsilon}$, whereas $\omega_{A^\epsilon}$ 
be the corresponding optimal microstructures for the homogenized tensor i.e. $A^{*}\in\partial\mathcal{G}_{\theta_A}$. Then in the  second step, we will treat the case $A^{*}\in int(\mathcal{G}_{\theta_A})$. 
\paragraph{Lower Bound :}
Let us establish the lower bound L1 on $(A^{*},B^{\#})$ for arbitrary microstructures  $\omega_{A^\epsilon},\omega_{B^\epsilon}$. 
We break it into two steps as mentioned above : first we treat $A^{*}\in\partial\mathcal{G}_{\theta_A}^L$ (and $A^{*}\in\partial\mathcal{G}_{\theta_A}^U$  can be dealt with analogously); next we
consider $A^{*}\in int(\mathcal{G}_{\theta_A})$.   
\paragraph{Step 1a : (Compactness property of $(\nabla u^\epsilon,A^\epsilon\eta)$)}
In this step, assuming $A^{*}\in \partial\mathcal{G}^L_{\theta_A}$, we state and prove a compactness property of
$(\nabla u^\epsilon,A^\epsilon\eta)$. We formulate it (cf.\eqref{eiz}) in such a way that it can be used to study 
the inflated system \eqref{eik} to derive the required lower trace bound L1.  
Recall the field $\nabla u^\epsilon$ possesses natural constraints \eqref{six}, i.e. 
\begin{equation*} \nabla u^\epsilon \rightharpoonup\nabla u \mbox{ weakly in } L^2(\Omega),\ \ -div(A^\epsilon\nabla u^\epsilon) \mbox{ is }H^{-1}(\Omega)\mbox{ convergent. }\end{equation*}
Now,  we impose one more constraint on the fields $\nabla u^\epsilon$ by restricting them inside the class of fields corresponding to microstructures providing  the saturation/optimality of the lower bound \eqref{lb5}. In order to define such class of optimal fields, we go back to the right hand side of \eqref{lb1} as well as its limit :
\begin{equation}\label{lb20} a_1\nabla u^{\epsilon}\cdot\nabla u^{\epsilon}- 2A^{\epsilon}\nabla u^{\epsilon}\cdot \eta \rightharpoonup a_1\nabla u\cdot\nabla u- 2\overline{A}\nabla u\cdot \eta + X\ \mbox{ in }\ \mathcal{D}^{\prime}(\Omega)\end{equation}
where $X$ is the $H$-measure correction term defined as in \eqref{eid} :
\begin{equation*} X = \langle\langle \varPi_V, Q(U,A\eta)\rangle\rangle  \quad\mbox{with } Q(U,A\eta)= a_1|U|^2-2AU\cdot\eta. \end{equation*}
Following \eqref{eie},\eqref{FL5} we had the lower bound on $X$ as,
\begin{equation*} X \geq\  \langle\langle\varPi_V,\underset{U \in \Lambda_{\xi}}{min}\ Q(U,A\eta)\rangle\rangle = - \frac{\theta_A(1-\theta_A)(a_2 -a_1)^2}{a_1} M_A\eta\cdot\eta\ = X_{min}\end{equation*}
It has been also noted that the equality of this above bound provides the saturation / optimality of the lower bound of $A^{*}$. 
In the following result, we will investigate the compactness property of $\nabla u^\epsilon$ satisfying \eqref{lb20}
with $X=X_{min}$.
\begin{theorem}[Compactness]\label{ub8}
Let us consider the following constrained oscillatory system :
\begin{equation}\label{eiw}
V_\epsilon=\ (\nabla u^\epsilon,A^\epsilon\eta) \rightharpoonup (\nabla u,\overline{A}\eta)=V_0\ \mbox{ in }L^2(\Omega)^{2N}\mbox{ weak }, \eta\in\mathbb{R}^N\smallsetminus\{0\},\end{equation}
\begin{equation}\begin{cases}\label{eiz}
a_1\Delta U_\epsilon - \nabla( div(A^\epsilon\eta) ) \in H^{-2}_{loc}(\Omega)\mbox{ convergent, } \\
\ U_\epsilon = \nabla u^\epsilon.
\end{cases}\end{equation}
Then 
\begin{equation}\label{eil} 
a_1\nabla u^{\epsilon}\cdot\nabla u^{\epsilon}- 2A^{\epsilon}\nabla u^{\epsilon}\cdot \eta \rightharpoonup a_1\nabla u\cdot\nabla u- 2\overline{A}\nabla u\cdot \eta + X_{min} \mbox{ in }\ \mathcal{D}^{\prime}(\Omega)
\end{equation}
where, $ X_{min}$ is defined above (cf.\eqref{FL5}). \\
\\
Conversely, if the sequence satisfying \eqref{eiw} possesses the property \eqref{eil}, then \eqref{eiz} must hold.
\end{theorem}
\begin{proof}
The oscillation variety $\vartheta_1$ of the above differential system \eqref{eiw},\eqref{eiz} is
\begin{equation*}
\vartheta_1=\ \{\ (\xi,U,A\eta) \in \mathbb{S}^{N-1}\times\mathbb{R}^N\times\mathbb{R}^N\ : \ U= \frac{(A\eta\cdot\xi)}{a_1|\xi|^2}\xi\ \}.
\end{equation*}
Note, both $\vartheta$ (see \eqref{os16}) and $\vartheta_1$ have some common constraints, namely $\xi_{i}U_j -\xi_{j}U_i = 0\ \forall i,j.$
The corresponding wave cone $\Lambda_1$ is as follows :
\begin{equation*}
\Lambda_1 =\ \{ (U,A\eta)\in\mathbb{R}^N\times\mathbb{R}^N;\ \exists\ \xi\in \mathbb{S}^{N-1}  \mbox{ such that } U= \frac{(A\eta\cdot\xi)}{a_1|\xi|^2}\xi \}.
\end{equation*}
Next we define $\Lambda_{1,\xi}\subset \Lambda_1$, $\xi\in \mathbb{S}^{N-1}$,
\begin{equation*}
\Lambda_{1,\xi}=\ \{(U,A\eta) \in \mathbb{R}^N\times \mathbb{R}^N;\ U= \frac{(A\eta\cdot\xi)}{a_1|\xi|^2}\xi \};\ \mbox{ So, } \underset{\xi\in\mathbb{S}^{N-1}}{\cup}\Lambda_{1,\xi}=\Lambda_1
\end{equation*}
Following that, we introduce a new linear map $q^{\prime\prime}_{1,\xi}:\mathbb{R}^N \times \mathbb{R}^{N} \longmapsto \mathbb{R}^N \times \mathbb{R}^{N}$, whose 
associated quadratic form is given by
\begin{equation*}
Q_{1,\xi}^{\prime\prime}(U,A\eta) :=\ q_{1,\xi}^{\prime\prime}(U,A\eta)\cdot(U,A\eta) =\ Q(U,A\eta)- \underset{U\in \Lambda_{1,\xi}}{min}Q(U,A\eta)
\end{equation*}
where, $Q(U,A\eta)=a_1|U|^2-2U\cdot A\eta$.\\
Since the above quadratic form $Q_{1,\xi}^{\prime\prime}(U,A\eta)$ is zero on the wave cone $\Lambda_{1,\xi}$, we get 
by applying Theorem \ref{dc10},
\begin{equation*} 
\langle\langle \varPi_V,Q_{1,\xi}^{\prime\prime}\rangle\rangle = trace\ (q_{1,\xi}^{\prime\prime}\varPi_V) = 0
\end{equation*}
where $\varPi_V$ is the $H$-measure of the sequence $(V_\epsilon-V_0)$. Thus we get using \eqref{sij}
\begin{align*}
X=\ \langle\langle\varPi_{V}, Q(U,A\eta)\rangle\rangle\ =\  \langle\langle\varPi_{V},\underset{U\in \Lambda_{1,\xi}}{min}Q(U,A\eta) \rangle\rangle &=\   \langle\langle\varPi_A, Q(\frac{(A\eta\cdot\xi)}{a_1|\xi|^2}\xi,A\eta) \rangle\rangle \notag\\ 
&=\ \langle\langle\varPi_A, -\frac{(A\eta\cdot\xi)^2}{a_1|\xi|^2} \rangle\rangle = X_{min}.
\end{align*}
Hence, under the oscillatory system \eqref{eiz}, the limit \eqref{eil} holds by Theorem \ref{dc10}.\\
\\ 
\textbf{Converse part : } The proof is inspired from  \cite[Chapter 28]{T}. Let us recall \eqref{eiy}, for $V_\epsilon \rightharpoonup V_0$ in $L^2(\Omega)$ weak, we had shown that : 
\begin{equation}\label{lb19} trace\ (q^{\prime}_{\xi}\varPi_V) \geq 0, \end{equation}
where the matrix $q^\prime_\xi$ introduced in \eqref{sit} is given by
\begin{equation}\label{eit}
q^\prime_\xi \in \mathcal{L}_{x,\xi}(\mathbb{R}^{2N};\mathbb{R}^{2N});\ \ q^\prime_\xi = 
\begin{bmatrix}
a_1I_{N\times N} & -I_{N\times N}\\
-I_{N\times N} & a_1^{-1}B 
\end{bmatrix}_{2N\times 2N}
\end{equation}
where, $\mathcal{L}_{x,\xi}(\mathbb{R}^{N_1};\mathbb{R}^{N_2})$ denotes the space of $N_1\times N_2$ matrix whose coefficients may depend on $(x,\xi)\in \mathbb{R}^N\times \mathbb{S}^{N-1}$. 
\begin{equation*} B\in \mathcal{L}_{x,\xi}(\mathbb{R}^{N};\mathbb{R}^{N});\ B= \begin{bmatrix}
\xi_1^2 & \xi_1\xi_2 & .. & \xi_1\xi_N \\
\xi_1\xi_2 & \xi_2^2 & .. & \xi_2\xi_N \\
..    &  ..     & .. &  ..         \\
\xi_1\xi_N & \xi_2\xi_N & .. & \xi_N^2
\end{bmatrix}_{N\times N}.
\end{equation*}
Notice that, $B^2=B$ as $\xi\in \mathbb{S}^{N-1}$.\\
\\
Here $\varPi_V\in \mathcal{M}(\Omega\times\mathbb{S}^{N-1};\mathbb{R}^{2N\times 2N})$ ($\varPi_V =\varPi_V^{*}$, Hermitian) is associated $H$-measure of the sequence $(V_\epsilon-V_0)$.
Since $V^\epsilon=(\nabla u^\epsilon, A^\epsilon\eta)$, where  $A^\epsilon = a^\epsilon I$ satisfies the constraints
\begin{equation*}
curl\ (\nabla u^\epsilon) = 0, \mbox{ and } -div\lb (A^\epsilon\eta)_j\nabla u^\epsilon\rb  \mbox{ converges strongly in } H^{-1}(\Omega), \mbox{ for each }j=1,..,N,
\end{equation*}
we have by the localization principle (cf. Theorem \ref{dc9}), $H$-measure $\varPi_V$ satisfies the following relations :\\
\\
(a): The sub-matrix $\{(\varPi_V)_{jk}\}_{1\leq j,k\leq N} $ of $\varPi_V$ satisfies : 
\begin{equation}\label{sin}
\xi_k(\varPi_V)_{jl} - \xi_j(\varPi_V)_{kl} = 0 \mbox{ for }j,k,l=1,..,N. 
\end{equation}
(b): The sub-matrix $\{(\varPi_V)_{jk}\}_{1\leq k\leq N, N+1\leq j\leq 2N}$ of $\varPi_V$ satisfies 
\begin{equation}\label{lb16}
\sum_{k=1}^N\xi_k(\varPi_V)_{jk} = 0 \mbox{ for }j=N+1,..,2N.
\end{equation} 
Now \eqref{eil} holds only if the equality holds in \eqref{lb19}, i.e.
\begin{equation}\label{eix} trace\ (q^{\prime}_{\xi}\varPi_V) = 0.\end{equation}
Now we will show \eqref{eix} implies \eqref{eiz}.\\
\\
In order to do that, we begin with localizing the sequence $V_\epsilon-V_0$ 
by $\phi(V_\epsilon-V_0)$, $\phi\in C_c^1(\Omega)$.
We will be using $\frac{\partial}{\partial x_j}=(-\Delta)^{\frac{1}{2}}R_j$ and
$\frac{{\partial}^2}{\partial x_j\partial x_k}=(-\Delta)R_jR_k$ and
where, $R_j$  is the Riesz operator (see \cite{STN}). Let us define the sequence
\begin{equation*}
\omega^\epsilon_k = a_1\sum_{j=1}^N R_jR_j \phi(U_\epsilon-U)_k - R_k\sum_{j=1}^N R_j \phi(A^\epsilon\eta-\overline{A}\eta)_j ; \ \ k=1,..,N.
\end{equation*}
We know 
\begin{equation*}
 \omega_\epsilon =\{\omega^\epsilon_k\}_{1\leq k\leq N}\ \rightharpoonup 0 \mbox{ in }L^2(\mathbb{R}^N)^N \mbox{ weak. } 
\end{equation*}
To prove \eqref{eiz}, it is enough to show that the sequence
\begin{equation}\label{lb15}
 \omega_\epsilon =\{\omega^\epsilon_k\}_{1\leq k\leq N}\ \rightarrow 0 \ \mbox{ in }L^2(\mathbb{R}^N)^N\mbox{ strong.}
\end{equation}
Let  $\varPi\in \mathcal{M}(\Omega\times\mathbb{S}^{N-1};\mathbb{R}^{N\times N})$ be $H$-measure associated to $\{\omega_k^\epsilon\}_{1\leq k\leq N}$
then we compute trace of $\varPi$ :
\begin{equation*}
 trace\ \varPi = -\sum_{k=1}^N a_1^2|\phi|^2(\varPi_V)_{kk} -2\sum_{k=1}^N\sum_{l=1}^{N} \xi_k|\phi|^2(\varPi_V)_{k,l+N}  -\sum_{k=1}^N\sum_{l=1}^N\xi_k\xi_l|\phi|^2(\varPi_V)_{k+N,l+N}.
\end{equation*}
The middle term of right hand side  vanishes because of \eqref{lb16}.  
Another simple computation using \eqref{eit} shows that the right hand side is equal to 
$a_1|\phi|^2trace\ (q^\prime_\xi\varPi_V))$ and hence  $trace\ \varPi =0$ because of \eqref{eix}. 
So \eqref{lb15} follows. Consequently, \eqref{eiz} follows.   
\hfill\end{proof}
\noindent
\textbf{Step 1b : ($H$-measure term) :}
Here onwards we will find the lower bounds by choosing  the field $\nabla u^\epsilon$ satisfying \eqref{six} and \eqref{eiz} or equivalently 
$A^{*}\in\partial\mathcal{G}_{\theta_A}^L$ and using them in the translated inequality \eqref{FL8} :
\begin{equation*}(B^{\epsilon}-b_1 I)(\nabla u^{\epsilon} +\eta)\cdot(\nabla u^{\epsilon} +\eta)\geq 0 \quad\mbox{a.e. in }\Omega.\end{equation*}
We expand the translated inequality to write,
\begin{equation}\label{lb8}B^{\epsilon}\nabla u^{\epsilon}\cdot\nabla u^{\epsilon} + (B^{\epsilon}- b_1 I)\eta \cdot \eta  -2b_1\nabla u^{\epsilon}\cdot\eta\ \geq\ b_1\nabla u^{\epsilon}\cdot\nabla u^{\epsilon} -2B^\epsilon\nabla u^{\epsilon}\cdot \eta\end{equation}
Passing to the limit on the left hand side is well known. On the other hand in order to pass to the limit in the right hand side, 
we use the $H$-measure techniques. Introducing a coupled variable $W_{\epsilon}=(\nabla u^{\epsilon}, B^{\epsilon}\eta)$, we write the right 
hand side of \eqref{lb8} in a quadratic form $q_1(W_{\epsilon})\cdot W_{\epsilon} $. 
Here $q_1$ is a linear map :
\begin{equation*} q_1:\mathbb{R}^N \times \mathbb{R}^{N} \longmapsto \mathbb{R}^N \times \mathbb{R}^{N}\ \mbox{ defined by }\ q_1(\nabla u^\epsilon,B^\epsilon\eta)=\ (b_1\nabla u^\epsilon-B^\epsilon\eta, -\nabla u^\epsilon ). \end{equation*} 
Introducing $W_0 =(\nabla u,\overline{B}\eta)$, as before we have 
\begin{equation*} q_1(W_{\epsilon})\cdot W_{\epsilon}\ =\ 2q_1(W_{\epsilon})\cdot W_0 -q_1(W_0)\cdot W_0 + q_1(W_{\epsilon}- W_0)\cdot(W_{\epsilon}- W_0).\end{equation*}
Denoting by $\varPi_W$ the $H$-measure of $(W_{\epsilon}- W_0)$, we thus pass to the limit in \eqref{lb8} by virtue of Theorem \ref{dc10} :
\begin{equation}\label{lb9}
B^{\#}\nabla u\cdot\nabla u + (\overline{B}-b_1 I)\eta \cdot \eta - 2b_1 \nabla u\cdot \eta\ \geq\ b_1\nabla u\cdot\nabla u- 2\overline{B}\nabla u\cdot \eta + Y  
\end{equation}
where $Y$ is the $H$-measure correction term defined by 
\begin{equation*}
Y = \underset{\epsilon \rightarrow 0}{\textrm{lim}}\ q_1(W_{\epsilon}- W_0)\cdot(W_{\epsilon}- W_0) = \int_{\mathbb{S}^{N-1}} trace\ (q_1\varPi_W (x,d\xi))
\end{equation*}
or equivalently, denoting the average w.r.t. directions $\xi\in\mathbb{S}^{N-1}$  by double angular bracket, we write with
$W = (U, B\eta) \in \mathbb{R}^N \times \mathbb{R}^N,$ we can write 
\begin{equation}\label{dc11}
Y=\langle\langle \varPi_W, Q_1(U,B\eta)\rangle\rangle, \mbox{ with }Q_1(U,B\eta) := b_1U\cdot U-2BU\cdot\eta. 
\end{equation}
The aim is to find tight lower bound for $H$-measure term $Y$. (Earlier we did
this job for $X$ (see \eqref{FL5}). As before, any such lower bound  must result
Theorem \ref{dc10} applied to an appropriate oscillatory system. This idea is carried out in
the next step.\\

\noindent\textbf{Step 1c : (Lower bound) :}
To take care of the interaction between microstructures $A^\epsilon$,$B^\epsilon$, we need to work with the following inflated oscillatory system with constraints : 
\begin{equation}\label{eik}
\begin{cases}
 W_\epsilon^\prime = (\nabla u^\epsilon, B^\epsilon\eta, A^\epsilon\eta ) \rightharpoonup  W_0^\prime= ( \nabla u, \overline{B}\eta, \overline{A}\eta )\ \mbox{ weakly in } L^2(\Omega)^{3N},\\ 
a_1\Delta U_\epsilon - \nabla( div(A^\epsilon\eta) ) \in H^{-2}_{loc}(\Omega)\mbox{ convergent, }\\
  U_\epsilon = \nabla u^\epsilon.
 \end{cases}
\end{equation}
We proceed in a fashion analogous to Step $2$ of Section \ref{os9}. 
We introduce a new linear form $q^{\prime}_{1,\xi}$, whose associated quadratic form is defined as 
\begin{equation}\label{eis} 
Q^\prime_{1,\xi}(U,B\eta,A\eta)= q^{\prime}_{1,\xi}(U,B\eta,A\eta)\cdot(U,B\eta,A\eta) := Q_1(U,B\eta)-\underset{U \in \Lambda_{1,\xi}}{min}\  Q_1(U,B\eta), 
\end{equation}
where $\Lambda_{1,\xi}$ is the corresponding wave cone for the inflated system \eqref{eik} :
\begin{equation*}\Lambda_{1,\xi}=\ \{(U,A\eta) \in \mathbb{R}^N\times \mathbb{R}^N;\ U= \frac{(A\eta\cdot\xi)}{a_1|\xi|^2}\xi \}.\end{equation*}
So \eqref{eis} becomes,
\begin{equation*} 
Q^\prime_{1,\xi}(U,B\eta,A\eta)= Q_1(U,B\eta)- Q_1(\frac{(A\eta\cdot\xi)}{a_1|\xi|^2}\xi ,B\eta).
\end{equation*}
Note that, $Q^{\prime}_{1,\xi}$ is zero on the wave cone $\Lambda_{1,\xi}$. Thus applying Theorem \ref{dc10} we get 
\begin{equation*} trace\ (q^{\prime}_{1,\xi}\varPi_{W^\prime}) = 0 \end{equation*}
which implies that, $Y$ (cf.\eqref{dc11}) 
\begin{align}
Y =\langle\langle\varPi_{W^\prime},  Q_1(\frac{(A\eta\cdot\xi)}{a_1|\xi|^2}\xi,B\eta)\rangle\rangle &= \langle\langle\varPi_{W^\prime},\lb b_1\frac{(A\eta\cdot\xi)^2}{a_1^2|\xi|^2}-2\frac{(A\eta\cdot\xi)(B\eta\cdot\xi)}{a_1|\xi|^2}\rb\rangle\rangle\notag\\
&= \langle\langle \varPi_{W^\prime}, b_1\lb \frac{(A\eta\cdot\xi)}{a_1|\xi|} - \frac{(B\eta\cdot\xi)}{b_1|\xi|}\rb^2\rangle\rangle -\langle\langle\varPi_{W^\prime},\frac{(B\eta\cdot\xi)^2}{b_1|\xi|^2}\rangle\rangle\notag\\    
&= \frac{1}{a_1^2b_1}\langle\langle\varPi_{AB}, \frac{((b_1A-a_1B)\eta\cdot\xi)^2}{|\xi|^2}  \rangle\rangle -\frac{1}{b_1}\langle\langle \varPi_B, \frac{(B\eta\cdot\xi)^2}{|\xi|^2}\rangle\rangle\notag \\
&=: R_1 + R_2 \mbox{ (say),}\label{FL6}
\end{align}
where $\varPi_{AB}\in \mathcal{M}(\Omega\times \mathbb{S}^{N-1};\mathbb{R}^{N\times N}) $ is $H$-measure of the sequence
$\{(b_1A^\epsilon-a_1B^\epsilon) -(b_1\overline{A}-a_1\overline{B})\}\eta$ and
$\varPi_B\in \mathcal{M}(\Omega\times \mathbb{S}^{N-1};\mathbb{R}^{N\times N})$ 
is $H$-measure of the sequence $(B^\epsilon-\overline{B})\eta$.\\
\\
Since, $\{(b_1A^\epsilon-a_1B^\epsilon) -(b_1\overline{A}-a_1\overline{B})\}\eta = \{b_1(a_1-a_2)(\A-\theta_A(x))-a_1(b_1-b_2)(\B-\theta_B(x))\}\eta$, 
the $H$-measure $\varPi_{AB}$ reduces to 
\begin{equation*}
(\varPi_{AB})_{ij} = (\nu_{AB})\eta_i\eta_j  \ \ \forall i,j =1,..,N
\end{equation*}
where, $\nu_{AB}$ is  $H$-measure of the scalar sequence $\{b_1(a_1-a_2)(\A-\theta_A(x))-a_1(b_1-b_2)(\B-\theta_B(x))\}$ with
\begin{equation*}
\nu_{AB}(x,d\xi)\geq 0\ \mbox{ and }\ \int_{\mathbb{S}^{N-1}}\nu_{AB}(x,d\xi) = L_{AB}
\end{equation*}
with 
\begin{align*}
L_{AB} & := L^\infty(\Omega)\mbox{ weak* limit of }\{b_1(a_1-a_2)(\A-\theta_A)-a_1(b_1-b_2)(\B-\theta_B)\}^2\notag\\
&= b_1^2(a_2-a_1)^2\theta_A(1-\theta_A) +a_1^2(b_2-b_1)^2\theta_B(1-\theta_B)-2b_1a_1(b_2-b_1)(a_2-a_1)(\theta_{AB}-\theta_A\theta_B)\notag\\
\end{align*}
where $\theta_{AB}$ is the $L^{\infty}(\Omega)$ weak* limit of $(\AB)$. Moreover, using \eqref{FG2} we have
\begin{align}\label{dc18} 
L_{AB} \geq\ & \  b_1^2(a_2-a_1)^2\theta_A(1-\theta_A) +a_1^2(b_2-b_1)^2\theta_B(1-\theta_B)\notag\\
             &\quad\qquad\qquad\qquad\qquad\qquad\qquad\qquad-2b_1a_1(b_2-b_1)(a_2-a_1)(\mbox{min}\{\theta_A,\theta_B\}-\theta_A\theta_B)\notag\\
             & =: L^0_{AB} \mbox{ (say).}
\end{align}
Thus
\begin{align*}
R_1 = \frac{1}{a_1^2b_1}\langle\langle\varPi_{AB}, \frac{((b_1A-a_1B)\eta\cdot\xi)^2}{|\xi|^2}  \rangle\rangle &= \frac{1}{a_1^2b_1} \int_{\mathbb{S}^{N-1}} \frac{(\eta \cdot \xi)^2}{|\xi|^2}\nu_{AB} (d\xi)
                                                             = \frac{1}{a_1^2b_1}L_{AB}\ M_{AB}\eta\cdot\eta
\end{align*}
where $M_{AB}$ is a non-negative matrix with unit trace defined by 
\begin{equation}\label{siv}
 M_{AB} = \frac{1}{L_{AB}}\int_{\mathbb{S}^{N-1}} \xi \otimes \xi\ \nu_{AB} (d\xi). 
\end{equation}
Hence, using \eqref{dc18} and \eqref{siv} we have  
\begin{align}\label{dc14}
R_1\ \geq\ &\{ \frac{b_1}{a_1^2}(a_2-a_1)^2\theta_A(1-\theta_A) +\frac{1}{b_1}(b_2-b_1)^2\theta_B(1-\theta_B)\notag\\
             &\qquad\qquad\qquad\qquad-\frac{2}{a_1}(b_2-b_1)(a_2-a_1)(\mbox{min}\{\theta_A,\theta_B\}-\theta_A\theta_B)\}M_{AB}\eta\cdot\eta.
\end{align}

We compute $R_2$ next. Since $(B^\epsilon-\overline{B})\eta=(b_2-b_1)(\B-\theta_B(x))\eta$, the $H$-measure
$\varPi_B$ reduces to 
\begin{equation*}
 (\varPi_B)_{ij} = (b_2-b_1)^2(\nu_B)\eta_i\eta_j  \ \ \forall i,j =1,..,N
\end{equation*}
where $\nu_B$ is the $H$-measure of the scalar sequence $(\B-\theta_B)$ satisfying 
\begin{equation*}
\nu_{AB}(x,d\xi)\geq 0\ \mbox{ and }\int_{\mathbb{S}^{N-1}}\nu_{AB}(x,d\xi) =\theta_B(1-\theta_B). 
\end{equation*}
Thus 
\begin{align}\label{dc16}
R_2 = -\frac{1}{b_1}\langle\langle\varPi_B, \frac{((B\eta\cdot\xi)^2}{|\xi|^2}  \rangle\rangle &= -\frac{(b_2-b_1)^2}{b_1} \int_{\mathbb{S}^{N-1}} \frac{(\eta \cdot \xi)^2}{|\xi|^2}\nu_B (d\xi)\notag\\
     &=- \frac{\theta_B(1-\theta_B)(b_2-b_1)^2}{b_1}M_B\eta\cdot\eta
\end{align}
where $M_B$ is the non-negative matrix with unit trace defined as 
\begin{equation}\label{dc19}
 M_B = \frac{1}{(1-\theta_B)\theta_B}\int_{\mathbb{S}^{N-1}} \xi\otimes\xi\ \nu_B(x,d\xi).
\end{equation}
Therefore from \eqref{FL6} and \eqref{dc14},\eqref{dc16} we have 
\begin{align}\label{dc13}
Y  &=  - \frac{\theta_B(1-\theta_B)(b_2-b_1)^2}{b_1}M_B\eta\cdot\eta + \frac{1}{a_1^2b_1}{L_{AB}}\ M_{AB}\eta\cdot\eta \notag\\
   &\geq - \frac{\theta_B(1-\theta_B)(b_2-b_1)^2}{b_1}M_B\eta\cdot\eta +\{\frac{b_1}{a_1^2}(a_2-a_1)^2\theta_A(1-\theta_A) +\frac{1}{b_1}(b_2-b_1)^2\theta_B(1-\theta_B)\notag\\
          &\qquad\qquad\qquad\qquad\qquad\qquad -\frac{2(b_2-b_1)(a_2-a_1)}{a_1}(\mbox{min}\{\theta_A,\theta_B\}-\theta_A\theta_B)\}M_{AB}\eta\cdot\eta\notag\\ 
&\ \ \ \ =: Y_{min} 
\end{align}
with 
\begin{equation}\label{dc17}
 trace\ Y\ \geq\ \frac{b_1}{a_1^2}(a_2 - a_1)^2\theta_A(1-\theta_A) - \frac{2}{a_1}(b_2-b_1)(a_2-a_1)(\mbox{min}\{\theta_A,\theta_B\}-\theta_A\theta_B).
\end{equation}
Rewriting \eqref{lb9}, we get
\begin{equation}\label{ub15}(B^{\#}- b_1 I)\nabla u\cdot\nabla u + (\overline{B}-b_1 I)\eta \cdot \eta + 2 (\overline{B}-b_1I)\nabla u\cdot \eta \geq\ Y\end{equation}
\textbf{Choice of $\nabla u$ : }
Let us take the point $x=x_0$, where $A^{*}(x_0)$, $B^{\#}(x_0)$ are defined as constant matrices with the proportion $\theta_A(x_0)$ and $\theta_B(x_0)$. We consider the oscillatory test function as in \eqref{zz11} namely,
\begin{equation*}
(i)\ \ \nabla u^\epsilon \rightharpoonup \zeta\in\mathbb{R}^N \mbox{ arbitrary and }(ii)\ \    div(A^\epsilon\nabla u^\epsilon) \mbox{ converges }H^{-1}(\Omega) \mbox{ strong.}
\end{equation*}
In the present context, we have additional restriction \eqref{eik}. According to  Theorem \ref{ub8} this restriction is satisfied by $u^\epsilon$ because $A^{*}\in\partial\mathcal{G}_{\theta_A}^L$.
We choose $\nabla u=\zeta$ to be the minimizer of left hand side of \eqref{lb4} as : 
\begin{equation}\label{ub4}
\zeta = -(\overline{A}(x_0)-a_1I)(A^{*}(x_0)- a_1 I)^{-1}\eta.
\end{equation}
\noindent\textbf{Matrix lower bound :}
Thus from \eqref{ub15} by using \eqref{ub4} and \eqref{dc13} we obtain the following estimate at $x_0$ :
\begin{align}\label{bs10}
&(B^{\#}- b_1 I)(\overline{A}-a_1I)^2(A^{*}- a_1 I)^{-2}\eta\cdot\eta  - 2 (\overline{B}-b_1I)(\overline{A}-a_1I)(A^{*}- a_1 I)^{-1}\eta\cdot\eta + (\overline{B}-b_1 I)\notag\\
&\eta\cdot\eta \geq - \frac{\theta_B(1-\theta_B)(b_2-b_1)^2}{b_1}M_B\eta\cdot\eta + \{\frac{b_1}{a_1^2}(a_2-a_1)^2\theta_A(1-\theta_A) +\frac{1}{b_1}(b_2-b_1)^2\theta_B(1-\theta_B)\notag\\
&\qquad\qquad\qquad\qquad\qquad\qquad\quad  -\frac{2(b_2-b_1)(a_2-a_1)}{a_1}(\mbox{min}\{\theta_A,\theta_B\}-\theta_A\theta_B)\}M_{AB}\eta\cdot\eta 
\end{align}
where, $A^{*}$ satisfies 
\begin{equation}\label{eif} 
(A^{*}-a_1I)^{-1}\eta\cdot\eta =\ (\overline{A}-a_1I)^{-1}\eta\cdot\eta + \frac{\theta_A}{(1-\theta_A)a_1}M_A\eta\cdot\eta. 
\end{equation}
Since $x_0$ is arbitrary, the matrix bound \eqref{bs10} is pointwise bound for $x$ almost everywhere. 
\noindent\textbf{Trace bound L1 : : $\theta_A\leq\theta_B$ almost everywhere in $x$ :}
Having obtained the matrix inequality \eqref{bs10}, we now prove trace bound L1. We simplify the above bound \eqref{bs10}, using \eqref{eif}, \eqref{dc17} to  obtain :
\begin{equation}\label{eig}
tr\ (B^{\#}-b_1I)(\overline{A}-a_1I)^2(A^{*}-a_1I)^{-2}\geq\ tr\ (\overline{B}-b_1I) +\frac{b_1}{a_1^2}(a_2-a_1)^2\theta_A(1-\theta_A).
\end{equation}
Hence the trace bound \eqref{tt} follows whenever $A^{*}\in\partial\mathcal{G}^L_{\theta_A}$.\hfill\qed\\
\\
Next, we give the examples of microstructures which possess the property  $\omega_{A^\epsilon}\subseteq\omega_{B^\epsilon}$ and such that
equality holds in \eqref{bs10} and hence in \eqref{eig}.
\begin{example}[Saturation/Optimality]\label{ot5}
The equality of this lower bound is achieved by the simple laminated materials $L^{\#}_1$ (cf. \eqref{lb10}),  at the
composites based on Hashin-Shtrikman construction given in \eqref{ED1} and at sequential laminates of $N$-rank.
\bpr   
The simple laminate, say in $e_1$ direction with $\theta_A\leq \theta_B$ is given by  
\begin{equation*}B^{\#} = diag\ (L^{\#}_1, \overline{b},..,\overline{b}\ )\ \mbox{ (cf.\eqref{lb10}) with }\ A^{*}= diag\ (\underline{a},\overline{a},..,\overline{a}).\end{equation*} 
It achieves equality in \eqref{bs10} with $M_B = M_{AB}= diag\ ( 1,0,..,0)$.
Indeed, 
\begin{align*}
(L^{\#}_1- b_1 )\frac{(\overline{a}-a_1)^2}{(\underline{a}- a_1 )^2}&=\ (L^{\#}_1- b_1 )\frac{(\theta_A(a_2 -a_1) +a_1)^2}{a_1^2}\\
&=\ \{\frac{b_2}{a_2^2} + \frac{(b_1-b_2)}{a_2^2}\theta_B + {b_1}(\frac{1}{a_1^2} -\frac{1}{a_2^2})\theta_A \} a_2^2 - \frac{b_1}{a_1^2}(\theta_A(a_2 -a_1) +a_1)^2 \\
&=\ (b_2-b_1)(1-\theta_B) + \frac{b_1}{a_1^2}(a_2-a_1)^2\theta_A(1-\theta_A).
\end{align*}
Similarly, the Hashin-Shtrikman construction given in \eqref{ED1}
i.e. with core $a_2 I$ and coating $a_1 I$ for $A_{B(0,1)}$ and core $b_2 I$ with coating $b_1 I$ for $B_{B(0,1)}$, with $\theta_A \leq \theta_B$
\begin{align*}
a^{*}&=\ a_1 + Na_1\frac{(1-\theta_A)(a_2-a_1)}{(N-\theta_A)a_1 +\theta_Aa_2}\\
b^{\#} &=\  b_1\ [\ 1 + \frac{N\theta_A(1-\theta_A)(a_2-a_1)^2}{(\theta_A a_2 + (N-\theta_A)a_1)^2}\ ] - \frac{(b_1-b_2)(Na_1)^2(1-\theta_B)}{(\theta_A a_2 + (N-\theta_A)a_1)^2}.                         
\end{align*}
achieves the equality in \eqref{bs10} with $M_B = M_{AB}= diag\ (\frac{1}{N},\frac{1}{N},..,\frac{1}{N})$.
\hfill\qed
\paragraph{Construction of Sequential Laminates with ${\omega}_{A^{\epsilon}}\subseteq {\omega}_{B^{\epsilon}} $ :}
First we are going to write down relations characterizing the $(p,p)$-sequential laminates $B^{\#}_{p,p}$
whenever $\omega_{A^{\epsilon}_p},\omega_{B^{\epsilon}_p}$ correspond to the $p$-sequential laminate
microstructures with $\omega_{A^{\epsilon}_p}\subseteq \omega_{B^{\epsilon}_p}$ in the same directions $\{e_i\}_{1\leq i\leq p}$.
Following the arguments presented  in Section \ref{os9}, by taking $A^{*}=A^{*}_p$ with matrix $a_1 I$ and core $a_2 I$
defined in \eqref{OP2} and considering the inequality \eqref{bs10}, it is natural to define $B^{\#}_{p,p}$ via the relation (namely equality of \eqref{bs10})
\begin{align*}
&(B^{\#}_{p,p}- b_1 I)(\overline{A}-a_1I)^2(A^{*}_p - a_1 I)^{-2} 
+ (\overline{B}-b_1 I)- 2 (\overline{B}-b_1I)(\overline{A}-a_1I)(A^{*}_p- a_1 I)^{-1}\notag\\
&=\{\frac{b_1(a_2 - a_1)^2}{a_1^2}\theta_A(1-\theta_A)- \frac{2(b_2-b_1)(a_2-a_1)}{a_1}(1-\theta_B)\theta_A\}(\sum_{i=1}^{p} m_i\frac{e_i\otimes e_i}{e_i.e_i}) \mbox{ with }\sum_{i=1}^{p} m_i =1.
\end{align*}
or using \eqref{OP2}, we may write
\begin{equation}\label{sie}
 (B^{\#}_{p,p}- b_1 I)(\overline{A}-a_1I)^2(A^{*}_p - a_1 I)^{-2} = (\overline{B}-b_1 I)+\frac{b_1(a_2-a_1)^2}{a_1^2}\theta_A(1-\theta_A)(\sum_{i=1}^{p} m_i\frac{e_i\otimes e_i}{e_i.e_i}).
\end{equation}
This can be also written as, (in an inverse form):
\begin{align}\label{hsa}
&\{\frac{(\overline{B}-b_1I)}{(\overline{A}-a_1I)}(A^{*}_p-a_1I)+\frac{b_1}{a_1}(\overline{A}-A^{*}_p)\}(B^{\#}_{p,p}-b_1I)^{-1}\{\frac{(\overline{B}-b_1I)}{(\overline{A}-a_1I)}(A^{*}_p-a_1I)+\frac{b_1}{a_1}(\overline{A}-A^{*}_p)\}\notag\\
&=\ (\overline{B}-b_1 I)+\frac{b_1(a_2-a_1)^2}{a_1^2}\theta_A(1-\theta_A)(\sum_{i=1}^{p} m_i\frac{e_i\otimes e_i}{e_i.e_i}).
\end{align}

We need to justify the above definition because it's not a priori clear that $B^{\#}_{p,p}$ defined in the above manner is indeed the relative limit of $B^\epsilon_p$. To this end, we begin with considering the sequence $A^\epsilon_p$ containing $p$-sequential laminate microstructures with matrix $a_1I$ and core $a_2I$ such that $A^\epsilon_p\xrightarrow{H}A^{*}_p\in \partial\mathcal{G}_{\theta_A}^L$. Then by considering the 
oscillatory test sequence $\nabla u^\epsilon$ satisfying \eqref{zz11} i.e. 
\begin{align*}
&(i)\ \nabla u^\epsilon \rightharpoonup \nabla u= -(\overline{A}-a_1I)(A^{*}_p- a_1 I)^{-1}\eta \mbox{ in }L^2(\Omega) \mbox{ weak }\\
\mbox{ and }\ \ &(ii)\    div(A^\epsilon_p\nabla u^\epsilon) \mbox{ converges }H^{-1}(\Omega) \mbox{ strong};
\end{align*}
and by using \eqref{FL5}, \eqref{UD10} we have 
\begin{align}\label{pol17}
(A^\epsilon_p-&a_1I)(\nabla u^\epsilon +\eta)\cdot(\nabla u^\epsilon +\eta)\notag\\
 &\rightharpoonup (A^{*}_p-a_1I)\nabla u\cdot\nabla u +2(\overline{A}-a_1I)\nabla u\cdot\eta +(\overline{A}-a_1I)\eta\cdot\eta - X_{min}= 0 \ \ \mbox{in }\mathcal{D}^\prime(\Omega)
\end{align}
where we recall, 
\begin{equation*}
X_{min}= - \frac{\theta_A(1-\theta_A)(a_2-a_1)^2}{a_1}(\sum_{i=1}^{p} m_i\frac{e_i\otimes e_i}{e_i.e_i})\ \mbox{ with }\sum_{i=1}^{p} m_i =1.
\end{equation*}
Next, by using $(A^\epsilon_p -a_1I) = (a_2-a_1)(1-\chi_{\omega_{A^{\epsilon}_p}})I$,  it follows that from \eqref{pol17} that 
\begin{equation}\label{pol11}
 \underset{\epsilon\rightarrow 0}{lim}\  |(1-\chi_{\omega_{A^{\epsilon}_p}})(\nabla u^\epsilon +\eta)|^2 =0. 
\end{equation}
Now, under the hypothesis $\omega_{A^{\epsilon}_p}\subseteq\omega_{B^{\epsilon}_p}$, from \eqref{pol11} it also follows that,
\begin{equation*}
0\ \leq\ \underset{\epsilon\rightarrow 0}{lim}\  |(1-\chi_{\omega_{B^{\epsilon}_p}})(\nabla u^\epsilon +\eta)|^2\ \leq\  \underset{\epsilon\rightarrow 0}{lim}\  |(1-\chi_{\omega_{A^{\epsilon}_p}})(\nabla u^\epsilon +\eta)|^2. 
\end{equation*}
Therefore, 
\begin{equation}\label{pol7}
 \underset{\epsilon\rightarrow 0}{lim}\  |(1-\chi_{\omega_{B^{\epsilon}_p}})(\nabla u^\epsilon +\eta)|^2 = 0.  
\end{equation}
Consequently, by using $(B^\epsilon_p -b_1I) =(b_2-b_1)(1-\chi_{\omega_{B^\epsilon_p}})I$, from \eqref{pol7} we have 
\begin{equation}\label{pol18}
 (B^\epsilon_p-b_1I)(\nabla u^\epsilon +\eta)\cdot(\nabla u^\epsilon +\eta)\rightharpoonup 0 \mbox{ in }\mathcal{D}^\prime(\Omega). 
\end{equation}
On the other hand, computations made in \textbf{Step 1c} show directly that for microstructures $(A^\epsilon_p$, $ B^\epsilon_p)$, we have $Y=Y_{min}$ with 
$M_A =M_{AB} = (\sum_{i=1}^{p} m_i\frac{e_i\otimes e_i}{e_i.e_i})$ with $\sum_{i=1}^{p} m_i =1$ such that,
\begin{equation*}
Y_{min} =\{\frac{b_1}{a_1^2}(a_2 - a_1)^2\theta_A(1-\theta_A) - \frac{2}{a_1}(b_2-b_1)(a_2-a_1)\theta_A(1-\theta_B)\}
(\sum_{i=1}^{p} m_i\frac{e_i\otimes e_i}{e_i.e_i}) \mbox{ with}\sum_{i=1}^{p} m_i =1.
\end{equation*}
Consequently, following \textbf{Step 1b}, we  have (cf. \eqref{lb9}) 
\begin{align}\label{pol19}
(B^\epsilon_p&-b_1I)(\nabla u^\epsilon +\eta)\cdot(\nabla u^\epsilon +\eta)\notag\\ 
&\rightharpoonup (B^{\#}_{p,p}-b_1I)\nabla u\cdot\nabla u +(\overline{B}-b_1I)\eta\cdot\eta + 2(\overline{B}-b_1I)\nabla u\cdot\eta - Y_{min}\ \mbox{ in }\ \mathcal{D}^\prime(\Omega). 
\end{align}
Next by comparing \eqref{pol18} and \eqref{pol19}, and finally using $\nabla u =  -(\overline{A}-a_1I)(A^{*}_p- a_1 I)^{-1}\eta$ we obtain 
that $B^{\#}_{p,p}$ satisfies the relation \eqref{sie}.\\ 
With $p=N$, the above defined $(N,N)$- sequential laminate $B^{\#}_{N,N}$ give the saturation / optimality of the lower trace bound \eqref{tt}. This follows from the very defining relation of $B^{\#}_{N,N}$ (cf.\eqref{sie}).
\hfill\epr
\end{example}
\begin{remark}
The difference between the present case with the one treated in Section \ref{os9} lies in the fact that 
we have now two phases $(b_1,b_2)$ with local volume $(\theta_B,1-\theta_B)$. In Section \ref{os9}, 
we had a single phase and so $\theta_B=1$ in \eqref{hsa}.
The effect of this difference is seen in the expression \eqref{FG4} and \eqref{hsa}, 
where the $p$-sequential laminate structure $A^{*}_p$ remains the same in both cases. 
\hfill\qed\end{remark}
\begin{remark}
In Section \ref{qw4}, we prove that the above lower bound \eqref{eig} is optimal in the sub-domain $\theta_A \leq \theta_B$. In general, it need not be optimal if $\theta_B < \theta_A$. 

\hfill\qed
\end{remark}
\noindent\textbf{Trace Bound L2 : $\theta_B(x) < \theta_A(x)$ almost everywhere $x$ : }
In this case we will be finding lower bounds on $(A^{*},B^{\#})$ over arbitrary $\omega_{B^\epsilon}$
and those optimal microstructures $\omega_{A^\epsilon}$ for $A^{*}$ which provides equality in  the optimal upper bound \eqref{FL11}.  This lower bound is also optimal under the same condition $\theta_B(x) < \theta_A(x)$ almost everywhere $x\in\Omega$.
\paragraph{Step 1d : ($H$-measure term)}
Instead of the translated inequality \eqref{FL8}, 
we consider the following one for the heat flux, which coincides with the known dual inequality for $A^\epsilon$ in the self-interacting case. As in the case of L1, our plan is to pass to the limit in this inequality using $H$-limit, relative limit and $H$-measure. The resulting inequality  \eqref{os12} will involve a new $H$-measure term. Let us now provide details. Introducing the constant vector $\eta$ in $\mathbb{R}^N$ and suitable constant $c$ to be chosen later and sequence $\sigma^\epsilon$ such that
\begin{equation}\label{tc}
((A^{\epsilon})^{-1}B^{\epsilon}(A^{\epsilon})^{-1} - c I)(\sigma^{\epsilon} + \eta)\cdot(\sigma^{\epsilon} + \eta)\ \geq\ 0 \quad\mbox{a.e. in }\Omega.\end{equation} 
where,
\begin{equation*}\sigma^{\epsilon} = A^{\epsilon}\nabla u^{\epsilon} \rightharpoonup \sigma\  (\ = A^{*}\nabla u )\mbox{ in }L^2(\Omega)\mbox{ and }-div\ (\sigma^{\epsilon}) \in H^{-1}(\Omega)\mbox{ convergent. }\end{equation*}
Now by expanding the above inequality we get, 
\begin{equation}\label{os11}
(A^{\epsilon})^{-1}B^{\epsilon}(A^{\epsilon})^{-1}\sigma^{\epsilon}\cdot\sigma^{\epsilon} +  ((A^{\epsilon})^{-1}B^{\epsilon}(A^{\epsilon})^{-1} -c I)\eta\cdot \eta -2c\sigma^{\epsilon}\cdot \eta\ \geq\ c\ \sigma^{\epsilon}\cdot\sigma^{\epsilon} - 2(A^{\epsilon})^{-1}B^{\epsilon}(A^{\epsilon})^{-1}\sigma^{\epsilon}\cdot\eta.  
\end{equation}
The choice of the constant $c$ will depend upon  $a_1,a_2,b_1,b_2$ and we will fix it later.\\ 
\\
It is easy to pass to the limit in the first term of the left hand side of \eqref{os11}, as
\begin{equation*} (A^{\epsilon})^{-1}B^{\epsilon}(A^{\epsilon})^{-1}\sigma^{\epsilon}\cdot\sigma^{\epsilon}= B^\epsilon\nabla u^\epsilon\cdot\nabla u^\epsilon \rightharpoonup  B^{\#}\nabla u\cdot \nabla u =\ {A^{*}}^{-1}B^{\#}{A^{*}}^{-1}\sigma\cdot \sigma\ \mbox{ in }\mathcal{D}^\prime(\Omega).\end{equation*}
Regarding the second term, we write 
\begin{align}\label{zz5}
(A^{\epsilon})^{-1}B^{\epsilon}(A^{\epsilon})^{-1} =& \{\frac{b_1}{a_1^2}\AB + \frac{b_1}{a_2^2}(\B - \AB) + \frac{b_2}{a_1^2}(\A - \AB)\notag\\
&\qquad+ \frac{b_2}{a_2^2}(1-\A - \B + \AB)\}I \notag\\
&\rightharpoonup \widetilde{L} \mbox{ (say),\ in }L^\infty(\Omega)\ \mbox{weak*.} 
\end{align}
Using the fact, that $\theta_{AB}\leq min\{\theta_A, \theta_B\}=\theta_B$ (in this case), where $\theta_{AB}$ is the 
$L^\infty(\Omega)$ weak* limit of the sequence $\AB$, we find
\begin{equation}\label{bs4}
 \widetilde{L}\ \geq\ \{\frac{b_2}{a_2^2} + (\frac{b_2}{a_1^2}-\frac{b_2}{a_2^2})\theta_A(x) - \frac{(b_2-b_1)}{a_1^2}\theta_B(x)\}I =: L =l(x)\hspace{1.4pt} I \mbox{ (say) }.
\end{equation}
It shows that, the $L^\infty(\Omega)$ weak* limit of $\{(A^{\epsilon})^{-1}B^{\epsilon}(A^{\epsilon})^{-1}\}\geq L$ along any convergent subsequence 
and it is equal to L for the choice $\omega_{B^\epsilon}\subset \omega_{A^\epsilon}$ (in which case $\theta_{AB}=\theta_B$). \\
\\
The optimal choice of the translated amount $c$ is 
\begin{equation}\label{os15} c = min\ \{ \frac{b_1}{a_1^2}, \frac{b_2}{a_2^2} \},\end{equation}
because then we will have $(\widetilde{L}-cI)\geq 0$ as shown by the following inequalities :
\begin{align*}
l=\ \frac{b_2}{a_2^2} + (\frac{b_2}{a_1^2}-\frac{b_2}{a_2^2})\theta_A - \frac{(b_2-b_1)}{a_1^2}\theta_B 
&\geq\ \frac{b_2}{a_2^2} + (\frac{b_1}{a_1^2}-\frac{b_2}{a_2^2})\theta_B\ \geq\ \frac{b_2}{a_2^2}, \mbox{ whenever } \frac{b_1}{a_1^2}\geq\frac{b_2}{a_2^2}\\
&\geq\ \frac{b_1}{a_1^2} + (\frac{b_2}{a_2^2}-\frac{b_1}{a_1^2})(1-\theta_B) \geq \frac{b_1}{a_1^2},  \mbox{ whenever } \frac{b_2}{a_2^2}\geq\frac{b_1}{a_1^2}.
\end{align*}
Thus, it is straight forward to pass to the limit in left hand side of \eqref{os11} to get   
\begin{equation*}{A^{*}}^{-1}B^{\#}{A^{*}}^{-1}\sigma\cdot \sigma +  (\widetilde{L} - c I)\eta\cdot \eta -2c\sigma\cdot \eta\geq\ \mbox{ limit of R.H.S. of }\eqref{os11} \end{equation*}
where $\widetilde{L}$ and $c$ are defined above. \\
\\
We use the notion of $H$-measure in order to pass to the limit in the right hand side of \eqref{os11}. 
Introducing a coupled variable $W^{\prime\prime}_{\epsilon}=(\sigma^{\epsilon}, (A^{\epsilon})^{-1}B^{\epsilon}(A^{\epsilon})^{-1}\eta)$, we write the right 
hand side of \eqref{os11} in a quadratic form $q_2(W^{\prime\prime}_{\epsilon})\cdot W^{\prime\prime}_{\epsilon} $. 
Here $q_2$ is a linear map 
\begin{equation*} q_2:\mathbb{R}^N \times \mathbb{R}^N \mapsto \mathbb{R}^N \times \mathbb{R}^N \mbox{ defined by } q_2(\sigma^\epsilon, (A^\epsilon)^{-1}B^\epsilon(A^\epsilon)^{-1}\eta)= (c\sigma^\epsilon-(A^\epsilon)^{-1}B^\epsilon(A^\epsilon)^{-1}\eta, -\sigma^\epsilon). \end{equation*} 
Introducing $W^{\prime\prime}_0 =(\sigma, \widetilde{L}\eta )$, we have 
\begin{equation*}q_2(W^{\prime\prime}_{\epsilon})\cdot W^{\prime\prime}_{\epsilon}=\ 2q_2(W^{\prime\prime}_{\epsilon})\cdot W^{\prime\prime}_0 -q_2(W^{\prime\prime}_0)\cdot W^{\prime\prime}_0 + q_2(W^{\prime\prime}_{\epsilon}- W^{\prime\prime}_0)\cdot(W^{\prime\prime}_{\epsilon}- W^{\prime\prime}_0).\end{equation*}
Denoting by $\varPi_{W^{\prime\prime}}$  $H$-measure of $(W^{\prime\prime}_{\epsilon}- W^{\prime\prime}_0)$, we  pass to the limit in \eqref{os11} by virtue of Theorem \ref{dc10} :
\begin{equation}\label{os12}
{A^{*}}^{-1}B^{\#}{A^{*}}^{-1}\sigma\cdot\sigma + (\widetilde{L} - c I)\eta\cdot \eta - 2c\ \sigma\cdot \eta\ \geq\ c\ \sigma\cdot\sigma - 2\widetilde{L}\sigma\cdot \eta + Y^\prime  
\end{equation}
where $Y^\prime$ is the $H$-measure correction term defined by 
\begin{equation*}Y^\prime =\ \underset{\epsilon \rightarrow 0}{\textrm{lim}}\ q_2(W^{\prime\prime}_{\epsilon}- W^{\prime\prime}_0)\cdot(W^{\prime\prime}_{\epsilon}- W^{\prime\prime}_0) = \int_{\mathbb{S}^{N-1}} trace\ (q_2\varPi_{W^{\prime\prime}} (x,d\xi))\end{equation*} 
or equivalently, denoting the average w.r.t. directions $\xi\in\mathbb{S}^{N-1}$  by double angular bracket, we write with $W^{\prime\prime} = (\varSigma, A^{-1}BA^{-1}\eta) \in \mathbb{R}^N \times \mathbb{R}^N$
\begin{equation}\label{OP10}
Y^\prime =\ \langle\langle \varPi_{W^{\prime\prime}}, Q_2(\varSigma,A^{-1}BA^{-1}\eta) \rangle\rangle\ \ \mbox{with }\ Q_2(\varSigma,A^{-1}BA^{-1}\eta)=\ c|\varSigma|^2 - 2A^{-1}BA^{-1}\varSigma\cdot \eta. 
\end{equation}
Now our objective is to find the lower bound on $Y^\prime$ in order to get the lower bound on  $B^{\#}$ or equivalently on ${A^{*}}^{-1}B^{\#}{A^{*}}^{-1}$
whenever $A^{*}\in\partial\mathcal{G}_{\theta_A}^U$. 
\paragraph{Step 1e :}
\paragraph{Special case :} $A^{\epsilon}= tB^{\epsilon}$ for some scalar $t>0$.\\
\\
In this case $c= (ta_2)^{-1}$ and so \eqref{os11} simply becomes 
\begin{equation}\label{FG8}
 (A^{\epsilon})^{-1}\sigma^{\epsilon}\cdot\sigma^{\epsilon} + ( (A^{\epsilon})^{-1} - a_2^{-1}I )\eta\cdot \eta - 2a_2^{-1}\sigma^{\epsilon}\cdot \eta \geq\ a_2^{-1}\hspace{1.5pt}\sigma^{\epsilon} \cdot \sigma^{\epsilon} - 2(A^{\epsilon})^{-1}\sigma^{\epsilon}\cdot \eta. 
\end{equation}
Since $A^{*}=tB^{\#}$, the task of proving L2 simply turns out to finding optimal lower bound on ${A^{*}}^{-1}$, which is known from classical results :
\begin{equation}\label{os20}
tr\ ({A^{*}}^{-1}-a_2^{-1} I)^{-1} \leq\  tr\ (\underline{A} ^{-1} - a_2^{-1}I)^{-1} + \frac{(1-\theta_A)a_2}{\theta_A}(N-1).
\end{equation}
Before considering the general case, it is appropriate to present some elements of the proof of \eqref{os20}.\\
\\
We begin with \eqref{FG8}. We pass to the limit in the left hand side of \eqref{FG8} by using homogenization theory and on the right hand side by using $H$-measures.
Let us introduce the linear map $q_3$ whose associated quadratic form is as follows :
\begin{equation*} q_3(V^\prime)\cdot V^\prime =\ a_2^{-1}\varSigma\cdot \varSigma - 2 A^{-1}\varSigma\cdot\eta,\  \mbox{ with }V^\prime =(\varSigma, A^{-1}\eta).\end{equation*}
The limit of \eqref{FG8} can be written as
\begin{equation}\label{FG11}
{A^{*}}^{-1}\sigma\cdot\sigma + ( \underline{A}^{-1} - a_2^{-1} )\eta\cdot \eta - 2a_2^{-1}\sigma\cdot \eta\ \geq\ a_2^{-1}\sigma\cdot\sigma - 2\underline{A}^{-1}\sigma\cdot \eta + X^\prime  
\end{equation}
where $\underline{A}^{-1}$ is the $L^{\infty}$ weak * limit of $(A^{\epsilon})^{-1}$ and $X^\prime$ is a new $H$-measure correction term defined by
\begin{equation*} X^\prime =\ \int_{\mathbb{S}^{N-1}} trace\ (q_3\varPi_{V^\prime}(x,d\xi)) =\ \langle\langle \varPi_{V^\prime} , Q_3(V^{\prime}) \rangle\rangle,\ \ \mbox{where } Q_3(V^{\prime})= q_3(V^{\prime})\cdot V^{\prime} \end{equation*} 
and $\varPi_{V^\prime}$ is the $H$-measure of $V^\prime_\epsilon -V^\prime = (\sigma^{\epsilon} - \sigma ,\ ((A^{\epsilon})^{-1} - \underline{A} ^{-1})\eta ).$\\
\\
Since $div\ \sigma^{\epsilon}$ is convergent in $H^{-1}(\Omega)$, we consider the oscillation variety which is as follows :
\begin{equation*}\vartheta^\prime =\ \{(\xi, \varSigma, A^{-1}\eta) \in \mathbb{S}^{N-1} \times \mathbb{R}^{N}\times\mathbb{R}^{N};\ \sum_{i=1}^N \xi_i \varSigma_i = 0\}  \end{equation*}
with its projection (wave cone) 
\begin{equation*}\Lambda^\prime = \{( \varSigma, A^{-1}\eta) \in \mathbb{R}^{N}\times\mathbb{R}^{N};\ \exists\ \xi\in\mathbb{S}^{N-1}\mbox{ such that }  (\xi, \varSigma,A) \in \vartheta^\prime\}.\end{equation*}
We define $\Lambda^\prime_{\xi}\subset \Lambda^\prime$, $\xi\in\mathbb{R}^N\smallsetminus\{0\}$ as
\begin{equation*}\Lambda^\prime_\xi = \{\ ( \varSigma,A^{-1}\eta) \in \mathbb{R}^{N}\times\mathbb{R}^{N};\  \sum_{i=1}^N \xi_i\varSigma_i=0 \};\ \mbox{So, }\ \underset{\xi\neq 0}{\cup}{\Lambda^\prime_\xi}=\Lambda^\prime.\end{equation*}
Based on this, one defines a new linear map $q_{3,\xi}^{\prime}$, whose associated quadratic form is 
\begin{equation*}q_{3,\xi}^\prime(\varSigma,A^{-1}\eta)\cdot(\varSigma, A^{-1}\eta ) =\ q_3(\varSigma,A^{-1}\eta)\cdot(\varSigma,A^{-1}\eta) - \underset{\varSigma \in \Lambda^\prime_{\xi} }{min}\ q_3(\varSigma,A^{-1}\eta)\cdot(\varSigma,A^{-1}\eta).\end{equation*}
As $q_{3,\xi}^\prime$ is non-negative on $\Lambda_\xi^\prime$, so by applying Theorem \ref{dc10}, we  get 
\begin{equation*}trace\ (q_{3,\xi}^\prime\varPi_{V^\prime}) \geq 0 .\end{equation*}
Finally introducing $\varPi^\prime_A, $  $H$-measure of $((A^{\epsilon})^{-1} - (\underline{A})^{-1})\eta$, we get 
\begin{equation}\label{FG9}
 X^\prime\ \geq\ \langle\langle \varPi_{V^\prime},\underset{\varSigma \in \Lambda^\prime_\xi}{min}\ q_3(\varSigma,A^{-1}\eta)\cdot(\varSigma,A^{-1}\eta)\rangle\rangle =\ \langle\langle \varPi^\prime_A, \underset{\varSigma \in \Lambda^\prime_\xi}{min}\ q_3(\varSigma,A^{-1}\eta)\cdot(\varSigma,A^{-1}\eta)\rangle\rangle.
\end{equation}
Now by introducing Lagrange multiplier corresponding to the constraint in $\Lambda^\prime$, it is straight forward to compute 
\begin{equation}\label{FG10}
\underset{\varSigma \in \Lambda^\prime_\xi}{min}\ q_3(\varSigma,A^{-1}\eta)\cdot(\varSigma,A^{-1}\eta) =\  -a_2(I-\frac{\xi\otimes \xi}{|\xi|^2}) A^{-1}\eta\cdot A^{-1}\eta
\end{equation}
with the minimizer $\varSigma_{min}$ in $\Lambda^\prime$ :
\begin{equation}\label{eih}\varSigma_{min} =\ a_2(I-\frac{\xi\otimes \xi}{|\xi|^2})A^{-1}\eta.\end{equation}
On the other hand, since $((A^{\epsilon})^{-1} - \underline{A}^{-1})\eta = (a_1 ^{-1} - a_2 ^{-1} )(\A - \theta_A)\eta $ , the $H$-measure $\varPi^\prime_A$ becomes 
\begin{equation}\label{os2}(\varPi^\prime_A)_{ij} = \frac{(a_2 - a_1)^{2}}{(a_1 a_2)^{2}}(\nu_A)\eta_i\eta_j  \ \ \forall i,j =1,..,N\end{equation}
where $\nu_A$ is  $H$-measure of $(\A-\theta_A)$ satisfying \eqref{hsz}. 
Then with the help of the matrix $M_A$ defined in \eqref{os19}, we obtain from \eqref{FG9} together with \eqref{FG10} and \eqref{os2} 
\begin{equation}\label{os5} X^\prime \geq\  -\langle\langle\varPi^\prime_A, a_2(I-\frac{\xi\otimes \xi}{|\xi|^2})A^{-1}\eta\cdot A^{-1}\eta \rangle\rangle = -\theta_A(1-\theta_A)a_2\frac{(a_2 - a_1 )^{2}}{(a_1a_2)^{2}}(I - M_A)\eta\cdot \eta :=X^\prime_{min}\end{equation}
Thus, 
\begin{equation}\begin{aligned}\label{ub2}
a_2^{-1}\sigma^\epsilon\cdot\sigma^\epsilon - 2(A^\epsilon)^{-1}\sigma^\epsilon\cdot \eta \rightharpoonup& \ a_2^{-1}\sigma\cdot\sigma - 2\underline{A}^{-1}\sigma\cdot \eta +X^\prime\ \mbox{ weakly in }\mathcal{D}^\prime(\Omega)\\
 &\geq\ a_2^{-1}\sigma\cdot\sigma - 2\underline{A}^{-1}\sigma\cdot \eta  -\theta_A(1-\theta_A)a_2\frac{(a_2 - a_1 )^{2}}{(a_1a_2)^{2}}(I - M_A)\eta\cdot \eta.
\end{aligned}\end{equation}
Using the above lower bound \eqref{ub2} in \eqref{FG11} one obtains  
\begin{equation*} {A^{*}}^{-1}\sigma\cdot\sigma + ( \underline{A}^{-1} - a_2^{-1} )\eta\cdot \eta - 2a_2^{-1}\sigma\cdot \eta\ \geq\ a_2^{-1}\sigma\cdot\sigma - 2\underline{A}^{-1}\sigma\cdot \eta  -\theta_A(1-\theta_A)a_2\frac{(a_2 - a_1 )^{2}}{(a_1a_2)^{2}}(I - M_A)\eta\cdot \eta.\end{equation*}
And then minimizing with respect to $\sigma$ with the minimizer 
$\sigma = -(\underline{A} ^{-1} - a_2^{-1}I)({A^{*}} ^{-1}- a_2^{-1}I)^{-1}\eta$ we get 
\begin{equation}\label{sig}
({A^{*}} ^{-1}- a_2^{-1}I)^{-1}\eta\cdot \eta \leq\  (\underline{A} ^{-1} - a_2^{-1}I)^{-1}\eta\cdot \eta +\frac{(1-\theta_A)}{\theta_A}a_2(I- M_A)\eta\cdot \eta.
\end{equation}
Finally, by taking the trace, \eqref{sig} yields the optimal lower bound \eqref{os20} on ${A^{*}}^{-1}$.
\hfill\qed
\paragraph{General Case :}
Having seen the special case, let us now treat the general case of
$(A^\epsilon,B^\epsilon)$. We go back to the equation \eqref{OP10}, in which we need to estimate the H-measure term $Y^{\prime}$ from below
optimally. As for L1, we will be needing compactness property for this
purpose. Following result which can be proved in a fashion similar
to Theorem \ref{ub8} states this property precisely :
\begin{theorem}[Compactness]\label{ad20}
Let us consider the following constrained oscillatory system :
\begin{align}
&V^\prime_\epsilon=\ (\sigma^\epsilon,(A^\epsilon)^{-1}\eta) \rightharpoonup (\sigma,\underline{A}^{-1}\eta)=V^\prime_0\ \mbox{ in }L^2(\Omega)^{2N}\mbox{ weak }, \eta\in\mathbb{R}^N\smallsetminus\{0\}, \label{ub9}\\
&\Delta( \sigma^\epsilon - a_2(A^\epsilon)^{-1}\eta) - a_2\nabla(div((A^\epsilon)^{-1}\eta) ) \in H^{-2}_{loc}(\Omega)\mbox{ convergent. }\label{ub10}
\end{align}
Then 
\begin{equation}\label{bs18}
a_2^{-1}\sigma^{\epsilon}\cdot\sigma^{\epsilon}- 2(A^{\epsilon})^{-1}\sigma^{\epsilon}\cdot\eta \rightharpoonup a_2^{-1} \sigma\cdot\sigma- 2\underline{A}^{-1}\sigma\cdot \eta + X^{\prime}_{min} \mbox{ in }\ \mathcal{D}^{\prime}(\Omega),
\end{equation}
where $ X^\prime_{min}$ is defined in \eqref{os5}. \\
\\
Conversely, if the sequence with \eqref{ub9} satisfies \eqref{bs18} then \eqref{ub10} must hold.
\hfill\qed\end{theorem}
\paragraph{Step 1f : (Lower bound)}
Taking this into account, in the translated inequality \eqref{os11} we restrict $\sigma^\epsilon$ 
to satisfy the following oscillatory system to work with  
\begin{equation}\label{os3}
\begin{cases}
W_\epsilon^{\prime\prime\prime} = (\sigma^\epsilon, (A^\epsilon)^{-1}B^\epsilon (A^\epsilon)^{-1}\eta, (A^\epsilon)^{-1}\eta ) \rightharpoonup \ W_0^{\prime\prime\prime} = ( \sigma , \widetilde{L}\eta, \underline{A}^{-1}\eta )\ \mbox{ weakly in }L^2(\Omega)^{3N},\\ 
\Delta( \sigma^\epsilon - a_2(A^\epsilon)^{-1}\eta) - a_2\nabla(div((A^\epsilon)^{-1}\eta) ) \in H^{-2}_{loc}(\Omega)\mbox{ convergent. }
\end{cases}
\end{equation}
Proceeding in a same as we did in Step 1c, here we compute $Y^\prime$ (cf.\eqref{OP10}) with $\Sigma=\Sigma_{min}$ given in \eqref{eih}, as follows :  
\begin{align}
Y^\prime &= \langle\langle \varPi_{W^{\prime\prime\prime}}, Q_2( a_2(I-\frac{\xi\otimes \xi}{|\xi|^2})A^{-1}\eta,A^{-1}BA^{-1}\eta\cdot\eta)\rangle\rangle\notag\\
         &= \langle\langle \varPi_{W^{\prime\prime\prime}},  c a_2^2(I-\frac{\xi\otimes \xi}{|\xi|^2})A^{-1}\eta\cdot A^{-1}\eta -2a_2A^{-1}BA^{-1}(I-\frac{\xi\otimes \xi}{|\xi|^2})A^{-1}\eta\cdot\eta \rangle\rangle\notag\\
         &= \langle\langle \varPi_{W^{\prime\prime\prime}}, c (I-\frac{\xi\otimes \xi}{|\xi|^2})(\frac{1}{c}A^{-1}BA^{-1}\eta-a_2A^{-1}\eta)\cdot(\frac{1}{c}A^{-1}BA^{-1}\eta-a_2A^{-1}\eta)\rangle\rangle \notag\\
         &\qquad-\frac{1}{c}\langle\langle \varPi_{W^{\prime\prime\prime}},(I-\frac{\xi\otimes \xi}{|\xi|^2})A^{-1}BA^{-1}\eta\cdot A^{-1}BA^{-1}\eta  \rangle\rangle \notag \\
         &= \langle\langle \varPi^\prime_{AB}, c (I-\frac{\xi\otimes \xi}{|\xi|^2})(\frac{1}{c}A^{-1}BA^{-1}-a_2A^{-1})\eta\cdot(\frac{1}{c}A^{-1}BA^{-1}-a_2A^{-1})\eta\rangle\rangle \notag\\
         &\qquad  -\frac{1}{c}\langle\langle \varPi^{\prime\prime}_{AB},(I-\frac{\xi\otimes \xi}{|\xi|^2})A^{-1}BA^{-1}\eta\cdot A^{-1}BA^{-1}\eta  \rangle\rangle, \notag \\
         &=: R^\prime_1 + R^\prime_2 \mbox{ (say)}\label{os6}
\end{align}
(where in the second line we have used $|(I-\frac{\xi\otimes \xi}{|\xi|^2})v|^2=\langle(I-\frac{\xi\otimes \xi}{|\xi|^2})v,v\rangle, v\in\mathbb{R}^N$).\\ 

\noindent
$\varPi^\prime_{AB}\in \mathcal{M}(\Omega\times \mathbb{S}^{N-1};\mathbb{R}^{N\times N}) $ is $H$-measure of the sequence
$\{\frac{1}{c}((A^\epsilon)^{-1}B^\epsilon(A^\epsilon)^{-1}-\widetilde{L})-a_2((A^\epsilon)^{-1}-(\underline{A})^{-1})\}\eta$,
and $\varPi^{\prime\prime}_{AB}\in \mathcal{M}(\Omega\times \mathbb{S}^{N-1};\mathbb{R}^{N\times N})$ is $H$-measure of the sequence
$\{(A^\epsilon)^{-1}B^\epsilon(A^\epsilon)^{-1}-\widetilde{L}\}\eta$; where $\widetilde{L}$ is the $L^\infty(\Omega)$ weak* limit of
$(A^\epsilon)^{-1}B^\epsilon (A^\epsilon)^{-1}$ (cf.\eqref{zz5}).\\
\\
Now both of these above sequences are scalar, as
\begin{align*}
((A^\epsilon)^{-1}B^\epsilon (A^\epsilon)^{-1} -\widetilde{L}) &=: f^\epsilon I \mbox{ (say);} \ \quad (f^\epsilon \rightharpoonup 0 \mbox{ in }L^\infty(\Omega)\mbox{ weak*})\\
\mbox{ and }\ ((A^\epsilon)^{-1}-(\underline{A})^{-1}) &= (\frac{1}{a_1}-\frac{1}{a_2})(\A-\theta_A)I\ \ ( \rightharpoonup 0 \mbox{ in }L^\infty(\Omega)\mbox{ weak*}).
\end{align*}
Then $H$-measure $\varPi^\prime_{AB}$ reduces to 
\begin{equation*}
(\varPi^\prime_{AB})_{ij} = (\nu^\prime_{AB})\eta_i\eta_j \ \ \forall i,j =1,..,N
\end{equation*}
where $\nu^\prime_{AB}$ is  $H$-measure of the scalar sequence $\{\frac{1}{c}f^\epsilon- a_2(\frac{1}{a_1}-\frac{1}{a_2})(\A-\theta_A)\}$ with
\begin{align}\label{zz6}
&\nu^\prime_{AB}(x,\xi)\geq 0\  \mbox{ and }\notag\\
&\int_{\mathbb{S}^{N-1}}\nu^\prime_{AB}(x,d\xi) = L^\infty(\Omega)\mbox{ weak* limit of }\{\frac{1}{c}f^\epsilon- a_2(\frac{1}{a_1}-\frac{1}{a_2})(\A-\theta_A)\}^2 := L^\prime_{AB} \mbox{ (say)}.
\end{align}
Thus
\begin{align}\label{dc15}
R^\prime_1 &= \langle\langle \varPi^\prime_{AB}, c (I-\frac{\xi\otimes \xi}{|\xi|^2})(a_2A^{-1}-\frac{1}{c}A^{-1}BA^{-1})\eta\cdot(a_2A^{-1}-\frac{1}{c}A^{-1}BA^{-1})\eta\rangle\rangle\notag\\ 
           &= cL^\prime_{AB}(I-M^\prime_{AB})\eta\cdot\eta
\end{align}
where $M^\prime_{AB}$ is a non-negative matrix with unit trace defined by 
\begin{equation}\label{zz9}
 M^\prime_{AB} = \frac{1}{L^\prime_{AB}}\int_{\mathbb{S}^{N-1}} \xi \otimes \xi\ \nu^\prime_{AB} (d\xi). 
\end{equation}
We compute $R_2^\prime$ next. We write
\begin{equation*}
 (\varPi^{\prime\prime}_B)_{ij} = (\nu^{\prime\prime}_B)\eta_i\eta_j  \ \ \forall i,j =1,..,N
\end{equation*}
where $\nu^{\prime\prime}_B$ is the $H$-measure of the scalar sequence $f^\epsilon$ satisfying 
\begin{equation}\label{zz7}
\nu^{\prime\prime}_{AB}(x,\xi)\geq 0 \mbox{ and }\int_{\mathbb{S}^{N-1}}\nu_{AB}(x,d\xi) = L^\infty(\Omega)\mbox{ weak* limit of }(f^\epsilon)^2 =:L^{\prime\prime}_{AB} \mbox{ (say)}.
\end{equation}
Thus 
\begin{align}\label{dc20}
 R_2^\prime &= -\frac{1}{c}\langle\langle \varPi^{\prime\prime}_{AB},(I-\frac{\xi\otimes \xi}{|\xi|^2})A^{-1}BA^{-1}\eta\cdot A^{-1}BA^{-1}\eta  \rangle\rangle\notag\\
            &= -\frac{1}{c}L^{\prime\prime}_{AB}(I-M^{\prime\prime}_{AB})\eta\cdot\eta 
\end{align} 
where $M^{\prime\prime}_{AB}$ is the non-negative matrix with unit trace defined as 
\begin{equation}\label{zz8}
 M^{\prime\prime}_{AB} = \frac{1}{I^{\prime\prime}_{AB}}\int_{\mathbb{S}^{N-1}} \xi\otimes\xi\ \nu^{\prime\prime}_B(x,d\xi).
\end{equation}
Therefore from \eqref{os6} and \eqref{dc15},\eqref{dc20} we have 
\begin{equation}\label{FG12}
Y^\prime =  cL^\prime_{AB}(I-M^\prime_{AB})\eta\cdot\eta-  \frac{1}{c}L^{\prime\prime}_{AB}(I-M^{\prime\prime}_{AB})\eta\cdot\eta 
\end{equation}
with 
\begin{equation}\label{zz2}
trace\ Y^\prime = \{cL^\prime_{AB} -\frac{1}{c}L^{\prime\prime}_{AB}\}(N-1).
\end{equation}
Recall that,
\begin{align}\label{zz1}
\{cL^\prime_{AB} -\frac{1}{c}L^{\prime\prime}_{AB}\}=&\ c\{\mbox{$L^\infty(\Omega)$ weak* limit of }\{\frac{1}{c}f^\epsilon- a_2(\frac{1}{a_1}-\frac{1}{a_2})(\A-\theta_A)\}^2\}\notag\\
                                                     &\qquad-\frac{1}{c}\{\mbox{$L^\infty(\Omega)$ weak* limit of }(f^\epsilon)^2\}.  
\end{align}
We want to find the lower bound the above quantity \eqref{zz1}. As usual the above quantity involves the parameter $\theta_{AB}(x)$, 
the $L^{\infty}(\Omega)$ weak* limit of $(\AB)$. Now keeping the estimate \eqref{bs4} in mind, we minimize $L^\prime_{AB}$ and according to that the other quantity $L^{\prime\prime}_{AB}$ also gets determined. For that, it is enough to make the choice $\omega_{B^\epsilon}\subset \omega_{A^\epsilon}$ which is possible in this present case ($\theta_B<\theta_A$), and we bound $trace\ Y^\prime$ from below as following :
\begin{align}\label{zz3}
trace\ Y^\prime\ \geq\ & L^\infty(\Omega) \mbox{ weak* limit of } \{ca_2^2(\frac{1}{a_1}-\frac{1}{a_2})^2(\A-\theta_A)^2-2a_2(\frac{1}{a_1}-\frac{1}{a_2})(\frac{b_1}{a_1^2}-\frac{b_2}{a_1^2})\notag\\
&  (\B-\theta_B)(\A-\theta_A)-2a_2(\frac{1}{a_1}-\frac{1}{a_2})(\frac{b_2}{a_1^2}-\frac{b_2}{a_2^2})(\A - \theta_A)^2\}(N-1)\notag\\
&=\{ca_2^2(\frac{1}{a_1}-\frac{1}{a_2})^2\theta_A(1-\theta_A)-2a_2(\frac{1}{a_1}-\frac{1}{a_2})\frac{(b_1-b_2)}{a_1^2}\theta_B(1-\theta_A)\notag\\
&\qquad\quad\quad\quad\qquad\qquad\qquad\qquad\qquad-2a_2(\frac{1}{a_1}-\frac{1}{a_2})(\frac{b_2}{a_1^2}-\frac{b_2}{a_2^2})\theta_A(1-\theta_A)\}(N-1).
\end{align}
Rewriting \eqref{os12}, we get
\begin{equation}\label{sip}
({A^{*}}^{-1}B^{\#}{A^{*}}^{-1}-cI)\sigma\cdot\sigma + (\widetilde{L} - c I)\eta\cdot \eta - 2(\widetilde{L}-cI)\ \sigma\cdot \eta\ \geq\ Y^\prime  
\end{equation}
\textbf{Choice of $\sigma$ :}
Now in order to obtain the final expression of the lower bound by eliminating $\sigma$ from \eqref{sip},
we choose $\sigma(x)$ (locally which can take any value in $\mathbb{R}^N$) as the minimizer of \eqref{sig} i.e.
\begin{equation}\label{ub5} \sigma = - (\underline{A}^{-1} - a_2^{-1}I)({A^{*}}^{-1}- a_2^{-1}I)^{-1}\eta\end{equation}
\textbf{Matrix lower bound :}
Then \eqref{sip} yields as
\begin{align}
&({A^{*}}^{-1}B^{\#}{A^{*}}^{-1}- c I)(\underline{A} ^{-1} - a_2^{-1}I)^2({A^{*}} ^{-1}- a_2^{-1}I)^{-2}\eta\cdot \eta + (\widetilde{L}-c I)\eta\cdot \eta\notag\\
&\quad\quad\qquad\qquad\qquad\qquad- 2(\widetilde{L}-c I)(\underline{A} ^{-1} - a_2^{-1}I)({A^{*}} ^{-1}- a_2^{-1}I)^{-1}\eta\cdot \eta\ \geq\ Y^\prime.\label{os14}
\end{align}
Now by using  $\widetilde{L}\geq L$ (cf.\eqref{bs4}) together with $A^{*}\eta\cdot\eta\geq \underline{A}\eta\cdot\eta$ we can replace $\widetilde{L}$ by L in \eqref{os14} to obtain :
\begin{align}
&({A^{*}}^{-1}B^{\#}{A^{*}}^{-1}- c I)(\underline{A} ^{-1} - a_2^{-1}I)^2({A^{*}} ^{-1}- a_2^{-1}I)^{-2}\eta\cdot \eta + (L-c I)\eta\cdot \eta\notag\\
&\quad\quad\qquad\qquad\qquad\qquad- 2(L-c I)(\underline{A} ^{-1} - a_2^{-1}I)({A^{*}} ^{-1}- a_2^{-1}I)^{-1}\eta\cdot \eta\ \geq\ Y^\prime \label{bs9} 
\end{align}
where, $A^{*}$ satisfies 
\begin{equation}\label{eij}
({A^{*}}^{-1}-a_2^{-1} I)^{-1}\eta\cdot\eta =\  (\underline{A} ^{-1} - a_2^{-1}I)^{-1}\eta\cdot\eta + \frac{(1-\theta_A)a_2}{\theta_A}(I-M_A)\eta\cdot\eta.
\end{equation}
We simplify the bound \eqref{bs9} with using \eqref{eij} and by taking trace using \eqref{zz3} one obtains :
\paragraph{Trace bound L2(a)\ when $c=\frac{b_2}{a_2^2}$ :}
\begin{align}\label{ub1}
tr\ \{({A^{*}}^{-1}B^{\#}{A^{*}}^{-1}- \frac{b_2}{a_2^2} I)(\underline{A} ^{-1} - a_2^{-1}I)^2({A^{*}} ^{-1}- a_2^{-1}I)^{-2}\} \geq\ &\frac{b_2}{a_2^2}\frac{(a_2-a_1)^2}{a_1^2}\theta_A(1-\theta_A)(N-1) \notag\\
                                                                                                                              &\qquad\quad + N(l-\frac{b_2}{a_2^2}). 
\end{align}
\textbf{Trace bound L2(b)\ when $c=\frac{b_1}{a_1^2}$ :}
\begin{align}\label{ub3}
&tr\ \{({A^{*}}^{-1}B^{\#}{A^{*}}^{-1}- \frac{b_1}{a_1^2} I)(\underline{A} ^{-1} - a_2^{-1}I)^2({A^{*}} ^{-1}- a_2^{-1}I)^{-2}\}\notag\\
&\geq\ \frac{b_1}{a_1^2}\frac{(a_2-a_1)^2}{a_1^2}\theta_A(1-\theta_A)(N-1)+2(\frac{b_2}{a_2^2}-\frac{b_1}{a_1^2})\frac{(a_2-a_1)}{a_1}(1-\theta_A)(N-1) + N(l-\frac{b_1}{a_1^2}). 
\end{align}
Hence, the trace bounds \eqref{to},\eqref{tn} follow whenever $A^{*}\in\partial\mathcal{G}_{\theta_A}^U$.
\hfill\qed
\\

Next, we give examples of microstructures which possess the property  $\omega_{B^\epsilon}\subset \omega_{A^\epsilon}$ and achieves the
equality in \eqref{bs9} as mentioned below.
\begin{example}[Saturation/Optimality]\label{Sd20}
The equality of this lower bound is achieved by the laminated materials $L^{\#}_2$ (cf. \eqref{lb10}),
at composites based on Hashin-Shtrikman construction given in \eqref{FG6} and at sequential laminates of $N$-rank to be constructed below.
\bpr
We consider the simple laminate say in $e_1$ direction with $\theta_B \leq \theta_A$, then 
\begin{equation*}B^{\#} = diag\ (L^{\#}_2, \overline{b},..,\overline{b})\ \mbox{(cf. \eqref{lb10})\ and }\ A^{*}=diag\ (\underline{a},\overline{a},..,\overline{a})\end{equation*} 
achieves the equality in \eqref{bs9} with $M^\prime_{AB} =M^{\prime\prime}_{AB}= diag\ ( 1,0,..,0)$. \\
\\
Similarly, the Hashin-Shtrikman construction given in \eqref{FG6} i.e. with core $a_1 I$ and coating $a_2 I$ for $A_{B(0,1)}$ and core $b_1 I$ with coating $b_2 I$ for $B_{B(0,1)}$, with $\theta_B \leq \theta_A$ 
\begin{align*}
a^{*}&=\ a_2 + N a_2 \frac{\theta_A(a_1-a_2)}{(1-\theta_A)a_1 + (N +\theta_A -1)a_2}\\
b^{\#} &=\  b_2\ [\ 1 + \frac{N\theta_A(1-\theta_A)(a_2-a_1)^2}{((1-\theta_A)a_1 + (N+\theta_A-1)a_2)^2}\ ] - \frac{(b_2-b_1)(Na_2)^2\theta_B}{((1-\theta_A)a_1 + (N+\theta_A-1)a_2)^2}.
\end{align*}
achieves the equality in \eqref{bs9} with $M^\prime_{AB} =M^{\prime\prime}_{AB}= diag\ (\frac{1}{N},\frac{1}{N},..,\frac{1}{N}\ )$.
\paragraph{Sequential Laminates when ${\omega}_{B^{\epsilon}}\subset {\omega}_{A^{\epsilon}} $ :}
Here we construct the $(p,p)$-sequential laminates $B^{\#}_{p,p}$ whenever ${\omega}_{B^{\epsilon}}\subset {\omega}_{A^{\epsilon}}$
and with having the same number of $p$ layers in the same directions $\{e_i\}_{1\leq i\leq p}$ as of
the sequential Laminates $A^{*}_p$ with matrix $a_1I$ and core $a_2I$ defined in \eqref{OP3}.
Similar to the Example \ref{ot5}, by considering the equality in \eqref{bs9}
with the above hypothesis, we define $B^{\#}_{p,p}$ through the following expression : 
\begin{align}\label{ot16}
&({A^{*}_p}^{-1}B^{\#}_{p,p}{A^{*}_p}^{-1}- c I)(\underline{A} ^{-1} - a_2^{-1}I)^2({A^{*}_p} ^{-1}- a_2^{-1}I)^{-2}\notag \\
&= (L-cI) + \{c\frac{(a_2-a_1)^2}{a_1^2}\theta_A(1-\theta_A) + 2(\frac{b_2}{a_2^2} -c)\frac{(a_2 - a_1)}{a_1}(1-\theta_A)\}(I-\sum_{i=1}^{p} m_i\frac{e_i\otimes e_i}{e_i.e_i}).
\end{align}
where $m_i\geq 0$ and $\sum_{i=1}^{p} m_i=1$.
This can be also written as, (in an inverse form):
\paragraph{Case L2(a): \ when $c=\frac{b_2}{a_2^2}$ :}
\begin{align*}
\{ \frac{(L-\frac{b_2}{a_2^2}I)}{(\underline{A} ^{-1} - a_2^{-1}I)}({A^{*}_p}^{-1} &- a_2^{-1}I) +\frac{b_2}{a_2}(\underline{A}^{-1}-{A^{*}_p}^{-1})\}^2({A^{*}_p}^{-1}B^{\#}_{p,p}{A^{*}_p}^{-1}- \frac{b_2}{a_2^2}I)^{-1}\notag\\
&=\ (L-\frac{b_2}{a_2^2}I) + \frac{b_2}{a_2^2}\theta_A(1-\theta_A)\frac{(a_2-a_1)^2}{a_1^2}(I-\sum_{i=1}^{p} m_i\frac{e_i\otimes e_i}{e_i.e_i}).
\end{align*}
\textbf{Case L2(b): \ when $c=\frac{b_1}{a_1^2}$ :}
\begin{align*}
&\{ \frac{(L-\frac{b_1}{a_1^2}I)}{(\underline{A} ^{-1} - a_2^{-1}I)}({A^{*}_p}^{-1} - a_2^{-1}I) +\frac{b_1}{a_1^2}a_2(\underline{A}^{-1}-{A^{*}_p}^{-1}) + 2(\frac{b_2}{a_2^2}-\frac{b_1}{a_1^2})\frac{(\underline{A}^{-1} - {A^{*}_p}^{-1})}{(\underline{A}^{-1}- a_2^{-1}I)}\}^2\notag\\
&\ .({A^{*}_p}^{-1}B^{\#}_{p,p}{A^{*}_p}^{-1}- \frac{b_1}{a_1^2}I)^{-1}\ = (L-\frac{b_1}{a_1^2}I) + \frac{b_1}{a_1^2}\theta_A(1-\theta_A)\frac{(a_2-a_1)^2}{a_1^2}(I-\sum_{i=1}^{p} m_i\frac{e_i\otimes e_i}{e_i.e_i})\notag\\
&\qquad\qquad\qquad\qquad\qquad\qquad\qquad+ 2(\frac{b_2}{a_2^2}-\frac{b_1}{a_1^2})(1-\theta_A)\frac{(a_2-a_1)}{a_1}(I-\sum_{i=1}^{p} m_i\frac{e_i\otimes e_i}{e_i.e_i}).
\end{align*}
The above defined $(N,N)$- sequential laminates give the saturation / optimality of the lower trace bound \eqref{to},\eqref{tn} respectively. This property is incorporated in the very definition of $B^{\#}_{N,N}$ (see \eqref{ot16}).
\hfill\epr
\end{example}
\paragraph{Step 2 : $A^{*}\in\mathcal{G}_{\theta_A}$ arbitrary :}
Here we will establish the trace bound L1 and trace bound L2 obtained in Step $1$, are also satisfied by any pair $(A^{*},B^{\#})$ corresponding to the arbitrary microstructures $\omega_{A^\epsilon},\omega_{B^\epsilon}$; i.e. need not be satisfying the oscillatory systems \eqref{eik},\eqref{os3}.  

We recall that the set $\mathcal{G}_{\theta_A(x)}$ consists of all $A^{*}(x)$ governed with two phase 
medium $(a_1,a_2)$ with its volume proportion $(\theta_A(x),1-\theta_A(x))$, and characterized by the inequalities given in \eqref{FL11}. We also recall that $\mathcal{K}_{\theta_A(x)}$ consists of all eigenvalues of $A^{*}(x)\in\mathcal{G}_{\theta_A(x)}$.   
In the sequel, we will interchangeably use $\mathcal{G}_{\theta_A(x)}$ and $\mathcal{K}_{\theta_A(x)}$ to mean the same set.   
$\mathcal{G}_{\theta_A(x)}$ is a convex region for fixed $x$.  The boundaries of $\mathcal{G}_{\theta_A(x)}$ are given by $\partial\mathcal{G}^{L}_{\theta_A(x)}$ and $\partial\mathcal{G}^{U}_{\theta_A(x)}$. The
$\partial\mathcal{G}^{L}_{\theta_A(x)}$ represents the set of all symmetric matrices with 
eigenvalues $\lambda_1(x),..,\lambda_N(x)$ satisfying $\underline{a}(x) \leq \lambda_i(x)\leq \overline{a}(x) \ \ \forall 1\leq i\leq N$ and the lower bound equality
\begin{equation*}\sum_{i=1}^N \frac{1}{\lambda_i(x)-a_1} = \frac{1}{\underline{a}(x)-a_1} + \frac{N-1}{\overline{a}(x)-a_1}\end{equation*}
and similarly, the $\partial\mathcal{G}^{U}_{\theta_A(x)}$ represents the set which satisfies the upper bound equality  
\begin{equation*}\sum_{i=1}^N \frac{1}{a_2 -\lambda_i(x)} = \frac{1}{a_2-\underline{a}(x)} + \frac{N-1}{a_2-\overline{a}(x)}.\end{equation*}
We have $\mathcal{G}_{\{\theta_A=0\}} = \{a_2\}$ and  $\mathcal{G}_{\{\theta_A=1\}} = \{a_1\}$. 
The family $\{\mathcal{G}_{\theta_A(x)}\}_{\theta_A(x)\in[0,1]}$ gives a continuum of convex sets for almost everywhere $x$, whose union is also a convex set \cite{CHA,MT1}. 
Moreover, for any $\theta_A(x)\in (0,1)$, there  exists an interval $[\widetilde{\theta_A}(x),\widetilde{\widetilde{\theta_A}}(x)]$ for $x$ almost everywhere such that, 
\begin{equation*}\lb \mathcal{G}_{\theta(x)} \cap \mathcal{G}_{\theta_A(x)} \rb = \emptyset \mbox{ for } \theta(x)\notin [\widetilde{\theta_A}(x),\widetilde{\widetilde{\theta_A}}(x)]
\mbox{ and } \underset{\theta(x)\in[\widetilde{\theta_A}(x),\widetilde{\widetilde{\theta_A}}(x)]}{\cup}\lb \mathcal{G}_{\theta(x)} \cap \mathcal{G}_{\theta_A(x)} \rb = \mathcal{G}_{\theta_A(x)}, x \mbox{ a.e.}\end{equation*}
The set can be also expressed through the continuum of boundaries $\partial\mathcal{G}_{\theta(x)}$ as follows :
\begin{equation}\label{eii} \mathcal{G}_{\theta_A(x)}\subset \underset{\theta\in[\widetilde{\theta_A},\theta_A]}{\cup}\partial\mathcal{G}^{L}_{\theta(x)}\ \mbox{ and }\ \mathcal{G}_{\theta_A(x)}\subset\underset{\theta\in[\theta_A,\widetilde{\widetilde{\theta_A}}]}{\cup}\partial\mathcal{G}^{U}_{\theta(x)}, \ x \mbox{ a.e. } \end{equation}
We will use the above notation to obtain our desired results.
\begin{figure}[H]
\begin{center}
 \includegraphics[width = 14cm]{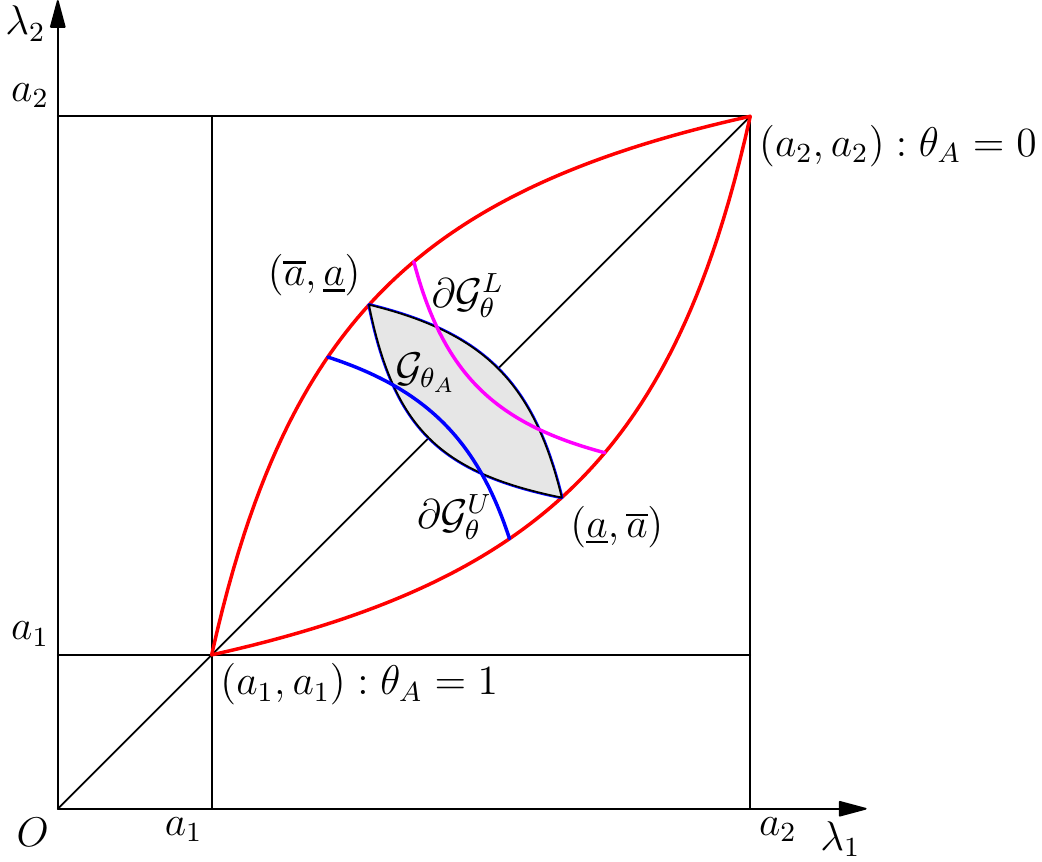}
 \caption{\textit{$N=2$ : The $\mathcal{G}_{\theta_A}$ set (gray region) can be seen as a continuum of the upper bound curve (in blue) $\partial\mathcal{G}^U_{\theta}$s and the lower bound curve (in magenta) $\partial\mathcal{G}^L_{\theta}$s, as the $\theta$ varies
 in some interval $[\widetilde{\theta_A},\widetilde{\widetilde{\theta_A}}]$ depending upon $\theta_A$.}}  
 \end{center}
\end{figure}
\noindent
Let us consider the interior points of $\mathcal{G}_{\theta_A(x)}$, or equivalently the continuum of the boundaries
$\{\partial\mathcal{G}^{L}_\theta(x)\}_{ \theta(x)\in[\widetilde{\theta_A}(x),\theta_A(x)]}$, $x$ a.e.; i.e. it says 
for a given $A^{*}(x)\in \mathcal{G}_{\theta_A(x)}$, there exist a unique $\theta(x)$ with  $\theta(x) < \theta_A(x)$, $x$ a.e. such that 
$A^{*}(x)\in \partial\mathcal{G}^{L}_{\theta(x)}$. So, it satisfies trace of \eqref{eif} where $\theta_A$ is replaced by $\theta$, i.e.
\begin{equation}\label{eio}
 tr\ (A^{*}(x)-a_1I)^{-1} = \frac{N}{(a_2-a_1)(1-\theta(x))} + \frac{\theta(x)}{(1-\theta(x))a_1}, \ \ A^{*}\in\mathcal{G}_{\theta_A(x)}\cap  \partial\mathcal{G}^{L}_{\theta(x)}.
\end{equation}
\noindent
Now in Step $1$ we have obtained the optimal `lower trace bound L1' on the pair $(A^{*},B^{\#})$ with $A^{*}\in\partial\mathcal{G}_{\theta(x)}^L$, 
which is nothing but considering the lower bound expression \eqref{eig} by replacing $\theta_A(x)$ by $\theta(x)$, i.e.
\begin{equation}\label{eip}
tr\ (B^{\#}-b_1I)(\overline{A}_\theta-a_1I)^2(A^{*}-a_1I)^{-2}\geq\ tr\ (\overline{B}-b_1I) +\frac{b_1}{a_1^2}(a_2-a_1)^2\theta(1-\theta).
\end{equation}
where $\overline{A}_\theta = \{a_1\theta + a_2(1-\theta)\}I$.\\
\\
Thus for any $A^{*}\in\mathcal{G}_{\theta_A}$, the pair $(A^{*},B^{\#})\in\mathcal{G}_{(\theta_A,\theta_B)}$ satisfies the following optimal lower trace bound whenever $\theta\leq\theta_A\leq\theta_B$
\begin{equation}\label{eim}
tr\ (B^{\#}-b_1I)(A^{*}-a_1I)^{-2}\geq \frac{N(b_2-b_1)(1-\theta_B)}{(a_2-a_1)^2(1-\theta(x))^2} +\frac{b_1}{a_1^2}\frac{\theta(x)}{(1-\theta(x))}
\end{equation}
Next one can find $\theta$ explicitly from \eqref{eio} in terms of $trace\ (A^{*}-a_1I)^{-1}$ where $A^{*}\in\mathcal{G}_{\theta_A}\cap \partial\mathcal{G}^{L}_{\theta}$ and substitute it
in \eqref{eip}  to eliminate $\theta$. We obtain :
\begin{equation}\label{FG15}
\theta = \frac{ a_1\{(a_2-a_1)\hspace{2pt} tr\ (A^{*}-a_1I)^{-1} - N\}}{(a_2-a_1)\{a_1tr\ (A^{*}-a_1I)^{-1} +1\}}\ \mbox{ and }\
(1-\theta) = \frac{a_2 + a_1(N-1)}{(a_2-a_1)\{a_1tr\ (A^{*}-a_1I)^{-1} +1\}}.
\end{equation}
Now putting $\theta$ and $(1-\theta)$ in \eqref{eip}, we get for all $A^{*}\in \mathcal{G}_{\theta_A}$ and $(A^{*},B^{\#})\in\mathcal{G}_{(\theta_A,\theta_B)}$ :
\begin{align}\label{eiq}
tr\ (B^{\#}-b_1I)(A^{*}-a_1I)^{-2}\geq\ & N(b_2-b_1)(1-\theta_B)\frac{(a_1tr\ (A^{*}-a_1I)^{-1} +1)^2}{(a_2 +a_1(N-1))^2}\notag\\
                                        &\qquad\qquad\qquad+\frac{b_1}{a_1}\frac{((a_2-a_1)\hspace{2pt} tr\ (A^{*}-a_1I)^{-1} - N)}{(a_2 +a_1(N-1))}.
\end{align}
\noindent
The above inequality coincides with \eqref{ub6} which is now proved.

Similarly, by using the inclusion $ \mathcal{G}_{\theta_A}\subset\underset{\theta\in[\theta_A,\widetilde{\widetilde{\theta_A}}]}{\cup}\partial\mathcal{G}^{U}_{\theta},$
one performs the same analysis with the lower trace bound L2 given in \eqref{ub1},\eqref{ub3} to generalize it for all  $A^{*}\in\mathcal{G}_{\theta_A}$
and establishes \eqref{eic},\eqref{to} and \eqref{tn} which provide the optimal lower bound whenever $\theta_B <\theta_A$.

\noindent
This completes the proof of establishing the lower trace bound L1, L2 announced in Section \ref{Sd9} and their saturation / optimality.
\hfill\qed
\begin{remark}[Pointwise bounds on energy densities]\label{siz}
This remark is analogous to Remark \ref{hsg} and here we find the bounds on energy densities $B^{\#}\nabla u\cdot\nabla u$ and
${A^{*}}^{-1}B^{\#}{A^{*}}^{-1}\sigma\cdot\sigma$ corresponding
to `trace bound L1' and `trace bound L2' respectively, which are useful in solving optimal oscillation-dissipation problem, as shown in Section \ref{qw5}. 
First we find the bound on $B^{\#}\nabla u\cdot\nabla u$ for $A^{*}\in\partial\mathcal{G}^{L}_{\theta_A}$ and $B^{\#}$ being associated to $A^{*}$. 
We consider the inequality obtained in \eqref{ub15} in Step $1c$ and using \eqref{ub4} for $A^{*}\in\partial\mathcal{G}^{L}_{\theta_A}$, we eliminate
$\eta$ to obtain the following inequality :
\begin{align}\label{tv}
B^{\#}\nabla u\cdot\nabla u \geq \{ b_1I  + 2 &(\overline{B}-b_1I)(\overline{A}-a_1I)^{-1}(A^{*}-a_1I)\notag\\
&\qquad+(Y- (\overline{B}-b_1 I))(\overline{A}-a_1I)^{-2}(A^{*}-a_1I)^2\}\nabla u\cdot\nabla u
\end{align}
where $Y$ is given in \eqref{dc13}.\\
(Previously, in Step $1c$, we used \eqref{ub4} to eliminate $\nabla u$ to obtain the matrix inequality \eqref{bs10}.) \\
\noindent
Then following the arguments presented in Step $2$ we generalize the above inequality \eqref{tv} for all $A^{*}\in\mathcal{G}_{\theta_A}$ and
the pair $(A^{*},B^{\#})\in\mathcal{G}_{(\theta_A,\theta_B)}$ by using the inclusion 
$\mathcal{G}_{\theta_A}\subset \underset{\theta\in[\widetilde{\theta_A},\theta_A]}{\cup}\partial\mathcal{G}^{L}_{\theta}$:
\begin{align}\label{FG19}
B^{\#}\nabla u\cdot\nabla u \geq \{ b_1I  + 2 &(\overline{B}-b_1I)(\overline{A}_{\theta}-a_1I)^{-1}(A^{*}-a_1I)\notag\\
&\qquad+(Y_{\theta}- (\overline{B}-b_1 I))(\overline{A}_{\theta}-a_1I)^{-2}(A^{*}-a_1I)^2\}\nabla u\cdot\nabla u
\end{align}
where $\theta_A$ in \eqref{tv} is replaced by $\theta$ which is explicitly given in \eqref{FG15}. 

The above bound is optimal whenever $\theta_A(x)\leq\theta_B(x)$, $x$ a.e. 
As it is shown above, for each $A^{*}\in\mathcal{G}_{\theta_A}$ there
exist a $B^{\#}$ such that the corresponding $(A^{*},B^{\#})$ achieves the equality
in the above bound \eqref{FG19}.
\hfill\qed\end{remark}
\begin{remark}\label{siz2}
Similarly, we consider the inequalities obtained in \eqref{sip} in Step $1f$ and using \eqref{ub5} 
for $A^{*}\in\partial\mathcal{G}_{\theta_A}^U$, we eliminate $\eta$ to obtain :
\begin{align}\label{ty}
{A^{*}}^{-1}B^{\#}{A^{*}}^{-1}\sigma\cdot\sigma \geq \{cI  + &2(L-cI)(\underline{A}^{-1} - a_2^{-1}I)^{-1}({A^{*}}^{-1}- a_2^{-1}I) +(Y^\prime- (L-cI))\notag\\
&\qquad\qquad\qquad\qquad (\underline{A}^{-1} - a_2^{-1}I)^{-2}({A^{*}}^{-1}- a_2^{-1}I)^{2}\}\sigma\cdot\sigma
\end{align}
where 
$c, L, Y^\prime$ are given in \eqref{os15},\eqref{bs4} and in \eqref{FG12} respectively.\\
\noindent
Then following the arguments presented in Step $2$ we generalize the above inequality \eqref{ty} 
for all $A^{*}\in\mathcal{G}_{\theta_A}$ and the pair $(A^{*},B^{\#})\in\mathcal{G}_{(\theta_A,\theta_B)}$
by using the inclusion $\mathcal{G}_{\theta_A}\subset\underset{\theta\in[\theta_A,\widetilde{\widetilde{\theta_A}}]}{\cup}\partial\mathcal{G}^{U}_{\theta}$:
\begin{align}\label{siu}
{A^{*}}^{-1}B^{\#}{A^{*}}^{-1}\sigma\cdot\sigma \geq \{cI  + &2(L-cI)(\underline{A}_\theta^{-1} - a_2^{-1}I)^{-1}({A^{*}}^{-1}- a_2^{-1}I) +(Y_\theta^\prime- (L_\theta-cI))\notag\\
&\qquad\qquad\qquad\qquad (\underline{A}_\theta^{-1} - a_2^{-1}I)^{-2}({A^{*}}^{-1}- a_2^{-1}I)^{2}\}\sigma\cdot\sigma
\end{align}
where $\theta_A$ in \eqref{tv} is replaced by $\theta$ which is uniquely determined by \eqref{FG15}. 
Moreover, the above bound is optimal whenever $\theta_B(x) < \theta_A(x)$, $x$ a.e. 
\hfill\qed\end{remark}
\begin{remark}\label{zz12}
By \eqref{zz11} we can take in \eqref{tv} $\nabla u=\zeta$ an arbitrary vector in $\mathbb{R}^N$. We apply \eqref{tv} with $A^{*}=A^{*}_N$ (cf.\eqref{OP2}) and $B^{\#}=B^{\#}_{N,N}$ (cf. Example \ref{ot5}) in which case, we get equality in \eqref{tv}. On the other hand, we apply \eqref{tv} with $A^{*}=A^{*}_N$ and $B^{\#}$ any relative limit corresponding to $A^{*}$. We obtain 
\begin{equation*} B^{\#}\zeta\cdot\zeta\ \geq\ B^{\#}_{N,N}\zeta\cdot\zeta.
\end{equation*}
\hfill\qed\end{remark}

\paragraph{Upper Bound : } 
In the final part of this section, we derive the upper bounds U1,U2 on
$(A^{*},B^{\#})$ introduced in Section \ref{ad18}. In deriving upper bound, we follow the same strategy as in lower bounds L1, L2. In the Step $1$ we will be finding bounds when $A^{*}\in\partial\mathcal{G}_{\theta_A}$, 
and then in Step $2$ we will generalize it for arbitrary $A^{*}\in \mathcal{G}_{\theta_A}$. 
\paragraph{Step 1a : ($H$-Measure term) }
Let us take $\eta\in\mathbb{R}^N$ arbitrary and consider the simple translated inequality 
\begin{equation*}(b_2I - B^{\epsilon})(\nabla u^{\epsilon} +\eta)\cdot(\nabla u^{\epsilon} +\eta)\geq 0 \quad\mbox{a.e. in }\Omega \end{equation*}
or equivalently, 
\begin{equation}\label{os13}
b_2\nabla u^{\epsilon}\cdot\nabla u^{\epsilon}-2B^\epsilon\nabla u^\epsilon\cdot\eta \geq\ B^{\epsilon}\nabla u^{\epsilon}\cdot\nabla u^\epsilon -(b_2I-B^\epsilon)\eta\cdot\eta-2b_2\nabla u^\epsilon\cdot\eta \ \ \mbox{a.e. in }\Omega.
\end{equation}
Let us pass to the limit in the above inequality. Passing to the limit in the right hand side is straight forward. 
The limit of left hand side of \eqref{os13} is 
\begin{equation*} lim\ \{b_2\nabla u^\epsilon\cdot\nabla u^\epsilon- 2B^\epsilon\nabla u^\epsilon\cdot\eta\} = b_2\nabla u\cdot\nabla u -2\overline{B}\nabla u\cdot\eta + Z \end{equation*} 
where $Z$ is a $H$-measure correction term. Combining these two we get  
\begin{equation}\label{OP1}
(b_2I - B^{\#})\nabla u\cdot\nabla u +2(b_2I-\overline{B})\nabla u\cdot\eta +(b_2I-\overline{B})\eta\cdot\eta + Z\geq\ 0.
\end{equation}
Now, $Z$ is given by
\begin{equation*} Z =\ \underset{\epsilon\rightarrow 0}{lim}\ q_4(\nabla u^\epsilon-\nabla u,(B^\epsilon-\overline{B})\eta)\cdot (\nabla u^\epsilon-\nabla u,(B^\epsilon-\overline{B})\eta)=\ \langle\langle\varPi_W, Q_4(U,B\eta)\rangle\rangle\end{equation*}
with, 
\begin{equation*}Q_4(W) =\ q_4(W)\cdot W=\ b_2|U|^2-2BU\cdot\eta,\ \ W=(U,B\eta)\in \mathbb{R}^N\times\mathbb{R}^N\end{equation*}
and $\varPi_{W}$ is the $H$-measure of the sequence $W_\epsilon-W_0= (\nabla u^\epsilon,B^\epsilon\eta)-(\nabla u,\overline{B}\eta)$. 
\paragraph{Step 1b : (Upper bound) :}
Here onwards we will find the upper bounds by choosing the field $\nabla u^\epsilon$ satisfying \eqref{six} and \eqref{eiz} or equivalently 
$A^{*}\in\partial\mathcal{G}_{\theta_A}^L$. Under the oscillatory system \eqref{eik} for $W_\epsilon^\prime$ where $W^\prime_\epsilon= (\nabla u^\epsilon,B^\epsilon\eta, A^\epsilon\eta)  \rightharpoonup (\nabla u,\overline{B}\eta,\overline{A}\eta)= W^\prime_0$ weakly in $L^2(\Omega)^{3N}$ and  $\varPi_{W^\prime}$ is the $H$-measure of the sequence $W^\prime_\epsilon-W^\prime_0$. Then following the very similar arguments 
presented in Step $1c$ for computing $Y$ in order to deduce the `lower trace bound L1', here we compute $Z$.
\begin{align}\label{os7}
Z= \langle\langle\varPi_{W}, Q_4(U,B\eta)\rangle\rangle &=  \langle\langle\varPi_{W^{\prime}}, Q_4(\frac{(A\eta\cdot\xi)}{a_1|\xi|^2}\xi,B\eta) \rangle\rangle \notag\\
&= \langle\langle \varPi_{W^\prime}, b_2\lb \frac{(A\eta\cdot\xi)}{a_1|\xi|} - \frac{(B\eta\cdot\xi)}{b_2|\xi|}\rb^2\rangle\rangle -\langle\langle\varPi_{W^\prime},\frac{(B\eta\cdot\xi)^2}{b_2|\xi|^2}\rangle\rangle\notag\\    
&= \frac{1}{a_1^2b_2}\langle\langle\widetilde{\varPi}_{AB}, \frac{((b_2A-a_1B)\eta\cdot\xi)^2}{|\xi|^2}  \rangle\rangle -\frac{1}{b_2}\langle\langle \varPi_B, \frac{(B\eta\cdot\xi)^2}{|\xi|^2}\rangle\rangle\notag \\
&= \widetilde{R_1} - \frac{\theta_B(1-\theta_B)(b_2-b_1)^2}{b_2}M_B\eta\cdot\eta
\end{align}
where $\widetilde{\varPi}_{AB}\in \mathcal{M}(\Omega\times \mathbb{S}^{N-1};\mathbb{R}^{N\times N}) $ is $H$-measure of the sequence
$\{(b_2A^\epsilon-a_1B^\epsilon) -(b_2\overline{A}-a_1\overline{B})\}\eta$ and $M_B$ 
is the non-negative matrix with unit trace defined in \eqref{dc19}.\\
The $H$-measure $\widetilde{\varPi}_{AB}$ reduces to 
\begin{equation*}
(\widetilde{\varPi}_{AB})_{ij} = (\widetilde{\nu}_{AB})\eta_i\eta_j  \ \ \forall i,j =1,..,N
\end{equation*}
where, $\widetilde{\nu}_{AB}$ is  $H$-measure of the scalar sequence $\{b_2(a_1-a_2)(\A-\theta_A(x))-a_1(b_1-b_2)(\B-\theta_B(x))\}$ with
\begin{equation*}
\widetilde{\nu}_{AB}(x,d\xi)\geq 0\ \mbox{ and }\ \int_{\mathbb{S}^{N-1}}\widetilde{\nu}_{AB}(x,d\xi) = U_{AB}
\end{equation*}
with 
\begin{align*}
U_{AB} & := L^\infty(\Omega)\mbox{ weak* limit of }\{b_2(a_1-a_2)(\A-\theta_A)-a_1(b_1-b_2)(\B-\theta_B)\}^2\notag\\
&= b_2^2(a_2-a_1)^2\theta_A(1-\theta_A) +a_1^2(b_2-b_1)^2\theta_B(1-\theta_B)-2b_2a_1(b_2-b_1)(a_2-a_1)(\theta_{AB}-\theta_A\theta_B)\notag\\
&\leq\  b_2^2(a_2-a_1)^2\theta_A(1-\theta_A) +a_1^2(b_2-b_1)^2\theta_B(1-\theta_B) +2b_2a_1(b_2-b_1)(a_2-a_1)\theta_A\theta_B\notag\\
&\ \ \ \ =: U^0_{AB} \mbox{ (say)}. 
\end{align*}
The above inequality $U_{AB}\leq U^{0}_{AB}$ follows by simply using $\theta_{AB}\geq 0$ and it becomes equal when $\theta_{AB}=0\ $ (e.g. $\omega_{A^\epsilon}\cap \omega_{B^\epsilon} =\emptyset$). 
Thus
\begin{align*}
\widetilde{R_1} = \frac{1}{a_1^2b_2}\langle\langle\widetilde{\varPi}_{AB}, \frac{((b_2A-a_1B)\eta\cdot\xi)^2}{|\xi|^2}  \rangle\rangle &= \frac{1}{a_1^2b_2} \int_{\mathbb{S}^{N-1}} \frac{(\eta \cdot \xi)^2}{|\xi|^2}\widetilde{\nu}_{AB} (d\xi)\notag\\
                                                             &= \frac{1}{a_1^2b_2}U_{AB}\ \widetilde{M}_{AB}\eta\cdot\eta \leq \frac{1}{a_1^2b_2}U^{0}_{AB}\ \widetilde{M}_{AB}\eta\cdot\eta
\end{align*}
where $\widetilde{M}_{AB}$ is a non-negative matrix with unit trace defined by 
\begin{equation*}
 \widetilde{M}_{AB} = \frac{1}{U_{AB}}\int_{\mathbb{S}^{N-1}} \xi \otimes \xi\ \widetilde{\nu}_{AB} (d\xi). 
\end{equation*}
Therefore from \eqref{os7} we have 
\begin{align}\label{zz10}
Z \leq\ &\{\frac{b_2}{a_1^2}(a_2-a_1)^2\theta_A(1-\theta_A) +\frac{(b_2-b_1)^2}{b_2}\theta_B(1-\theta_B)+\frac{2}{a_1}(b_2-b_1)(a_2-a_1)\theta_A\theta_B\}\widetilde{M}_{AB}\eta\cdot\eta\notag\\
        &\quad - \frac{\theta_B(1-\theta_B)(b_2-b_1)^2}{b_2}M_B\eta\cdot\eta.
\end{align}
with 
\begin{equation}\label{FG14}
 trace\ Z\ \leq\ \frac{b_2}{a_1^2}(a_2 - a_1)^2\theta_A(1-\theta_A) + \frac{2}{a_1}(b_2-b_1)(a_2-a_1)\theta_A\theta_B.
\end{equation}
\textbf{Matrix upper bound :}
Next we consider \eqref{OP1} with the upper bound \eqref{zz10} on $Z$ and choosing $\nabla u = - (\overline{A}-a_1I)(A^{*}- a_1 I)^{-1}\eta$ (cf.\eqref{ub4}) to obtain : 
\begin{align}
&(b_2I-B^{\#})(\overline{A}-a_1I)^2(A^{*}- a_1 I)^{-2}\eta\cdot\eta - 2 (b_2I-\overline{B})(\overline{A}-a_1I)(A^{*}- a_1 I)^{-1}\eta\cdot \eta \notag\\
&\qquad\quad+(b_2I-\overline{B})\eta\cdot\eta - \frac{\theta_B(1-\theta_B)(b_2-b_1)^2}{b_2}M_B\eta\cdot\eta+\{\frac{b_2}{a_1^2}(a_2-a_1)^2\theta_A(1-\theta_A)\notag\\
&\qquad\qquad\qquad+\frac{(b_2-b_1)^2}{b_2}\theta_B(1-\theta_B)+\frac{2}{a_1}(b_2-b_1)(a_2-a_1)\theta_A\theta_B\}\widetilde{M}_{AB}\eta\cdot\eta \geq 0 \label{OP4}
\end{align}
\textbf{Trace bound U1 : : $\theta_A+\theta_B\leq 1$ almost everywhere in $x$ :}
We simplify the above bound with using \eqref{eif} as $A^{*}\in\partial\mathcal{G}^L_{\theta_A}$
and then by taking trace using \eqref{FG14} to obtain :
\begin{equation}\label{FL2}
tr\ (b_2I-B^{\#})(\overline{A}-a_1I)^2(A^{*}-a_1I)^{-2}\geq\ tr\ (b_2I-\overline{B}) -\frac{b_2}{a_1^2}(a_2-a_1)^2\theta_A(1-\theta_A).
\end{equation}
However $(b_2I-B^{\#})$ need not be positive definite (cf. \eqref{Sd3}), so we slightly modify \eqref{FL2} by using $A^{*}\in \partial\mathcal{G}^L_{\theta_A}$ to obtain : 
\begin{equation}\label{lb18}
tr\ (\frac{b_2}{a_1}A^{*}-B^{\#})(\overline{A}-a_1I)^2(A^{*}(x)- a_1 I)^{-2} \geq\ tr\ (b_2I-\overline{B})+ tr\ \frac{b_2}{a_1}(\overline{A}-a_1I).
\end{equation}
Thus \eqref{hsm} follows. Note that, $(\frac{b_2}{a_1}A^{*}-B^{\#})$ is positive definite matrix (cf. \eqref{Sd3}), that's why the above inequality turns out to be an upper bound.
 \hfill\qed
\\

Next we give examples of microstructures which possess the property  $\omega_{A^\epsilon}\cap\omega_{B^\epsilon}=\emptyset$ and such that
equality holds in \eqref{OP4} as shown below.
\begin{example}[Saturation/Optimality]\label{FG5}
The equality of this lower bound is achieved by the laminated materials $U^{\#}_1$ (cf. \eqref{tp}),
at composites based on Hashin-Shtrikman construction given in \eqref{FG7} and at sequential laminates of higher rank.
\bpr
The composite based on Hashin-Shtrikman construction given in \eqref{FG7} i.e.
with core $a_2 I$ and coating $a_1 I$ for $A_{B(0,1)}$ and core $b_1 I$ with coating $b_2 I$ for $B_{B(0,1)}$, with $\theta_A + \theta_B\leq 1$ 
\begin{align*}
a^{*}&=\ a_1 + N a_1 \frac{(1-\theta_A)(a_2-a_1)}{(N-\theta_A)a_1 + \theta_A a_2}\\
b^{\#} &=\  b_2\ [\ 1 + \frac{N\theta_A(1-\theta_A)(a_2-a_1)^2}{(\theta_Aa_2 + (N-\theta_A)a_1)^2}\ ] - \frac{(b_2-b_1)(Na_1)^2\theta_B}{(\theta_Aa_2 + (N-\theta_A)a_1)^2}.
\end{align*}
achieves the equality in \eqref{OP4} with $M_B=\widetilde{M}_{AB} = diag\ (\frac{1}{N},\frac{1}{N},..,\frac{1}{N}\ )$.\\
\\
Similarly, the simple laminates say in $e_1$ direction with $\theta_A + \theta_B\leq 1$
\begin{equation*}B^{\#} = diag\ (U^\#_1, \overline{b},..,\overline{b})\ \mbox{ (cf.\eqref{tp}) and }\ A^{*}=diag\ (\underline{a},\overline{a},..,\overline{a})\end{equation*} 
achieves the equality in \eqref{OP4} with $M_B = \widetilde{M}_{AB}= diag\ ( 1,0,..,0)$. 
\paragraph{Sequential Laminates when ${\omega}_{A^{\epsilon}}\cap {\omega}_{B^{\epsilon}}=\emptyset $ :}
Here we define $(p,p)$-sequential laminates $B^{\#}_{p,p}$ whenever ${\omega}_{A^{\epsilon}}\cap {\omega}_{B^{\epsilon}}=\emptyset$
and with having the same number of $p$ layers in the same directions $\{e_i\}_{1\leq i\leq p}$.
Following the arguments presented before, by taking $A^{*}=A^{*}_p$ with matrix $a_1 I$ and core $a_2 I$ defined in \eqref{OP2}
and considering the equality in \eqref{OP4} we have
\begin{equation*}
(b_2I-B^{\#}_{p,p})(\overline{A}-a_1I)^{2}(A^{*}_p - a_1 I)^{-2} = (b_2I-\overline{B})-\frac{b_2(a_2 - a_1)^2}{a_1^2}\theta_A(1-\theta_A)(\sum_{i=1}^{p} m_i\frac{e_i\otimes e_i}{e_i.e_i});
\end{equation*}
or,
\begin{align*}
\{\frac{(b_2I-\overline{B})}{(\overline{A}-a_1I)}(A^{*}_p-a_1I)+\frac{b_2}{a_1}(A^{*}_p-a_1I)\}^2(\frac{b_2}{a_1}A^{*}_p-B^{\#}_{p,p})^{-1} = \frac{b_2}{a_1}(\overline{A}-a_1I) + (b_2I-\overline{B}).
\end{align*}
The above defined $(N,N)$- sequential laminates give the saturation / optimality of the upper trace bound \eqref{hsm}.
\hfill\epr
\end{example}
\begin{remark}
The above upper bound \eqref{lb18} need not be an optimal bound whenever $\theta_A + \theta_B > 1$. 

\hfill\qed\end{remark}
\noindent
\textbf{Trace bound U2 : $\theta_A(x) +\theta_B(x) > 1$ almost everywhere in $x$: } In this case, we establish the optimal upper bound by proceeding in a similar
way as we did in the case of proving lower trace bound L2. 
\paragraph{Step 1c : ($H$-Measure term) :}
We begin by considering the translated inequality :
\begin{equation*}(\frac{b_2}{a_1^2}I-(A^{\epsilon})^{-1}B^{\epsilon}(A^{\epsilon})^{-1})(\sigma^{\epsilon} + \eta)\cdot(\sigma^{\epsilon} + \eta)\ \geq\ 0 \quad\mbox{a.e. in }\Omega,\ \ \ \eta\in \mathbb{R}^N.\end{equation*} 
Note that here the optimal translated amount is $\frac{b_2}{a_1^2}$. By passing to the limit in the  of the above inequality we write  
\begin{equation}\label{OP5}
(\frac{b_2}{a_1^2}I - {A^{*}}^{-1}B^{\#}{A^{*}}^{-1})\sigma \cdot\sigma  + 2(\frac{b_2}{a_1^2}I-\widetilde{L})\sigma \cdot\eta +(b_2I-\widetilde{L}) + Z^\prime \geq\ 0
\end{equation}
where, $\widetilde{L}$ is the $L^{\infty}$ weak* limit of $(A^{\epsilon})^{-1}B^{\epsilon}(A^{\epsilon})^{-1}$ with having the following upper bound in this case $\theta_A +\theta_B>1$ as :
\begin{equation}\label{OP6}
 \widetilde{L}\ \leq\ \Theta^{*} := \{\frac{b_2}{a_1^2}+\frac{(b_1-b_2)}{a_1^2}\theta_B + b_1(\frac{1}{a_2^2}-\frac{1}{a_1^2})(1-\theta_A)\}I.
\end{equation}
Note that, by making the choice $\omega_{A^\epsilon}, \omega_{B^\epsilon}$ such that  $\omega_{B^\epsilon}^c\subset \omega_{A^\epsilon}$ (which is possible in the
present case), we obtain equality in the above inequality.\\
\\
In \eqref{OP5}, $Z^\prime$ is $H$-measure correction term defined as :
\begin{equation*}
lim\ \{\frac{b_2}{a_1^2}\sigma^\epsilon\cdot\sigma^\epsilon- 2(A^\epsilon)^{-1}B^\epsilon(A^\epsilon)^{-1}\sigma^\epsilon\cdot\eta\} = \frac{b_2}{a_1^2}\sigma\cdot\sigma -2\Theta^{*}\sigma\cdot\eta + Z^\prime \end{equation*} 
with 
\begin{equation*}
Z^\prime= \underset{\epsilon\rightarrow 0}{lim}\ q_5(\sigma^\epsilon-\sigma,((A^\epsilon)^{-1}B^\epsilon(A^\epsilon)^{-1}-\widetilde{L})\eta)\cdot (\sigma^\epsilon-\sigma,((A^\epsilon)^{-1}B^\epsilon(A^\epsilon)^{-1}-\widetilde{L})\eta)=\langle\langle\varPi_{W^{\prime\prime}}, Q_5(W^{\prime\prime})\rangle\rangle
\end{equation*}
where, 
\begin{equation*}Q_5(W^{\prime\prime}) =\ q_5(W^{\prime\prime})\cdot W^{\prime\prime}=\ \frac{b_2}{a_1^2}|\varSigma|^2-2A^{-1}BA^{-1}\varSigma\cdot\eta,\ \ W^{\prime\prime}= (\varSigma,A^{-1}BA^{-1})\in \mathbb{R}^N\times\mathbb{R}^N\end{equation*}
and $\varPi_{W^{\prime\prime}}$ is the $H$-measure of the sequence $W^{\prime\prime}_\epsilon-W^{\prime\prime}_0= (\sigma^\epsilon,(A^\epsilon)^{-1}B^\epsilon(A^\epsilon)^{-1}\eta)-(\sigma,\widetilde{L}\eta)$.
\paragraph{Step 1d : (Upper bound) :}
Here onwards we will find the upper bounds by choosing the field $\sigma^\epsilon$ satisfying \eqref{ub9},\eqref{ub10} or equivalently 
$A^{*}\in\partial\mathcal{G}_{\theta_A}^U$. Under the oscillatory system \eqref{os3} for $W_\epsilon^{\prime\prime\prime}$ where $W_\epsilon^{\prime\prime\prime} = (\sigma^\epsilon, (A^\epsilon)^{-1}B^\epsilon (A^\epsilon)^{-1}\eta, (A^\epsilon)^{-1}\eta ) \rightharpoonup  ( \sigma , \widetilde{L}\eta, \underline{A}^{-1}\eta )= W_0^{\prime\prime\prime}$ weakly in $L^2(\Omega)^{3N}$ and 
 $\varPi_{W^{\prime\prime\prime}}$ is the $H$-measure of the sequence $W^{\prime\prime\prime}_\epsilon-W^{\prime\prime\prime}_0$. Then following the very similar arguments presented in Step $1f$ for computing $Y^\prime$ in order to deduce the `lower trace bound L2', here we compute $Z^\prime$.
\begin{align}\label{OP7}
Z^\prime\ &=\ \langle\langle\varPi_{W^{\prime\prime}}, Q_5(\varSigma,A^{-1}BA^{-1}\eta)\rangle\rangle\ =\  \langle\langle\varPi_{W^{\prime\prime\prime}}, Q_5( a_2(I-\frac{\xi\otimes \xi}{|\xi|^2})A^{-1}\eta,A^{-1}BA^{-1}) \rangle\rangle \notag\\
          &=\ \langle\langle \widetilde{\varPi}^\prime_{AB}, \frac{b_2}{a_1^2} (I-\frac{\xi\otimes \xi}{|\xi|^2})(\frac{a_1^2}{b_2}A^{-1}BA^{-1}-a_2A^{-1})\eta\cdot(\frac{a_1^2}{b_2}A^{-1}BA^{-1}-a_2A^{-1})\eta\rangle\rangle \notag\\
          &\qquad  -\frac{a_1^2}{b_2}\langle\langle \widetilde{\varPi}^{\prime\prime}_{AB},(I-\frac{\xi\otimes \xi}{|\xi|^2})A^{-1}BA^{-1}\eta\cdot A^{-1}BA^{-1}\eta  \rangle\rangle, \notag \\
          &= \widetilde{R}^\prime_1 -\frac{a_1^2}{b_2}L^{\prime\prime}_{AB}(I-M^{\prime\prime}_{AB})\eta\cdot\eta
\end{align}
where $\widetilde{\varPi}^\prime_{AB}\in \mathcal{M}(\Omega\times \mathbb{S}^{N-1};\mathbb{R}^{N\times N}) $ is $H$-measure of the sequence
$\{\frac{a_1^2}{b_2}((A^\epsilon)^{-1}B^\epsilon(A^\epsilon)^{-1}-\widetilde{L})-a_2((A^\epsilon)^{-1}-(\underline{A})^{-1})\}\eta$
and the weak* limit $L^{\prime\prime}_{AB}$ and the non-negative matrix  $M^{\prime\prime}_{AB}$ with unit trace
are defined in \eqref{zz7} and \eqref{zz8} respectively. \\
Then $H$-measure $\widetilde{\varPi}^\prime_{AB}$ reduces to 
\begin{equation*}
(\widetilde{\varPi}^\prime_{AB})_{ij} = (\widetilde{\nu}^\prime_{AB})\eta_i\eta_j \ \ \forall i,j =1,..,N
\end{equation*}
where $\widetilde{\nu}^\prime_{AB}$ is  $H$-measure of the scalar sequence $\{\frac{a_1^2}{b_2}f^\epsilon- a_2(\frac{1}{a_1}-\frac{1}{a_2})(\A-\theta_A)\}$ with
\begin{align}
&\widetilde{\nu}^\prime_{AB}(x,d\xi)\geq 0\  \mbox{ and }\notag\\
&\int_{\mathbb{S}^{N-1}}\widetilde{\nu}^\prime_{AB}(x,d\xi) = L^\infty(\Omega)\mbox{ weak* limit of }\{\frac{a_1^2}{b_2}f^\epsilon- a_2(\frac{1}{a_1}-\frac{1}{a_2})(\A-\theta_A)\}^2 := \widetilde{L}^\prime_{AB}.
\end{align}
So
\begin{equation*}
\widetilde{R}^\prime_1 = \frac{b_2}{a_1^2}\widetilde{L}^\prime_{AB}\ (I-\widetilde{M}^\prime_{AB})\eta\cdot\eta
\end{equation*}
where $\widetilde{M}^\prime_{AB}$ is a non-negative matrix with unit trace defined by 
\begin{equation*}
\widetilde{M}^\prime_{AB} = \frac{1}{\widetilde{L}^\prime_{AB}}\int_{\mathbb{S}^{N-1}} \xi \otimes \xi\ \widetilde{\nu}^\prime_{AB} (d\xi). 
\end{equation*}
Therefore from \eqref{OP7} we have 
\begin{equation}\label{os8}
Z^\prime =  \frac{b_2}{a_1^2}\widetilde{L}^\prime_{AB}(I-\widetilde{M}^\prime_{AB})\eta\cdot\eta-  \frac{a_1^2}{b_2}L^{\prime\prime}_{AB}(I-M^{\prime\prime}_{AB})\eta\cdot\eta 
\end{equation}
with 
\begin{equation}\label{FL4}
trace\ Z^\prime = \{\frac{b_2}{a_1^2}\widetilde{L}^\prime_{AB} -\frac{a_1^2}{b_2}L^{\prime\prime}_{AB}\}(N-1).
\end{equation}
Recall that, 
\begin{align}\label{FG20}
\{\frac{b_2}{a_1^2}\widetilde{L}^\prime_{AB} -\frac{a_1^2}{b_2}L^{\prime\prime}_{AB}\}=&\ \frac{b_2}{a_1^2}\{\mbox{$L^\infty(\Omega)$ weak* limit of }\{\frac{a_1^2}{b_2}f^\epsilon- a_2(\frac{1}{a_1}-\frac{1}{a_2})(\A-\theta_A)\}^2\}\notag\\
                                                     &\qquad-\frac{a_1^2}{b_2}\{\mbox{$L^\infty(\Omega)$ weak* limit of }(f^\epsilon)^2\}.  
\end{align}
We want to find the upper bound the above quantity \eqref{FG20}. As usual the above quantity involves the parameter $\theta_{AB}(x)$, the $L^{\infty}(\Omega)$ weak* limit of $(\AB)$. 
Now keeping the estimate \eqref{OP6} in mind, we maximize $\widetilde{L}^\prime_{AB}$ and according to that the other quantity $L^{\prime\prime}_{AB}$ also gets determined. For that, it is enough to make the choice  $\omega_{B^\epsilon}^c\subset \omega_{A^\epsilon}$ which is possible in this present case ($\theta_A +\theta_B>1$), we bound $trace\ Z^\prime$ from above as following :
\begin{align}\label{os4}
&trace\  Z^\prime\ \leq \ L^\infty(\Omega) \mbox{ weak* limit of }\{\frac{b_2}{a_1^2}a_2^2(\frac{1}{a_1}-\frac{1}{a_2})^2(\A-\theta_A)^2-2a_2(\frac{1}{a_1}-\frac{1}{a_2})(\frac{b_1}{a_1^2}-\frac{b_1}{a_2^2})\notag\\
&\quad\ (\A-\theta_A)^2 + 2a_2(\frac{1}{a_1}-\frac{1}{a_2})\frac{(b_1-b_2)}{a_1^2}((1-\B)-(1-\theta_B))(\A-\theta_A)\}(N-1)\notag\\
&\qquad= \{\frac{b_2}{a_1^2}a_2^2(\frac{1}{a_1}-\frac{1}{a_2})^2\theta_A(1-\theta_A)-2a_2(\frac{1}{a_1} - \frac{1}{a_2})(\frac{b_1}{a_1^2} -\frac{b_1}{a_2^2})\theta_A(1-\theta_A)\notag\\
&\quad\qquad\qquad\qquad\qquad\qquad\qquad+2a_2(\frac{1}{a_1}-\frac{1}{a_2})\frac{(b_1-b_2)}{a_1^2}(1-\theta_B)(1-\theta_A)\}(N-1).
\end{align}
\textbf{Matrix upper bound :}
Next we consider \eqref{OP5} with $Z^\prime$ in \eqref{os8} and choosing 
$\sigma = - (\underline{A}^{-1} - a_2^{-1}I)({A^{*}}^{-1}- a_2^{-1}I)^{-1}\eta$ 
(cf.\eqref{ub5}) and using $\widetilde{L}\leq \Theta^{*}$ (cf.\eqref{OP6}) we obtain 
\begin{align}\label{OP9}
&(\frac{b_2}{a_1^2}I-{A^{*}}^{-1}B^{\#}{A^{*}}^{-1})(\underline{A}^{-1} - a_2^{-1}I)^2({A^{*}} ^{-1}- a_2^{-1}I)^{-2}\eta\cdot\eta + (\frac{b_2}{a_1^2}I-\Theta^{*})\eta\cdot\eta\notag\\
&\qquad\qquad\qquad\qquad\qquad- 2(\frac{b_2}{a_1^2}I-\Theta^{*})(\underline{A}^{-1} - a_2^{-1}I)({A^{*}}^{-1}- a_2^{-1}I)^{-1}\eta\cdot\eta + Z^{\prime}  \geq\ 0 
\end{align} 
\textbf{Trace bound U2 : $\theta_A+\theta_B>1$ almost everywhere in $x$ :} 
We simplify the above bound with using \eqref{eij} as $A^{*}\in\partial\mathcal{G}_{\theta_A}^U$ 
and then by taking trace using \eqref{os4} to obtain : 
\begin{align}\label{lb17}
tr \ &(\frac{b_2a_2}{a_1^2}{A^{*}}^{-1}-{A^{*}}^{-1}B^{\#}{A^{*}}^{-1})(\underline{A}^{-1} - a_2^{-1}I)^2({A^{*}}^{-1} - a_2^{-1}I)^{-2}\notag\\
&\geq\  tr\ (\frac{b_2}{a_1^2}I-\Theta^{*})+ tr\ \frac{b_2a_2}{a_1^2}(\underline{A}^{-1}-a_2^{-1}I)- 2\frac{(b_2-b_1)(a_2-a_1)}{a_1^3}(1-\theta_A(x))(N-1). 
\end{align}
Thus \eqref{tm} follows. Note that $(\frac{b_2a_2}{a_1}{A^{*}}^{-1}-{A^{*}}^{-1}B^{\#}{A^{*}}^{-1})$ is a positive definite matrix (cf.\eqref{Sd3}),
that is why the above inequality turns out to be an upper bound.
\hfill\qed
\begin{example}[Saturation/Optimality]
The equality of the above upper bound is achieved by the laminated materials $U^{\#}_2$ (cf. \eqref{tp}), at the
composites based on Hashin-Shtrikman construction given in \eqref{bs19} and at the sequential laminates of higher rank.
\paragraph{Sequential Laminates when $\omega_{B^\epsilon}^c \subset \omega_{A^\epsilon}$ :}
Here we define the $(p,p)$-sequential laminates $B^{\#}_{p,p}$ whenever $ ({\omega}_{B^{\epsilon}})^c \subset\omega_{A^\epsilon}$
and with having the same number of $p$ layers in the same directions $\{e_i\}_{1\leq i\leq p}$.
By taking $A^{*}=A^{*}_p$ with matrix $a_1 I$ and core $a_2 I$ defined in \eqref{OP3}
and considering the equality in \eqref{OP9} we have 
\begin{align*}
&(\frac{b_2a_2}{a_1^2}{A^{*}_p}^{-1}-{A^{*}_p}^{-1}B^{\#}_{p,p}{A^{*}_p}^{-1})(\underline{A}^{-1} - a_2^{-1}I)^2({A^{*}_p} ^{-1}- a_2^{-1}I)^{-2} \\
&\ = (\frac{b_2a_2}{a_1^2}\underline{A}^{-1}-\Theta^{*})- 2\frac{(b_2-b_1)(a_2-a_1)}{a_1^3}(1-\theta_A)(I-\sum_{i=1}^{p} m_i\frac{e_i\otimes e_i}{e_i.e_i}); \mbox{ with }\sum_{i=1}^p m_i =1.
\end{align*}
The above defined $(N,N)$- sequential laminates give the saturation / optimality of the upper trace bound \eqref{tm}.
\hfill\qed
\end{example}
\noindent\textbf{Step 2 :} Proceeding in a exactly same way as we have done for the lower trace
bounds L1 and L2, the upper trace bounds U1 and U2 obtained in the above Step $1$ 
with the hypothesis $A^{*}\in\partial\mathcal{G}_{\theta_A}$ can be generalized  
for any pair of $(A^{*},B^{\#})\in\mathcal{G}_{(\theta_A,\theta_B)}$ with $A^{*}\in\mathcal{G}_{\theta_A}$. \\
\\
This completes our discussion of proof of the main results announced in Section \ref{Sd9}.
\hfill\qed
\begin{remark}\label{OP8}
Results analogous to Remark \ref{siz},\ref{siz2} and Remark \ref{zz12} are valid for U1, U2 also. In particular, for
$A^{*}=A^{*}_N$ (cf.\eqref{OP2}) and $B^{\#}$ any relative limit corresponding to $A^{*}$, we will have for $\zeta\in\mathbb{R}^N$ :
\begin{equation*} B^{\#}\zeta\cdot\zeta\ \leq\ B^{\#}_{N,N}\zeta\cdot\zeta \end{equation*}
where, $B^{\#}_{N,N}$ is given in Example \ref{FG5}. 
\hfill\qed\end{remark}

\section{Optimality of Regions}\label{qw4}
\setcounter{equation}{0}
This section is devoted to the proof of  Theorem \ref{qw6}. Let us recall the sets $\mathcal{K}_{(\theta_A,\theta_B)}$ and $\mathcal{K}^{f}_{(\theta_A,\theta_B)}(A^{*})$ introduced in \eqref{kab} and \eqref{kfab} respectively to begin with the following result :
\begin{lemma}
 Recall that the Hausdorff distance $d_H$ between two compact sets in a metric space is 
\begin{equation}
 d_H\lb K_1,K_2\rb = \underset{x_1\in K_1}{max}\ \underset{x_2\in K_2}{min}\ d(x_1,x_2) + \underset{x_2\in K_2}{max}\ \underset{x_1\in K_1}{min}\ d(x_2,x_1).
\end{equation}
There exist positive constants $C>0,\delta_A>0,\delta_B>0$ such that, for any $\theta_{A^1},\theta_{A^2}\in [0,1]$ and  $\theta_{B^1},\theta_{B^2}\in [0,1]$, we have
\begin{align}
d_H\lb\mathcal{G}_{\theta_{A^1}},\mathcal{G}_{\theta_{A^2}}\rb &\leq\ C|\theta_{A^1} -\theta_{A^2}|^{\delta_A};\label{qw9}\\
d_H\lb\mathcal{K}^{f}_{(\theta_{A^1},\theta_{B^1})}(A^{*,1}),
\mathcal{K}^{f}_{(\theta_{A^2},\theta_{B^2})}(A^{*,2})\rb
&\leq\ C\lb|\theta_{A^1} -\theta_{A^2}|^{\delta_A} + |\theta_{B^1} -\theta_{B^2}|^{\delta_B}\rb. \label{qw8}
\end{align}
Here $\mathcal{G}_{\theta_{A}}$ stands for the \textit{G-closure set} and $\mathcal{K}^f_{(\theta_A,\theta_B)}(A^{*})$ stands for the fibre over $A^{*}$ and it is defined in Remark \ref{qw1}.
\end{lemma}
\begin{proof}
A proof of \eqref{qw9} is found in \cite{A}, who uses periodic microstructures. Here we argue directly without using them. 
It is enough to prove that, for any sequence of characteristic functions $\chi_{\omega_{A^{\epsilon,1}}}$ with $\chi_{\omega_{A^{\epsilon,1}}}\rightharpoonup \theta_{A^1}$ in $L^\infty(\Omega)$ weak*, there exists another sequence of characteristic functions $\chi_{\omega_{A^{\epsilon,2}}}$  
with $\chi_{\omega_{A^{\epsilon,2}}}\rightharpoonup \theta_{A^2}$ in $L^\infty(\Omega)$ weak* such that $|A^{*,1}-A^{*,2}|\leq C|\theta_{A^1}-\theta_{A^2}|^{\delta_A}$ where $A^{*,1}$ and $A^{*,2}$ defined as $A^{\epsilon,i}\xrightarrow{H} A^{*,i}$ for $i=1,2$ with $A^{\epsilon,i}
= \{a_1\chi_{\omega_A^{\epsilon,i}} + a_2(1-\chi_{\omega_A^{\epsilon,i}})\}I$ for $i=1,2$. 
Without loss of generality we assume that, $\theta_{A^1}\geq \theta_{A^2}$, and use the following Lypunov type result :
If there is a sequence $\chi_{\omega_{A^{\epsilon,1}}}\rightharpoonup \theta_{A^1}$ in $L^\infty(\Omega)$ weak* with  $\theta_{A^1}\geq \theta_{A^2}$, then there exists a sequence $\chi_{\omega_{A^{\epsilon,2}}}\rightharpoonup \theta_{A^2}$ in $L^\infty(\Omega)$ weak* such that $(\chi_{\omega_{A^{\epsilon,1}}}-\chi_{\omega_{A^{\epsilon,2}}})\geq 0$ for $x\in\Omega$ a.e. Then by using \eqref{FL16} one obtains $|A^{*,1}-A^{*,2}|\leq C|\theta_{A^1}-\theta_{A^2}|^{\delta_A}$ and consequently \eqref{qw9}. 

Similarly, one shows for any sequence of characteristic functions $\chi_{\omega_{B^{\epsilon,1}}},$ with $\chi_{\omega_{B^{\epsilon,1}}}\rightharpoonup \theta_{B^1}$ in $L^\infty(\Omega)$ weak*, there exists another sequence of characteristic function $\chi_{\omega_{B^{\epsilon,2}}}$  with 
with $\chi_{\omega_{B^{\epsilon,2}}}\rightharpoonup \theta_{B^2}$ in $L^\infty(\Omega)$ weak* such that together with the above mentioned $\chi_{\omega_{A^{\epsilon,1}}},\chi_{\omega_{A^{\epsilon,2}}}$, one gets $|B^{\#,1}-B^{\#,2}|\leq C\lb
|\theta_{A^1}-\theta_{A^2}|^{\delta_A} + |\theta_{B^1}-\theta_{B^2}|^{\delta_B}\rb $ from \eqref{FL14} and it gives \eqref{qw8}. 
\hfill\end{proof}
\noindent
\paragraph{Characterization of $\mathcal{G}_{(\theta_A,\theta_B)}$/Optimality of the regions $(Li,Uj)$ :}
Proving optimality of the regions is equivalent to obtaining   
the following characterization of $\mathcal{G}_{(\theta_A(x),\theta_B(x))}$, $x$ almost everywhere in terms of $\mathcal{K}_{(\theta_A,\theta_B)}$ which was defined in Remark \ref{qw1} :
\begin{align}\label{qw2}
\mathcal{G}_{(\theta_A(x),\theta_B(x))}= \{ (A^{*}(x),B^{\#}(x))\in \mathcal{M}(a_1,a_2;\Omega)&\times  \mathcal{M}(b_1,\widetilde{b_2};\Omega)\ |\ \notag\\ 
\  ( A^{*}(x)&,  B^{\#}(x))\in \mathcal{K}_{(\theta_A(x),\theta_B(x))}\mbox{ a.e. }x\in \Omega\}.
\end{align}
\bpr[Proof of \eqref{qw2} :]
Similar pointwise characterization of $\mathcal{G}_{\theta_A(x)}$ in terms of bounds on $A^{*}$ is a celebrated theorem in this subject \cite{TF,MT,A}.
The above result describes its extension by including relative limits $B^{\#}$ too. Let us underline some new features : there are four possible phase regions in the place of a  single region for $A^{*}$.  Secondly, due to non-commutativity of $A^{*},B^{\#}$, the bounds are formulated in terms of traces of product of matrices instead of their individual spectra. The proof is presented in three parts.
\paragraph{First part : Macroscopically homogeneous case i.e. $(A^{*},B^{\#})$, $(\theta_A,\theta_B)$ are constants :}
Here we assume that $A^{*}, B^{\#}$ are constant matrices and $\theta_A,\theta_B$ are constant functions.  Examples of microstructures  which are macroscopically homogeneous include periodic ones and Hashin-Shtrikman structures.
Depending upon the values of $\theta_A$ and $\theta_B$, there are four sets of possible lower trace bounds and upper trace bounds $\{Li,Uj\}$, $i,j=1,2$ introduced in Section \ref{ad18}. Each pair defines a region denoted as $(Li,Uj)$. We consider, for instance, the region $(L1,U1)$, where 
$\theta_A\leq\theta_B$ and $\theta_A+\theta_B\leq 1$. Now let us take $(A^{*},B^{\#})$ lying in this region. Of course  $A^{*}\in\mathcal{G}_{\theta_A}$, and $B^{\#}$ satisfies \eqref{Sd3}. We want to show that there exists $(A^\epsilon(x),B^\epsilon(x))$ satisfying \eqref{ta},\eqref{tb} such that $A^\epsilon(x) \xrightarrow{H} A^{*}$ and $B^\epsilon(x)\xrightarrow{A^\epsilon(x)} B^{\#}$ in $\Omega$. 
(Recall that in Section \ref{sil}, the converse of the above assertion was shown). \\

\noindent
We divide the proof into several sub-cases. First we treat matrices $A^{*}$ similar to $A^{*}_N= diag(\lambda_1,..,\lambda_N)$ with $\lambda_1\geq\ldots\geq\lambda_N$. Next, we show how the general case can be reduced to the above case. \\

\noindent
\textit{Case (1) :}
Let us consider $A^{*}\in\partial\mathcal{G}_{\theta_A}^L$ and any pair $(A^{*},B^{\#})$ satisfying the  
equality in L1 bound. Let's say spectrum of $A^{*} = \{\lambda_1,..,\lambda_N\}$. Without loss of generality, we assume the ordering
$\lambda_1\geq\ldots\geq\lambda_i\geq\ldots\geq\lambda_N$.
Then there exists an orthogonal matrix $P$ such that $A^{*}= P(diag\ (\lambda_1,..,\lambda_N))P^{-1}$. The required $A^\epsilon$ can be constructed as follows by change of variables technique. 
Let us denote $\Omega_0  = P^{-1}(\Omega)$. It is classical to construct (see \cite{A}) $A^\epsilon_N(x_0)$, $x_0\in\Omega_0$ 
given by $N$-sequential laminated microstructures such that $A^\epsilon_N(x_0)$ $H$-converges to $A^{*}_N$ in $\Omega_0$ where $A^{*}_N = diag(\lambda_1,\ldots,\lambda_N)$. 
Next we define $A^\epsilon(x)\stackrel{def}{=} PA^\epsilon_N(x_0)P^{-1}$, $x\in\Omega$, $x_0\in\Omega_0$ with $x_0=P^{-1}x$.
Then by the covariance property of $H$-convergence \cite[Lemma 21.1]{T},
$A^\epsilon(x)$ $H$-converges to $PA^{*}_NP^{-1} = A^{*}$\ \ in $P(\Omega_0)=\Omega$.

We constructed (cf. Example \ref{ot5}) $B^\epsilon_{N,N}(x_0)$, $x_0\in\Omega_0$ and diagonal matrix $B^{\#}_{N,N}$ given by $N$-sequential laminated microstructures with $\omega_{A^\epsilon_N}\subseteq \omega_{B^\epsilon_{N,N}}$ (possible since $\theta_A\leq \theta_B$) such that $B^\epsilon_{N,N}(x_0)$ converges to $B^{\#}_{N,N}$ relative to $A^{\epsilon}_N(x_0)$ in $\Omega_0$ and the pair $(A^{*}_N,B^{\#}_{N,N})$ satisfies the equality in L1 bound. 

Having treated diagonal matrices representing $N$-rank laminates, let us now consider other matrices $(A^{*},B^{\#})$.
Since the right hand side of L1 remains same for both $A^{*}$ and $A^{*}_N$ and since we are considering
equality in L1 bound, we have obviously
\begin{equation}\label{zz20}
 tr \ \{(A^{*}-a_1I)^{-1}(B^{\#}-b_1I)(A^{*}-a_1I)^{-1}\} =  tr \ \{(A^{*}_N-a_1I)^{-1}(B^{\#}_{N,N}-b_1I)(A^{*}_N-a_1I)^{-1}\}.
\end{equation}
We now show that  $B^{\#}$ and $B^{\#}_{N,N}$ have the same spectra and that $A^{*}$, $B^{\#}$ commute. For that, we prove the following result stated as :
\begin{lemma}[Optimality-Commutativity]\label{zz14}
For $A^{*}\in\mathcal{M}(a_1,a_2;\Omega)$ and $B^{\#}\in\mathcal{M}(b_1,\widetilde{b_2};\Omega)$, we have
\begin{align}\label{zz13}
&tr \ \{(A^{*}-a_1I)^{-1}(B^{\#}-b_1I)(A^{*}-a_1I)^{-1} \}\notag\\
&\geq   tr \{ (diag\ (\lambda_1-a_1,..,\lambda_N-a_1))^{-1}(diag\ (\mu_1-b_1,..,\mu_N-b_1))(diag\ (\lambda_1-a_1,..,\lambda_N-a_1))^{-1}\}
\end{align}
where spectrum of $A^{*}=\{\lambda_i\}_{i=1}^N$ with $\lambda_1\geq..\geq\lambda_i\geq..\geq\lambda_N$ 
and spectrum of $B^{\#}=\{\mu_i\}_{i=1}^N$ with $\mu_1\geq..\geq\mu_i\geq..\geq\mu_N$.\\
The equality in the above inequality \eqref{zz13} takes place if and only if $A^{*}$ commutes with $B^{\#}$.
\end{lemma}
\bpr
The above inequality \eqref{zz13} is a simple consequence of the following result from \cite{BT}, which says for any two positive semi-definite matrices $E$ and $F$ :
\begin{center}
 $tr\ EF \geq \sum_{i=1}^N \sigma_i(E)\sigma_{N-i+1}(F)$
\end{center}
where $\sigma_1(.)\leq..\leq\sigma_i(.)\leq..\leq\sigma_N(.)$ are the eigenvalues in increasing order. Further the equality holds only if $EF=FE$ in which case there exists an orthogonal matrix $Q$ satisfying
$Q^{-1}EQ =diag(\sigma_1(E),..,\sigma_N(E))$ and $Q^{-1}FQ =diag(\sigma_1(F),..,\sigma_N(F))$.
We apply the above result with the choice $E= (A^{*}-a_1I)^{-2}$ and $F=(B^{\#}-b_1I).$
\hfill\epr
\noindent
In our case, using successively \eqref{zz13}, the inequality $B^{\#}\geq B^{\#}_{N,N}$ (cf. Remark \ref{zz12}) and finally \eqref{zz20} in that order, we see that equality holds in \eqref{zz13}. As a consequence, it follows that
$B^{\#}$ and $B^{\#}_{N,N}$ have the same spectra and that $A^{*}$, $B^{\#}$ commute. Therefore, there exists an orthogonal matrix $Q$ diagonalizing $A^{*}$ and $B^{\#}$ simultaneously to give 
$A^{*}= Q(diag\ (\lambda_1,..,\lambda_N))Q^{-1}$ and $B^{\#}= Q(diag\ (\mu_1,..,\mu_N))Q^{-1}$. Now we can use as before change of variables technique using $Q$ instead of $P$. More precisely, we define $\Omega^\prime = Q^{-1}(\Omega)$ and 
$A^\epsilon(x) = QA^\epsilon_N(x^\prime)Q^{-1}$, 
$B^\epsilon(x) = QB^\epsilon_{N,N}(x^\prime)Q^{-1}$, $x\in\Omega$, $x^\prime\in\Omega^\prime$, with $x^\prime = Q^{-1}x$.
Applying Lemma \ref{pol6} it follows that $A^\epsilon \xrightarrow{H} A^{*}$ and  $B^\epsilon(x)\xrightarrow{A^\epsilon} QB^{\#}_{N,N}Q^{-1}=B^{\#}$ in $\Omega$.
\\

\noindent
\textit{Case (2) :}
Now let us consider $A^{*}\in int(\mathcal{G}_{\theta_A})$ arbitrary and continue to assume the equality in L1 bound.  This implies  $A^{*}\in \partial\mathcal{G}_{\theta}^L$ for some $\theta<\theta_A$.
Recall the spectrum of $A^{*}$ is denoted $\{\lambda_i\}_{i=1}^N$ with $\lambda_1\geq\ldots\geq\lambda_i\geq\ldots\geq\lambda_N$. In this case also, there exists $\widetilde{A}^\epsilon_N(\widetilde{x})$ with $\widetilde{x}\in\widetilde{\Omega}$ constructed through  
sequential laminated microstructures  such that $\widetilde{A}^\epsilon_N \xrightarrow{H} \widetilde{A}^{*}_N$
in $\widetilde{\Omega}$ with $\widetilde{A}^{*}_N= diag(\lambda_1,..,\lambda_N)$ with local proportion $\theta_A$ (see \cite[Page no. 125]{A}). Using this $\widetilde{A}^\epsilon_N$, we construct $\widetilde{B}^\epsilon_{N,N}$ (see Example \ref{ot5}) such that $\widetilde{B}^\epsilon_{N,N}\xrightarrow{\widetilde{A}^\epsilon_N}\widetilde{B}_{N,N}^{\#}$ in $\widetilde{\Omega}$ 
and the pair $(\widetilde{A}^{*}_N, \widetilde{B}^{\#}_{N,N})$ achieves the equality in the L1 bound. Then as in the previous case, we apply Remark \ref{zz12} over the fibre on $A^{*}\in \partial\mathcal{G}_{\theta}^L$ and Lemma \ref{zz14} to deduce that the spectra of $\widetilde{B}^{\#}_{N,N}$ and $B^{\#}$ is same and consequently the only if part of the lemma gives that there exists an orthogonal matrix $R$ such that 
$R\widetilde{A}^{*}_NR^{-1}=A^{*}$ and $R\widetilde{B}^{\#}_{N,N}R^{-1}=B^{\#}$ in $\Omega$. Now by following the change of variables techniques $R: \widetilde{\Omega} \mapsto \Omega$ as above, we obtain the pair $(A^\epsilon,B^\epsilon)$ satisfying \eqref{ta},\eqref{tb} such that $A^\epsilon\xrightarrow{H} A^{*}$ and $B^\epsilon\xrightarrow{A^\epsilon} B^{\#}$ in $\Omega$.
Let us remark that the hypothesis $\theta_A\leq\theta_B$ is needed in the construction of $\widetilde{B}^{\epsilon}_{N,N}$ with $\omega_{\widetilde{A}^\epsilon_N}\subseteq \omega_{\widetilde{B}^\epsilon_{N,N}}$.\\

\noindent
\textit{Case (3) :}
The two cases above complete the proof of the saturation / optimality of the L1 bound.  Similarly one can establish the saturation / optimality of the bound U1 using the hypothesis that $\theta_A +\theta_B\leq 1$. \\

\noindent
\textit{Case (4) :}
Here we prove the optimality of the region $(L1,U1)$, assuming of course that $\theta_A\leq \theta_B$ and $\theta_A+\theta_B\leq 1$.  Take $(A^{*},B^{\#})$ lying in the region (L1,U1). Of course we will have  $A^{*}\in\mathcal{G}_{\theta_A}$. We treat the special case $A^{*}\in\partial\mathcal{G}_{\theta_A}^L$ 
(for $A^{*}\in int(\mathcal{G}_{\theta_A})$, we follow the arguments of the Case (2) above).  
As before, let $A^{*}_N = diag(\lambda_1,\ldots,\lambda_N)$, where $\{\lambda_i\}_{1\leq i\leq N}=$ spectrum of $A^{*}$. It is classical to  construct $N$-rank laminate $A^\epsilon_N$ (core $a_2$ and matrix $a_1$) such that $A^\epsilon_N \xrightarrow{H} A^{*}_N$ in $\Omega$. 
Using the saturation / optimality of the L1 and U1 individually (case (3)), we can construct as before $B^{\epsilon,1}_{N,N}$ and $B^{\epsilon,2}_{N,N}$ which converge in the relative sense with respect to $A^\epsilon_N$  where $\omega_{A^\epsilon_N}\subseteq \omega_{B^{\epsilon,1}_{N,N}}$ and $\omega_{A^\epsilon_N}\subseteq \omega^c_{B^{\epsilon,2}_{N,N}}$, (i.e. underline $A^\epsilon_N$ is common for both $B^{\epsilon,1}_{N,N}$ and $B^{\epsilon,2}_{N,N}$) to, say, $B^{\#,1}_{N,N}$ and $B^{\#,2}_{N,N}$ respectively, which are diagonal, such that $(A^{*}_N,B^{\#,1}_{N,N})$ and $(A^{*}_N,B^{\#,2}_{N,N})$ achieve the equality in L1 and U1 trace bounds respectively. 

\noindent
As $(A^{*},B^{\#})$ with $A^{*}\in\partial\mathcal{G}^L_{\theta_A}$ satisfies  both trace bounds L1 (cf.\eqref{eig}) and U1 (cf.\eqref{FL2}), i.e. 
\begin{align}
&tr\ \{(B^{\#}-b_1I)(A^{*}-a_1I)^{-2}\} = \alpha_1 \ (\mbox{say})\ \geq \ tr\ \{(B^{\#,1}_{N,N}-b_1I)(A^{*}_N-a_1I)^{-2}\} =\widetilde{\alpha}_1 \ (\mbox{say}) \label{FL3}\\
&tr\ \{(b_2I-B^{\#})(A^{*}-a_1I)^{-2}\} = \alpha_2\ (\mbox{say})\ \geq \  tr\ \{(b_2I-B^{\#,2}_{N,N})(A^{*}_N-a_1I)^{-2}\}=\widetilde{\alpha}_2\ (\mbox{say}), \label{bs15}
\end{align} 
we can find some scalars $\beta_1\geq 0$ and $\beta_2\geq 0$ such that :
\begin{align}
tr\ \{((B^{\#}-\beta_1I)-b_1I)(A^{*}-a_1I)^{-2}\} &= tr\   \{(B^{\#,1}_{N,N}-b_1I)(A^{*}_N-a_1I)^{-2}\} \label{pol3}\\
tr\ \{(b_2I-(B^{\#}+\beta_2I))(A^{*}-a_1I)^{-2}\} & = tr\ \{(b_2I-B^{\#,2}_{N,N})(A^{*}_N-a_1I)^{-2}\}. \label{pol4}
\end{align}
In fact, they are given by 
\begin{equation}
 \beta_1 = \frac{\alpha_1-\widetilde{\alpha}_1}{tr\ (A^{*}-a_1I)^{-2}}\ \mbox{ and } \ \beta_2 = \frac{\alpha_2-\widetilde{\alpha}_2}{tr\ (A^{*}-a_1I)^{-2}}.
\end{equation}
Let us define
\begin{equation}\label{ED6}
B^\epsilon_{N,N}\stackrel{def}{=} 
\frac{1}{(\beta_1+\beta_2)}(\beta_2B^{\epsilon,1}_{N,N}+ \beta_1B^{\epsilon,2}_{N,N}) 
\end{equation}
then using Lemma \ref{pol5}, it is easily checked that 
\begin{equation}\label{ED7}
B^{\epsilon}_{N,N} \xrightarrow{A^\epsilon_N} \frac{1}{(\beta_1+\beta_2)}(\beta_2B^{\#,1}_{N,N}+ \beta_1B^{\#,2}_{N,N}) \stackrel{def}{=} B^{\#}_{N,N} \mbox{ in }\Omega.
\end{equation}
Simple computation using \eqref{pol3},\eqref{pol4} and \eqref{ED7} also shows that
\begin{align}
tr\ \{(B^{\#}-b_1I)(A^{*}-a_1I)^{-2}\} & = tr\   \{(B^{\#}_{N,N}-b_1I)(A^{*}_N-a_1I)^{-2}\}\label{ED8}\\
tr\ \{(b_2I-B^{\#})(A^{*}-a_1I)^{-2}\}  &= tr\ \{(b_2I-B^{\#}_{N,N})(A^{*}_N-a_1I)^{-2}\}.\label{ED9}
\end{align}
Now from \eqref{pol3} by using the similar arguments presented in case$(1)$ we get :
\begin{equation}\label{eg1}
\mbox{Spectrum of $(B^{\#}-\beta_1I)=$ Spectrum of $B^{\#,1}_{N,N}$,}
\end{equation}
and similarly, from \eqref{pol4} by using the similar arguments presented in case$(3)$ we get :
\begin{equation}\label{eg2}
\mbox{Spectrum of $(\beta_2I +B^{\#})=$ Spectrum of $B^{\#,2}_{N,N}$}.
\end{equation}
We claim now that 
\begin{equation}\label{ED10}
\mbox{Spectrum of $B^{\#}=$ Spectrum of $B^{\#}_{N,N}$. }
\end{equation}
Proof of the claim \eqref{ED10} : Let us assume that $\mu_k$ is an eigenvalue of $B^{\#}$ with eigenvector $u_k$, i.e. $B^{\#}u_k =\mu_k u_k$ with $\mu_1\geq\ldots\geq\mu_N$. Then from \eqref{eg1}
shows that the spectrum of $(B^{\#}-\beta_1 I)$ is $\{\mu_i-\beta_1\}$ in decreasing order. Since $(B^{\#}-\beta_1 I)$ is a diagonal matrix , it follows that $(B^{\#}-\beta_1 I) = diag\ \{\mu_i-\beta_1\}$. Similarly \eqref{eg2} shows that, $(\beta_2 I+B^{\#}) = diag\ \{\beta_2+\mu_i\}$.
Now from the definitaion \eqref{ED7} of $B^{\#}_{N,N}$, it follows that $B^{\#}_{N,N} = diag\ \{\mu_i\}$. In particular, the claim \eqref{ED10} follows as a consequence. \\
Thanks to \eqref{ED10}, we are in position to apply the converse part of the  Lemma \ref{zz14} in \eqref{ED8} or \eqref{ED9}, we conclude that $A^{*}$, $B^{\#}$ commute. So there exists an orthogonal matrix $S$ (say) such that $SA^{*}_NS^{-1} = A^{*}$ and $SB^{\#}_{N,N}S^{-1}=B^{\#}$ in $\Omega$ and consequently, $A^\epsilon \stackrel{def}{=}SA^\epsilon_NS^{-1}\xrightarrow{H} A^{*}$ and $B^\epsilon 
\stackrel{def}{=}SB^\epsilon_{N,N}S^{-1}\xrightarrow{A^\epsilon} B^{\#}$ in $\Omega$.
\begin{remark}
 It is important that there should be one common sequence $A^\epsilon$ in the proof of optimality of the region (L1,U1). That is why, we have to go through these arguments.
\end{remark}
\noindent
\textit{Case (5) :}
The optimality of the other regions $(Li,Uj)$ can be shown analogously.
\paragraph{Second Part : Reduction to Macroscopically homogeneous case i.e. $(A^{*},B^{\#})$ and $(\theta_A,\theta_B)$ are constants :}
Here we will show how to reduce the general case of variables $A^{*}(x),B^{\#}(x)$ and $\theta_A(x),\theta_B(x)$  to the special case where 
they are constants. 
It is based on piecewise constant approximation of functions. Such a procedure was adapted in the context of $H$-limits $A^{*}$ (see \cite[Theorem 2.1.2]{A}). With the same inspiration, we want to extend it to include the relative limits $B^{\#}$ also. To this end, we make use of fibre-wise convexity and Lemma \ref{zz15}, Remarks \ref{ub13}, \ref{sii}, \ref{qw1}.  

Let us begin by recalling few elements of the reduction for $A^{*}$.
Let $A^{*}(x)\in \mathcal{G}_{\theta_A(x)}$,  $x\in\Omega$ a.e. be given.  
Let $\{\omega^n_k\}_{1\leq k\leq n}$ be an arbitrary family of disjoint open subsets covering $\Omega$ upto null sets, such that the maximal diameter of the collection goes to $0$ as $n\rightarrow\infty$. 
Let us define a piecewise constant function $\theta^n_A\in L^\infty(\Omega)$ by 
\begin{equation*} \theta^n_A(x) = \sum_{j=1}^n\theta_{A,k}^n\chi_{\omega_k^n}(x) \mbox{ with } \theta_{A,k}^n = \frac{1}{|\omega_k^n|}\int_{\omega_k^n}\theta_A(x) dx. 
\end{equation*}
Then $\theta_A^n(x)\rightarrow \theta_A(x)$ strongly in $L^p(\Omega)$ for $1\leq p<\infty$. We recall the definition of the closed convex set $\mathcal{G}_{\theta_{A,k}^n}$ for $\theta_{A,k}^n$ being a constant :
\begin{align*}
  \mathcal{G}_{\theta_{A,k}^n} = \{ A^{*}\ | \ A^{*} \mbox{ is a constant matrix and it satisfies bounds \eqref{FL11}  } \ &\\
  \mbox{ with constant in proportion }\theta_{A,k}^n\}.&
 \end{align*}
By what is proved in the First Part above, $\mathcal{G}_{\theta_{A,k}^n}$  can also be characterized as 
 \begin{align*}
 \mathcal{G}_{\theta_{A,k}^n} = \{ A^{*}\ | \  A^{*} \mbox{ is a constant matrix and it is a $H$-limit in }\omega_k^n \ & \\
 \mbox{ with $(a_1,a_2)$ constant proportions } (\theta_{A,k}^n,1-\theta_{A,k}^n)\}. &
\end{align*} 
Now, in each open set $\omega_k^n$, we set 
\begin{equation*}
\widetilde{A}^n(x)  \stackrel{def}{=} \mbox{ projection of }A^{*}(x) \mbox{ onto } \mathcal{G}_{\theta_{A,k}^n}.
\end{equation*}
Then for a.e. $x\in\omega_k^n$, following \eqref{qw9} we have 
\begin{equation*}
 |A^{*}(x)-\widetilde{A}^n(x)| \leq C|\theta_A(x)-\theta_{A,k}^n|^{\delta_A},
\end{equation*}
for some $C>0,\delta_A>0$
and therefore the sequence $\widetilde{A}^n$ converges strongly to $A^{*}$ in $L^p(\Omega)$ for any $1\leq p<\infty$. Each matrix  $\widetilde{A}^n$  is not yet piecewise constant. Therefore, we define a sequence of piecewise constant matrices
\begin{equation*}
 \widehat{A}^n(x) = \sum_{k=1}^n\widehat{A}^n_k\chi_{\omega_k^n}(x) \mbox{ with } \widehat{A}^n_k = \frac{1}{|\omega_k^n|}\int_{\omega_k^n}\widetilde{A}^n(x) dx .
\end{equation*}
It follows that the sequence $\widehat{A}^n$ also converges strongly to $A^{*}$ in $L^p(\Omega)$. 
Unfortunately, there is no guarantee that
each $\widehat{A}^n$ belongs to $\mathcal{G}_{\theta_{A,k}^n}$. 
Therefore, we define a constant matrix $A^n_k$ as
\begin{equation*}
A^n_k  \stackrel{def}{=} \mbox{ projection of }\widehat{A}^n_k \mbox{ onto } \mathcal{G}_{\theta_{A,k}^n}.
\end{equation*}
This yields a piecewise constant matrix
\begin{equation*}
 A^n(x) =\sum_{k=1}^nA^n_k\chi_{\omega_k^n}(x).
\end{equation*}
Then $A^n(x)\in \mathcal{G}_{\theta_{A,k}^n}$, $x\in\omega_k^n$
and $A^n$ converges to $A^{*}$ strongly in $L^p(\Omega)$ for $1\leq p<\infty$ and so $A^n$ $H$-converges to $A^{*}$ (see Remark \ref{sii}).

Thanks to Lemma \ref{zz15} and the fibre-wise convexity structure (cf. Remark \ref{qw1}), we can follow the above arguments to obtain the required approximation of $B^{\#}(x)$ too. Indeed, let $(A^{*}(x),B^{\#}(x))\in\mathcal{G}_{(\theta_A(x),\theta_B(x))}$ with $A^{*}(x)\in\mathcal{G}_{\theta_A(x)}$, $x\in\Omega$ a.e. be given. Working with the same partition $\{\omega_k^n\}_k$ and proceeding as before, we define a piecewise constant function $\theta^n_B\in L^\infty(\Omega)$ by 
\begin{equation*} 
\theta^n_B(x) = \sum_{j=1}^n\theta_{B,k}^n\chi_{\omega_k^n}(x) \mbox{ with } \theta_{B,k}^n = \frac{1}{|\omega_k^n|}\int_{\omega_k^n}\theta_B(x) dx
\end{equation*}
We recall that the closed convex set $\mathcal{K}^f_{(\theta_{A,k}^n,\theta_{B,k}^n)}(A^{*})$ for $\theta_{A,k}^n,\theta_{B,k}^n$ being constants, can be characterized as shown in the First Part above, i.e. 
\begin{align*}
 \mathcal{K}^f_{(\theta_{A,k}^n,\theta_{B,k}^n)}(A^{*}) = \{ & B^{\#};\ \ B^{\#} \mbox{ is a constant matrix and it is a relative limit in }\omega_k^n \\
& \mbox{associated with }A^{*}\in \mathcal{G}_{\theta_{A,k}^n}, 
 \mbox{ with constant proportions } (\theta_{A,k}^n,\theta_{B,k}^n)\}. 
\end{align*}
Now, in each open set $\omega_k^n$, we define 
\begin{equation*}
 \widetilde{B}^n(x) \stackrel{def}{=}\mbox{ projection of }B^{\#}(x) \mbox{ onto } \mathcal{K}^f_{(\theta_{A,k}^n,\theta_{B,k}^n)}(A^{*}(x)).
\end{equation*}
Note that because the fibre over $A^{*}$ is convex and closed, this is well defined. \\
\\
Then for a.e. $x\in\omega_k^n$, using \eqref{qw8} we have 
\begin{equation*}
 |B^{\#}(x)-\widetilde{B}^n(x)| \leq C\lb|\theta_A(x)-\theta_{A,k}^n|^{\delta_A} + |\theta_B(x)-\theta_{B,k}^n|^{\delta_B}\rb,
\end{equation*}
for some $C>0$, $\delta_A,\delta_B>0$ and therefore the sequence $\widetilde{B}^n$ converges strongly to $B^{\#}$ in $L^p(\Omega)$ for any $1\leq p<\infty$. Each matrix  $\widetilde{B}^n$  is not yet piecewise constant. Therefore, we define a sequence of piecewise constant matrices
\begin{equation*}
 \widehat{B}^n(x) = \sum_{k=1}^n\widehat{B}^n_k\chi_{\omega_k^n}(x) \mbox{ with } \widehat{B}^n_k = \frac{1}{|\omega_k^n|}\int_{\omega_k^n}\widetilde{B}^n(x) dx .
\end{equation*}
It follows that the sequence $\widehat{B}^n$ also converges strongly to $B^{\#}$ in $L^p(\Omega)$. 
Next, we define a constant matrix $B_j^n$ as
\begin{equation*}
 B^n_k \stackrel{def}{=}\mbox{ projection of }\widetilde{B}^n_k \mbox{ onto } \mathcal{K}^f_{(\theta_{A,k}^n,\theta_{B,k}^n)}(A^{*}).
\end{equation*}
This yields a piecewise constant matrix
\begin{equation*}
 B^n(x) =\sum_{k=1}^nB^n_k\chi_{\omega_k^n}(x), 
\end{equation*}
which is a relative-limit, i.e., belongs to $\mathcal{K}^f_{(\theta_{A,k}^n,\theta_{B,k}^n)}(A^{*})$ for fixed $A^{*}\in \mathcal{G}_{\theta_{A,k}^n}$. Let us prove that the
sequence $B^n$ converges to $B^{\#}$ strongly in $L^p(\Omega)$. 
Then by Remark \ref{sii} $B^n$ will converge to $B^{\#}$ relative to $A^n$.  
By construction, the projection $B^n_k$ satisfies
\begin{equation}
 |B^n(x)-\widehat{B}^n(x)| \leq |\widetilde{B}^n(x)-\widehat{B}^n(x)|.
\end{equation}
Therefore,
\begin{align*}
 |B^n(x)-B^{\#}(x)| \leq 2|\widehat{B}^n(x)-\widetilde{B}^n(x)| +|\widetilde{B}^n(x)-B^{\#}(x)|.
\end{align*}
We know that both of these terms in the right hand side converges strongly to $0$ in $L^p$, so we deduce $B^n$ converges
strongly to $B^{\#}$ in $L^p(\Omega)$ for any $1\leq p<\infty$. Hence 
$B^n$ converges to $B^{\#}$  relative to the sequence $A^n$ in $\Omega$. \\
\\
In our problem, we have four regions $(Li,Uj)$ $i,j=1,2$, and the corresponding physical domains $\Omega_{(Li,Uj)}$ $i,j=1,2$ are defined as follows :
\begin{equation}\begin{aligned}\label{FL1}
\Omega_{(L1,U1)} &= \{ x\in\Omega : \theta_A(x)\leq \theta_B(x),  \theta_A(x)+\theta_B(x)\leq 1\},\\
\Omega_{(L1,U2)} &= \{ x\in\Omega : \theta_A(x)\leq \theta_B(x), \theta_A(x)+\theta_B(x)> 1\},\\
\Omega_{(L2,U1)} &= \{ x\in\Omega : \theta_B(x)< \theta_A(x), \theta_A(x)+\theta_B(x)\leq 1\},\\
\Omega_{(L2,U2)} &= \{ x\in\Omega : \theta_B(x)< \theta_A(x), \theta_A(x)+\theta_B(x)> 1\}.
\end{aligned}\end{equation}
These domains provide a measurable disjoint cover for $\Omega$ : $\Omega = \underset{i,j=1,2}{\cup}\Omega_{(Li,Uj)}$.  \\
\\
In the above arguments we may replace the open set $\Omega$ by $int(\Omega_{(Li,Uj)})$ and deal with its open covering given by  $\{int(\Omega_{(Li,Uj)})\cap \omega^n_k\}_k$. Combination of the optimality result of first part and the approximation result of the second part yields the optimality of the region $(Li,Uj)$ in the sub-domain $int(\Omega_{(Li,Uj)})$. Since these step is analogous to the treatment of the classical case $A^{*}$, we will not give details. 

\paragraph{Third part : Case (a) : }  
If the interior of $\Omega_{(Li,Uj)}$, $i,j=1,2$ cover $\Omega$ upto a null-set, we can prove the optimality of all the four regions as follows using the optimality of individual four regions established in First and Second parts.
Let us assume that
$$|\Omega\ \smallsetminus\ \underset{i,j=1,2}{\cup}\ int\hspace{2pt}(\Omega_{(Li,Uj)})|=0.$$
Then according to First and Second parts, given $(A^{*}(x),B^{\#}(x))$, $x\in \Omega$, by restricting it on $int(\Omega_{(Li,Uj)})$, there exist $ (A^\epsilon_{ij}, B^\epsilon_{ij})$ with the required property such that 
\begin{align*}
 A^\epsilon_{ij} \rightarrow A^{*}_{ij} \mbox{ in }L^p(int\hspace{2pt}(\Omega_{(Li,Uj)}),
\quad  B^\epsilon_{ij} \rightarrow B^{\#}_{ij} \mbox{ in }L^p(int\hspace{2pt}(\Omega_{(Li,Uj)}), \ \ i,j =1,2, \ \ 1\leq p<\infty.
\end{align*}
Putting them together, we define ($A^\epsilon,B^\epsilon$) on $\Omega$ by
\begin{align*}
 A^\epsilon = A^\epsilon_{ij} \mbox{ on } int\hspace{2pt}(\Omega_{(Li,Uj)}),\quad B^\epsilon = B^\epsilon_{ij} \mbox{ on } int\hspace{2pt}(\Omega_{(Li,Uj)}).
\end{align*}
Then we have 
\begin{align*}
 A^\epsilon \rightarrow A^{*} \mbox{ in }L^p(\Omega),
\quad  B^\epsilon \rightarrow B^{\#} \mbox{ in }L^p(\Omega), \ \ i,j =1,2. 
\end{align*}
Consequently,
\begin{align*}
 A^\epsilon \xrightarrow{H} A^{*} \mbox{ in }\Omega,
\quad  B^\epsilon \xrightarrow{A^\epsilon} B^{\#} \mbox{ in }\Omega, \ \ i,j =1,2. 
\end{align*}
\textbf{Case (b) : } If
\begin{center}$|\Omega \smallsetminus \underset{i,j=1,2}{\cup}\ int\hspace{2pt}(\Omega_{(Li,Uj)})|> 0 $.
\end{center}
Covering arguments given in previous case fail. \\
\\
However,the following covering arguments can be advanced. Indeed using 
Vitali covering (cf. \cite[Theorem 17.1]{DiB}), we have a countable collection of open cubes  $\{\omega_{ij}^k\}_{k=1}^{\infty}$ with finite diameter and with pairwise disjoint interiors,  such that 
\begin{equation*}
|\Omega_{(Li,Uj)}\ \smallsetminus\ \underset{k}{\cup}\ \omega_{ij}^k | =0 \ \mbox{ for each }i,j=1,2. 
\end{equation*}
So, we have 
\begin{equation*}
|\Omega\ \smallsetminus\ \underset{i,j=1,2}{\cup}\underset{k}{\cup}\lb \Omega \cap \omega_{ij}^k\rb |=0.
\end{equation*}
Then again following the First and Second parts, given $(A^{*}(x),B^{\#}(x))$, $x\in \Omega$, by restricting it on $\Omega\cap\omega_{ij}^k$, there exist $((A^\epsilon_{ij})_k, (B^\epsilon_{ij})_k)$ with required properties such that  
\begin{align*}
 (A^\epsilon_{ij})_k \rightarrow A^{*} \mbox{ in }L^p(\Omega\cap \omega_{ij}^k),
 \mbox{ and }  (B^\epsilon_{ij})_k \rightarrow B^{\#} \mbox{ in }L^p(\Omega\cap \omega_{ij}^k),  \ i,j =1,2,\ k=1,2,.. 
\end{align*}
Putting them together, we define $(A^\epsilon,B^\epsilon)$ on $\Omega$ via 
\begin{align*}
 A^\epsilon = (A^\epsilon_{ij})_k \mbox{ on } \Omega_{(Li,Uj)}\cap\omega_{ij}^k \mbox{ a.e. }\quad B^\epsilon = (B^\epsilon_{ij})_k \mbox{ on } \Omega_{(Li,Uj)}\cap\omega_{ij}^k \mbox{ a.e. }.
\end{align*}
Then we have 
\begin{align*}
 A^\epsilon \rightarrow A^{*} \mbox{ in }L^p(\Omega),
\quad  B^\epsilon \rightarrow B^{\#} \mbox{ in }L^p(\Omega), \ \ i,j =1,2. 
\end{align*}
Consequently,
\begin{align*}
 A^\epsilon \xrightarrow{H} A^{*} \mbox{ in }\Omega,
\quad  B^\epsilon \xrightarrow{A^\epsilon} B^{\#} \mbox{ in }\Omega, \ \ i,j =1,2. 
\end{align*}
\hfill\end{proof}
\begin{remark}\label{eg3}
Before closing this section, we wish to highlight one important result which was intermediary. 
In the First Part of the proof, we have shown that for any given $(A^{*},B^{\#})\in \mathcal{K}_{(\theta_A,\theta_B)}$ i.e. being in one of the regions $(Li,Uj)$ $i,j=1,2$, i.e. $B^{\#}$ lying on the fibre over $A^{*}$, we have the commutation relation :  
\begin{equation}\label{eg4}
A^{*}B^{\#} = B^{\#}A^{*}.
\end{equation} 
The above result is extended later to
\begin{equation}
 A^{*}(x)B^{\#}(x) = B^{\#}(x)A^{*}(x), \ \ x \mbox{ a.e. in }\Omega.  
\end{equation}
by the arguments in Second and Third Parts. This is somewhat a surprising property.
\hfill\qed\end{remark}
\section{Applications in Calculus of Variations}\label{qw5}
\setcounter{equation}{0}
In this section we mention two applications of our earlier results. In particular we make use of the optimality of the bounds. 
\subsection{Application to an Optimal Design Problem}
Here we mention one application in optimal design problem (ODP). Let us consider the following model problem in ODP already treated in literature \cite{GRB}. 
Our contribution is merely to point out the convergence of integrands (and hence integrals) and the advantage of using $I^{\#}$  in the relaxation process.
We introduce some notations. Let us denote by Char($\Omega$) the set of all characteristic functions of measurable subsets of $\Omega$, i.e.
\begin{equation*} \mbox{Char}(\Omega) = \{\chi :\Omega \mapsto \{0,1\} \mbox{ measurable}\}.  \end{equation*}
For a given $\delta_A\in (0,1)$, let us consider the set $\mathcal{C}_{\delta_A}$ of classical microstructures defined by 
\begin{equation}\label{hsu} \mathcal{C}_{\delta_A} = \{\chi \in \mbox{Char}(\Omega) : \frac{1}{|\Omega|}\int_{\Omega} \chi(x) dx = {\delta_A}\}\end{equation}
and for any $\chi_{\omega_A}\in $Char$(\Omega$), we define the functional $J(\chi_{\omega_A})$ as follows 
\begin{equation}\begin{aligned}\label{tz}
 J :\mbox{Char}(\Omega) &\mapsto \mathbb{R}\\ 
 \chi_{\omega_A} &\mapsto J(\chi_{\omega_A}) := \int_{\Omega} \nabla u_{\omega_A} \cdot\nabla u_{\omega_A}\ dx,
\end{aligned}\end{equation}
where $u_{\omega_A}\in H^1_0(\Omega)$ is the solution the following state equation with a given $g\in H^{-1}(\Omega)$ :
\begin{equation}\label{tx}
\begin{aligned}       -div(A\nabla u_{\omega_A})&= g\quad\mbox{in } \Omega\\
                              u_{\omega_A}&= 0 \quad\mbox{on }\partial\Omega.
\end{aligned}
\end{equation}
with $A\in \mathcal{M}(a_1,a_2;\Omega)$ is governed with two-phase medium as 
\begin{equation*} A(x) = \{a_1\chi_{\omega_A}(x) +a_2(1-\chi_{\omega_A}(x))\}I,\ \ x\in\Omega. \end{equation*}
We are interested in finding the infimum value and minimizers of
\begin{equation}\label{eiv}
m=\underset{\chi_{\omega_A}\in \mathcal{C}_{\delta_A}}{\mbox{inf }}J(\chi_{\omega_A}). \end{equation}
To this end, we do relaxation. For any ${\delta_A}\in (0,1)$, let us consider the set $\mathcal{D}_{\delta_A}$ of generalized microstructures
defined by 
\begin{equation}\label{hsv} \mathcal{D}_{\delta_A} = \{\theta \in L^\infty(\Omega;[0,1]): \frac{1}{|\Omega|}\int_\Omega \theta(x) dx  = {\delta_A} \}.\end{equation}
To find the relaxed version of \eqref{eiv}, we consider a sequence
$A^\epsilon\in\mathcal{M}(a_1,a_2;\Omega)$ as in \eqref{ta} with $\A\in \mathcal{C}_{\delta_A}$ and the corresponding state $u^\epsilon\in H^1_0(\Omega)$ solving \eqref{tx}. The functional becomes
\begin{equation}\label{ot6} J(\chi_{\omega_{A^\epsilon}}) = \int_{\Omega}\nabla u^{\epsilon}\cdot\nabla u^{\epsilon}\ dx; \end{equation}
The microstructures $\A$ constitute the following admissible set in which $\delta_A$ is given :  
\begin{align}\label{ED5}
\mathcal{S} = &\mbox{ Collection of all microstructures } \A \mbox{ satisfying the following conditions : }\notag\\
 &\frac{|\omega_{A^\epsilon}|}{|\Omega|}=\delta_A\ \forall \epsilon,\ \A \rightharpoonup \theta_A(x)\mbox{ in $L^\infty(\Omega)$ weak* with }\frac{1}{|\Omega|}\int_\Omega \theta_A(x)dx =\delta_A, \notag\\
 &\mbox{ and } A^\epsilon(x)\xrightarrow{H} A^{*}(x).
\end{align}
Then $A^{*}\in\mathcal{G}_{\theta_A}$ and \eqref{eiv} becomes 
\begin{equation}\label{ot7} m = \underset{\chi_{\omega_{A^\epsilon}}\in \mathcal{S}}{inf}\lb\underset{\epsilon\rightarrow 0}{lim\hspace{2pt}inf}\ \int_\Omega \nabla u^\epsilon\cdot\nabla u^\epsilon\ dx\rb.\end{equation}
Our main claim in this paragraph is that 
\begin{equation}\label{ot8} m = \underset{\theta_A\in \mathcal{D}_{\delta_A}}{min}\lb\underset{A^{*}\in\mathcal{G}_{\theta_A}}{min}\ \int_\Omega \frac{|(a_2I-A^{*}(x))\nabla u_{\theta_A}(x)|^2}{a_2(a_2-a_1)\theta_A(x)} dx+\frac{1}{a_2}\langle g,u_{\theta_A}\rangle \rb,\end{equation}
where, $u_{\theta_A}\in H^1_0(\Omega)$ is the solution of the following homogenized equation with $A^{*}\in\mathcal{G}_{\theta_A}$ :
\begin{equation}\label{ot19}\begin{aligned}
-div(A^{*}\nabla u_{\theta_A}) &= g \mbox{ in }\Omega\\
u_{\theta_A} &=0 \mbox{ on } \Omega .
\end{aligned}
\end{equation}
In order to show the above relaxation identity \eqref{ot8}, we pass to the limit $\epsilon\rightarrow 0$ in \eqref{ot6} and we have
\begin{equation}
J(\theta_A,A^{*},I^{\#}):=\  \underset{\epsilon\rightarrow 0}{lim\hspace{2pt}inf}\  J(\chi_{\omega_{A^\epsilon}}) = \int_{\Omega}I^{\#}\nabla u_{\theta_A}\cdot\nabla u_{\theta_A}\ dx \end{equation}
where $ u^{\epsilon} \rightharpoonup u_{\theta_A} \mbox{ weakly in }H^{1}_{0}(\Omega),\ A^{\epsilon}\nabla u^{\epsilon} \rightharpoonup A^{*}\nabla u_{\theta_A} \mbox{ weakly in }(L^{2}(\Omega))^N$, with $A^{*}\in\mathcal{G}_{\theta_A}$
and the pair $(A^{*},I^{\#})\in \mathcal{G}_{(\theta_A,1)}$, the set defined by the lower trace bound L (cf.\eqref{tw}) and upper trace bound U (cf.\eqref{tq}) with $b=1$. (See \eqref{qw2} for the characterization of $\mathcal{G}_{(\theta_A,1)}$). The optimality of the region defined by the above bounds gives that the relaxation of \eqref{eiv} :
\begin{equation}\label{lb13} 
m =\ \mbox{min }\{ J(\theta_A, A^{*},I^{\#});\ \theta_A\in\mathcal{D}_{\delta_A},\ (A^{*},I^{\#})\in\mathcal{G}_{(\theta_A,1)}\}.
\end{equation}
Now using the lower bound \eqref{bs1} in Remark \ref{hsg}, we have the pointwise lower bound on the integrand $I^{\#}\nabla u_{\theta_A}\cdot\nabla u_{\theta_A}$ :
\begin{equation}\label{os1}
I^{\#}\nabla u_{\theta_A}\cdot\nabla u_{\theta_A} \geq\ {a_2}^{-1}\{A^{*} + (a_2I-\overline{A})^{-1}(a_2I-A^{*})^2\}\nabla u_{\theta_A}\cdot \nabla u_{\theta_A}.
\end{equation}
Therefore, 
\begin{equation}\label{eg8}
m \geq \underset{\theta_A\in \mathcal{D}_{\delta_A}}{min}\lb\underset{A^{*}\in\mathcal{G}_{\theta_A}}{min}\ \int_\Omega h(A^{*}) \nabla u_{\theta_A}\cdot \nabla u_{\theta_A}dx\rb
\end{equation}
where $h$ is the map
$$ A^{*}\mapsto h(A^{*}):= {a_2}^{-1}\{A^{*} + (a_2I-\overline{A})^{-1}(a_2I-A^{*})^2\}.$$
Now we vary $A^{*}\in\mathcal{G}_{\theta_A}$ along the horizontal segment connecting  $A^{*}\in\mathcal{G}_{\theta_A}$ to the $N$-rank laminate  $A^{*}_N\in\partial\mathcal{G}^{U}_{\theta_A}$ in the phase space of eigenvalues of  $A^{*}$. Note that, the eigenvalues $\{\lambda_2,..,\lambda_N\}$ of matrices of $A^{*}$ do not change whereas $\lambda_1$ 
increases along the segment. Let us restrict $h$ to the above segment. Such a restriction is a function of a single variable $t$ along the segment. It is easy to check that
$$\frac{d h}{d t}|_{t=\lambda_1} = \frac{(2t-a_2-\overline{a})}{a_2(a_2-\overline{a})}|_{t=\lambda_1} = \frac{(2\lambda_1-a_2-\overline{a})}{a_2(a_2-\overline{a})} <0 $$
as $\lambda_1\leq\overline{a}=(a_1\theta_A+(1-\theta_A)a_2) \mbox{ and }\overline{a}<a_2$.\\
\\
Thus $h$ is decreasing as $\lambda_1$ increases. Hence we get $h(A^{*})\geq h(A^{*}_N)$.\\
\\
Now from the Remark \ref{hsg}, for each $A^{*}_N\in\partial\mathcal{G}^{U}_{\theta_A}$ 
given by \eqref{OP3} with matrix $a_2I$ and core $a_1I$, the corresponding $N$- sequential laminates $I^{\#}_N$ given in \eqref{bs13} achieve the equality in the minimization problem \eqref{eg8}.
Thus from \eqref{lb13}, it follows that
$$m\  \ = \int_{\Omega} I^{\#}_N \nabla u^{*}_{\theta_A}\cdot\nabla u^{*}_{\theta_A}\ dx. $$
Simple computations show that 
$$m = \int_\Omega \frac{|(a_2I-A^{*}_N(x))\nabla u^{*}_{\theta_A}(x)|^2}{a_2(a_2-a_1)\theta_A(x)}dx +\frac{1}{a_2}\langle g(x),u^{*}_{\theta_A}(x)\rangle $$
where, $u^{*}_{\theta_A}$ is the solution of \eqref{ot19} corresponding to the $N$-sequential laminates $A^{*}_N$ with matrix $a_2I$ and core $a_1I$. 
In particular, \eqref{ot8} follows, describing its minimum value as well as the minimizers and the underlying microstructures.
\hfill\qed

\subsection{Application to an Optimal Oscillation-Dissipation Problem}\label{qw10}
Here we mention another application to a minimization problem, where the where the results of Section \ref{ad18} are useful. Let us define the functional $J(\chi_{\omega_A},\chi_{\omega_B})$ as
\begin{equation}\begin{aligned}\label{ub11}
 J :\mbox{Char}(\Omega)\times \mbox{Char}(\Omega)  &\mapsto \mathbb{R}\\ 
 (\chi_{\omega_A},\chi_{\omega_B}) &\mapsto J(\chi_{\omega_A},\chi_{\omega_B}) := \int_{\Omega} B\nabla u_{\omega_A} \cdot\nabla u_{\omega_A} dx,
\end{aligned}\end{equation}
where $u_{\omega_A}$ solves \eqref{tx}, and $B$ is governed with the 
two phase medium as
\begin{equation*} B(x) = \{b_1\chi_{\omega_B} +b_2(1-\chi_{\omega_B})\}I, \ x\in\Omega \end{equation*}
with $\chi_{\omega_B}(x)\in \mathcal{C}_{\delta_B}$ for some given ${\delta_B}\in (0,1)$, as introduced in \eqref{hsu}.\\
\\
We are interested in solving the following non-convex minimization problem : 
\begin{equation}\label{ub20} m^\prime =\underset{(\chi_{\omega_A},\chi_{\omega_B})\in (\mathcal{C}_{{\delta_A}}\times \mathcal{C}_{{\delta_B}})}{\mbox{inf }}J(\chi_{\omega_A},\chi_{\omega_B}). \end{equation}
In classical problems in Calculus of Variations in which the homogenization theory is applied, we usually have 
minimization with respect to $\chi_{\omega_A}$ but not with respect to $\chi_{\omega_B}$. Minimization with respect
to $\chi_{\omega_B}$ is a new aspect. Interpretation of such problems was given in Introduction. As usual, there are no minimizers among characteristic
functions and there is a need to relax the problem. 

To find the relaxation of \eqref{ub20}, we consider sequence  $A^\epsilon\in\mathcal{M}(a_1,a_2;\Omega)$ given in \eqref{ta} with $\A\in \mathcal{C}_{\delta_A}$, 
and the corresponding state sequence $u^\epsilon\in H^1_0(\Omega)$ which solves \eqref{tx}.
Let us also consider a sequence $B^\epsilon\in\mathcal{M}(b_1,b_2;\Omega)$ given in \eqref{tb} with $\B\in\mathcal{C}_{\delta_B}$
and the functional becomes  
\begin{equation}\label{Sd16} J(\chi_{\omega_{A^\epsilon}},\chi_{\omega_{B^\epsilon}}) = \int_{\Omega}B^\epsilon\nabla u^{\epsilon}\cdot\nabla u^{\epsilon}\ dx. \end{equation}
The pair of microstructures $(\A,\B)$ constitute the following admissible set in which $\delta_A$,$\delta_B$ are given :
\begin{align}\label{ub17}
\mathcal{S} =  &\mbox{ Collection of all pair of microstructures } (\A,\B) \mbox{ having the properties : }\notag\\
&(\A,\B) \ |\ \lb\frac{|\omega_{A^\epsilon}|}{|\Omega|},\frac{|\omega_{A^\epsilon}|}{|\Omega|}\rb=({\delta_A},{\delta_B})\ \forall\epsilon, (\A,\B) \rightharpoonup (\theta_A(x),\theta_B(x))\notag\\ 
&\mbox{in $L^\infty(\Omega)$ weak*, with }\lb\frac{1}{|\Omega|}\int_\Omega \theta_A(x)dx,\frac{1}{|\Omega|}\int_\Omega \theta_B(x)dx\rb =({\delta_A},{\delta_B}),\notag\\
&\mbox{ and }\ A^\epsilon(x) \xrightarrow{H}A^{*}(x),\ B^\epsilon(x)\xrightarrow{A^\epsilon}B^{\#}(x).
\end{align}
Then $A^{*}\in\mathcal{G}_{\theta_A}$ and $(A^{*},B^{\#})\in\mathcal{G}_{(\theta_A,\theta_B)}$ a set which has been introduced in Section \ref{ad18}. Note that we need here the convergence of $B^\epsilon$ towards $B^{\#}$ relative to $A^\epsilon$.\\
Let $(\A,\B)\in\mathcal{S}$ and let $u^\epsilon\in H^1_0(\Omega)$ solve \eqref{tx}. Then \eqref{ub20} becomes :
\begin{equation}\label{Sd17} m^\prime = \underset{(\chi_{\omega_{A^\epsilon}},\chi_{\omega_{B^\epsilon}})\in \mathcal{S}}{inf}\lb\underset{\epsilon\rightarrow 0}{lim\hspace{2pt}inf}\ J(\chi_{\omega_{A^\epsilon}},\chi_{\omega_{B^\epsilon}})\rb=\underset{(\chi_{\omega_{A^\epsilon}},\chi_{\omega_{B^\epsilon}})\in \mathcal{S}}{inf}\lb\underset{\epsilon\rightarrow 0}{lim\hspace{2pt}inf}\ \int_{\Omega}B^\epsilon\nabla u^{\epsilon}\cdot\nabla u^{\epsilon}\ dx\rb.\end{equation}
By passing to the limit as $\epsilon\rightarrow 0$ in \eqref{Sd17} we obtain 
\begin{equation}\label{ub14}
J(\theta_A,\theta_B,A^{*},B^{\#}):= \underset{\epsilon\rightarrow 0}{lim\hspace{2pt}inf}  J(\chi_{\omega_{A^\epsilon}},\chi_{\omega_{B^\epsilon}}) = \int_{\Omega}B^{\#}\nabla u_{\theta_A}\cdot\nabla u_{\theta_A}\ dx
\end{equation}
where, $u_{\theta_A}\in H^1_0(\Omega)$ is the solution of \eqref{ot19} with $A^{*}\in\mathcal{G}_{\theta_A}$ and as a pair $(A^{*},B^{\#})\in\mathcal{G}_{(\theta_A,\theta_B)}$.\\
\\
The optimality result (cf. Remark \ref{ED4} and $(5)$ of Theorem \ref{qw6}) shows that the relaxation of \eqref{ub20} is given by :
\begin{equation}\label{lb11}
 m^\prime = \mbox{min\ }\{ J(\theta_A,\theta_B,A^{*},B^{\#}) ; \ (\theta_A,\theta_B)\in\mathcal{D}_{\delta_A}\times \mathcal{D}_{\delta_B},\ A^{*}\in\mathcal{G}_{\theta_A},(A^{*},B^{\#})\in\mathcal{G}_{(\theta_A,\theta_B)}\}
\end{equation}
where, $\mathcal{D}_{\delta_A},\mathcal{D}_{\delta_B}$ are defined for given ${\delta_A},{\delta_B}\in(0,1)$, as introduced in \eqref{hsv}.\\
\\
We split the integral $J(\theta_A,\theta_B,A^{*},B^{\#})$ in \eqref{ub14} into two parts as follows : 
\begin{equation}\label{ub16}
 J(\theta_A,\theta_B,A^{*},B^{\#}) = \int_{\Omega_{L1}} B^{\#}\nabla u_{\theta_A}\cdot\nabla u_{\theta_A}\ dx + \int_{\Omega_{L2}}{A^{*}}^{-1}B^{\#}{A^{*}}^{-1}\sigma_{\theta_A}\cdot\sigma_{\theta_A}\ dx 
\end{equation}
where, $\Omega_{L1} = \{x\in\Omega\ |\ \theta_A(x) \leq \theta_B(x)\}$  and $\Omega_{L2} = \{x\in\Omega\ |\ \theta_B(x)< \theta_A(x)\}$ and 
$\sigma_{\theta_A}=A^{*}\nabla u_{\theta_A}$. Let us assume that $|int\ \Omega_{L1}|=|\Omega_{L1}|$ and $|int\ \Omega_{L2}|=|\Omega_{L2}|$.\\
\\
Now using the pointwise bound on the above integrands $B^{\#}\nabla u_{\theta_A}\cdot\nabla u_{\theta_A}$ and ${A^{*}}^{-1}B^{\#}{A^{*}}^{-1}\sigma_{\theta_A}\cdot\sigma_{\theta_A}$ 
obtained in Remark \ref{siz}, \ref{siz2} for $A^{*}\in\mathcal{G}_{\theta_A}$ and the associated pair $(A^{*},B^{\#})\in\mathcal{G}_{(\theta_A,\theta_B)}$, we have :
\begin{equation}\label{tu}
J(\theta_A,\theta_B,A^{*},B^{\#})\geq J_1(\theta_A,\theta_B,A^{*},M_{AB},M_B) + J_2(\theta_A,\theta_B,A^{*},M^{\prime}_{AB},M^{\prime\prime}_{AB})
\end{equation}
where using \eqref{FG19} we define 
\begin{align*}
J_1(\theta_A,\theta_B,A^{*},M_{AB},M_B)= \int_{\Omega_{L1}}& \{ b_1I  + 2 (\overline{B}-b_1I)(\overline{A}_{\theta}-a_1I)^{-1}(A^{*}-a_1I)\notag\\
&+(Y_{\theta}- (\overline{B}-b_1 I))(\overline{A}_{\theta}-a_1I)^{-2}(A^{*}-a_1I)^2\}\nabla u_{\theta_A}\cdot\nabla u_{\theta_A}\ dx
\end{align*}
with $\overline{A}_\theta = (a_1\theta +a_2(1-\theta))I$, 
\begin{align*}Y_\theta &=\{\frac{b_1(a_2-a_1)^2}{a_1^2}\theta(1-\theta) +\frac{(b_2-b_1)^2}{b_1}\theta_B(1-\theta_B)-\frac{2(b_2-b_1)(a_2-a_1)}{a_1}\theta(1-\theta_B)\}M_{AB}\eta\cdot\eta\\  
                            &\qquad- \frac{\theta_B(1-\theta_B)(b_2-b_1)^2}{b_1}M_B\eta\cdot\eta. 
\end{align*}
$M_B,M_{AB}$ are symmetric non-negative definite matrices with unit trace, whose expression can be found in \eqref{siv} and \eqref{dc19} respectively 
replacing $\theta_A$ by $\theta$, where $\theta$ is uniquely determined by \eqref{eib}.\\
\\
Analogously using \eqref{siu} we define 
\begin{align*}
J_2(\theta_A,\theta_B,A^{*},M^{\prime}_{AB},M^{\prime\prime}_{AB})= \int_{\Omega_{L2}}& \{ cI  + 2(L_\theta-cI)(\underline{A}_{\theta}^{-1} - a_2^{-1}I)^{-1}({A^{*}}^{-1}- a_2^{-1}I)\notag\\
&+(Y^\prime_\theta- (L_\theta-cI))(\underline{A}_\theta^{-1} - a_2^{-1}I)^{-2}({A^{*}}^{-1}- a_2^{-1}I)^{2}\}\sigma_{\theta_A}\cdot\sigma_{\theta_A}\ dx
\end{align*}
with 
\begin{align*} 
&\underline{A}_\theta^{-1} =(\frac{\theta}{a_1} +\frac{1-\theta}{a_2})I,\ c = min\{\frac{b_2}{a_2^2},\frac{b_1}{a_1^2}\}, \ L_\theta= \{\frac{b_1}{a_1^2}\theta_B(x) + \frac{b_2}{a_1^2}(\theta(x) - \theta_B(x)) + \frac{b_2}{a_2^2}(1-\theta(x))\}I,\\
&Y^\prime_\theta = cL^\prime_{AB}(I-M^\prime_{AB})\eta\cdot\eta-  \frac{1}{c}L^{\prime\prime}_{AB}(I-M^{\prime\prime}_{AB})\eta\cdot\eta \ \ (\mbox{cf. \eqref{FG12} where $\theta_A$ is replaced by $\theta$.})
\end{align*}
where $M^\prime_{AB},M^{\prime\prime}_{AB}$ are symmetric non-negative definite matrices with unit trace can be found in \eqref{zz9} and \eqref{zz8} respectively replacing $\theta_A$ by $\theta$
and here $\theta$ is uniquely determined by \eqref{eic} and it depends only on $A^{*}$. 

For the $J_1$- term, we need to use the optimality of the equality in L1 bound described in (1) of Theorem \ref{qw6}. In fact, the construction in Example \ref{ot5} shows that the optimal microstructure is given by  $A^{*,1}_N$, the $N$ sequential laminate with core $a_2I$ and matrix $a_1I$ with $\theta_A\in \mathcal{D}_{\delta_A}$,
and the corresponding $(N,N)$-sequential laminates $B^{\#,1}_{N,N}$ (cf.\eqref{sie}) associated with $A^{*,1}_N$ under the condition $\omega_{A^\epsilon}\subset\omega_{B^\epsilon}$.
This construction can be carried out in $int(\Omega_{L1})$ because $\theta_A\leq\theta_B$. We have also computed the associated matrices $M_B,M_{AB}$ as $M_B=M_{AB}=\sum_{i=1}^N m_i\frac{e_i\otimes e_i}{e_i\cdot e_i}$
where, $m_i\geq 0$ for $i=1,..,N$ and $\sum_i m_i =1$. And further it is found that 
$$
J_1 = \int_{\Omega_{L1}}B^{\#,1}_{N,N}\nabla \widetilde{u_{\theta_A}}\cdot\nabla \widetilde{u_{\theta_A}}\ dx.
$$
Similarly, for the $J_2$- term, we need to use the optimality of the equality in L2 bound described in (2) of Theorem \ref{qw6}. In fact, the construction in Example \ref{Sd20} shows that the optimal microstructure is given by  $A^{*,2}_N$ the $N$ sequential laminate with core $a_1I$ and matrix $a_2I$ with $\theta_A\in \mathcal{D}_{\delta_A}$,
and the corresponding $(N,N)$-sequential laminates $B^{\#,2}_{N,N}$ (cf.\eqref{ot16}) associated with $A^{*,2}_N$ under the condition $\omega_{B^\epsilon}\subset\omega_{A^\epsilon}$.
This construction is possible in $\Omega_{L2}$ since $\theta_B<\theta_A$. Moreover, we have 
$$J_2 = \int_{\Omega_{L2}}(A^{*,2}_N)^{-1}B^{\#,2}_{N,N}(A^{*,2}_N)^{-1}\widetilde{\sigma_{\theta_A}}\cdot\widetilde{\sigma_{\theta_A}}\ dx.$$
In the above we have used $\widetilde{u_{\theta_A}}$ ($\widetilde{\sigma}_{\theta_A}$ is the associated homogenized flux) for the homogenized limit of $\widetilde{u^\epsilon_{\theta_A}}$ which is the solution of \eqref{tx} with coefficient matrix $\widetilde{A^{\epsilon}}$ defined in $\Omega$ by  
\begin{align*}
\widetilde{A^{\epsilon}} &= A^{\epsilon,1} \mbox{ in }\Omega_{L1} \mbox{ \ (the $N$-sequential laminate microstructure with core $a_2I$ and matrix $a_1I$)},\\
                         &= A^{\epsilon,2} \mbox{ in }\Omega_{L2} \mbox{ \ (the $N$-sequential laminate microstructure with core $a_1I$ and matrix $a_2I$)}.
\end{align*}
The associated homogenized matrix $\widetilde{A^{*}}$ in $\Omega$  is also given by 
\begin{align*}
\widetilde{A^{*}} &= A^{*,1}_N \mbox{ in }\Omega_{L1} \mbox{ \ (the $N$-sequential laminate with core $a_2I$ and matrix $a_1I$)},\\
                         &= A^{*,2}_N \mbox{ in }\Omega_{L2} \mbox{ \ (the $N$-sequential laminate with core $a_1I$ and matrix $a_2I$)}.
\end{align*}
Thus,
\begin{equation*}J \geq \int_{\Omega_{L1}}B^{\#,1}_{N,N}\nabla \widetilde{u_{\theta_A}}\cdot\nabla\widetilde{u_{\theta_A}}\ dx + \int_{\Omega_{L2}}(A^{*,2}_N)^{-1}B^{\#,2}_{N,N}(A^{*,2}_N)^{-1}\widetilde{\sigma_{\theta_A}}\cdot\widetilde{\sigma_{\theta_A}}\ dx.\end{equation*}
Therefore, from \eqref{lb11} it follows that 
\begin{equation*} m^\prime = \int_{\Omega_{L1}}B^{\#,1}_{N,N}\nabla\widetilde{u_{\theta_A}}\cdot\nabla\widetilde{u_{\theta_A}}\ dx + \int_{\Omega_{L2}}B^{\#,2}_{N,N}\nabla\widetilde{u_{\theta_A}}\cdot\nabla\widetilde{u_{\theta_A}}\ dx.\end{equation*}
Thus we have solved the problem \eqref{ub20} completely describing the minimum value as well the minimizers $(A^{*},B^{\#})$ and the underlying microstructures for both $A^{*}$ and $B^{\#}$.
They are found among ``piecewise'' $N$-laminates. It is interesting to see minimizer admitting an interface defined by $\{\theta_A=\theta_B\}$ across which it interchanges core and matrix values. 
\hfill\qed
\begin{remark}\label{bs6}
We have solved the minimization problem \eqref{ub20} as an application of Theorem \ref{qw6} $(1),(2),(5)$. One may solve other types of optimization problems as an application of other parts of Theorem \ref{qw6}. 
\begin{equation*}
(i)\ \underset{\chi_{\omega_A}\in \mathcal{C}_{{\delta_A}}}{\mbox{inf }}\ \underset{\chi_{\omega_B}\in \mathcal{C}_{{\delta_B}}}{\mbox{sup }}J(\chi_{\omega_A},\chi_{\omega_B}),
\ \mbox{ or, }\ (ii)\ \underset{\chi_{\omega_B}\in \mathcal{C}_{{\delta_B}}}{\mbox{inf }}\ \underset{\chi_{\omega_A}\in \mathcal{C}_{{\delta_A}}}{\mbox{sup }}J(\chi_{\omega_A},\chi_{\omega_B}) 
\end{equation*}
or,
\begin{equation*}
(iii)\ \underset{(\chi_{\omega_A},\chi_{\omega_B})\in (\mathcal{C}_{{\delta_A}}\times \mathcal{C}_{{\delta_B}})}{\mbox{sup }}J(\chi_{\omega_A},\chi_{\omega_B}).
\end{equation*}
For instance, resolution of $(i)$ requires optimality of the region $(Li,Uj)$, $i,j=1,2$ described in Theorem \ref{qw6} $(3),(4)$. It is more complicated than \eqref{ub20}. For details, see \cite{TG-MV}. 
\hfill\qed\end{remark}
\paragraph{Acknowledgement :}
This work has been carried out within a project supported by the 
Airbus Group Corporate Foundation Chair ``Mathematics of Complex Systems'' 
established at Tata Institute Of fundamental Research (TIFR) - Centre for Applicable Mathematics.

\bibliographystyle{plain}
\bibliography{Master_bibfile}
\end{document}